\documentclass[preprint,1p]{amsart}

\usepackage{amsfonts}
\usepackage{amsmath}
\usepackage{amsthm}
\usepackage[utf8]{inputenc}
\usepackage[english]{babel}
\usepackage{epstopdf}
\usepackage{bmpsize}
\usepackage{epsfig}
\usepackage{hyperref}
\usepackage[english]{varioref}
\usepackage{amssymb}
\usepackage{multicol}
\usepackage{dcolumn}
\usepackage{geometry}
\usepackage{fancyhdr}
\usepackage[mathcal]{eucal}
\usepackage{mathrsfs}
\usepackage{color, colortbl}
\usepackage{microtype}
\usepackage{longtable}
\usepackage[toc,page]{appendix}
\usepackage{bm}
\usepackage{pifont}
\usepackage{fleqn}
\usepackage{graphicx}
\usepackage{txfonts}
\usepackage{subfig}
\usepackage[pagewise]{lineno}
\DeclareMathAlphabet{\mathbbm}{U}{bbm}{m}{n}
\ifx\figforTeXisloaded\relax \else\global\let\figforTeXisloaded=\relax\fi
\catcode`\@=11
\ifx\ctr@ln@m\undefined\else%
    \immediate\write16{*** Fig4TeX WARNING : \string\ctr@ln@m\space already defined.}\fi
\def\ctr@ln@m#1{\ifx#1\undefined\else%
    \immediate\write16{*** Fig4TeX WARNING : \string#1 already defined.}\fi}
\ctr@ln@m\ctr@ld@f
\def\ctr@ld@f#1#2{\ctr@ln@m#2#1#2}
\ctr@ld@f\def\ctr@ln@w#1#2{\ctr@ln@m#2\csname#1\endcsname#2}
{\catcode`\/=0 \catcode`/\=12 /ctr@ld@f/gdef/BS@{\}}
\ctr@ld@f\def\ctr@lcsn@m#1{\expandafter\ifx\csname#1\endcsname\relax\else%
    \immediate\write16{*** Fig4TeX WARNING : \BS@\expandafter\string#1\space already defined.}\fi}
\ctr@ld@f\edef\colonc@tcode{\the\catcode`\:}
\ctr@ld@f\edef\semicolonc@tcode{\the\catcode`\;}
\ctr@ld@f\def\t@stc@tcodech@nge{{\let\c@tcodech@nged=\z@%
    \ifnum\colonc@tcode=\the\catcode`\:\else\let\c@tcodech@nged=\@ne\fi%
    \ifnum\semicolonc@tcode=\the\catcode`\;\else\let\c@tcodech@nged=\@ne\fi%
    \ifx\c@tcodech@nged\@ne%
    \immediate\write16{}
    \immediate\write16{!!!=============================================================!!!}
    \immediate\write16{ Fig4TeX WARNING:}
    \immediate\write16{ The category code of some characters has been changed, which will}
    \immediate\write16{ result in an error (message "Runaway argument?").}
    \immediate\write16{ This probably comes from another package that changed the category}
    \immediate\write16{ code after Fig4TeX was loaded. If that proves to be exact, the}
    \immediate\write16{ solution is to exchange the loading commands on top of your file}
    \immediate\write16{ so that Fig4TeX is loaded last. For example, in LaTeX, we should}
    \immediate\write16{ say :}
    \immediate\write16{\BS@ usepackage[french]{babel}}
    \immediate\write16{\BS@ usepackage{fig4tex}}
    \immediate\write16{!!!=============================================================!!!}
    \immediate\write16{}
    \fi}}
\ctr@ld@f\def\FigforTeX{F\kern-.05em i\kern-.05em g\kern-.1em\raise-.14em\hbox{4}\kern-.19em\TeX}
\ctr@ld@f\def\W@rnmesoldA#1{\W@rnmesold}
\ctr@ld@f\def\W@rnmesoldAB#1(#2){\W@rnmesold}
\ctr@ld@f\def\W@rnmesold{%
    \immediate\write16{}
    \immediate\write16{!!!=============================================================!!!}
    \immediate\write16{ Fig4TeX WARNING:}
    \immediate\write16{ The file to be compiled is not compatible with the current version}
    \immediate\write16{ of Fig4TeX. To fix that, upgrade the source file (mainly change \BS@ ps*}
    \immediate\write16{ macros by \BS@ fig* macros), or use fig4tex184.tex instead (\BS@ input fig4tex184}
    \immediate\write16{ or \BS@ usepackage{fig4tex184}).}
    \immediate\write16{!!!=============================================================!!!}
    \immediate\write16{}}
\ctr@ln@m\psbeginfig\let\psbeginfig\W@rnmesoldA
\ctr@ln@m\psset\let\psset\W@rnmesoldAB
\ctr@ln@m\pssetdefault\let\pssetdefault\W@rnmesoldAB
\ctr@ln@m\pssetupdate\let\pssetupdate\W@rnmesoldA
\ctr@ln@w{newdimen}\epsil@n\epsil@n=0.00005pt
\ctr@ln@w{newdimen}\Cepsil@n\Cepsil@n=0.005pt
\ctr@ln@w{newdimen}\dcq@\dcq@=254pt
\ctr@ln@w{newdimen}\PI@\PI@=3.141592pt
\ctr@ln@w{newdimen}\DemiPI@deg\DemiPI@deg=90pt
\ctr@ln@w{newdimen}\PI@deg\PI@deg=180pt
\ctr@ln@w{newdimen}\DePI@deg\DePI@deg=360pt
\ctr@ld@f\chardef\t@n=10
\ctr@ld@f\chardef\c@nt=100
\ctr@ld@f\chardef\@lxxiv=74
\ctr@ld@f\chardef\@xci=91
\ctr@ld@f\mathchardef\@nMnCQn=9949
\ctr@ld@f\chardef\@vi=6
\ctr@ld@f\chardef\@xxx=30
\ctr@ld@f\chardef\@lvi=56
\ctr@ld@f\chardef\@@lxxi=71
\ctr@ld@f\chardef\@lxxxv=85
\ctr@ld@f\mathchardef\@@mmmmlxviii=4068
\ctr@ld@f\mathchardef\@ccclx=360
\ctr@ld@f\mathchardef\@dccxx=720
\ctr@ln@w{newcount}\p@rtent \ctr@ln@w{newcount}\f@ctech \ctr@ln@w{newcount}\result@tent
\ctr@ln@w{newdimen}\v@lmin \ctr@ln@w{newdimen}\v@lmax \ctr@ln@w{newdimen}\v@leur
\ctr@ln@w{newdimen}\result@t\ctr@ln@w{newdimen}\result@@t
\ctr@ln@w{newdimen}\mili@u \ctr@ln@w{newdimen}\c@rre \ctr@ln@w{newdimen}\delt@
\ctr@ld@f\def\degT@rd{0.017453 }  
\ctr@ld@f\def\rdT@deg{57.295779 } 
\ctr@ln@m\v@leurseule
{\catcode`p=12 \catcode`t=12 \gdef\v@leurseule#1pt{#1}}
\ctr@ld@f\def\repdecn@mb#1{\expandafter\v@leurseule\the#1\space}
\ctr@ld@f\def\arct@n#1(#2,#3){{\v@lmin=#2\v@lmax=#3%
    \maxim@m{\mili@u}{-\v@lmin}{\v@lmin}\maxim@m{\c@rre}{-\v@lmax}{\v@lmax}%
    \delt@=\mili@u\m@ech\mili@u%
    \ifdim\c@rre>\@nMnCQn\mili@u\divide\v@lmax\tw@\c@lATAN\v@leur(\z@,\v@lmax)
    \else%
    \maxim@m{\mili@u}{-\v@lmin}{\v@lmin}\maxim@m{\c@rre}{-\v@lmax}{\v@lmax}%
    \m@ech\c@rre%
    \ifdim\mili@u>\@nMnCQn\c@rre\divide\v@lmin\tw@
    \maxim@m{\mili@u}{-\v@lmin}{\v@lmin}\c@lATAN\v@leur(\mili@u,\z@)%
    \else\c@lATAN\v@leur(\delt@,\v@lmax)\fi\fi%
    \ifdim\v@lmin<\z@\v@leur=-\v@leur\ifdim\v@lmax<\z@\advance\v@leur-\PI@%
    \else\advance\v@leur\PI@\fi\fi%
    \global\result@t=\v@leur}#1=\result@t}
\ctr@ld@f\def\m@ech#1{\ifdim#1>1.646pt\divide\mili@u\t@n\divide\c@rre\t@n\m@ech#1\fi}
\ctr@ld@f\def\c@lATAN#1(#2,#3){{\v@lmin=#2\v@lmax=#3\v@leur=\z@\delt@=\tw@ pt%
    \un@iter{0.785398}{\v@lmax<}%
    \un@iter{0.463648}{\v@lmax<}%
    \un@iter{0.244979}{\v@lmax<}%
    \un@iter{0.124355}{\v@lmax<}%
    \un@iter{0.062419}{\v@lmax<}%
    \un@iter{0.031240}{\v@lmax<}%
    \un@iter{0.015624}{\v@lmax<}%
    \un@iter{0.007812}{\v@lmax<}%
    \un@iter{0.003906}{\v@lmax<}%
    \un@iter{0.001953}{\v@lmax<}%
    \un@iter{0.000976}{\v@lmax<}%
    \un@iter{0.000488}{\v@lmax<}%
    \un@iter{0.000244}{\v@lmax<}%
    \un@iter{0.000122}{\v@lmax<}%
    \un@iter{0.000061}{\v@lmax<}%
    \un@iter{0.000030}{\v@lmax<}%
    \un@iter{0.000015}{\v@lmax<}%
    \global\result@t=\v@leur}#1=\result@t}
\ctr@ld@f\def\un@iter#1#2{%
    \divide\delt@\tw@\edef\dpmn@{\repdecn@mb{\delt@}}%
    \mili@u=\v@lmin%
    \ifdim#2\z@%
      \advance\v@lmin-\dpmn@\v@lmax\advance\v@lmax\dpmn@\mili@u%
      \advance\v@leur-#1pt%
    \else%
      \advance\v@lmin\dpmn@\v@lmax\advance\v@lmax-\dpmn@\mili@u%
      \advance\v@leur#1pt%
    \fi}
\ctr@ld@f\def\c@ssin#1#2#3{\expandafter\ifx\csname COS@\number#3\endcsname\relax\c@lCS{#3pt}%
    \expandafter\xdef\csname COS@\number#3\endcsname{\repdecn@mb\result@t}%
    \expandafter\xdef\csname SIN@\number#3\endcsname{\repdecn@mb\result@@t}\fi%
    \edef#1{\csname COS@\number#3\endcsname}\edef#2{\csname SIN@\number#3\endcsname}}
\ctr@ld@f\def\c@lCS#1{{\mili@u=#1\p@rtent=\@ne%
    \relax\ifdim\mili@u<\z@\red@ng<-\else\red@ng>+\fi\f@ctech=\p@rtent%
    \relax\ifdim\mili@u<\z@\mili@u=-\mili@u\f@ctech=-\f@ctech\fi\c@@lCS}}
\ctr@ld@f\def\c@@lCS{\v@lmin=\mili@u\c@rre=-\mili@u\advance\c@rre\DemiPI@deg\v@lmax=\c@rre%
    \mili@u\@@lxxi\mili@u\divide\mili@u\@@mmmmlxviii%
    \edef\v@larg{\repdecn@mb{\mili@u}}\mili@u=-\v@larg\mili@u%
    \edef\v@lmxde{\repdecn@mb{\mili@u}}%
    \c@rre\@@lxxi\c@rre\divide\c@rre\@@mmmmlxviii%
    \edef\v@largC{\repdecn@mb{\c@rre}}\c@rre=-\v@largC\c@rre%
    \edef\v@lmxdeC{\repdecn@mb{\c@rre}}%
    \fctc@s\mili@u\v@lmin\global\result@t\p@rtent\v@leur%
    \let\t@mp=\v@larg\let\v@larg=\v@largC\let\v@largC=\t@mp%
    \let\t@mp=\v@lmxde\let\v@lmxde=\v@lmxdeC\let\v@lmxdeC=\t@mp%
    \fctc@s\c@rre\v@lmax\global\result@@t\f@ctech\v@leur}
\ctr@ld@f\def\fctc@s#1#2{\v@leur=#1\relax\ifdim#2<\@lxxxv\p@\cosser@h\else\sinser@t\fi}
\ctr@ld@f\def\cosser@h{\advance\v@leur\@lvi\p@\divide\v@leur\@lvi%
    \v@leur=\v@lmxde\v@leur\advance\v@leur\@xxx\p@%
    \v@leur=\v@lmxde\v@leur\advance\v@leur\@ccclx\p@%
    \v@leur=\v@lmxde\v@leur\advance\v@leur\@dccxx\p@\divide\v@leur\@dccxx}
\ctr@ld@f\def\sinser@t{\v@leur=\v@lmxdeC\p@\advance\v@leur\@vi\p@%
    \v@leur=\v@largC\v@leur\divide\v@leur\@vi}
\ctr@ld@f\def\red@ng#1#2{\relax\ifdim\mili@u#1#2\DemiPI@deg\advance\mili@u#2-\PI@deg%
    \p@rtent=-\p@rtent\red@ng#1#2\fi}
\ctr@ld@f\def\pr@c@lCS#1#2#3{\ctr@lcsn@m{COS@\number#3 }%
    \expandafter\xdef\csname COS@\number#3\endcsname{#1}%
    \expandafter\xdef\csname SIN@\number#3\endcsname{#2}}
\pr@c@lCS{1}{0}{0}
\pr@c@lCS{0.7071}{0.7071}{45}\pr@c@lCS{0.7071}{-0.7071}{-45}
\pr@c@lCS{0}{1}{90}          \pr@c@lCS{0}{-1}{-90}
\pr@c@lCS{-1}{0}{180}        \pr@c@lCS{-1}{0}{-180}
\pr@c@lCS{0}{-1}{270}        \pr@c@lCS{0}{1}{-270}
\ctr@ld@f\def\invers@#1#2{{\v@leur=#2\maxim@m{\v@lmax}{-\v@leur}{\v@leur}%
    \f@ctech=\@ne\m@inv@rs%
    \multiply\v@leur\f@ctech\edef\v@lv@leur{\repdecn@mb{\v@leur}}%
    \p@rtentiere{\p@rtent}{\v@leur}\v@lmin=\p@\divide\v@lmin\p@rtent%
    \inv@rs@\multiply\v@lmax\f@ctech\global\result@t=\v@lmax}#1=\result@t}
\ctr@ld@f\def\m@inv@rs{\ifdim\v@lmax<\p@\multiply\v@lmax\t@n\multiply\f@ctech\t@n\m@inv@rs\fi}
\ctr@ld@f\def\inv@rs@{\v@lmax=-\v@lmin\v@lmax=\v@lv@leur\v@lmax%
    \advance\v@lmax\tw@ pt\v@lmax=\repdecn@mb{\v@lmin}\v@lmax%
    \delt@=\v@lmax\advance\delt@-\v@lmin\ifdim\delt@<\z@\delt@=-\delt@\fi%
    \ifdim\delt@>\epsil@n\v@lmin=\v@lmax\inv@rs@\fi}
\ctr@ld@f\def\minim@m#1#2#3{\relax\ifdim#2<#3#1=#2\else#1=#3\fi}
\ctr@ld@f\def\maxim@m#1#2#3{\relax\ifdim#2>#3#1=#2\else#1=#3\fi}
\ctr@ld@f\def\p@rtentiere#1#2{#1=#2\divide#1by65536 }
\ctr@ld@f\def\r@undint#1#2{{\v@leur=#2\divide\v@leur\t@n\p@rtentiere{\p@rtent}{\v@leur}%
    \v@leur=\p@rtent pt\global\result@t=\t@n\v@leur}#1=\result@t}
\ctr@ld@f\def\sqrt@#1#2{{\v@leur=#2%
    \minim@m{\v@lmin}{\p@}{\v@leur}\maxim@m{\v@lmax}{\p@}{\v@leur}%
    \f@ctech=\@ne\m@sqrt@\sqrt@@%
    \mili@u=\v@lmin\advance\mili@u\v@lmax\divide\mili@u\tw@\multiply\mili@u\f@ctech%
    \global\result@t=\mili@u}#1=\result@t}
\ctr@ld@f\def\m@sqrt@{\ifdim\v@leur>\dcq@\divide\v@leur\c@nt\v@lmax=\v@leur%
    \multiply\f@ctech\t@n\m@sqrt@\fi}
\ctr@ld@f\def\sqrt@@{\mili@u=\v@lmin\advance\mili@u\v@lmax\divide\mili@u\tw@%
    \c@rre=\repdecn@mb{\mili@u}\mili@u%
    \ifdim\c@rre<\v@leur\v@lmin=\mili@u\else\v@lmax=\mili@u\fi%
    \delt@=\v@lmax\advance\delt@-\v@lmin\ifdim\delt@>\epsil@n\sqrt@@\fi}
\ctr@ld@f\def\extrairelepremi@r#1\de#2{\expandafter\lepremi@r#2@#1#2}
\ctr@ld@f\def\lepremi@r#1,#2@#3#4{\def#3{#1}\def#4{#2}\ignorespaces}
\ctr@ld@f\def\@cfor#1:=#2\do#3{%
  \edef\@fortemp{#2}%
  \ifx\@fortemp\empty\else\@cforloop#2,\@nil,\@nil\@@#1{#3}\fi}
\ctr@ln@m\@nextwhile
\ctr@ld@f\def\@cforloop#1,#2\@@#3#4{%
  \def#3{#1}%
  \ifx#3\Fig@nnil\let\@nextwhile=\Fig@fornoop\else#4\relax\let\@nextwhile=\@cforloop\fi%
  \@nextwhile#2\@@#3{#4}}

\ctr@ld@f\def\@ecfor#1:=#2\do#3{%
  \def\@@cfor{\@cfor#1:=}%
  \edef\@@@cfor{#2}%
  \expandafter\@@cfor\@@@cfor\do{#3}}
\ctr@ld@f\def\Fig@nnil{\@nil}
\ctr@ld@f\def\Fig@fornoop#1\@@#2#3{}
\ctr@ln@m\list@@rg
\ctr@ld@f\def\trtlis@rg#1#2{\def\list@@rg{#1}%
    \@ecfor\p@rv@l:=\list@@rg\do{\expandafter#2\p@rv@l|}}
\ctr@ld@f\def\trtlis@rgtok#1{\let@xte={}\let\n@xt\addt@t@xt\addt@t@xt #1}
\ctr@ln@m\M@cro
\ctr@ln@m\n@xt
\ctr@ld@f\def\addt@t@xt#1{\if#1|\let\n@xt\relax\else%
    \if#1,\expandafter\M@cro\the\let@xte|\let@xte={}%
    \else\let@xte=\expandafter{\the\let@xte #1}\fi\fi\n@xt}
\ctr@ln@w{newbox}\b@xvisu
\ctr@ln@w{newtoks}\let@xte
\ctr@ln@w{newif}\ifitis@K
\ctr@ln@w{newcount}\s@mme
\ctr@ln@w{newcount}\l@mbd@un \ctr@ln@w{newcount}\l@mbd@de
\ctr@ln@w{newcount}\superc@ntr@l\superc@ntr@l=\@ne        
\ctr@ln@w{newcount}\typec@ntr@l\typec@ntr@l=\superc@ntr@l 
\ctr@ln@w{newdimen}\v@lX  \ctr@ln@w{newdimen}\v@lY  \ctr@ln@w{newdimen}\v@lZ
\ctr@ln@w{newdimen}\v@lXa \ctr@ln@w{newdimen}\v@lYa \ctr@ln@w{newdimen}\v@lZa
\ctr@ln@w{newdimen}\unit@\unit@=\p@ 
\ctr@ld@f\def\unit@util{pt}
\ctr@ld@f\def\ptT@ptps{0.996264 }
\ctr@ld@f\def\ptpsT@pt{1.00375 }
\ctr@ld@f\def\ptT@unit@{1} 
\ctr@ld@f\def\setunit@#1{\def\unit@util{#1}\setunit@@#1:\invers@{\result@t}{\unit@}%
    \edef\ptT@unit@{\repdecn@mb\result@t}}
\ctr@ld@f\def\setunit@@#1#2:{\ifcat#1a\unit@=\@ne#1#2\else\unit@=#1#2\fi}
\ctr@ld@f\def\d@fm@cdim#1#2{{\v@leur=#2\v@leur=\ptT@unit@\v@leur\xdef#1{\repdecn@mb\v@leur}}}
\ctr@ln@w{newif}\ifBdingB@x\BdingB@xtrue
\ctr@ln@w{newdimen}\c@@rdXmin \ctr@ln@w{newdimen}\c@@rdYmin  
\ctr@ln@w{newdimen}\c@@rdXmax \ctr@ln@w{newdimen}\c@@rdYmax
\ctr@ld@f\def\b@undb@x#1#2{\ifBdingB@x%
    \relax\ifdim#1<\c@@rdXmin\global\c@@rdXmin=#1\fi%
    \relax\ifdim#2<\c@@rdYmin\global\c@@rdYmin=#2\fi%
    \relax\ifdim#1>\c@@rdXmax\global\c@@rdXmax=#1\fi%
    \relax\ifdim#2>\c@@rdYmax\global\c@@rdYmax=#2\fi\fi}
\ctr@ld@f\def\b@undb@xP#1{{\Figg@tXY{#1}\b@undb@x{\v@lX}{\v@lY}}}
\ctr@ld@f\def\ellBB@x#1;#2,#3(#4,#5,#6){{\s@uvc@ntr@l\et@tellBB@x%
    \setc@ntr@l{2}\figptell-2::#1;#2,#3(#4,#6)\b@undb@xP{-2}%
    \figptell-2::#1;#2,#3(#5,#6)\b@undb@xP{-2}%
    \c@ssin{\C@}{\S@}{#6}\v@lmin=\C@ pt\v@lmax=\S@ pt%
    \mili@u=#3\v@lmin\delt@=#2\v@lmax\arct@n\v@leur(\delt@,\mili@u)%
    \mili@u=-#3\v@lmax\delt@=#2\v@lmin\arct@n\c@rre(\delt@,\mili@u)%
    \v@leur=\rdT@deg\v@leur\advance\v@leur-\DePI@deg%
    \c@rre=\rdT@deg\c@rre\advance\c@rre-\DePI@deg%
    \v@lmin=#4pt\v@lmax=#5pt%
    \loop\ifdim\v@leur<\v@lmax\ifdim\v@leur>\v@lmin%
    \edef\@ngle{\repdecn@mb\v@leur}\figptell-2::#1;#2,#3(\@ngle,#6)%
    \b@undb@xP{-2}\fi\advance\v@leur\PI@deg\repeat%
    \loop\ifdim\c@rre<\v@lmax\ifdim\c@rre>\v@lmin%
    \edef\@ngle{\repdecn@mb\c@rre}\figptell-2::#1;#2,#3(\@ngle,#6)%
    \b@undb@xP{-2}\fi\advance\c@rre\PI@deg\repeat%
    \resetc@ntr@l\et@tellBB@x}\ignorespaces}
\ctr@ld@f\def\initb@undb@x{\c@@rdXmin=\maxdimen\c@@rdYmin=\maxdimen%
    \c@@rdXmax=-\maxdimen\c@@rdYmax=-\maxdimen}
\ctr@ld@f\def\c@ntr@lnum#1{%
    \relax\ifnum\typec@ntr@l=\@ne%
    \ifnum#1<\z@%
    \immediate\write16{*** Forbidden point number (#1). Abort.}\end\fi\fi%
    \set@bjc@de{#1}}
\ctr@ln@m\objc@de
\ctr@ld@f\def\set@bjc@de#1{\edef\objc@de{@BJ\ifnum#1<\z@ M\romannumeral-#1\else\romannumeral#1\fi}}
\s@mme=\m@ne\loop\ifnum\s@mme>-19
  \set@bjc@de{\s@mme}\ctr@lcsn@m\objc@de\ctr@lcsn@m{\objc@de T}
\advance\s@mme\m@ne\repeat
\s@mme=\@ne\loop\ifnum\s@mme<6
  \set@bjc@de{\s@mme}\ctr@lcsn@m\objc@de\ctr@lcsn@m{\objc@de T}
\advance\s@mme\@ne\repeat
\ctr@ld@f\def\setc@ntr@l#1{\ifnum\superc@ntr@l>#1\typec@ntr@l=\superc@ntr@l%
    \else\typec@ntr@l=#1\fi}
\ctr@ld@f\def\resetc@ntr@l#1{\global\superc@ntr@l=#1\setc@ntr@l{#1}}
\ctr@ld@f\def\s@uvc@ntr@l#1{\edef#1{\the\superc@ntr@l}}
\ctr@ln@m\c@lproscal
\ctr@ld@f\def\c@lproscalDD#1[#2,#3]{{\Figg@tXY{#2}%
    \edef\Xu@{\repdecn@mb{\v@lX}}\edef\Yu@{\repdecn@mb{\v@lY}}\Figg@tXY{#3}%
    \global\result@t=\Xu@\v@lX\global\advance\result@t\Yu@\v@lY}#1=\result@t}
\ctr@ld@f\def\c@lproscalTD#1[#2,#3]{{\Figg@tXY{#2}\edef\Xu@{\repdecn@mb{\v@lX}}%
    \edef\Yu@{\repdecn@mb{\v@lY}}\edef\Zu@{\repdecn@mb{\v@lZ}}%
    \Figg@tXY{#3}\global\result@t=\Xu@\v@lX\global\advance\result@t\Yu@\v@lY%
    \global\advance\result@t\Zu@\v@lZ}#1=\result@t}
\ctr@ld@f\def\c@lprovec#1{%
    \det@rmC\v@lZa(\v@lX,\v@lY,\v@lmin,\v@lmax)%
    \det@rmC\v@lXa(\v@lY,\v@lZ,\v@lmax,\v@leur)%
    \det@rmC\v@lYa(\v@lZ,\v@lX,\v@leur,\v@lmin)%
    \Figv@ctCreg#1(\v@lXa,\v@lYa,\v@lZa)}
\ctr@ld@f\def\det@rm#1[#2,#3]{{\Figg@tXY{#2}\Figg@tXYa{#3}%
    \delt@=\repdecn@mb{\v@lX}\v@lYa\advance\delt@-\repdecn@mb{\v@lY}\v@lXa%
    \global\result@t=\delt@}#1=\result@t}
\ctr@ld@f\def\det@rmC#1(#2,#3,#4,#5){{\global\result@t=\repdecn@mb{#2}#5%
    \global\advance\result@t-\repdecn@mb{#3}#4}#1=\result@t}
\ctr@ld@f\def\getredf@ctDD#1(#2,#3){{\maxim@m{\v@lXa}{-#2}{#2}\maxim@m{\v@lYa}{-#3}{#3}%
    \maxim@m{\v@lXa}{\v@lXa}{\v@lYa}
    \ifdim\v@lXa>\@xci pt\divide\v@lXa\@xci%
    \p@rtentiere{\p@rtent}{\v@lXa}\advance\p@rtent\@ne\else\p@rtent=\@ne\fi%
    \global\result@tent=\p@rtent}#1=\result@tent\ignorespaces}
\ctr@ld@f\def\getredf@ctTD#1(#2,#3,#4){{\maxim@m{\v@lXa}{-#2}{#2}\maxim@m{\v@lYa}{-#3}{#3}%
    \maxim@m{\v@lZa}{-#4}{#4}\maxim@m{\v@lXa}{\v@lXa}{\v@lYa}%
    \maxim@m{\v@lXa}{\v@lXa}{\v@lZa}
    \ifdim\v@lXa>\@lxxiv pt\divide\v@lXa\@lxxiv%
    \p@rtentiere{\p@rtent}{\v@lXa}\advance\p@rtent\@ne\else\p@rtent=\@ne\fi%
    \global\result@tent=\p@rtent}#1=\result@tent\ignorespaces}
\ctr@ln@m\getredf@ctB
\ctr@ld@f\def\getredf@ctBDD#1{\getredf@ctDD#1(\v@lX,\v@lY)}
\ctr@ld@f\def\getredf@ctBTD#1{\getredf@ctTD#1(\v@lX,\v@lY,\v@lZ)}
\ctr@ld@f\def\FigptintercircB@zDD#1:#2:#3,#4[#5,#6,#7,#8]{{\s@uvc@ntr@l\et@tfigptintercircB@zDD%
    \setc@ntr@l{2}\figvectPDD-1[#5,#8]\Figg@tXY{-1}\getredf@ctDD\f@ctech(\v@lX,\v@lY)%
    \mili@u=#4\unit@\divide\mili@u\f@ctech\c@rre=\repdecn@mb{\mili@u}\mili@u%
    \figptBezierDD-5::#3[#5,#6,#7,#8]%
    \v@lmin=#3\p@\v@lmax=\v@lmin\advance\v@lmax0.1\p@%
    \loop\edef\T@{\repdecn@mb{\v@lmax}}\figptBezierDD-2::\T@[#5,#6,#7,#8]%
    \figvectPDD-1[-5,-2]\n@rmeucCDD{\delt@}{-1}\ifdim\delt@<\c@rre\v@lmin=\v@lmax%
    \advance\v@lmax0.1\p@\repeat%
    \loop\mili@u=\v@lmin\advance\mili@u\v@lmax%
    \divide\mili@u\tw@\edef\T@{\repdecn@mb{\mili@u}}\figptBezierDD-2::\T@[#5,#6,#7,#8]%
    \figvectPDD-1[-5,-2]\n@rmeucCDD{\delt@}{-1}\ifdim\delt@>\c@rre\v@lmax=\mili@u%
    \else\v@lmin=\mili@u\fi\v@leur=\v@lmax\advance\v@leur-\v@lmin%
    \ifdim\v@leur>\epsil@n\repeat\figptcopyDD#1:#2/-2/%
    \resetc@ntr@l\et@tfigptintercircB@zDD}\ignorespaces}
\ctr@ln@m\figptinterlines
\ctr@ld@f\def\inters@cDD#1:#2[#3,#4;#5,#6]{{\s@uvc@ntr@l\et@tinters@cDD%
    \setc@ntr@l{2}\vecunit@{-1}{#4}\vecunit@{-2}{#6}%
    \Figg@tXY{-1}\setc@ntr@l{1}\Figg@tXYa{#3}%
    \edef\A@{\repdecn@mb{\v@lX}}\edef\B@{\repdecn@mb{\v@lY}}%
    \v@lmin=\B@\v@lXa\advance\v@lmin-\A@\v@lYa%
    \Figg@tXYa{#5}\setc@ntr@l{2}\Figg@tXY{-2}%
    \edef\C@{\repdecn@mb{\v@lX}}\edef\D@{\repdecn@mb{\v@lY}}%
    \v@lmax=\D@\v@lXa\advance\v@lmax-\C@\v@lYa%
    \delt@=\A@\v@lY\advance\delt@-\B@\v@lX%
    \invers@{\v@leur}{\delt@}\edef\v@ldelta{\repdecn@mb{\v@leur}}%
    \v@lXa=\A@\v@lmax\advance\v@lXa-\C@\v@lmin%
    \v@lYa=\B@\v@lmax\advance\v@lYa-\D@\v@lmin%
    \v@lXa=\v@ldelta\v@lXa\v@lYa=\v@ldelta\v@lYa%
    \setc@ntr@l{1}\Figp@intregDD#1:{#2}(\v@lXa,\v@lYa)%
    \resetc@ntr@l\et@tinters@cDD}\ignorespaces}
\ctr@ld@f\def\inters@cTD#1:#2[#3,#4;#5,#6]{{\s@uvc@ntr@l\et@tinters@cTD%
    \setc@ntr@l{2}\figvectNVTD-1[#4,#6]\figvectNVTD-2[#6,-1]\figvectPTD-1[#3,#5]%
    \r@pPSTD\v@leur[-2,-1,#4]\edef\v@lcoef{\repdecn@mb{\v@leur}}%
    \figpttraTD#1:{#2}=#3/\v@lcoef,#4/\resetc@ntr@l\et@tinters@cTD}\ignorespaces}
\ctr@ld@f\def\r@pPSTD#1[#2,#3,#4]{{\Figg@tXY{#2}\edef\Xu@{\repdecn@mb{\v@lX}}%
    \edef\Yu@{\repdecn@mb{\v@lY}}\edef\Zu@{\repdecn@mb{\v@lZ}}%
    \Figg@tXY{#3}\v@lmin=\Xu@\v@lX\advance\v@lmin\Yu@\v@lY\advance\v@lmin\Zu@\v@lZ%
    \Figg@tXY{#4}\v@lmax=\Xu@\v@lX\advance\v@lmax\Yu@\v@lY\advance\v@lmax\Zu@\v@lZ%
    \invers@{\v@leur}{\v@lmax}\global\result@t=\repdecn@mb{\v@leur}\v@lmin}%
    #1=\result@t}
\ctr@ln@m\n@rminf
\ctr@ld@f\def\n@rminfDD#1#2{{\Figg@tXY{#2}\maxim@m{\v@lX}{\v@lX}{-\v@lX}%
    \maxim@m{\v@lY}{\v@lY}{-\v@lY}\maxim@m{\global\result@t}{\v@lX}{\v@lY}}%
    #1=\result@t}
\ctr@ld@f\def\n@rminfTD#1#2{{\Figg@tXY{#2}\maxim@m{\v@lX}{\v@lX}{-\v@lX}%
    \maxim@m{\v@lY}{\v@lY}{-\v@lY}\maxim@m{\v@lZ}{\v@lZ}{-\v@lZ}%
    \maxim@m{\v@lX}{\v@lX}{\v@lY}\maxim@m{\global\result@t}{\v@lX}{\v@lZ}}%
    #1=\result@t}
\ctr@ln@m\n@rmeucC
\ctr@ld@f\def\n@rmeucCDD#1#2{\Figg@tXY{#2}\divide\v@lX\f@ctech\divide\v@lY\f@ctech%
    #1=\repdecn@mb{\v@lX}\v@lX\v@lX=\repdecn@mb{\v@lY}\v@lY\advance#1\v@lX}
\ctr@ld@f\def\n@rmeucCTD#1#2{\Figg@tXY{#2}%
    \divide\v@lX\f@ctech\divide\v@lY\f@ctech\divide\v@lZ\f@ctech%
    #1=\repdecn@mb{\v@lX}\v@lX\v@lX=\repdecn@mb{\v@lY}\v@lY\advance#1\v@lX%
    \v@lX=\repdecn@mb{\v@lZ}\v@lZ\advance#1\v@lX}
\ctr@ln@m\n@rmeucSV
\ctr@ld@f\def\n@rmeucSVDD#1#2{{\Figg@tXY{#2}%
    \v@lXa=\repdecn@mb{\v@lX}\v@lX\v@lYa=\repdecn@mb{\v@lY}\v@lY%
    \advance\v@lXa\v@lYa\sqrt@{\global\result@t}{\v@lXa}}#1=\result@t}
\ctr@ld@f\def\n@rmeucSVTD#1#2{{\Figg@tXY{#2}\v@lXa=\repdecn@mb{\v@lX}\v@lX%
    \v@lYa=\repdecn@mb{\v@lY}\v@lY\v@lZa=\repdecn@mb{\v@lZ}\v@lZ%
    \advance\v@lXa\v@lYa\advance\v@lXa\v@lZa\sqrt@{\global\result@t}{\v@lXa}}#1=\result@t}
\ctr@ln@m\n@rmeuc
\ctr@ld@f\def\n@rmeucDD#1#2{{\Figg@tXY{#2}\getredf@ctDD\f@ctech(\v@lX,\v@lY)%
    \divide\v@lX\f@ctech\divide\v@lY\f@ctech%
    \v@lXa=\repdecn@mb{\v@lX}\v@lX\v@lYa=\repdecn@mb{\v@lY}\v@lY%
    \advance\v@lXa\v@lYa\sqrt@{\global\result@t}{\v@lXa}%
    \global\multiply\result@t\f@ctech}#1=\result@t}
\ctr@ld@f\def\n@rmeucTD#1#2{{\Figg@tXY{#2}\getredf@ctTD\f@ctech(\v@lX,\v@lY,\v@lZ)%
    \divide\v@lX\f@ctech\divide\v@lY\f@ctech\divide\v@lZ\f@ctech%
    \v@lXa=\repdecn@mb{\v@lX}\v@lX%
    \v@lYa=\repdecn@mb{\v@lY}\v@lY\v@lZa=\repdecn@mb{\v@lZ}\v@lZ%
    \advance\v@lXa\v@lYa\advance\v@lXa\v@lZa\sqrt@{\global\result@t}{\v@lXa}%
    \global\multiply\result@t\f@ctech}#1=\result@t}
\ctr@ln@m\vecunit@
\ctr@ld@f\def\vecunit@DD#1#2{{\Figg@tXY{#2}\getredf@ctDD\f@ctech(\v@lX,\v@lY)%
    \divide\v@lX\f@ctech\divide\v@lY\f@ctech%
    \Figv@ctCreg#1(\v@lX,\v@lY)\n@rmeucSV{\v@lYa}{#1}%
    \invers@{\v@lXa}{\v@lYa}\edef\v@lv@lXa{\repdecn@mb{\v@lXa}}%
    \v@lX=\v@lv@lXa\v@lX\v@lY=\v@lv@lXa\v@lY%
    \Figv@ctCreg#1(\v@lX,\v@lY)\multiply\v@lYa\f@ctech\global\result@t=\v@lYa}}
\ctr@ld@f\def\vecunit@TD#1#2{{\Figg@tXY{#2}\getredf@ctTD\f@ctech(\v@lX,\v@lY,\v@lZ)%
    \divide\v@lX\f@ctech\divide\v@lY\f@ctech\divide\v@lZ\f@ctech%
    \Figv@ctCreg#1(\v@lX,\v@lY,\v@lZ)\n@rmeucSV{\v@lYa}{#1}%
    \invers@{\v@lXa}{\v@lYa}\edef\v@lv@lXa{\repdecn@mb{\v@lXa}}%
    \v@lX=\v@lv@lXa\v@lX\v@lY=\v@lv@lXa\v@lY\v@lZ=\v@lv@lXa\v@lZ%
    \Figv@ctCreg#1(\v@lX,\v@lY,\v@lZ)\multiply\v@lYa\f@ctech\global\result@t=\v@lYa}}
\ctr@ld@f\def\vecunitC@TD[#1,#2]{\Figg@tXYa{#1}\Figg@tXY{#2}%
    \advance\v@lX-\v@lXa\advance\v@lY-\v@lYa\advance\v@lZ-\v@lZa\c@lvecunitTD}
\ctr@ld@f\def\vecunitCV@TD#1{\Figg@tXY{#1}\c@lvecunitTD}
\ctr@ld@f\def\c@lvecunitTD{\getredf@ctTD\f@ctech(\v@lX,\v@lY,\v@lZ)%
    \divide\v@lX\f@ctech\divide\v@lY\f@ctech\divide\v@lZ\f@ctech%
    \v@lXa=\repdecn@mb{\v@lX}\v@lX%
    \v@lYa=\repdecn@mb{\v@lY}\v@lY\v@lZa=\repdecn@mb{\v@lZ}\v@lZ%
    \advance\v@lXa\v@lYa\advance\v@lXa\v@lZa\sqrt@{\v@lYa}{\v@lXa}%
    \invers@{\v@lXa}{\v@lYa}\edef\v@lv@lXa{\repdecn@mb{\v@lXa}}%
    \v@lX=\v@lv@lXa\v@lX\v@lY=\v@lv@lXa\v@lY\v@lZ=\v@lv@lXa\v@lZ}
\ctr@ln@m\figgetangle
\ctr@ld@f\def\figgetangleDD#1[#2,#3,#4]{\ifGR@cri{\s@uvc@ntr@l\et@tfiggetangleDD\setc@ntr@l{2}%
    \figvectPDD-1[#2,#3]\figvectPDD-2[#2,#4]\vecunit@{-1}{-1}%
    \c@lproscalDD\delt@[-2,-1]\figvectNVDD-1[-1]\c@lproscalDD\v@leur[-2,-1]%
    \arct@n\v@lmax(\delt@,\v@leur)\v@lmax=\rdT@deg\v@lmax%
    \ifdim\v@lmax<\z@\advance\v@lmax\DePI@deg\fi\xdef#1{\repdecn@mb{\v@lmax}}%
    \resetc@ntr@l\et@tfiggetangleDD}\ignorespaces\fi}
\ctr@ld@f\def\figgetangleTD#1[#2,#3,#4,#5]{\ifGR@cri{\s@uvc@ntr@l\et@tfiggetangleTD\setc@ntr@l{2}%
    \figvectPTD-1[#2,#3]\figvectPTD-2[#2,#5]\figvectNVTD-3[-1,-2]%
    \figvectPTD-2[#2,#4]\figvectNVTD-4[-3,-1]%
    \vecunit@{-1}{-1}\c@lproscalTD\delt@[-2,-1]\c@lproscalTD\v@leur[-2,-4]%
    \arct@n\v@lmax(\delt@,\v@leur)\v@lmax=\rdT@deg\v@lmax%
    \ifdim\v@lmax<\z@\advance\v@lmax\DePI@deg\fi\xdef#1{\repdecn@mb{\v@lmax}}%
    \resetc@ntr@l\et@tfiggetangleTD}\ignorespaces\fi}    
\ctr@ld@f\def\figgetdist#1[#2,#3]{\ifGR@cri{\s@uvc@ntr@l\et@tfiggetdist\setc@ntr@l{2}%
    \figvectP-1[#2,#3]\n@rmeuc{\v@lX}{-1}\v@lX=\ptT@unit@\v@lX\xdef#1{\repdecn@mb{\v@lX}}%
    \resetc@ntr@l\et@tfiggetdist}\ignorespaces\fi}
\ctr@ld@f\def\figget#1=#2[#3]{\keln@mun#1|%
    \def\n@mref{a}\ifx\l@debut\n@mref\figgetangle#2[#3]\else
    \def\n@mref{d}\ifx\l@debut\n@mref\figgetdist#2[#3]\else
    \W@rnmeskwd{figget}{#1}\fi\fi\ignorespaces}
\ctr@ld@f\def\Figg@tT#1{\c@ntr@lnum{#1}%
    {\expandafter\expandafter\expandafter\extr@ctT\csname\objc@de\endcsname:%
     \ifnum\B@@ltxt=\z@\ptn@me{#1}\else\csname\objc@de T\endcsname\fi}}
\ctr@ld@f\def\extr@ctT#1,#2,#3/#4:{\def\B@@ltxt{#3}}
\ctr@ld@f\def\Figg@tXY#1{\c@ntr@lnum{#1}%
    \expandafter\expandafter\expandafter\extr@ctC\csname\objc@de\endcsname:}
\ctr@ln@m\extr@ctC
\ctr@ld@f\def\extr@ctCDD#1/#2,#3,#4:{\v@lX=#2\v@lY=#3}
\ctr@ld@f\def\extr@ctCTD#1/#2,#3,#4:{\v@lX=#2\v@lY=#3\v@lZ=#4}
\ctr@ld@f\def\Figg@tXYa#1{\c@ntr@lnum{#1}%
    \expandafter\expandafter\expandafter\extr@ctCa\csname\objc@de\endcsname:}
\ctr@ln@m\extr@ctCa
\ctr@ld@f\def\extr@ctCaDD#1/#2,#3,#4:{\v@lXa=#2\v@lYa=#3}
\ctr@ld@f\def\extr@ctCaTD#1/#2,#3,#4:{\v@lXa=#2\v@lYa=#3\v@lZa=#4}
\ctr@ln@m\t@xt@
\ctr@ld@f\def\figinit#1{\t@stc@tcodech@nge\initpr@lim\Figinit@#1,:\initpss@ttings\ignorespaces}
\ctr@ld@f\def\Figinit@#1,#2:{\setunit@{#1}\def\t@xt@{#2}\ifx\t@xt@\empty\else\Figinit@@#2:\fi}
\ctr@ld@f\def\Figinit@@#1#2:{\if#12 \else\Figs@tproj{#1}\initTD@\fi}
\ctr@ln@w{newif}\ifTr@isDim
\ctr@ld@f\def\UnD@fined{UNDEFINED}
\ctr@ln@m\@utoFN
\ctr@ln@m\@utoFInDone
\ctr@ln@m\disob@unit
\ctr@ld@f\def\initpr@lim{\initb@undb@x\figsetmark{}\figsetptname{$A_{##1}$}\def\Sc@leFact{1}%
    \initDD@\figsetroundcoord{yes}\GR@critrue\expandafter\setupd@te\D@FTupdate:%
    \edef\disob@unit{\UnD@fined}\edef\t@rgetpt{\UnD@fined}\gdef\@utoFInDone{1}\gdef\@utoFN{0}}
\ctr@ld@f\def\initDD@{\Tr@isDimfalse%
    \ifPDFm@ke%
     \let\Ps@rcerc=\Ps@rcercBz%
     \let\Ps@rell=\Ps@rellBz%
    \fi
    \let\c@lDCUn=\c@lDCUnDD%
    \let\c@lDCDeux=\c@lDCDeuxDD%
    \let\c@ldefproj=\relax%
    \let\c@lproscal=\c@lproscalDD%
    \let\c@lprojSP=\relax%
    \let\extr@ctC=\extr@ctCDD%
    \let\extr@ctCa=\extr@ctCaDD%
    \let\extr@ctCF=\extr@ctCFDD%
    \let\Figp@intreg=\Figp@intregDD%
    \let\Figpts@xes=\Figpts@xesDD%
    \let\getredf@ctB=\getredf@ctBDD%
    \let\n@rmeucSV=\n@rmeucSVDD\let\n@rmeuc=\n@rmeucDD\let\n@rmeucC\n@rmeucCDD\let\n@rminf=\n@rminfDD%
    \let\pr@dMatV=\pr@dMatVDD%
    \let\Q@@xes=\Q@@xesDD%
    \let\vecunit@=\vecunit@DD%
    \let\figcoord=\figcoordDD%
    \let\figgetangle=\figgetangleDD%
    \let\figpt=\figptDD%
    \let\figptBezier=\figptBezierDD%
    \let\figptbary=\figptbaryDD%
    \let\figptcirc=\figptcircDD%
    \let\figptcircumcenter=\figptcircumcenterDD%
    \let\figptcopy=\figptcopyDD%
    \let\figptcurvcenter=\figptcurvcenterDD%
    \let\figptell=\figptellDD%
    \let\figptendnormal=\figptendnormalDD%
    \let\figptinterlineplane=\figptinterlineplaneDD%
    \let\figptinterlines=\inters@cDD%
    \let\figptorthocenter=\figptorthocenterDD%
    \let\figptorthoprojline=\figptorthoprojlineDD%
    \let\figptorthoprojplane=\figptorthoprojplaneDD%
    \let\figptrot=\figptrotDD%
    \let\figptscontrol=\figptscontrolDD%
    \let\figptsintercirc=\figptsintercircDD%
    \let\figptsinterlinell=\figptsinterlinellDD%
    \let\figptsorthoprojline=\figptsorthoprojlineDD%
    \let\figptorthoprojplane=\figptorthoprojplaneDD%
    \let\figptsrot=\figptsrotDD%
    \let\figptssym=\figptssymDD%
    \let\figptstra=\figptstraDD%
    \let\figptsym=\figptsymDD%
    \let\figpttraC=\figpttraCDD%
    \let\figpttra=\figpttraDD%
    \let\figptvisilimSL=\figptvisilimSLDD%
    \let\figsetobdist=\figsetobdistDD%
    \let\figsettarget=\figsettargetDD%
    \let\figsetview=\figsetviewDD%
    \let\figvectDBezier=\figvectDBezierDD%
    \let\figvectN=\figvectNDD%
    \let\figvectNV=\figvectNVDD%
    \let\figvectP=\figvectPDD%
    \let\figvectU=\figvectUDD%
    \let\figdrawarccircP=\Q@arccircPDD%
    \let\figdrawarccirc=\Q@arccircDD%
    \let\figdrawarcell=\Q@arcellDD%
    \let\figdrawarcellPA=\Q@arcellPADD%
    \let\figdrawarrowBezier=\Q@arrowBezierDD%
    \let\figdrawarrowcircP=\Q@arrowcircPDD%
    \let\figdrawarrowcirc=\Q@arrowcircDD%
    \let\figdrawarrowhead=\Q@arrowheadDD%
    \let\figdrawarrow=\Q@arrowDD%
    \let\figdrawBezier=\Q@BezierDD%
    \let\figdrawcirc=\Q@circDD%
    \let\figdrawcurve=\Q@curveDD%
    \let\figdrawnormal=\Q@normalDD%
    }
\ctr@ld@f\def\initTD@{\Tr@isDimtrue\initb@undb@xTD\newt@rgetptfalse\newdis@bfalse%
    \let\c@lDCUn=\c@lDCUnTD%
    \let\c@lDCDeux=\c@lDCDeuxTD%
    \let\c@ldefproj=\c@ldefprojTD%
    \let\c@lproscal=\c@lproscalTD%
    \let\extr@ctC=\extr@ctCTD%
    \let\extr@ctCa=\extr@ctCaTD%
    \let\extr@ctCF=\extr@ctCFTD%
    \let\Figp@intreg=\Figp@intregTD%
    \let\Figpts@xes=\Figpts@xesTD%
    \let\getredf@ctB=\getredf@ctBTD%
    \let\n@rmeucSV=\n@rmeucSVTD\let\n@rmeuc=\n@rmeucTD\let\n@rmeucC\n@rmeucCTD\let\n@rminf=\n@rminfTD%
    \let\pr@dMatV=\pr@dMatVTD%
    \let\Q@@xes=\Q@@xesTD%
    \let\vecunit@=\vecunit@TD%
    \let\figcoord=\figcoordTD%
    \let\figgetangle=\figgetangleTD%
    \let\figpt=\figptTD%
    \let\figptBezier=\figptBezierTD%
    \let\figptbary=\figptbaryTD%
    \let\figptcirc=\figptcircTD%
    \let\figptcircumcenter=\figptcircumcenterTD%
    \let\figptcopy=\figptcopyTD%
    \let\figptcurvcenter=\figptcurvcenterTD%
    \let\figptinterlineplane=\figptinterlineplaneTD%
    \let\figptinterlines=\inters@cTD%
    \let\figptorthocenter=\figptorthocenterTD%
    \let\figptorthoprojline=\figptorthoprojlineTD%
    \let\figptorthoprojplane=\figptorthoprojplaneTD%
    \let\figptrot=\figptrotTD%
    \let\figptscontrol=\figptscontrolTD%
    \let\figptsintercirc=\figptsintercircTD%
    \let\figptsorthoprojline=\figptsorthoprojlineTD%
    \let\figptsorthoprojplane=\figptsorthoprojplaneTD%
    \let\figptsrot=\figptsrotTD%
    \let\figptssym=\figptssymTD%
    \let\figptstra=\figptstraTD%
    \let\figptsym=\figptsymTD%
    \let\figpttraC=\figpttraCTD%
    \let\figpttra=\figpttraTD%
    \let\figptvisilimSL=\figptvisilimSLTD%
    \let\figsetobdist=\figsetobdistTD%
    \let\figsettarget=\figsettargetTD%
    \let\figsetview=\figsetviewTD%
    \let\figvectDBezier=\figvectDBezierTD%
    \let\figvectN=\figvectNTD%
    \let\figvectNV=\figvectNVTD%
    \let\figvectP=\figvectPTD%
    \let\figvectU=\figvectUTD%
    \let\figdrawarccircP=\Q@arccircPTD%
    \let\figdrawarccirc=\Q@arccircTD%
    \let\figdrawarcell=\Q@arcellTD%
    \let\figdrawarcellPA=\Q@arcellPATD%
    \let\figdrawarrowBezier=\Q@arrowBezierTD%
    \let\figdrawarrowcircP=\Q@arrowcircPTD%
    \let\figdrawarrowcirc=\Q@arrowcircTD%
    \let\figdrawarrowhead=\Q@arrowheadTD%
    \let\figdrawarrow=\Q@arrowTD%
    \let\figdrawBezier=\Q@BezierTD%
    \let\figdrawcirc=\Q@circTD%
    \let\figdrawcurve=\Q@curveTD%
    }
\ctr@ld@f\def\un@v@ilable#1{\immediate\write16{*** The macro #1 is not available in the current context.}}
\ctr@ld@f\def\figinsert#1{{\def\t@xt@{#1}\relax%
    \ifx\t@xt@\empty\ifnum\@utoFInDone>\z@\Figinsert@\DefGIfilen@me,:\fi%
    \else\expandafter\FiginsertNu@#1 :\fi}\ignorespaces}
\ctr@ld@f\def\FiginsertNu@#1 #2:{\def\t@xt@{#1}\relax\ifx\t@xt@\empty\def\t@xt@{#2}%
    \ifx\t@xt@\empty\ifnum\@utoFInDone>\z@\Figinsert@\DefGIfilen@me,:\fi%
    \else\FiginsertNu@#2:\fi\else\expandafter\FiginsertNd@#1 #2:\fi}
\ctr@ld@f\def\FiginsertNd@#1#2:{\ifcat#1a\Figinsert@#1#2,:\else%
    \ifnum\@utoFInDone>\z@\Figinsert@\DefGIfilen@me,#1#2,:\fi\fi}
\ctr@ln@m\Sc@leFact
\ctr@ld@f\def\Figinsert@#1,#2:{\def\t@xt@{#2}\ifx\t@xt@\empty\xdef\Sc@leFact{1}\else%
    \X@rgdeux@#2\xdef\Sc@leFact{\@rgdeux}\fi%
    \Figdisc@rdLTS{#1}{\t@xt@}\@psfgetbb{\t@xt@}%
    \v@lX=\@psfllx\p@\v@lX=\ptpsT@pt\v@lX\v@lX=\Sc@leFact\v@lX%
    \v@lY=\@psflly\p@\v@lY=\ptpsT@pt\v@lY\v@lY=\Sc@leFact\v@lY%
    \b@undb@x{\v@lX}{\v@lY}%
    \v@lX=\@psfurx\p@\v@lX=\ptpsT@pt\v@lX\v@lX=\Sc@leFact\v@lX%
    \v@lY=\@psfury\p@\v@lY=\ptpsT@pt\v@lY\v@lY=\Sc@leFact\v@lY%
    \b@undb@x{\v@lX}{\v@lY}%
    \ifPDFm@ke\Figinclud@PDF{\t@xt@}{\Sc@leFact}\else%
    \v@lX=\c@nt pt\v@lX=\Sc@leFact\v@lX\edef\F@ct{\repdecn@mb{\v@lX}}%
    \ifx\TeXturesonMacOSltX\special{postscriptfile #1 vscale=\F@ct\space hscale=\F@ct}%
    \else\includegraphics{#1}\fi\fi%
    \message{[\t@xt@]}\ignorespaces}
\ctr@ld@f\def\Figdisc@rdLTS#1#2{\expandafter\Figdisc@rdLTS@#1 :#2}
\ctr@ld@f\def\Figdisc@rdLTS@#1 #2:#3{\def#3{#1}\relax\ifx#3\empty\expandafter\Figdisc@rdLTS@#2:#3\fi}
\ctr@ld@f\def\figinsertE#1{\FiginsertE@#1,:\ignorespaces}
\ctr@ld@f\def\FiginsertE@#1,#2:{{\def\t@xt@{#2}\ifx\t@xt@\empty\xdef\Sc@leFact{1}\else%
    \X@rgdeux@#2\xdef\Sc@leFact{\@rgdeux}\fi%
    \Figdisc@rdLTS{#1}{\t@xt@}\pdfximage{\t@xt@}%
    \setbox\Gb@x=\hbox{\pdfrefximage\pdflastximage}%
    \v@lX=\z@\v@lY=-\Sc@leFact\dp\Gb@x\b@undb@x{\v@lX}{\v@lY}%
    \advance\v@lX\Sc@leFact\wd\Gb@x\advance\v@lY\Sc@leFact\dp\Gb@x%
    \advance\v@lY\Sc@leFact\ht\Gb@x\b@undb@x{\v@lX}{\v@lY}%
    \v@lX=\Sc@leFact\wd\Gb@x\pdfximage width \v@lX {\t@xt@}%
    \rlap{\pdfrefximage\pdflastximage}\message{[\t@xt@]}}\ignorespaces}
\ctr@ld@f\def\X@rgdeux@#1,{\edef\@rgdeux{#1}}
\ctr@ln@m\figpt
\ctr@ld@f\def\figptDD#1:#2(#3,#4){\ifGR@cri\c@ntr@lnum{#1}%
    {\v@lX=#3\unit@\v@lY=#4\unit@\Fig@dmpt{#2}{\z@}}\ignorespaces\fi}
\ctr@ld@f\def\Fig@dmpt#1#2{\def\t@xt@{#1}\ifx\t@xt@\empty\def\B@@ltxt{\z@}%
    \else\expandafter\gdef\csname\objc@de T\endcsname{#1}\def\B@@ltxt{\@ne}\fi%
    \expandafter\xdef\csname\objc@de\endcsname{\ifitis@vect@r\C@dCl@svect%
    \else\C@dCl@spt\fi,\z@,\B@@ltxt/\the\v@lX,\the\v@lY,#2}}
\ctr@ld@f\def\C@dCl@spt{P}
\ctr@ld@f\def\C@dCl@svect{V}
\ctr@ln@m\c@@rdYZ
\ctr@ln@m\c@@rdY
\ctr@ld@f\def\figptTD#1:#2(#3,#4){\ifGR@cri\c@ntr@lnum{#1}%
    \def\c@@rdYZ{#4,0,0}\extrairelepremi@r\c@@rdY\de\c@@rdYZ%
    \extrairelepremi@r\c@@rdZ\de\c@@rdYZ%
    {\v@lX=#3\unit@\v@lY=\c@@rdY\unit@\v@lZ=\c@@rdZ\unit@\Fig@dmpt{#2}{\the\v@lZ}%
    \b@undb@xTD{\v@lX}{\v@lY}{\v@lZ}}\ignorespaces\fi}
\ctr@ln@m\Figp@intreg
\ctr@ld@f\def\Figp@intregDD#1:#2(#3,#4){\c@ntr@lnum{#1}%
    {\result@t=#4\v@lX=#3\v@lY=\result@t\Fig@dmpt{#2}{\z@}}\ignorespaces}
\ctr@ld@f\def\Figp@intregTD#1:#2(#3,#4){\c@ntr@lnum{#1}%
    \def\c@@rdYZ{#4,\z@,\z@}\extrairelepremi@r\c@@rdY\de\c@@rdYZ%
    \extrairelepremi@r\c@@rdZ\de\c@@rdYZ%
    {\v@lX=#3\v@lY=\c@@rdY\v@lZ=\c@@rdZ\Fig@dmpt{#2}{\the\v@lZ}%
    \b@undb@xTD{\v@lX}{\v@lY}{\v@lZ}}\ignorespaces}
\ctr@ln@m\figptBezier
\ctr@ld@f\def\figptBezierDD#1:#2:#3[#4,#5,#6,#7]{\ifGR@cri{\s@uvc@ntr@l\et@tfigptBezierDD%
    \FigptBezier@#3[#4,#5,#6,#7]\Figp@intregDD#1:{#2}(\v@lX,\v@lY)%
    \resetc@ntr@l\et@tfigptBezierDD}\ignorespaces\fi}
\ctr@ld@f\def\figptBezierTD#1:#2:#3[#4,#5,#6,#7]{\ifGR@cri{\s@uvc@ntr@l\et@tfigptBezierTD%
    \FigptBezier@#3[#4,#5,#6,#7]\Figp@intregTD#1:{#2}(\v@lX,\v@lY,\v@lZ)%
    \resetc@ntr@l\et@tfigptBezierTD}\ignorespaces\fi}
\ctr@ld@f\def\FigptBezier@#1[#2,#3,#4,#5]{\setc@ntr@l{2}%
    \edef\T@{#1}\v@leur=\p@\advance\v@leur-#1pt\edef\UNmT@{\repdecn@mb{\v@leur}}%
    \figptcopy-4:/#2/\figptcopy-3:/#3/\figptcopy-2:/#4/\figptcopy-1:/#5/%
    \l@mbd@un=-4 \l@mbd@de=-\thr@@\p@rtent=\m@ne\c@lDecast%
    \l@mbd@un=-4 \l@mbd@de=-\thr@@\p@rtent=-\tw@\c@lDecast%
    \l@mbd@un=-4 \l@mbd@de=-\thr@@\p@rtent=-\thr@@\c@lDecast\Figg@tXY{-4}}
\ctr@ln@m\c@lDCUn
\ctr@ld@f\def\c@lDCUnDD#1#2{\Figg@tXY{#1}\v@lX=\UNmT@\v@lX\v@lY=\UNmT@\v@lY%
    \Figg@tXYa{#2}\advance\v@lX\T@\v@lXa\advance\v@lY\T@\v@lYa%
    \Figp@intregDD#1:(\v@lX,\v@lY)}
\ctr@ld@f\def\c@lDCUnTD#1#2{\Figg@tXY{#1}\v@lX=\UNmT@\v@lX\v@lY=\UNmT@\v@lY\v@lZ=\UNmT@\v@lZ%
    \Figg@tXYa{#2}\advance\v@lX\T@\v@lXa\advance\v@lY\T@\v@lYa\advance\v@lZ\T@\v@lZa%
    \Figp@intregTD#1:(\v@lX,\v@lY,\v@lZ)}
\ctr@ld@f\def\c@lDecast{\relax\ifnum\l@mbd@un<\p@rtent\c@lDCUn{\l@mbd@un}{\l@mbd@de}%
    \advance\l@mbd@un\@ne\advance\l@mbd@de\@ne\c@lDecast\fi}
\ctr@ld@f\def\figptmap#1:#2=#3/#4/#5/{\ifGR@cri{\s@uvc@ntr@l\et@tfigptmap%
    \setc@ntr@l{2}\figvectP-1[#4,#3]\Figg@tXY{-1}%
    \pr@dMatV/#5/\figpttra#1:{#2}=#4/1,-1/%
    \resetc@ntr@l\et@tfigptmap}\ignorespaces\fi}
\ctr@ln@m\pr@dMatV
\ctr@ld@f\def\pr@dMatVDD/#1,#2;#3,#4/{\v@lXa=#1\v@lX\advance\v@lXa#2\v@lY%
    \v@lYa=#3\v@lX\advance\v@lYa#4\v@lY\Figv@ctCreg-1(\v@lXa,\v@lYa)}
\ctr@ld@f\def\pr@dMatVTD/#1,#2,#3;#4,#5,#6;#7,#8,#9/{%
    \v@lXa=#1\v@lX\advance\v@lXa#2\v@lY\advance\v@lXa#3\v@lZ%
    \v@lYa=#4\v@lX\advance\v@lYa#5\v@lY\advance\v@lYa#6\v@lZ%
    \v@lZa=#7\v@lX\advance\v@lZa#8\v@lY\advance\v@lZa#9\v@lZ%
    \Figv@ctCreg-1(\v@lXa,\v@lYa,\v@lZa)}
\ctr@ln@m\figptbary
\ctr@ld@f\def\figptbaryDD#1:#2[#3;#4]{\ifGR@cri{\edef\list@num{#3}\extrairelepremi@r\p@int\de\list@num%
    \s@mme=\z@\@ecfor\c@ef:=#4\do{\advance\s@mme\c@ef}%
    \edef\listec@ef{#4,0}\extrairelepremi@r\c@ef\de\listec@ef%
    \Figg@tXY{\p@int}\divide\v@lX\s@mme\divide\v@lY\s@mme%
    \multiply\v@lX\c@ef\multiply\v@lY\c@ef%
    \@ecfor\p@int:=\list@num\do{\extrairelepremi@r\c@ef\de\listec@ef%
           \Figg@tXYa{\p@int}\divide\v@lXa\s@mme\divide\v@lYa\s@mme%
           \multiply\v@lXa\c@ef\multiply\v@lYa\c@ef%
           \advance\v@lX\v@lXa\advance\v@lY\v@lYa}%
    \Figp@intregDD#1:{#2}(\v@lX,\v@lY)}\ignorespaces\fi}
\ctr@ld@f\def\figptbaryTD#1:#2[#3;#4]{\ifGR@cri{\edef\list@num{#3}\extrairelepremi@r\p@int\de\list@num%
    \s@mme=\z@\@ecfor\c@ef:=#4\do{\advance\s@mme\c@ef}%
    \edef\listec@ef{#4,0}\extrairelepremi@r\c@ef\de\listec@ef%
    \Figg@tXY{\p@int}\divide\v@lX\s@mme\divide\v@lY\s@mme\divide\v@lZ\s@mme%
    \multiply\v@lX\c@ef\multiply\v@lY\c@ef\multiply\v@lZ\c@ef%
    \@ecfor\p@int:=\list@num\do{\extrairelepremi@r\c@ef\de\listec@ef%
           \Figg@tXYa{\p@int}\divide\v@lXa\s@mme\divide\v@lYa\s@mme\divide\v@lZa\s@mme%
           \multiply\v@lXa\c@ef\multiply\v@lYa\c@ef\multiply\v@lZa\c@ef%
           \advance\v@lX\v@lXa\advance\v@lY\v@lYa\advance\v@lZ\v@lZa}%
    \Figp@intregTD#1:{#2}(\v@lX,\v@lY,\v@lZ)}\ignorespaces\fi}
\ctr@ld@f\def\figptbaryR#1:#2[#3;#4]{\ifGR@cri{%
    \v@leur=\z@\@ecfor\c@ef:=#4\do{\maxim@m{\v@lmax}{\c@ef pt}{-\c@ef pt}%
    \ifdim\v@lmax>\v@leur\v@leur=\v@lmax\fi}%
    \ifdim\v@leur<\p@\f@ctech=\@M\else\ifdim\v@leur<\t@n\p@\f@ctech=\@m\else%
    \ifdim\v@leur<\c@nt\p@\f@ctech=\c@nt\else\ifdim\v@leur<\@m\p@\f@ctech=\t@n\else%
    \f@ctech=\@ne\fi\fi\fi\fi%
    \def\listec@ef{0}%
    \@ecfor\c@ef:=#4\do{\sc@lec@nvRI{\c@ef pt}\edef\listec@ef{\listec@ef,\the\s@mme}}%
    \extrairelepremi@r\c@ef\de\listec@ef\figptbary#1:#2[#3;\listec@ef]}\ignorespaces\fi}
\ctr@ld@f\def\sc@lec@nvRI#1{\v@leur=#1\p@rtentiere{\s@mme}{\v@leur}\advance\v@leur-\s@mme\p@%
    \multiply\v@leur\f@ctech\p@rtentiere{\p@rtent}{\v@leur}%
    \multiply\s@mme\f@ctech\advance\s@mme\p@rtent}
\ctr@ln@m\figptcirc
\ctr@ld@f\def\figptcircDD#1:#2:#3;#4(#5){\ifGR@cri{\s@uvc@ntr@l\et@tfigptcircDD%
    \c@lptellDD#1:{#2}:#3;#4,#4(#5)\resetc@ntr@l\et@tfigptcircDD}\ignorespaces\fi}
\ctr@ld@f\def\figptcircTD#1:#2:#3,#4,#5;#6(#7){\ifGR@cri{\s@uvc@ntr@l\et@tfigptcircTD%
    \setc@ntr@l{2}\c@lExtAxes#3,#4,#5(#6)\figptellP#1:{#2}:#3,-4,-5(#7)%
    \resetc@ntr@l\et@tfigptcircTD}\ignorespaces\fi}
\ctr@ln@m\figptcircumcenter
\ctr@ld@f\def\figptcircumcenterDD#1:#2[#3,#4,#5]{\ifGR@cri{\s@uvc@ntr@l\et@tfigptcircumcenterDD%
    \setc@ntr@l{2}\figvectNDD-5[#3,#4]\figptbaryDD-3:[#3,#4;1,1]%
                  \figvectNDD-6[#4,#5]\figptbaryDD-4:[#4,#5;1,1]%
    \resetc@ntr@l{2}\inters@cDD#1:{#2}[-3,-5;-4,-6]%
    \resetc@ntr@l\et@tfigptcircumcenterDD}\ignorespaces\fi}
\ctr@ld@f\def\figptcircumcenterTD#1:#2[#3,#4,#5]{\ifGR@cri{\s@uvc@ntr@l\et@tfigptcircumcenterTD%
    \setc@ntr@l{2}\figvectNTD-1[#3,#4,#5]%
    \figvectPTD-3[#3,#4]\figvectNVTD-5[-1,-3]\figptbaryTD-3:[#3,#4;1,1]%
    \figvectPTD-4[#4,#5]\figvectNVTD-6[-1,-4]\figptbaryTD-4:[#4,#5;1,1]%
    \resetc@ntr@l{2}\inters@cTD#1:{#2}[-3,-5;-4,-6]%
    \resetc@ntr@l\et@tfigptcircumcenterTD}\ignorespaces\fi}
\ctr@ln@m\figptcopy
\ctr@ld@f\def\figptcopyDD#1:#2/#3/{\ifGR@cri{\Figg@tXY{#3}%
    \Figp@intregDD#1:{#2}(\v@lX,\v@lY)}\ignorespaces\fi}
\ctr@ld@f\def\figptcopyTD#1:#2/#3/{\ifGR@cri{\Figg@tXY{#3}%
    \Figp@intregTD#1:{#2}(\v@lX,\v@lY,\v@lZ)}\ignorespaces\fi}
\ctr@ln@m\figptcurvcenter
\ctr@ld@f\def\figptcurvcenterDD#1:#2:#3[#4,#5,#6,#7]{\ifGR@cri{\s@uvc@ntr@l\et@tfigptcurvcenterDD%
    \setc@ntr@l{2}\c@lcurvradDD#3[#4,#5,#6,#7]\edef\Sprim@{\repdecn@mb{\result@t}}%
    \figptBezierDD-1::#3[#4,#5,#6,#7]\figpttraDD#1:{#2}=-1/\Sprim@,-5/%
    \resetc@ntr@l\et@tfigptcurvcenterDD}\ignorespaces\fi}
\ctr@ld@f\def\figptcurvcenterTD#1:#2:#3[#4,#5,#6,#7]{\ifGR@cri{\s@uvc@ntr@l\et@tfigptcurvcenterTD%
    \setc@ntr@l{2}\figvectDBezierTD -5:1,#3[#4,#5,#6,#7]%
    \figvectDBezierTD -6:2,#3[#4,#5,#6,#7]\vecunit@TD{-5}{-5}%
    \edef\Sprim@{\repdecn@mb{\result@t}}\figvectNVTD-1[-6,-5]%
    \figvectNVTD-5[-5,-1]\c@lproscalTD\v@leur[-6,-5]%
    \invers@{\v@leur}{\v@leur}\v@leur=\Sprim@\v@leur\v@leur=\Sprim@\v@leur%
    \figptBezierTD-1::#3[#4,#5,#6,#7]\edef\Sprim@{\repdecn@mb{\v@leur}}%
    \figpttraTD#1:{#2}=-1/\Sprim@,-5/\resetc@ntr@l\et@tfigptcurvcenterTD}\ignorespaces\fi}
\ctr@ld@f\def\c@lcurvradDD#1[#2,#3,#4,#5]{{\figvectDBezierDD -5:1,#1[#2,#3,#4,#5]%
    \figvectDBezierDD -6:2,#1[#2,#3,#4,#5]\vecunit@DD{-5}{-5}%
    \edef\Sprim@{\repdecn@mb{\result@t}}\figvectNVDD-5[-5]\c@lproscalDD\v@leur[-6,-5]%
    \invers@{\v@leur}{\v@leur}\v@leur=\Sprim@\v@leur\v@leur=\Sprim@\v@leur%
    \global\result@t=\v@leur}}
\ctr@ln@m\figptell
\ctr@ld@f\def\figptellDD#1:#2:#3;#4,#5(#6,#7){\ifGR@cri{\s@uvc@ntr@l\et@tfigptell%
    \c@lptellDD#1::#3;#4,#5(#6)\figptrotDD#1:{#2}=#1/#3,#7/%
    \resetc@ntr@l\et@tfigptell}\ignorespaces\fi}
\ctr@ld@f\def\c@lptellDD#1:#2:#3;#4,#5(#6){\c@ssin{\C@}{\S@}{#6}\v@lmin=\C@ pt\v@lmax=\S@ pt%
    \v@lmin=#4\v@lmin\v@lmax=#5\v@lmax%
    \edef\Xc@mp{\repdecn@mb{\v@lmin}}\edef\Yc@mp{\repdecn@mb{\v@lmax}}%
    \setc@ntr@l{2}\figvectC-1(\Xc@mp,\Yc@mp)\figpttraDD#1:{#2}=#3/1,-1/}
\ctr@ld@f\def\figptellP#1:#2:#3,#4,#5(#6){\ifGR@cri{\s@uvc@ntr@l\et@tfigptellP%
    \setc@ntr@l{2}\figvectP-1[#3,#4]\figvectP-2[#3,#5]%
    \v@leur=#6pt\c@lptellP{#3}{-1}{-2}\figptcopy#1:{#2}/-3/%
    \resetc@ntr@l\et@tfigptellP}\ignorespaces\fi}
\ctr@ln@m\@ngle
\ctr@ld@f\def\c@lptellP#1#2#3{\edef\@ngle{\repdecn@mb\v@leur}\c@ssin{\C@}{\S@}{\@ngle}%
    \figpttra-3:=#1/\C@,#2/\figpttra-3:=-3/\S@,#3/}
\ctr@ln@m\figptendnormal
\ctr@ld@f\def\figptendnormalDD#1:#2:#3,#4[#5,#6]{\ifGR@cri{\s@uvc@ntr@l\et@tfigptendnormal%
    \Figg@tXYa{#5}\Figg@tXY{#6}%
    \advance\v@lX-\v@lXa\advance\v@lY-\v@lYa%
    \setc@ntr@l{2}\Figv@ctCreg-1(\v@lX,\v@lY)\vecunit@{-1}{-1}\Figg@tXY{-1}%
    \delt@=#3\unit@\maxim@m{\delt@}{\delt@}{-\delt@}\edef\l@ngueur{\repdecn@mb{\delt@}}%
    \v@lX=\l@ngueur\v@lX\v@lY=\l@ngueur\v@lY%
    \delt@=\p@\advance\delt@-#4pt\edef\l@ngueur{\repdecn@mb{\delt@}}%
    \figptbaryR-1:[#5,#6;#4,\l@ngueur]\Figg@tXYa{-1}%
    \advance\v@lXa\v@lY\advance\v@lYa-\v@lX%
    \setc@ntr@l{1}\Figp@intregDD#1:{#2}(\v@lXa,\v@lYa)\resetc@ntr@l\et@tfigptendnormal}%
    \ignorespaces\fi}
\ctr@ld@f\def\figptexcenter#1:#2[#3,#4,#5]{\ifGR@cri{\let@xte={-}
    \Figptexinsc@nter#1:#2[#3,#4,#5]}\ignorespaces\fi}
\ctr@ld@f\def\figptincenter#1:#2[#3,#4,#5]{\ifGR@cri{\let@xte={}
    \Figptexinsc@nter#1:#2[#3,#4,#5]}\ignorespaces\fi}
\ctr@ld@f
\ctr@ld@f\def\Figptexinsc@nter#1:#2[#3,#4,#5]{%
    \figgetdist\LA@[#4,#5]\figgetdist\LB@[#3,#5]\figgetdist\LC@[#3,#4]%
    \figptbaryR#1:{#2}[#3,#4,#5;\the\let@xte\LA@,\LB@,\LC@]}
\ctr@ln@m\figptinterlineplane
\ctr@ld@f\def\figptinterlineplaneDD{\un@v@ilable{figptinterlineplane}}
\ctr@ld@f\def\figptinterlineplaneTD#1:#2[#3,#4;#5,#6]{\ifGR@cri{\s@uvc@ntr@l\et@tfigptinterlineplane%
    \setc@ntr@l{2}\figvectPTD-1[#3,#5]\vecunit@TD{-2}{#6}%
    \r@pPSTD\v@leur[-2,-1,#4]\edef\v@lcoef{\repdecn@mb{\v@leur}}%
    \figpttraTD#1:{#2}=#3/\v@lcoef,#4/\resetc@ntr@l\et@tfigptinterlineplane}\ignorespaces\fi}
\ctr@ln@m\figptorthocenter
\ctr@ld@f\def\figptorthocenterDD#1:#2[#3,#4,#5]{\ifGR@cri{\s@uvc@ntr@l\et@tfigptorthocenterDD%
    \setc@ntr@l{2}\figvectNDD-3[#3,#4]\figvectNDD-4[#4,#5]%
    \resetc@ntr@l{2}\inters@cDD#1:{#2}[#5,-3;#3,-4]%
    \resetc@ntr@l\et@tfigptorthocenterDD}\ignorespaces\fi}
\ctr@ld@f\def\figptorthocenterTD#1:#2[#3,#4,#5]{\ifGR@cri{\s@uvc@ntr@l\et@tfigptorthocenterTD%
    \setc@ntr@l{2}\figvectNTD-1[#3,#4,#5]%
    \figvectPTD-2[#3,#4]\figvectNVTD-3[-1,-2]%
    \figvectPTD-2[#4,#5]\figvectNVTD-4[-1,-2]%
    \resetc@ntr@l{2}\inters@cTD#1:{#2}[#5,-3;#3,-4]%
    \resetc@ntr@l\et@tfigptorthocenterTD}\ignorespaces\fi}
\ctr@ln@m\figptorthoprojline
\ctr@ld@f\def\figptorthoprojlineDD#1:#2=#3/#4,#5/{\ifGR@cri{\s@uvc@ntr@l\et@tfigptorthoprojlineDD%
    \setc@ntr@l{2}\figvectPDD-3[#4,#5]\figvectNVDD-4[-3]\resetc@ntr@l{2}%
    \inters@cDD#1:{#2}[#3,-4;#4,-3]\resetc@ntr@l\et@tfigptorthoprojlineDD}\ignorespaces\fi}
\ctr@ld@f\def\figptorthoprojlineTD#1:#2=#3/#4,#5/{\ifGR@cri{\s@uvc@ntr@l\et@tfigptorthoprojlineTD%
    \setc@ntr@l{2}\figvectPTD-1[#4,#3]\figvectPTD-2[#4,#5]\vecunit@TD{-2}{-2}%
    \c@lproscalTD\v@leur[-1,-2]\edef\v@lcoef{\repdecn@mb{\v@leur}}%
    \figpttraTD#1:{#2}=#4/\v@lcoef,-2/\resetc@ntr@l\et@tfigptorthoprojlineTD}\ignorespaces\fi}
\ctr@ln@m\figptorthoprojplane
\ctr@ld@f\def\figptorthoprojplaneDD{\un@v@ilable{figptorthoprojplane}}
\ctr@ld@f\def\figptorthoprojplaneTD#1:#2=#3/#4,#5/{\ifGR@cri{\s@uvc@ntr@l\et@tfigptorthoprojplane%
    \setc@ntr@l{2}\figvectPTD-1[#3,#4]\vecunit@TD{-2}{#5}%
    \c@lproscalTD\v@leur[-1,-2]\edef\v@lcoef{\repdecn@mb{\v@leur}}%
    \figpttraTD#1:{#2}=#3/\v@lcoef,-2/\resetc@ntr@l\et@tfigptorthoprojplane}\ignorespaces\fi}
\ctr@ld@f\def\figpthom#1:#2=#3/#4,#5/{\ifGR@cri{\s@uvc@ntr@l\et@tfigpthom%
    \setc@ntr@l{2}\figvectP-1[#4,#3]\figpttra#1:{#2}=#4/#5,-1/%
    \resetc@ntr@l\et@tfigpthom}\ignorespaces\fi}
\ctr@ld@f\def\figptinv#1:#2=#3/#4,#5/{\ifGR@cri{\s@uvc@ntr@l\et@tfigptinv%
    \setc@ntr@l{2}\figvectP-1[#4,#3]\Figg@tXY{-1}%
    \getredf@ctB\f@ctech\n@rmeucC{\delt@}{-1}%
    \delt@=\ptT@unit@\delt@\delt@=\ptT@unit@\delt@%
    \invers@{\delt@}{\delt@}\multiply\f@ctech\f@ctech\divide\delt@\f@ctech%
    \delt@=#5\delt@\edef\v@lcoef{\repdecn@mb{\delt@}}\figpttra#1:{#2}=#4/\v@lcoef,-1/%
    \resetc@ntr@l\et@tfigptinv}\ignorespaces\fi}
\ctr@ln@m\figptrot
\ctr@ld@f\def\figptrotDD#1:#2=#3/#4,#5/{\ifGR@cri{\s@uvc@ntr@l\et@tfigptrotDD%
    \c@ssin{\C@}{\S@}{#5}\setc@ntr@l{2}\figvectPDD-1[#4,#3]\Figg@tXY{-1}%
    \v@lXa=\C@\v@lX\advance\v@lXa-\S@\v@lY%
    \v@lYa=\S@\v@lX\advance\v@lYa\C@\v@lY%
    \Figv@ctCreg-1(\v@lXa,\v@lYa)\figpttraDD#1:{#2}=#4/1,-1/%
    \resetc@ntr@l\et@tfigptrotDD}\ignorespaces\fi}
\ctr@ld@f\def\figptrotTD#1:#2=#3/#4,#5,#6/{\ifGR@cri{\s@uvc@ntr@l\et@tfigptrotTD%
    \c@ssin{\C@}{\S@}{#5}%
    \setc@ntr@l{2}\figptorthoprojplaneTD-3:=#4/#3,#6/\figvectPTD-2[-3,#3]%
    \n@rmeucTD\v@leur{-2}\ifdim\v@leur<\Cepsil@n\Figg@tXYa{#3}\else%
    \edef\v@lcoef{\repdecn@mb{\v@leur}}\figvectNVTD-1[#6,-2]%
    \Figg@tXYa{-1}\v@lXa=\v@lcoef\v@lXa\v@lYa=\v@lcoef\v@lYa\v@lZa=\v@lcoef\v@lZa%
    \v@lXa=\S@\v@lXa\v@lYa=\S@\v@lYa\v@lZa=\S@\v@lZa\Figg@tXY{-2}%
    \advance\v@lXa\C@\v@lX\advance\v@lYa\C@\v@lY\advance\v@lZa\C@\v@lZ%
    \Figg@tXY{-3}\advance\v@lXa\v@lX\advance\v@lYa\v@lY\advance\v@lZa\v@lZ\fi%
    \Figp@intregTD#1:{#2}(\v@lXa,\v@lYa,\v@lZa)\resetc@ntr@l\et@tfigptrotTD}\ignorespaces\fi}
\ctr@ln@m\figptsym
\ctr@ld@f\def\figptsymDD#1:#2=#3/#4,#5/{\ifGR@cri{\s@uvc@ntr@l\et@tfigptsymDD%
    \resetc@ntr@l{2}\figptorthoprojlineDD-5:=#3/#4,#5/\figvectPDD-2[#3,-5]%
    \figpttraDD#1:{#2}=#3/2,-2/\resetc@ntr@l\et@tfigptsymDD}\ignorespaces\fi}
\ctr@ld@f\def\figptsymTD#1:#2=#3/#4,#5/{\ifGR@cri{\s@uvc@ntr@l\et@tfigptsymTD%
    \resetc@ntr@l{2}\figptorthoprojplaneTD-3:=#3/#4,#5/\figvectPTD-2[#3,-3]%
    \figpttraTD#1:{#2}=#3/2,-2/\resetc@ntr@l\et@tfigptsymTD}\ignorespaces\fi}
\ctr@ln@m\figpttra
\ctr@ld@f\def\figpttraDD#1:#2=#3/#4,#5/{\ifGR@cri{\Figg@tXYa{#5}\v@lXa=#4\v@lXa\v@lYa=#4\v@lYa%
    \Figg@tXY{#3}\advance\v@lX\v@lXa\advance\v@lY\v@lYa%
    \Figp@intregDD#1:{#2}(\v@lX,\v@lY)}\ignorespaces\fi}
\ctr@ld@f\def\figpttraTD#1:#2=#3/#4,#5/{\ifGR@cri{\Figg@tXYa{#5}\v@lXa=#4\v@lXa\v@lYa=#4\v@lYa%
    \v@lZa=#4\v@lZa\Figg@tXY{#3}\advance\v@lX\v@lXa\advance\v@lY\v@lYa%
    \advance\v@lZ\v@lZa\Figp@intregTD#1:{#2}(\v@lX,\v@lY,\v@lZ)}\ignorespaces\fi}
\ctr@ln@m\figpttraC
\ctr@ld@f\def\figpttraCDD#1:#2=#3/#4,#5/{\ifGR@cri{\v@lXa=#4\unit@\v@lYa=#5\unit@%
    \Figg@tXY{#3}\advance\v@lX\v@lXa\advance\v@lY\v@lYa%
    \Figp@intregDD#1:{#2}(\v@lX,\v@lY)}\ignorespaces\fi}
\ctr@ld@f\def\figpttraCTD#1:#2=#3/#4,#5,#6/{\ifGR@cri{\v@lXa=#4\unit@\v@lYa=#5\unit@\v@lZa=#6\unit@%
    \Figg@tXY{#3}\advance\v@lX\v@lXa\advance\v@lY\v@lYa\advance\v@lZ\v@lZa%
    \Figp@intregTD#1:{#2}(\v@lX,\v@lY,\v@lZ)}\ignorespaces\fi}
\ctr@ld@f\def\figptsaxes#1:#2(#3){\ifGR@cri{\an@lys@xes#3,:\ifx\t@xt@\empty%
    \ifTr@isDim\Figpts@xes#1:#2(0,#3,0,#3,0,#3)\else\Figpts@xes#1:#2(0,#3,0,#3)\fi%
    \else\Figpts@xes#1:#2(#3)\fi}\ignorespaces\fi}
\ctr@ln@m\Figpts@xes
\ctr@ld@f\def\Figpts@xesDD#1:#2(#3,#4,#5,#6){%
    \s@mme=#1\figpttraC\the\s@mme:$x$=#2/#4,0/%
    \advance\s@mme\@ne\figpttraC\the\s@mme:$y$=#2/0,#6/}
\ctr@ld@f\def\Figpts@xesTD#1:#2(#3,#4,#5,#6,#7,#8){%
    \s@mme=#1\figpttraC\the\s@mme:$x$=#2/#4,0,0/%
    \advance\s@mme\@ne\figpttraC\the\s@mme:$y$=#2/0,#6,0/%
    \advance\s@mme\@ne\figpttraC\the\s@mme:$z$=#2/0,0,#8/}
\ctr@ld@f\def\figptsmap#1=#2/#3/#4/{\ifGR@cri{\s@uvc@ntr@l\et@tfigptsmap%
    \setc@ntr@l{2}\def\list@num{#2}\s@mme=#1%
    \@ecfor\p@int:=\list@num\do{\figvectP-1[#3,\p@int]\Figg@tXY{-1}%
    \pr@dMatV/#4/\figpttra\the\s@mme:=#3/1,-1/\advance\s@mme\@ne}%
    \resetc@ntr@l\et@tfigptsmap}\ignorespaces\fi}
\ctr@ln@m\figptscontrol
\ctr@ld@f\def\figptscontrolDD#1[#2,#3,#4,#5]{\ifGR@cri{\s@uvc@ntr@l\et@tfigptscontrolDD\setc@ntr@l{2}%
    \v@lX=\z@\v@lY=\z@\Figtr@nptDD{-5}{#2}\Figtr@nptDD{2}{#5}%
    \divide\v@lX\@vi\divide\v@lY\@vi%
    \Figtr@nptDD{3}{#3}\Figtr@nptDD{-1.5}{#4}\Figp@intregDD-1:(\v@lX,\v@lY)%
    \v@lX=\z@\v@lY=\z@\Figtr@nptDD{2}{#2}\Figtr@nptDD{-5}{#5}%
    \divide\v@lX\@vi\divide\v@lY\@vi\Figtr@nptDD{-1.5}{#3}\Figtr@nptDD{3}{#4}%
    \s@mme=#1\advance\s@mme\@ne\Figp@intregDD\the\s@mme:(\v@lX,\v@lY)%
    \figptcopyDD#1:/-1/\resetc@ntr@l\et@tfigptscontrolDD}\ignorespaces\fi}
\ctr@ld@f\def\figptscontrolTD#1[#2,#3,#4,#5]{\ifGR@cri{\s@uvc@ntr@l\et@tfigptscontrolTD\setc@ntr@l{2}%
    \v@lX=\z@\v@lY=\z@\v@lZ=\z@\Figtr@nptTD{-5}{#2}\Figtr@nptTD{2}{#5}%
    \divide\v@lX\@vi\divide\v@lY\@vi\divide\v@lZ\@vi%
    \Figtr@nptTD{3}{#3}\Figtr@nptTD{-1.5}{#4}\Figp@intregTD-1:(\v@lX,\v@lY,\v@lZ)%
    \v@lX=\z@\v@lY=\z@\v@lZ=\z@\Figtr@nptTD{2}{#2}\Figtr@nptTD{-5}{#5}%
    \divide\v@lX\@vi\divide\v@lY\@vi\divide\v@lZ\@vi\Figtr@nptTD{-1.5}{#3}\Figtr@nptTD{3}{#4}%
    \s@mme=#1\advance\s@mme\@ne\Figp@intregTD\the\s@mme:(\v@lX,\v@lY,\v@lZ)%
    \figptcopyTD#1:/-1/\resetc@ntr@l\et@tfigptscontrolTD}\ignorespaces\fi}
\ctr@ld@f\def\Figtr@nptDD#1#2{\Figg@tXYa{#2}\v@lXa=#1\v@lXa\v@lYa=#1\v@lYa%
    \advance\v@lX\v@lXa\advance\v@lY\v@lYa}
\ctr@ld@f\def\Figtr@nptTD#1#2{\Figg@tXYa{#2}\v@lXa=#1\v@lXa\v@lYa=#1\v@lYa\v@lZa=#1\v@lZa%
    \advance\v@lX\v@lXa\advance\v@lY\v@lYa\advance\v@lZ\v@lZa}
\ctr@ld@f\def\figptscontrolcurve#1,#2[#3]{\ifGR@cri{\s@uvc@ntr@l\et@tfigptscontrolcurve%
    \def\list@num{#3}\extrairelepremi@r\Ak@\de\list@num%
    \extrairelepremi@r\Ai@\de\list@num\extrairelepremi@r\Aj@\de\list@num%
    \s@mme=#1\figptcopy\the\s@mme:/\Ai@/%
    \setc@ntr@l{2}\figvectP -1[\Ak@,\Aj@]%
    \@ecfor\Ak@:=\list@num\do{\advance\s@mme\@ne\figpttra\the\s@mme:=\Ai@/\curv@roundness,-1/%
       \figvectP -1[\Ai@,\Ak@]\advance\s@mme\@ne\figpttra\the\s@mme:=\Aj@/-\curv@roundness,-1/%
       \advance\s@mme\@ne\figptcopy\the\s@mme:/\Aj@/%
       \edef\Ai@{\Aj@}\edef\Aj@{\Ak@}}\advance\s@mme-#1\divide\s@mme\thr@@%
       \xdef#2{\the\s@mme}%
    \resetc@ntr@l\et@tfigptscontrolcurve}\ignorespaces\fi}
\ctr@ln@m\figptsintercirc
\ctr@ld@f\def\figptsintercircDD#1[#2,#3;#4,#5]{\ifGR@cri{\s@uvc@ntr@l\et@tfigptsintercircDD%
    \setc@ntr@l{2}\let\c@lNVintc=\c@lNVintcDD\Figptsintercirc@#1[#2,#3;#4,#5]%
    \resetc@ntr@l\et@tfigptsintercircDD}\ignorespaces\fi}
\ctr@ld@f\def\figptsintercircTD#1[#2,#3;#4,#5;#6]{\ifGR@cri{\s@uvc@ntr@l\et@tfigptsintercircTD%
    \setc@ntr@l{2}\let\c@lNVintc=\c@lNVintcTD\vecunitC@TD[#2,#6]%
    \Figv@ctCreg-3(\v@lX,\v@lY,\v@lZ)\Figptsintercirc@#1[#2,#3;#4,#5]%
    \resetc@ntr@l\et@tfigptsintercircTD}\ignorespaces\fi}
\ctr@ld@f\def\Figptsintercirc@#1[#2,#3;#4,#5]{\figvectP-1[#2,#4]%
    \vecunit@{-1}{-1}\delt@=\result@t\f@ctech=\result@tent%
    \s@mme=#1\advance\s@mme\@ne\figptcopy#1:/#2/\figptcopy\the\s@mme:/#4/%
    \ifdim\delt@=\z@\else%
    \v@lmin=#3\unit@\v@lmax=#5\unit@\v@leur=\v@lmin\advance\v@leur\v@lmax%
    \ifdim\v@leur>\delt@%
    \v@leur=\v@lmin\advance\v@leur-\v@lmax\maxim@m{\v@leur}{\v@leur}{-\v@leur}%
    \ifdim\v@leur<\delt@%
    \divide\v@lmin\f@ctech\divide\v@lmax\f@ctech\divide\delt@\f@ctech%
    \v@lmin=\repdecn@mb{\v@lmin}\v@lmin\v@lmax=\repdecn@mb{\v@lmax}\v@lmax%
    \invers@{\v@leur}{\delt@}\advance\v@lmax-\v@lmin%
    \v@lmax=-\repdecn@mb{\v@leur}\v@lmax\advance\delt@\v@lmax\delt@=.5\delt@%
    \v@lmax=\delt@\multiply\v@lmax\f@ctech%
    \edef\t@ille{\repdecn@mb{\v@lmax}}\figpttra-2:=#2/\t@ille,-1/%
    \delt@=\repdecn@mb{\delt@}\delt@\advance\v@lmin-\delt@%
    \sqrt@{\v@leur}{\v@lmin}\multiply\v@leur\f@ctech\edef\t@ille{\repdecn@mb{\v@leur}}%
    \c@lNVintc\figpttra#1:=-2/-\t@ille,-1/\figpttra\the\s@mme:=-2/\t@ille,-1/\fi\fi\fi}
\ctr@ld@f\def\c@lNVintcDD{\Figg@tXY{-1}\Figv@ctCreg-1(-\v@lY,\v@lX)} 
\ctr@ld@f\def\c@lNVintcTD{{\Figg@tXY{-3}\v@lmin=\v@lX\v@lmax=\v@lY\v@leur=\v@lZ%
    \Figg@tXY{-1}\c@lprovec{-3}\vecunit@{-3}{-3}
    \Figg@tXY{-1}\v@lmin=\v@lX\v@lmax=\v@lY%
    \v@leur=\v@lZ\Figg@tXY{-3}\c@lprovec{-1}}} 
\ctr@ln@m\figptsinterlinell
\ctr@ld@f\def\figptsinterlinellDD#1[#2,#3,#4,#5;#6,#7]{\ifGR@cri{\s@uvc@ntr@l\et@tfigptsinterlinellDD%
    \figptcopy#1:/#6/\s@mme=#1\advance\s@mme\@ne\figptcopy\the\s@mme:/#7/%
    \v@lmin=#3\unit@\v@lmax=#4\unit@
    \setc@ntr@l{2}\figptbaryDD-4:[#6,#7;1,1]\figptsrotDD-3=-4,#7/#2,-#5/
    \Figg@tXY{-3}\Figg@tXYa{#2}\advance\v@lX-\v@lXa\advance\v@lY-\v@lYa
    \figvectP-1[-3,-2]\Figg@tXYa{-1}\figvectP-3[-4,#7]\Figptsint@rLE{#1}
    \resetc@ntr@l\et@tfigptsinterlinellDD}\ignorespaces\fi}
\ctr@ld@f\def\figptsinterlinellP#1[#2,#3,#4;#5,#6]{\ifGR@cri{\s@uvc@ntr@l\et@tfigptsinterlinellP%
    \figptcopy#1:/#5/\s@mme=#1\advance\s@mme\@ne\figptcopy\the\s@mme:/#6/\setc@ntr@l{2}%
    \figvectP-1[#2,#3]\vecunit@{-1}{-1}\v@lmin=\result@t
    \figvectP-2[#2,#4]\vecunit@{-2}{-2}\v@lmax=\result@t
    \figptbary-4:[#5,#6;1,1]
    \figvectP-3[#2,-4]\c@lproscal\v@lX[-3,-1]\c@lproscal\v@lY[-3,-2]
    \figvectP-3[-4,#6]\c@lproscal\v@lXa[-3,-1]\c@lproscal\v@lYa[-3,-2]
    \Figptsint@rLE{#1}\resetc@ntr@l\et@tfigptsinterlinellP}\ignorespaces\fi}
\ctr@ld@f\def\Figptsint@rLE#1{%
    \getredf@ctDD\f@ctech(\v@lmin,\v@lmax)%
    \getredf@ctDD\p@rtent(\v@lX,\v@lY)\ifnum\p@rtent>\f@ctech\f@ctech=\p@rtent\fi%
    \getredf@ctDD\p@rtent(\v@lXa,\v@lYa)\ifnum\p@rtent>\f@ctech\f@ctech=\p@rtent\fi%
    \divide\v@lmin\f@ctech\divide\v@lmax\f@ctech\divide\v@lX\f@ctech\divide\v@lY\f@ctech%
    \divide\v@lXa\f@ctech\divide\v@lYa\f@ctech%
    \c@rre=\repdecn@mb\v@lXa\v@lmax\mili@u=\repdecn@mb\v@lYa\v@lmin%
    \getredf@ctDD\f@ctech(\c@rre,\mili@u)%
    \c@rre=\repdecn@mb\v@lX\v@lmax\mili@u=\repdecn@mb\v@lY\v@lmin%
    \getredf@ctDD\p@rtent(\c@rre,\mili@u)\ifnum\p@rtent>\f@ctech\f@ctech=\p@rtent\fi%
    \divide\v@lmin\f@ctech\divide\v@lmax\f@ctech\divide\v@lX\f@ctech\divide\v@lY\f@ctech%
    \divide\v@lXa\f@ctech\divide\v@lYa\f@ctech%
    \v@lmin=\repdecn@mb{\v@lmin}\v@lmin\v@lmax=\repdecn@mb{\v@lmax}\v@lmax%
    \edef\G@xde{\repdecn@mb\v@lmin}\edef\P@xde{\repdecn@mb\v@lmax}%
    \c@rre=-\v@lmax\v@leur=\repdecn@mb\v@lY\v@lY\advance\c@rre\v@leur\c@rre=\G@xde\c@rre%
    \v@leur=\repdecn@mb\v@lX\v@lX\v@leur=\P@xde\v@leur\advance\c@rre\v@leur
    \v@lmin=\repdecn@mb\v@lYa\v@lmin\v@lmax=\repdecn@mb\v@lXa\v@lmax%
    \mili@u=\repdecn@mb\v@lX\v@lmax\advance\mili@u\repdecn@mb\v@lY\v@lmin
    \v@lmax=\repdecn@mb\v@lXa\v@lmax\advance\v@lmax\repdecn@mb\v@lYa\v@lmin
    \ifdim\v@lmax>\epsil@n%
    \maxim@m{\v@leur}{\c@rre}{-\c@rre}\maxim@m{\v@lmin}{\mili@u}{-\mili@u}%
    \maxim@m{\v@leur}{\v@leur}{\v@lmin}\maxim@m{\v@lmin}{\v@lmax}{-\v@lmax}%
    \maxim@m{\v@leur}{\v@leur}{\v@lmin}\p@rtentiere{\p@rtent}{\v@leur}\advance\p@rtent\@ne%
    \divide\c@rre\p@rtent\divide\mili@u\p@rtent\divide\v@lmax\p@rtent%
    \delt@=\repdecn@mb{\mili@u}\mili@u\v@leur=\repdecn@mb{\v@lmax}\c@rre%
    \advance\delt@-\v@leur\ifdim\delt@<\z@\else\sqrt@\delt@\delt@%
    \invers@\v@lmax\v@lmax\edef\Uns@rAp{\repdecn@mb\v@lmax}%
    \v@leur=-\mili@u\advance\v@leur-\delt@\v@leur=\Uns@rAp\v@leur%
    \edef\t@ille{\repdecn@mb\v@leur}\figpttra#1:=-4/\t@ille,-3/\s@mme=#1\advance\s@mme\@ne%
    \v@leur=-\mili@u\advance\v@leur\delt@\v@leur=\Uns@rAp\v@leur%
    \edef\t@ille{\repdecn@mb\v@leur}\figpttra\the\s@mme:=-4/\t@ille,-3/\fi\fi}
\ctr@ln@m\figptsorthoprojline
\ctr@ld@f\def\figptsorthoprojlineDD#1=#2/#3,#4/{\ifGR@cri{\s@uvc@ntr@l\et@tfigptsorthoprojlineDD%
    \setc@ntr@l{2}\figvectPDD-3[#3,#4]\figvectNVDD-4[-3]\resetc@ntr@l{2}%
    \def\list@num{#2}\s@mme=#1\@ecfor\p@int:=\list@num\do{%
    \inters@cDD\the\s@mme:[\p@int,-4;#3,-3]\advance\s@mme\@ne}%
    \resetc@ntr@l\et@tfigptsorthoprojlineDD}\ignorespaces\fi}
\ctr@ld@f\def\figptsorthoprojlineTD#1=#2/#3,#4/{\ifGR@cri{\s@uvc@ntr@l\et@tfigptsorthoprojlineTD%
    \setc@ntr@l{2}\figvectPTD-2[#3,#4]\vecunit@TD{-2}{-2}%
    \def\list@num{#2}\s@mme=#1\@ecfor\p@int:=\list@num\do{%
    \figvectPTD-1[#3,\p@int]\c@lproscalTD\v@leur[-1,-2]%
    \edef\v@lcoef{\repdecn@mb{\v@leur}}\figpttraTD\the\s@mme:=#3/\v@lcoef,-2/%
    \advance\s@mme\@ne}\resetc@ntr@l\et@tfigptsorthoprojlineTD}\ignorespaces\fi}
\ctr@ln@m\figptsorthoprojplane
\ctr@ld@f\def\figptsorthoprojplaneDD{\un@v@ilable{figptsorthoprojplane}}
\ctr@ld@f\def\figptsorthoprojplaneTD#1=#2/#3,#4/{\ifGR@cri{\s@uvc@ntr@l\et@tfigptsorthoprojplane%
    \setc@ntr@l{2}\vecunit@TD{-2}{#4}%
    \def\list@num{#2}\s@mme=#1\@ecfor\p@int:=\list@num\do{\figvectPTD-1[\p@int,#3]%
    \c@lproscalTD\v@leur[-1,-2]\edef\v@lcoef{\repdecn@mb{\v@leur}}%
    \figpttraTD\the\s@mme:=\p@int/\v@lcoef,-2/\advance\s@mme\@ne}%
    \resetc@ntr@l\et@tfigptsorthoprojplane}\ignorespaces\fi}
\ctr@ld@f\def\figptshom#1=#2/#3,#4/{\ifGR@cri{\s@uvc@ntr@l\et@tfigptshom%
    \setc@ntr@l{2}\def\list@num{#2}\s@mme=#1%
    \@ecfor\p@int:=\list@num\do{\figvectP-1[#3,\p@int]%
    \figpttra\the\s@mme:=#3/#4,-1/\advance\s@mme\@ne}%
    \resetc@ntr@l\et@tfigptshom}\ignorespaces\fi}
\ctr@ld@f\def\figptsinv#1=#2/#3,#4/{\ifGR@cri{\s@uvc@ntr@l\et@tfigptsinv%
    \setc@ntr@l{2}\def\list@num{#2}\s@mme=#1%
    \@ecfor\p@int:=\list@num\do{\figvectP-1[#3,\p@int]\Figg@tXY{-1}%
    \getredf@ctB\f@ctech\n@rmeucC{\delt@}{-1}%
    \delt@=\ptT@unit@\delt@\delt@=\ptT@unit@\delt@%
    \invers@{\delt@}{\delt@}\multiply\f@ctech\f@ctech\divide\delt@\f@ctech%
    \delt@=#4\delt@\edef\v@lcoef{\repdecn@mb{\delt@}}\figpttra\the\s@mme:=#3/\v@lcoef,-1/%
    \advance\s@mme\@ne}\resetc@ntr@l\et@tfigptsinv}\ignorespaces\fi}
\ctr@ln@m\figptsrot
\ctr@ld@f\def\figptsrotDD#1=#2/#3,#4/{\ifGR@cri{\s@uvc@ntr@l\et@tfigptsrotDD%
    \c@ssin{\C@}{\S@}{#4}\setc@ntr@l{2}\def\list@num{#2}\s@mme=#1%
    \@ecfor\p@int:=\list@num\do{\figvectPDD-1[#3,\p@int]\Figg@tXY{-1}%
    \v@lXa=\C@\v@lX\advance\v@lXa-\S@\v@lY%
    \v@lYa=\S@\v@lX\advance\v@lYa\C@\v@lY%
    \Figv@ctCreg-1(\v@lXa,\v@lYa)\figpttraDD\the\s@mme:=#3/1,-1/\advance\s@mme\@ne}%
    \resetc@ntr@l\et@tfigptsrotDD}\ignorespaces\fi}
\ctr@ld@f\def\figptsrotTD#1=#2/#3,#4,#5/{\ifGR@cri{\s@uvc@ntr@l\et@tfigptsrotTD%
    \c@ssin{\C@}{\S@}{#4}%
    \setc@ntr@l{2}\def\list@num{#2}\s@mme=#1%
    \@ecfor\p@int:=\list@num\do{\figptorthoprojplaneTD-3:=#3/\p@int,#5/%
    \figvectPTD-2[-3,\p@int]%
    \figvectNVTD-1[#5,-2]\n@rmeucTD\v@leur{-2}\edef\v@lcoef{\repdecn@mb{\v@leur}}%
    \Figg@tXYa{-1}\v@lXa=\v@lcoef\v@lXa\v@lYa=\v@lcoef\v@lYa\v@lZa=\v@lcoef\v@lZa%
    \v@lXa=\S@\v@lXa\v@lYa=\S@\v@lYa\v@lZa=\S@\v@lZa\Figg@tXY{-2}%
    \advance\v@lXa\C@\v@lX\advance\v@lYa\C@\v@lY\advance\v@lZa\C@\v@lZ%
    \Figg@tXY{-3}\advance\v@lXa\v@lX\advance\v@lYa\v@lY\advance\v@lZa\v@lZ%
    \Figp@intregTD\the\s@mme:(\v@lXa,\v@lYa,\v@lZa)\advance\s@mme\@ne}%
    \resetc@ntr@l\et@tfigptsrotTD}\ignorespaces\fi}
\ctr@ln@m\figptssym
\ctr@ld@f\def\figptssymDD#1=#2/#3,#4/{\ifGR@cri{\s@uvc@ntr@l\et@tfigptssymDD%
    \setc@ntr@l{2}\figvectPDD-3[#3,#4]\Figg@tXY{-3}\Figv@ctCreg-4(-\v@lY,\v@lX)%
    \resetc@ntr@l{2}\def\list@num{#2}\s@mme=#1%
    \@ecfor\p@int:=\list@num\do{\inters@cDD-5:[#3,-3;\p@int,-4]\figvectPDD-2[\p@int,-5]%
    \figpttraDD\the\s@mme:=\p@int/2,-2/\advance\s@mme\@ne}%
    \resetc@ntr@l\et@tfigptssymDD}\ignorespaces\fi}
\ctr@ld@f\def\figptssymTD#1=#2/#3,#4/{\ifGR@cri{\s@uvc@ntr@l\et@tfigptssymTD%
    \setc@ntr@l{2}\vecunit@TD{-2}{#4}\def\list@num{#2}\s@mme=#1%
    \@ecfor\p@int:=\list@num\do{\figvectPTD-1[\p@int,#3]%
    \c@lproscalTD\v@leur[-1,-2]\v@leur=2\v@leur\edef\v@lcoef{\repdecn@mb{\v@leur}}%
    \figpttraTD\the\s@mme:=\p@int/\v@lcoef,-2/\advance\s@mme\@ne}%
    \resetc@ntr@l\et@tfigptssymTD}\ignorespaces\fi}
\ctr@ln@m\figptstra
\ctr@ld@f\def\figptstraDD#1=#2/#3,#4/{\ifGR@cri{\Figg@tXYa{#4}\v@lXa=#3\v@lXa\v@lYa=#3\v@lYa%
    \def\list@num{#2}\s@mme=#1\@ecfor\p@int:=\list@num\do{\Figg@tXY{\p@int}%
    \advance\v@lX\v@lXa\advance\v@lY\v@lYa%
    \Figp@intregDD\the\s@mme:(\v@lX,\v@lY)\advance\s@mme\@ne}}\ignorespaces\fi}
\ctr@ld@f\def\figptstraTD#1=#2/#3,#4/{\ifGR@cri{\Figg@tXYa{#4}\v@lXa=#3\v@lXa\v@lYa=#3\v@lYa%
    \v@lZa=#3\v@lZa\def\list@num{#2}\s@mme=#1\@ecfor\p@int:=\list@num\do{\Figg@tXY{\p@int}%
    \advance\v@lX\v@lXa\advance\v@lY\v@lYa\advance\v@lZ\v@lZa%
    \Figp@intregTD\the\s@mme:(\v@lX,\v@lY,\v@lZ)\advance\s@mme\@ne}}\ignorespaces\fi}
\ctr@ln@m\figptvisilimSL
\ctr@ld@f\def\figptvisilimSLDD{\un@v@ilable{figptvisilimSL}}
\ctr@ld@f\def\figptvisilimSLTD#1:#2[#3,#4;#5,#6]{\ifGR@cri{\s@uvc@ntr@l\et@tfigptvisilimSLTD%
    \setc@ntr@l{2}\figvectP-1[#3,#4]\n@rminf{\delt@}{-1}%
    \ifcase\CUR@proj\v@lX=\cxa@\p@\v@lY=-\p@\v@lZ=\cxb@\p@
    \Figv@ctCreg-2(\v@lX,\v@lY,\v@lZ)\figvectP-3[#5,#6]\figvectNV-1[-2,-3]%
    \or\figvectP-1[#5,#6]\vecunitCV@TD{-1}\v@lmin=\v@lX\v@lmax=\v@lY
    \v@leur=\v@lZ\v@lX=\cza@\p@\v@lY=\czb@\p@\v@lZ=\czc@\p@\c@lprovec{-1}%
    \or\c@ley@pt{-2}\figvectN-1[#5,#6,-2]\fi
    \edef\Ai@{#3}\edef\Aj@{#4}\figvectP-2[#5,\Ai@]\c@lproscal\v@leur[-1,-2]%
    \ifdim\v@leur>\z@\p@rtent=\@ne\else\p@rtent=\m@ne\fi%
    \figvectP-2[#5,\Aj@]\c@lproscal\v@leur[-1,-2]%
    \ifdim\p@rtent\v@leur>\z@\figptcopy#1:#2/#3/%
    \message{*** \BS@ figptvisilimSL: points are on the same side.}\else%
    \figptcopy-3:/#3/\figptcopy-4:/#4/%
    \loop\figptbary-5:[-3,-4;1,1]\figvectP-2[#5,-5]\c@lproscal\v@leur[-1,-2]%
    \ifdim\p@rtent\v@leur>\z@\figptcopy-3:/-5/\else\figptcopy-4:/-5/\fi%
    \divide\delt@\tw@\ifdim\delt@>\epsil@n\repeat%
    \figptbary#1:#2[-3,-4;1,1]\fi\resetc@ntr@l\et@tfigptvisilimSLTD}\ignorespaces\fi}
\ctr@ld@f\def\c@ley@pt#1{\t@stp@r\ifitis@K\v@lX=\cza@\p@\v@lY=\czb@\p@\v@lZ=\czc@\p@%
    \Figv@ctCreg-1(\v@lX,\v@lY,\v@lZ)\Figp@intreg-2:(\wd\Bt@rget,\ht\Bt@rget,\dp\Bt@rget)%
    \figpttra#1:=-2/-\disob@intern,-1/\else\end\fi}
\ctr@ld@f\def\t@stp@r{\itis@Ktrue\ifnewt@rgetpt\else\itis@Kfalse%
    \message{*** \BS@ figptvisilimXX: target point undefined.}\fi\ifnewdis@b\else%
    \itis@Kfalse\message{*** \BS@ figptvisilimXX: observation distance undefined.}\fi%
    \ifitis@K\else\message{*** This macro must be called after \BS@ figdrawbegin or after
    having set the missing parameter(s) with \BS@ figset proj()}\fi}
\ctr@ld@f\def\figscan#1(#2,#3){{\s@uvc@ntr@l\et@tfigscan\@psfgetbb{#1}\if@psfbbfound\else%
    \def\@psfllx{0}\def\@psflly{20}\def\@psfurx{540}\def\@psfury{640}\fi\figscan@{#2}{#3}%
    \resetc@ntr@l\et@tfigscan}\ignorespaces}
\ctr@ld@f\def\figscan@#1#2{%
    \unit@=\@ne bp\setc@ntr@l{2}\figsetmark{}%
    \def\minst@p{20pt}%
    \v@lX=\@psfllx\p@\v@lX=\Sc@leFact\v@lX\r@undint\v@lX\v@lX%
    \v@lY=\@psflly\p@\v@lY=\Sc@leFact\v@lY\ifdim\v@lY>\z@\r@undint\v@lY\v@lY\fi%
    \delt@=\@psfury\p@\delt@=\Sc@leFact\delt@%
    \advance\delt@-\v@lY\v@lXa=\@psfurx\p@\v@lXa=\Sc@leFact\v@lXa\v@leur=\minst@p%
    \edef\valv@lY{\repdecn@mb{\v@lY}}\edef\LgTr@it{\the\delt@}%
    \loop\ifdim\v@lX<\v@lXa\edef\valv@lX{\repdecn@mb{\v@lX}}%
    \figptDD -1:(\valv@lX,\valv@lY)\figwriten -1:\hbox{\vrule height\LgTr@it}(0)%
    \ifdim\v@leur<\minst@p\else\figsetmark{\raise-8bp\hbox{$\scriptscriptstyle\triangle$}}%
    \figwrites -1:\@ffichnb{0}{\valv@lX}(6)\v@leur=\z@\figsetmark{}\fi%
    \advance\v@leur#1pt\advance\v@lX#1pt\repeat%
    \def\minst@p{10pt}%
    \v@lX=\@psfllx\p@\v@lX=\Sc@leFact\v@lX\ifdim\v@lX>\z@\r@undint\v@lX\v@lX\fi%
    \v@lY=\@psflly\p@\v@lY=\Sc@leFact\v@lY\r@undint\v@lY\v@lY%
    \delt@=\@psfurx\p@\delt@=\Sc@leFact\delt@%
    \advance\delt@-\v@lX\v@lYa=\@psfury\p@\v@lYa=\Sc@leFact\v@lYa\v@leur=\minst@p%
    \edef\valv@lX{\repdecn@mb{\v@lX}}\edef\LgTr@it{\the\delt@}%
    \loop\ifdim\v@lY<\v@lYa\edef\valv@lY{\repdecn@mb{\v@lY}}%
    \figptDD -1:(\valv@lX,\valv@lY)\figwritee -1:\vbox{\hrule width\LgTr@it}(0)%
    \ifdim\v@leur<\minst@p\else\figsetmark{$\triangleright$\kern4bp}%
    \figwritew -1:\@ffichnb{0}{\valv@lY}(6)\v@leur=\z@\figsetmark{}\fi%
    \advance\v@leur#2pt\advance\v@lY#2pt\repeat}
\ctr@ld@f
\ctr@ld@f\def\figscan@E#1(#2,#3){{\s@uvc@ntr@l\et@tfigscan@E%
    \Figdisc@rdLTS{#1}{\t@xt@}\pdfximage{\t@xt@}%
    \setbox\Gb@x=\hbox{\pdfrefximage\pdflastximage}%
    \edef\@psfllx{0}\v@lY=-\dp\Gb@x\edef\@psflly{\repdecn@mb{\v@lY}}%
    \edef\@psfurx{\repdecn@mb{\wd\Gb@x}}%
    \v@lY=\dp\Gb@x\advance\v@lY\ht\Gb@x\edef\@psfury{\repdecn@mb{\v@lY}}%
    \figscan@{#2}{#3}\resetc@ntr@l\et@tfigscan@E}\ignorespaces}
\ctr@ld@f\def\figshowpts[#1,#2]{{\figsetmark{$\bullet$}\figsetptname{\bf ##1}%
    \p@rtent=#2\relax\ifnum\p@rtent<\z@\p@rtent=\z@\fi%
    \s@mme=#1\relax\ifnum\s@mme<\z@\s@mme=\z@\fi%
    \loop\ifnum\s@mme<\p@rtent\pt@rvect{\s@mme}%
    \ifitis@K\figwriten{\the\s@mme}:(4pt)\fi\advance\s@mme\@ne\repeat%
    \pt@rvect{\s@mme}\ifitis@K\figwriten{\the\s@mme}:(4pt)\fi}\ignorespaces}
\ctr@ld@f\def\pt@rvect#1{\set@bjc@de{#1}%
    \expandafter\expandafter\expandafter\inqpt@rvec\csname\objc@de\endcsname:}
\ctr@ld@f\def\inqpt@rvec#1#2:{\if#1\C@dCl@spt\itis@Ktrue\else\itis@Kfalse\fi}
\ctr@ld@f\def\figshowsettings{{%
    \immediate\write16{====================================================================}%
    \immediate\write16{ Current settings are (DDV means "with dynamic default value"):}%
    \immediate\write16{ --- GENERAL ---}%
    \immediate\write16{Scale factor and Unit = \unit@util\space (\the\unit@)
     \space -> \BS@ figinit{ScaleFactorUnit}}%
    \immediate\write16{Update mode = \ifGRupdatem@de yes\else no\fi
     \space-> \BS@ figset(update=yes/no) or \BS@ figsetdefault(update=yes/no)}%
    \immediate\write16{ --- WRITING ---}%
    \immediate\write16{Implicit point name = \ptn@me{i} \space-> \BS@ figset write(ptname={Name})}%
    \immediate\write16{Point marker = \the\c@nsymb \space -> \BS@ figset write(mark=Mark)}%
    \immediate\write16{Print rounded coordinates = \ifr@undcoord yes\else no\fi
     \space-> \BS@ figset write(roundcoord=yes/no)}%
    \immediate\write16{ --- GRAPHICAL (general) ---}%
    \immediate\write16{Color = \CUR@color \space-> \BS@ figset(color=ColorDefinition)}%
    \immediate\write16{Filling mode = \iffillm@de yes\else no\fi
     \space-> \BS@ figset(fillmode=yes/no)}%
    \immediate\write16{Line join = \CUR@join \space-> \BS@ figset(join=miter/round/bevel)}%
    \immediate\write16{Line style = \CUR@dash \space-> \BS@ figset(dash=Index/Pattern)}%
    \immediate\write16{Line width = \CUR@width
     \space-> \BS@ figset(width=real in PostScript units)}%
    \immediate\write16{ --- GRAPHICAL (specific) ---}%
    \immediate\write16{Altitude (all the following attributes are DDV):}%
    \immediate\write16{ Base line color =
     \ifx\DDV@blcolor\D@FTref general color\else\DDV@blcolor\fi
     \space-> \BS@ figset altitude(blcolor=ColorDefinition)}%
    \immediate\write16{ Base line style =
     \ifx\DDV@bldash\D@FTref general style\else\DDV@bldash\fi
     \space-> \BS@ figset altitude(bldash=Index/Pattern)}%
    \immediate\write16{ Base line width =
     \ifx\DDV@blwidth\D@FTref general width\else\DDV@blwidth\fi
     \space-> \BS@ figset altitude(blwidth=real in PostScript units)}%
    \immediate\write16{ Square line color =
     \ifx\DDV@sqcolor\D@FTref general color\else\DDV@sqcolor\fi
     \space-> \BS@ figset altitude(sqcolor=ColorDefinition)}%
    \immediate\write16{ Square line style =
     \ifx\DDV@sqdash\D@FTref general style\else\DDV@sqdash\fi
     \space-> \BS@ figset altitude(sqdash=Index/Pattern)}%
    \immediate\write16{ Square line width =
     \ifx\DDV@sqwidth\D@FTref general width\else\DDV@sqwidth\fi
     \space-> \BS@ figset altitude(sqwidth=real in PostScript units)}%
    \immediate\write16{Arrowhead:}%
    \immediate\write16{ (half-)Angle = \@rrowheadangle
     \space-> \BS@ figset arrowhead(angle=real in degrees)}%
    \immediate\write16{ Filling mode = \if@rrowhfill yes\else no\fi
     \space-> \BS@ figset arrowhead(fillmode=yes/no)}%
    \immediate\write16{ "Outside" = \if@rrowhout yes\else no\fi
     \space-> \BS@ figset arrowhead(out=yes/no)}%
    \immediate\write16{ Length = \@rrowheadlength
     \if@rrowratio\space(not active)\else\space(active)\fi
     \space-> \BS@ figset arrowhead(length=real in user coord.)}%
    \immediate\write16{ Ratio = \@rrowheadratio
     \if@rrowratio\space(active)\else\space(not active)\fi
     \space-> \BS@ figset arrowhead(ratio=real in [0,1])}%
    \immediate\write16{Curve:}%
    \immediate\write16{ Roundness = \curv@roundness
     \space-> \BS@ figset curve(roundness=real in [0,0.5])}%
    \immediate\write16{Flow chart:}%
    \immediate\write16{ Arrow position = \@rrowp@s
     \space-> \BS@ figset flowchart(arrowposition=real in [0,1])}%
    \immediate\write16{ Arrow reference point = \ifcase\@rrowr@fpt start\else end\fi
     \space-> \BS@ figset flowchart(arrowrefpt = start/end)}%
    \immediate\write16{ Background color = \fcbgc@lor
     \space-> \BS@ figset flowchart(bgcolor=ColorDefinition)}%
    \immediate\write16{ Line type = \ifcase\fclin@typ@ curve\else polygon\fi
     \space-> \BS@ figset flowchart(line=polygon/curve)}%
    \immediate\write16{ Padding = (\Xp@dd, \Yp@dd)
     \space-> \BS@ figset flowchart(padding = real in user coord.)}%
    \immediate\write16{\space\space\space\space(or
     \BS@ figset flowchart(xpadding=real, ypadding=real) )}%
    \immediate\write16{ Radius = \fclin@r@d
     \space-> \BS@ figset flowchart(radius=positive real in user coord.)}%
    \immediate\write16{ Shape = \fcsh@pe
     \space-> \BS@ figset flowchart(shape = rectangle, ellipse or lozenge)}%
    \immediate\write16{ Thickness color (DDV) = 
     \ifx\DDV@thickcolor\D@FTref general color\else\DDV@thickcolor\fi
     \space-> \BS@ figset flowchart(thickcolor=ColorDefinition)}%
    \immediate\write16{ Thickness = \thickn@ss
     \space-> \BS@ figset flowchart(thickness = real in user coord.)}%
    \immediate\write16{Mesh:}%
    \immediate\write16{ Diagonal = \c@ntrolmesh
     \space-> \BS@ figset mesh(diag=integer in {-1,0,1})}%
    \immediate\write16{ Lines color (DDV) =
     \ifx\DDV@meshcolor\D@FTref general color\else\DDV@meshcolor\fi
     \space-> \BS@ figset mesh(color=ColorDefinition)}%
    \immediate\write16{ Lines style (DDV) =
     \ifx\DDV@meshdash\D@FTref general style\else\DDV@meshdash\fi
     \space-> \BS@ figset mesh(dash=Index/Pattern)}%
    \immediate\write16{ Lines width (DDV) =
     \ifx\DDV@meshwidth\D@FTref general width\else\DDV@meshwidth\fi
     \space-> \BS@ figset mesh(width=real in PostScript units)}%
    \immediate\write16{Trimesh:}%
    \immediate\write16{ Lines color (DDV) =
     \ifx\DDV@tmeshcolor\D@FTref general color\else\DDV@tmeshcolor\fi
     \space-> \BS@ figset trimesh(color=ColorDefinition)}%
    \immediate\write16{ Lines style (DDV) =
     \ifx\DDV@tmeshdash\D@FTref general style\else\DDV@tmeshdash\fi
     \space-> \BS@ figset trimesh(dash=Index/Pattern)}%
    \immediate\write16{ Lines width (DDV) =
     \ifx\DDV@tmeshwidth\D@FTref general width\else\DDV@tmeshwidth\fi
     \space-> \BS@ figset trimesh(width=real in PostScript units)}%
    \ifTr@isDim%
    \immediate\write16{ --- 3D to 2D PROJECTION ---}%
    \immediate\write16{Projection : \typ@proj \space-> \BS@ figinit{ScaleFactorUnit, ProjType}}%
    \immediate\write16{Longitude (psi) = \v@lPsi \space-> \BS@ figset proj(psi=real in degrees)}%
    \ifcase\CUR@proj\immediate\write16{Depth coeff. (Lambda)
     \space = \v@lTheta \space-> \BS@ figset proj(lambda=real in [0,1])}%
    \else\immediate\write16{Latitude (theta)
     \space = \v@lTheta \space-> \BS@ figset proj(theta=real in degrees)}%
    \fi%
    \ifnum\CUR@proj=\tw@%
    \immediate\write16{Observation distance = \disob@unit
     \space-> \BS@ figset proj(dist=real in user coord.)}%
    \immediate\write16{Target point = \t@rgetpt \space-> \BS@ figset proj(targetpt=pt number)}%
     \v@lX=\ptT@unit@\wd\Bt@rget\v@lY=\ptT@unit@\ht\Bt@rget\v@lZ=\ptT@unit@\dp\Bt@rget%
    \immediate\write16{ Its coordinates are
     (\repdecn@mb{\v@lX}, \repdecn@mb{\v@lY}, \repdecn@mb{\v@lZ})}%
    \fi%
    \fi%
    \immediate\write16{====================================================================}%
    \ignorespaces}}
\ctr@ln@w{newif}\ifitis@vect@r
\ctr@ld@f\def\figvectC#1(#2,#3){{\itis@vect@rtrue\figpt#1:(#2,#3)}\ignorespaces}
\ctr@ld@f\def\Figv@ctCreg#1(#2,#3){{\itis@vect@rtrue\Figp@intreg#1:(#2,#3)}\ignorespaces}
\ctr@ln@m\figvectDBezier
\ctr@ld@f\def\figvectDBezierDD#1:#2,#3[#4,#5,#6,#7]{\ifGR@cri{\s@uvc@ntr@l\et@tfigvectDBezierDD%
    \FigvectDBezier@#2,#3[#4,#5,#6,#7]\v@lX=\c@ef\v@lX\v@lY=\c@ef\v@lY%
    \Figv@ctCreg#1(\v@lX,\v@lY)\resetc@ntr@l\et@tfigvectDBezierDD}\ignorespaces\fi}
\ctr@ld@f\def\figvectDBezierTD#1:#2,#3[#4,#5,#6,#7]{\ifGR@cri{\s@uvc@ntr@l\et@tfigvectDBezierTD%
    \FigvectDBezier@#2,#3[#4,#5,#6,#7]\v@lX=\c@ef\v@lX\v@lY=\c@ef\v@lY\v@lZ=\c@ef\v@lZ%
    \Figv@ctCreg#1(\v@lX,\v@lY,\v@lZ)\resetc@ntr@l\et@tfigvectDBezierTD}\ignorespaces\fi}
\ctr@ld@f\def\FigvectDBezier@#1,#2[#3,#4,#5,#6]{\setc@ntr@l{2}%
    \edef\T@{#2}\v@leur=\p@\advance\v@leur-#2pt\edef\UNmT@{\repdecn@mb{\v@leur}}%
    \ifnum#1=\tw@\def\c@ef{6}\else\def\c@ef{3}\fi%
    \figptcopy-4:/#3/\figptcopy-3:/#4/\figptcopy-2:/#5/\figptcopy-1:/#6/%
    \l@mbd@un=-4 \l@mbd@de=-\thr@@\p@rtent=\m@ne\c@lDecast%
    \ifnum#1=\tw@\c@lDCDeux{-4}{-3}\c@lDCDeux{-3}{-2}\c@lDCDeux{-4}{-3}\else%
    \l@mbd@un=-4 \l@mbd@de=-\thr@@\p@rtent=-\tw@\c@lDecast%
    \c@lDCDeux{-4}{-3}\fi\Figg@tXY{-4}}
\ctr@ln@m\c@lDCDeux
\ctr@ld@f\def\c@lDCDeuxDD#1#2{\Figg@tXY{#2}\Figg@tXYa{#1}%
    \advance\v@lX-\v@lXa\advance\v@lY-\v@lYa\Figp@intregDD#1:(\v@lX,\v@lY)}
\ctr@ld@f\def\c@lDCDeuxTD#1#2{\Figg@tXY{#2}\Figg@tXYa{#1}\advance\v@lX-\v@lXa%
    \advance\v@lY-\v@lYa\advance\v@lZ-\v@lZa\Figp@intregTD#1:(\v@lX,\v@lY,\v@lZ)}
\ctr@ln@m\figvectN
\ctr@ld@f\def\figvectNDD#1[#2,#3]{\ifGR@cri{\Figg@tXYa{#2}\Figg@tXY{#3}%
    \advance\v@lX-\v@lXa\advance\v@lY-\v@lYa%
    \Figv@ctCreg#1(-\v@lY,\v@lX)}\ignorespaces\fi}
\ctr@ld@f\def\figvectNTD#1[#2,#3,#4]{\ifGR@cri{\vecunitC@TD[#2,#4]\v@lmin=\v@lX\v@lmax=\v@lY%
    \v@leur=\v@lZ\vecunitC@TD[#2,#3]\c@lprovec{#1}}\ignorespaces\fi}
\ctr@ln@m\figvectNV
\ctr@ld@f\def\figvectNVDD#1[#2]{\ifGR@cri{\Figg@tXY{#2}\Figv@ctCreg#1(-\v@lY,\v@lX)}\ignorespaces\fi}
\ctr@ld@f\def\figvectNVTD#1[#2,#3]{\ifGR@cri{\vecunitCV@TD{#3}\v@lmin=\v@lX\v@lmax=\v@lY%
    \v@leur=\v@lZ\vecunitCV@TD{#2}\c@lprovec{#1}}\ignorespaces\fi}
\ctr@ln@m\figvectP
\ctr@ld@f\def\figvectPDD#1[#2,#3]{\ifGR@cri{\Figg@tXYa{#2}\Figg@tXY{#3}%
    \advance\v@lX-\v@lXa\advance\v@lY-\v@lYa%
    \Figv@ctCreg#1(\v@lX,\v@lY)}\ignorespaces\fi}
\ctr@ld@f\def\figvectPTD#1[#2,#3]{\ifGR@cri{\Figg@tXYa{#2}\Figg@tXY{#3}%
    \advance\v@lX-\v@lXa\advance\v@lY-\v@lYa\advance\v@lZ-\v@lZa%
    \Figv@ctCreg#1(\v@lX,\v@lY,\v@lZ)}\ignorespaces\fi}
\ctr@ln@m\figvectU
\ctr@ld@f\def\figvectUDD#1[#2]{\ifGR@cri{\n@rmeuc\v@leur{#2}\invers@\v@leur\v@leur%
    \delt@=\repdecn@mb{\v@leur}\unit@\edef\v@ldelt@{\repdecn@mb{\delt@}}%
    \Figg@tXY{#2}\v@lX=\v@ldelt@\v@lX\v@lY=\v@ldelt@\v@lY%
    \Figv@ctCreg#1(\v@lX,\v@lY)}\ignorespaces\fi}
\ctr@ld@f\def\figvectUTD#1[#2]{\ifGR@cri{\n@rmeuc\v@leur{#2}\invers@\v@leur\v@leur%
    \delt@=\repdecn@mb{\v@leur}\unit@\edef\v@ldelt@{\repdecn@mb{\delt@}}%
    \Figg@tXY{#2}\v@lX=\v@ldelt@\v@lX\v@lY=\v@ldelt@\v@lY\v@lZ=\v@ldelt@\v@lZ%
    \Figv@ctCreg#1(\v@lX,\v@lY,\v@lZ)}\ignorespaces\fi}
\ctr@ld@f\def\figvisu#1#2#3{\c@ldefproj\initb@undb@x\xdef\figforTeXFigno{\figforTeXnextFigno}%
    \s@mme=\figforTeXnextFigno\advance\s@mme\@ne\xdef\figforTeXnextFigno{\number\s@mme}%
    \setbox\b@xvisu=\hbox{\ifnum\@utoFN>\z@\figinsert{}\gdef\@utoFInDone{0}\fi\ignorespaces#3}%
    \gdef\@utoFInDone{1}\gdef\@utoFN{0}%
    \v@lXa=-\c@@rdYmin\v@lYa=\c@@rdYmax\advance\v@lYa-\c@@rdYmin%
    \v@lX=\c@@rdXmax\advance\v@lX-\c@@rdXmin%
    \setbox#1=\hbox{#2}\v@lY=-\v@lX\maxim@m{\v@lX}{\v@lX}{\wd#1}%
    \advance\v@lY\v@lX\divide\v@lY\tw@\advance\v@lY-\c@@rdXmin%
    \setbox#1=\vbox{\parindent\z@\hsize=\v@lX\vskip\v@lYa%
    \rlap{\hskip\v@lY\smash{\raise\v@lXa\box\b@xvisu}}%
    \def\t@xt@{#2}\ifx\t@xt@\empty\else\medskip\centerline{#2}\fi}\wd#1=\v@lX}
\ctr@ld@f\def\figDecrementFigno{{\xdef\figforTeXnextFigno{\figforTeXFigno}%
    \s@mme=\figforTeXFigno\advance\s@mme\m@ne\xdef\figforTeXFigno{\number\s@mme}}}
\ctr@ln@w{newbox}\Bt@rget\setbox\Bt@rget=\null
\ctr@ln@w{newbox}\BminTD@\setbox\BminTD@=\null
\ctr@ln@w{newbox}\BmaxTD@\setbox\BmaxTD@=\null
\ctr@ln@w{newif}\ifnewt@rgetpt\ctr@ln@w{newif}\ifnewdis@b
\ctr@ld@f\def\b@undb@xTD#1#2#3{%
    \relax\ifdim#1<\wd\BminTD@\global\wd\BminTD@=#1\fi%
    \relax\ifdim#2<\ht\BminTD@\global\ht\BminTD@=#2\fi%
    \relax\ifdim#3<\dp\BminTD@\global\dp\BminTD@=#3\fi%
    \relax\ifdim#1>\wd\BmaxTD@\global\wd\BmaxTD@=#1\fi%
    \relax\ifdim#2>\ht\BmaxTD@\global\ht\BmaxTD@=#2\fi%
    \relax\ifdim#3>\dp\BmaxTD@\global\dp\BmaxTD@=#3\fi}
\ctr@ld@f\def\c@ldefdisob{{\ifdim\wd\BminTD@<\maxdimen\v@leur=\wd\BmaxTD@\advance\v@leur-\wd\BminTD@%
    \delt@=\ht\BmaxTD@\advance\delt@-\ht\BminTD@\maxim@m{\v@leur}{\v@leur}{\delt@}%
    \delt@=\dp\BmaxTD@\advance\delt@-\dp\BminTD@\maxim@m{\v@leur}{\v@leur}{\delt@}%
    \v@leur=5\v@leur\else\v@leur=800pt\fi\c@ldefdisob@{\v@leur}}}
\ctr@ln@m\disob@intern
\ctr@ln@m\disob@
\ctr@ln@m\divf@ctproj
\ctr@ld@f\def\c@ldefdisob@#1{{\v@leur=#1\ifdim\v@leur<\p@\v@leur=800pt\fi%
    \xdef\disob@intern{\repdecn@mb{\v@leur}}%
    \delt@=\ptT@unit@\v@leur\xdef\disob@unit{\repdecn@mb{\delt@}}%
    \f@ctech=\@ne\loop\ifdim\v@leur>\t@n pt\divide\v@leur\t@n\multiply\f@ctech\t@n\repeat%
    \xdef\disob@{\repdecn@mb{\v@leur}}\xdef\divf@ctproj{\the\f@ctech}}%
    \global\newdis@btrue}
\ctr@ln@m\t@rgetpt
\ctr@ld@f\def\c@ldeft@rgetpt{\newt@rgetpttrue\def\t@rgetpt{CenterBoundBox}{%
    \delt@=\wd\BmaxTD@\advance\delt@-\wd\BminTD@\divide\delt@\tw@%
    \v@leur=\wd\BminTD@\advance\v@leur\delt@\global\wd\Bt@rget=\v@leur%
    \delt@=\ht\BmaxTD@\advance\delt@-\ht\BminTD@\divide\delt@\tw@%
    \v@leur=\ht\BminTD@\advance\v@leur\delt@\global\ht\Bt@rget=\v@leur%
    \delt@=\dp\BmaxTD@\advance\delt@-\dp\BminTD@\divide\delt@\tw@%
    \v@leur=\dp\BminTD@\advance\v@leur\delt@\global\dp\Bt@rget=\v@leur}}
\ctr@ln@m\c@ldefproj
\ctr@ld@f\def\c@ldefprojTD{\ifnewt@rgetpt\else\c@ldeft@rgetpt\fi\ifnewdis@b\else\c@ldefdisob\fi}
\ctr@ld@f\def\c@lprojcav{
    \v@lZa=\cxa@\v@lY\advance\v@lX\v@lZa%
    \v@lZa=\cxb@\v@lY\v@lY=\v@lZ\advance\v@lY\v@lZa\ignorespaces}
\ctr@ln@m\v@lcoef
\ctr@ld@f\def\c@lprojrea{
    \advance\v@lX-\wd\Bt@rget\advance\v@lY-\ht\Bt@rget\advance\v@lZ-\dp\Bt@rget%
    \v@lZa=\cza@\v@lX\advance\v@lZa\czb@\v@lY\advance\v@lZa\czc@\v@lZ%
    \divide\v@lZa\divf@ctproj\advance\v@lZa\disob@ pt\invers@{\v@lZa}{\v@lZa}%
    \v@lZa=\disob@\v@lZa\edef\v@lcoef{\repdecn@mb{\v@lZa}}%
    \v@lXa=\cxa@\v@lX\advance\v@lXa\cxb@\v@lY\v@lXa=\v@lcoef\v@lXa%
    \v@lY=\cyb@\v@lY\advance\v@lY\cya@\v@lX\advance\v@lY\cyc@\v@lZ%
    \v@lY=\v@lcoef\v@lY\v@lX=\v@lXa\ignorespaces}
\ctr@ld@f\def\c@lprojort{
    \v@lXa=\cxa@\v@lX\advance\v@lXa\cxb@\v@lY%
    \v@lY=\cyb@\v@lY\advance\v@lY\cya@\v@lX\advance\v@lY\cyc@\v@lZ%
    \v@lX=\v@lXa\ignorespaces}
\ctr@ld@f\def\Figptpr@j#1:#2/#3/{{\Figg@tXY{#3}\superc@lprojSP%
    \Figp@intregDD#1:{#2}(\v@lX,\v@lY)}\ignorespaces}
\ctr@ln@m\figsetobdist
\ctr@ld@f\def\figsetobdistDD{\un@v@ilable{figsetobdist}}
\ctr@ld@f\def\figsetobdistTD(#1){{\ifCUR@PS\W@rnmesIgn{figset proj(dist=...)}%
    \else\v@leur=#1\unit@\c@ldefdisob@{\v@leur}\fi}\ignorespaces}
\ctr@ln@m\c@lprojSP
\ctr@ln@m\CUR@proj
\ctr@ln@m\typ@proj
\ctr@ln@m\superc@lprojSP
\ctr@ld@f\def\Figs@tproj#1{%
    \if#13 \def@ultproj\else\if#1c\def@ultproj%
    \else\if#1o\xdef\CUR@proj{1}\xdef\typ@proj{orthogonal}%
         \figsetviewTD(\def@ultpsi,\def@ulttheta)%
         \global\let\c@lprojSP=\c@lprojort\global\let\superc@lprojSP=\c@lprojort%
    \else\if#1r\xdef\CUR@proj{2}\xdef\typ@proj{realistic}%
         \figsetviewTD(\def@ultpsi,\def@ulttheta)%
         \global\let\c@lprojSP=\c@lprojrea\global\let\superc@lprojSP=\c@lprojrea%
    \else\def@ultproj\message{*** Unknown projection. Cavalier projection assumed.}%
    \fi\fi\fi\fi}
\ctr@ld@f\def\def@ultproj{\xdef\CUR@proj{0}\xdef\typ@proj{cavalier}\figsetviewTD(\def@ultpsi,0.5)%
         \global\let\c@lprojSP=\c@lprojcav\global\let\superc@lprojSP=\c@lprojcav}
\ctr@ln@m\figsettarget
\ctr@ld@f\def\figsettargetDD{\un@v@ilable{figsettarget}}
\ctr@ld@f\def\figsettargetTD[#1]{{\ifCUR@PS\W@rnmesIgn{figset proj(targetpt=...)}%
    \else\global\newt@rgetpttrue\xdef\t@rgetpt{#1}\Figg@tXY{#1}\global\wd\Bt@rget=\v@lX%
    \global\ht\Bt@rget=\v@lY\global\dp\Bt@rget=\v@lZ\fi}\ignorespaces}
\ctr@ln@m\figsetview
\ctr@ld@f\def\figsetviewDD{\un@v@ilable{figsetview}}
\ctr@ld@f\def\figsetviewTD(#1){\ifCUR@PS\W@rnmesIgn{figset proj(Psi|Theta|Lambda=...)}%
     \else\Figsetview@#1,:\fi\ignorespaces}
\ctr@ld@f\def\Figsetview@#1,#2:{{\xdef\v@lPsi{#1}\def\t@xt@{#2}%
    \ifx\t@xt@\empty\def\@rgdeux{\v@lTheta}\else\X@rgdeux@#2\fi%
    \c@ssin{\costhet@}{\sinthet@}{#1}\v@lmin=\costhet@ pt\v@lmax=\sinthet@ pt%
    \ifcase\CUR@proj%
    \v@leur=\@rgdeux\v@lmin\xdef\cxa@{\repdecn@mb{\v@leur}}%
    \v@leur=\@rgdeux\v@lmax\xdef\cxb@{\repdecn@mb{\v@leur}}\v@leur=\@rgdeux pt%
    \relax\ifdim\v@leur>\p@\message{*** Lambda too large ! See \BS@ figset proj() !}\fi%
    \else%
    \v@lmax=-\v@lmax\xdef\cxa@{\repdecn@mb{\v@lmax}}\xdef\cxb@{\costhet@}%
    \ifx\t@xt@\empty\edef\@rgdeux{\def@ulttheta}\fi\c@ssin{\C@}{\S@}{\@rgdeux}%
    \v@lmax=-\S@ pt%
    \v@leur=\v@lmax\v@leur=\costhet@\v@leur\xdef\cya@{\repdecn@mb{\v@leur}}%
    \v@leur=\v@lmax\v@leur=\sinthet@\v@leur\xdef\cyb@{\repdecn@mb{\v@leur}}%
    \xdef\cyc@{\C@}\v@lmin=-\C@ pt%
    \v@leur=\v@lmin\v@leur=\costhet@\v@leur\xdef\cza@{\repdecn@mb{\v@leur}}%
    \v@leur=\v@lmin\v@leur=\sinthet@\v@leur\xdef\czb@{\repdecn@mb{\v@leur}}%
    \xdef\czc@{\repdecn@mb{\v@lmax}}\fi%
    \xdef\v@lTheta{\@rgdeux}}}
\ctr@ld@f\def\def@ultpsi{40}
\ctr@ld@f\def\def@ulttheta{25}
\ctr@ln@m\l@debut
\ctr@ln@m\n@mref
\ctr@ld@f\def\Figsetpr@j#1=#2|{\keln@mtr#1|%
    \def\n@mref{dep}\ifx\l@debut\n@mref\Figsetd@p{#2}\else
    \def\n@mref{dis}\ifx\l@debut\n@mref%
     \ifnum\CUR@proj=\tw@\figsetobdist(#2)\else\Figset@rr\fi\else
    \def\n@mref{lam}\ifx\l@debut\n@mref\Figsetd@p{#2}\else
    \def\n@mref{lat}\ifx\l@debut\n@mref\Figsetth@{#2}\else
    \def\n@mref{lon}\ifx\l@debut\n@mref\figsetview(#2)\else
    \def\n@mref{psi}\ifx\l@debut\n@mref\figsetview(#2)\else
    \def\n@mref{tar}\ifx\l@debut\n@mref%
     \ifnum\CUR@proj=\tw@\figsettarget[#2]\else\Figset@rr\fi\else
    \def\n@mref{the}\ifx\l@debut\n@mref\Figsetth@{#2}\else
    \W@rnmesAttr{figset proj}{#1}\fi\fi\fi\fi\fi\fi\fi\fi}
\ctr@ld@f\def\Figsetd@p#1{\ifnum\CUR@proj=\z@\figsetview(\v@lPsi,#1)\else\Figset@rr\fi}
\ctr@ld@f\def\Figsetth@#1{\ifnum\CUR@proj=\z@\Figset@rr\else\figsetview(\v@lPsi,#1)\fi}
\ctr@ld@f\def\Figset@rr{\message{*** \BS@ figset proj(): Attribute "\n@mref" ignored, incompatible
    with current projection}}
\ctr@ld@f\def\initb@undb@xTD{\wd\BminTD@=\maxdimen\ht\BminTD@=\maxdimen\dp\BminTD@=\maxdimen%
    \wd\BmaxTD@=-\maxdimen\ht\BmaxTD@=-\maxdimen\dp\BmaxTD@=-\maxdimen}
\ctr@ln@w{newbox}\Gb@x      
\ctr@ln@w{newbox}\Gb@xSC    
\ctr@ln@w{newtoks}\c@nsymb  
\ctr@ln@w{newif}\ifr@undcoord\ctr@ln@w{newif}\ifunitpr@sent
\ctr@ld@f\def\unssqrttw@{0.707106 }
\ctr@ld@f\def\figAst{\raise-1.15ex\hbox{$\ast$}}
\ctr@ld@f\def\figBullet{\raise-1.15ex\hbox{$\bullet$}}
\ctr@ld@f\def\figCirc{\raise-1.15ex\hbox{$\circ$}}
\ctr@ld@f\def\figDiamond{\raise-1.15ex\hbox{$\diamond$}}%
\ctr@ld@f\def\boxit#1#2{\leavevmode\hbox{\vrule\vbox{\hrule\vglue#1%
    \vtop{\hbox{\kern#1{#2}\kern#1}\vglue#1\hrule}}\vrule}}
\ctr@ld@f
\ctr@ld@f
\ctr@ld@f\def\c@nterpt{\ignorespaces%
    \kern-.5\wd\Gb@xSC%
    \raise-.5\ht\Gb@xSC\rlap{\hbox{\raise.5\dp\Gb@xSC\hbox{\copy\Gb@xSC}}}%
    \kern .5\wd\Gb@xSC\ignorespaces}
\ctr@ld@f\def\b@undb@xSC#1#2{{\v@lXa=#1\v@lYa=#2%
    \v@leur=\ht\Gb@xSC\advance\v@leur\dp\Gb@xSC%
    \advance\v@lXa-.5\wd\Gb@xSC\advance\v@lYa-.5\v@leur\b@undb@x{\v@lXa}{\v@lYa}%
    \advance\v@lXa\wd\Gb@xSC\advance\v@lYa\v@leur\b@undb@x{\v@lXa}{\v@lYa}}}
\ctr@ln@m\Dist@n
\ctr@ln@m\l@suite
\ctr@ld@f\def\@keldist#1#2{\edef\Dist@n{#2}\y@tiunit{\Dist@n}%
    \ifunitpr@sent#1=\Dist@n\else#1=\Dist@n\unit@\fi}
\ctr@ld@f\def\y@tiunit#1{\unitpr@sentfalse\expandafter\y@tiunit@#1:}
\ctr@ld@f\def\y@tiunit@#1#2:{\ifcat#1a\unitpr@senttrue\else\def\l@suite{#2}%
    \ifx\l@suite\empty\else\y@tiunit@#2:\fi\fi}
\ctr@ln@m\figcoord
\ctr@ld@f\def\figcoordDD#1{{\v@lX=\ptT@unit@\v@lX\v@lY=\ptT@unit@\v@lY%
    \ifr@undcoord\ifcase#1\v@leur=0.5pt\or\v@leur=0.05pt\or\v@leur=0.005pt%
    \or\v@leur=0.0005pt\else\v@leur=\z@\fi%
    \ifdim\v@lX<\z@\advance\v@lX-\v@leur\else\advance\v@lX\v@leur\fi%
    \ifdim\v@lY<\z@\advance\v@lY-\v@leur\else\advance\v@lY\v@leur\fi\fi%
    (\@ffichnb{#1}{\repdecn@mb{\v@lX}},\ifmmode\else\thinspace\fi%
    \@ffichnb{#1}{\repdecn@mb{\v@lY}})}}
\ctr@ld@f\def\@ffichnb#1#2{{\def\@@ffich{\@ffich#1(}\edef\n@mbre{#2}%
    \expandafter\@@ffich\n@mbre)}}
\ctr@ld@f\def\@ffich#1(#2.#3){{#2\ifnum#1>\z@.\fi\def\dig@ts{#3}\s@mme=\z@%
    \loop\ifnum\s@mme<#1\expandafter\@ffichdec\dig@ts:\advance\s@mme\@ne\repeat}}
\ctr@ld@f\def\@ffichdec#1#2:{\relax#1\def\dig@ts{#20}}
\ctr@ld@f\def\figcoordTD#1{{\v@lX=\ptT@unit@\v@lX\v@lY=\ptT@unit@\v@lY\v@lZ=\ptT@unit@\v@lZ%
    \ifr@undcoord\ifcase#1\v@leur=0.5pt\or\v@leur=0.05pt\or\v@leur=0.005pt%
    \or\v@leur=0.0005pt\else\v@leur=\z@\fi%
    \ifdim\v@lX<\z@\advance\v@lX-\v@leur\else\advance\v@lX\v@leur\fi%
    \ifdim\v@lY<\z@\advance\v@lY-\v@leur\else\advance\v@lY\v@leur\fi%
    \ifdim\v@lZ<\z@\advance\v@lZ-\v@leur\else\advance\v@lZ\v@leur\fi\fi%
    (\@ffichnb{#1}{\repdecn@mb{\v@lX}},\ifmmode\else\thinspace\fi%
     \@ffichnb{#1}{\repdecn@mb{\v@lY}},\ifmmode\else\thinspace\fi%
     \@ffichnb{#1}{\repdecn@mb{\v@lZ}})}}
\ctr@ld@f\def\figsetroundcoord#1{\expandafter\Figsetr@undcoord#1:\ignorespaces}
\ctr@ld@f\def\Figsetr@undcoord#1#2:{\if#1n\r@undcoordfalse\else\r@undcoordtrue\fi}
\ctr@ld@f\def\Figsetwr@te#1=#2|{\keln@mun#1|%
    \def\n@mref{m}\ifx\l@debut\n@mref\figsetmark{#2}\else
    \def\n@mref{p}\ifx\l@debut\n@mref\figsetptname{#2}\else
    \def\n@mref{r}\ifx\l@debut\n@mref\figsetroundcoord{#2}\else
    \W@rnmesAttr{figset write}{#1}\fi\fi\fi}
\ctr@ld@f\def\figsetmark#1{\c@nsymb={#1}\setbox\Gb@xSC=\hbox{\the\c@nsymb}\ignorespaces}
\ctr@ln@m\ptn@me
\ctr@ld@f\def\figsetptname#1{\def\ptn@me##1{#1}\ignorespaces}
\ctr@ld@f\def\FigWrit@L#1:#2(#3,#4){\ignorespaces\@keldist\v@leur{#3}\@keldist\delt@{#4}%
    \C@rp@r@m\def\list@num{#1}\@ecfor\p@int:=\list@num\do{\FigWrit@pt{\p@int}{#2}}}
\ctr@ld@f\def\FigWrit@pt#1#2{\FigWp@r@m{#1}{#2}\Vc@rrect\figWp@si%
    \ifdim\wd\Gb@xSC>\z@\b@undb@xSC{\v@lX}{\v@lY}\fi\figWBB@x}
\ctr@ld@f\def\FigWp@r@m#1#2{\Figg@tXY{#1}%
    \setbox\Gb@x=\hbox{\def\t@xt@{#2}\ifx\t@xt@\empty\Figg@tT{#1}\else#2\fi}\c@lprojSP}
\ctr@ld@f\let\Vc@rrect=\relax
\ctr@ld@f\let\C@rp@r@m=\relax
\ctr@ld@f\def\figwrite[#1]#2{{\ignorespaces\def\list@num{#1}\@ecfor\p@int:=\list@num\do{%
    \setbox\Gb@x=\hbox{\def\t@xt@{#2}\ifx\t@xt@\empty\Figg@tT{\p@int}\else#2\fi}%
    \Figwrit@{\p@int}}}\ignorespaces}
\ctr@ld@f\def\Figwrit@#1{\Figg@tXY{#1}\c@lprojSP%
    \rlap{\kern\v@lX\raise\v@lY\hbox{\unhcopy\Gb@x}}\v@leur=\v@lY%
    \advance\v@lY\ht\Gb@x\b@undb@x{\v@lX}{\v@lY}\advance\v@lX\wd\Gb@x%
    \v@lY=\v@leur\advance\v@lY-\dp\Gb@x\b@undb@x{\v@lX}{\v@lY}}
\ctr@ld@f\def\figwritec[#1]#2{{\ignorespaces\def\list@num{#1}%
    \@ecfor\p@int:=\list@num\do{\Figwrit@c{\p@int}{#2}}}\ignorespaces}
\ctr@ld@f\def\Figwrit@c#1#2{\FigWp@r@m{#1}{#2}%
    \rlap{\kern\v@lX\raise\v@lY\hbox{\rlap{\kern-.5\wd\Gb@x%
    \raise-.5\ht\Gb@x\hbox{\raise.5\dp\Gb@x\hbox{\unhcopy\Gb@x}}}}}%
    \v@leur=\ht\Gb@x\advance\v@leur\dp\Gb@x%
    \advance\v@lX-.5\wd\Gb@x\advance\v@lY-.5\v@leur\b@undb@x{\v@lX}{\v@lY}%
    \advance\v@lX\wd\Gb@x\advance\v@lY\v@leur\b@undb@x{\v@lX}{\v@lY}}
\ctr@ld@f\def\figwritep[#1]{{\ignorespaces\def\list@num{#1}\setbox\Gb@x=\hbox{\c@nterpt}%
    \@ecfor\p@int:=\list@num\do{\Figwrit@{\p@int}}}\ignorespaces}
\ctr@ld@f\def\figwritew#1:#2(#3){\figwritegcw#1:{#2}(#3,0pt)}
\ctr@ld@f\def\figwritee#1:#2(#3){\figwritegce#1:{#2}(#3,0pt)}
\ctr@ld@f\def\figwriten#1:#2(#3){{\def\Vc@rrect{\v@lZ=\v@leur\advance\v@lZ\dp\Gb@x}%
    \Figwrit@NS#1:{#2}(#3)}\ignorespaces}
\ctr@ld@f\def\figwrites#1:#2(#3){{\def\Vc@rrect{\v@lZ=-\v@leur\advance\v@lZ-\ht\Gb@x}%
    \Figwrit@NS#1:{#2}(#3)}\ignorespaces}
\ctr@ld@f\def\Figwrit@NS#1:#2(#3){\let\figWp@si=\FigWp@siNS\let\figWBB@x=\FigWBB@xNS%
    \FigWrit@L#1:{#2}(#3,0pt)}
\ctr@ld@f\def\FigWp@siNS{\rlap{\kern\v@lX\raise\v@lY\hbox{\rlap{\kern-.5\wd\Gb@x%
    \raise\v@lZ\hbox{\unhcopy\Gb@x}}\c@nterpt}}}
\ctr@ld@f\def\FigWBB@xNS{\advance\v@lY\v@lZ%
    \advance\v@lY-\dp\Gb@x\advance\v@lX-.5\wd\Gb@x\b@undb@x{\v@lX}{\v@lY}%
    \advance\v@lY\ht\Gb@x\advance\v@lY\dp\Gb@x%
    \advance\v@lX\wd\Gb@x\b@undb@x{\v@lX}{\v@lY}}
\ctr@ld@f\def\figwritenw#1:#2(#3){{\let\figWp@si=\FigWp@sigW\let\figWBB@x=\FigWBB@xgWE%
    \def\C@rp@r@m{\v@leur=\unssqrttw@\v@leur\delt@=\v@leur%
    \ifdim\delt@=\z@\delt@=\epsil@n\fi}\let@xte={-}\FigWrit@L#1:{#2}(#3,0pt)}\ignorespaces}
\ctr@ld@f\def\figwritesw#1:#2(#3){{\let\figWp@si=\FigWp@sigW\let\figWBB@x=\FigWBB@xgWE%
    \def\C@rp@r@m{\v@leur=\unssqrttw@\v@leur\delt@=-\v@leur%
    \ifdim\delt@=\z@\delt@=-\epsil@n\fi}\let@xte={-}\FigWrit@L#1:{#2}(#3,0pt)}\ignorespaces}
\ctr@ld@f\def\figwritene#1:#2(#3){{\let\figWp@si=\FigWp@sigE\let\figWBB@x=\FigWBB@xgWE%
    \def\C@rp@r@m{\v@leur=\unssqrttw@\v@leur\delt@=\v@leur%
    \ifdim\delt@=\z@\delt@=\epsil@n\fi}\let@xte={}\FigWrit@L#1:{#2}(#3,0pt)}\ignorespaces}
\ctr@ld@f\def\figwritese#1:#2(#3){{\let\figWp@si=\FigWp@sigE\let\figWBB@x=\FigWBB@xgWE%
    \def\C@rp@r@m{\v@leur=\unssqrttw@\v@leur\delt@=-\v@leur%
    \ifdim\delt@=\z@\delt@=-\epsil@n\fi}\let@xte={}\FigWrit@L#1:{#2}(#3,0pt)}\ignorespaces}
\ctr@ld@f\def\figwritegw#1:#2(#3,#4){{\let\figWp@si=\FigWp@sigW\let\figWBB@x=\FigWBB@xgWE%
    \let@xte={-}\FigWrit@L#1:{#2}(#3,#4)}\ignorespaces}
\ctr@ld@f\def\figwritege#1:#2(#3,#4){{\let\figWp@si=\FigWp@sigE\let\figWBB@x=\FigWBB@xgWE%
    \let@xte={}\FigWrit@L#1:{#2}(#3,#4)}\ignorespaces}
\ctr@ld@f\def\FigWp@sigW{\v@lXa=\z@\v@lYa=\ht\Gb@x\advance\v@lYa\dp\Gb@x%
    \ifdim\delt@>\z@\relax%
    \rlap{\kern\v@lX\raise\v@lY\hbox{\rlap{\kern-\wd\Gb@x\kern-\v@leur%
          \raise\delt@\hbox{\raise\dp\Gb@x\hbox{\unhcopy\Gb@x}}}\c@nterpt}}%
    \else\ifdim\delt@<\z@\relax\v@lYa=-\v@lYa%
    \rlap{\kern\v@lX\raise\v@lY\hbox{\rlap{\kern-\wd\Gb@x\kern-\v@leur%
          \raise\delt@\hbox{\raise-\ht\Gb@x\hbox{\unhcopy\Gb@x}}}\c@nterpt}}%
    \else\v@lXa=-.5\v@lYa%
    \rlap{\kern\v@lX\raise\v@lY\hbox{\rlap{\kern-\wd\Gb@x\kern-\v@leur%
          \raise-.5\ht\Gb@x\hbox{\raise.5\dp\Gb@x\hbox{\unhcopy\Gb@x}}}\c@nterpt}}%
    \fi\fi}
\ctr@ld@f\def\FigWp@sigE{\v@lXa=\z@\v@lYa=\ht\Gb@x\advance\v@lYa\dp\Gb@x%
    \ifdim\delt@>\z@\relax%
    \rlap{\kern\v@lX\raise\v@lY\hbox{\c@nterpt\kern\v@leur%
          \raise\delt@\hbox{\raise\dp\Gb@x\hbox{\unhcopy\Gb@x}}}}%
    \else\ifdim\delt@<\z@\relax\v@lYa=-\v@lYa%
    \rlap{\kern\v@lX\raise\v@lY\hbox{\c@nterpt\kern\v@leur%
          \raise\delt@\hbox{\raise-\ht\Gb@x\hbox{\unhcopy\Gb@x}}}}%
    \else\v@lXa=-.5\v@lYa%
    \rlap{\kern\v@lX\raise\v@lY\hbox{\c@nterpt\kern\v@leur%
          \raise-.5\ht\Gb@x\hbox{\raise.5\dp\Gb@x\hbox{\unhcopy\Gb@x}}}}%
    \fi\fi}
\ctr@ld@f\def\FigWBB@xgWE{\advance\v@lY\delt@%
    \advance\v@lX\the\let@xte\v@leur\advance\v@lY\v@lXa\b@undb@x{\v@lX}{\v@lY}%
    \advance\v@lX\the\let@xte\wd\Gb@x\advance\v@lY\v@lYa\b@undb@x{\v@lX}{\v@lY}}
\ctr@ld@f\def\figwritegcw#1:#2(#3,#4){{\let\figWp@si=\FigWp@sigcW\let\figWBB@x=\FigWBB@xgcWE%
    \let@xte={-}\FigWrit@L#1:{#2}(#3,#4)}\ignorespaces}
\ctr@ld@f\def\figwritegce#1:#2(#3,#4){{\let\figWp@si=\FigWp@sigcE\let\figWBB@x=\FigWBB@xgcWE%
    \let@xte={}\FigWrit@L#1:{#2}(#3,#4)}\ignorespaces}
\ctr@ld@f\def\FigWp@sigcW{\rlap{\kern\v@lX\raise\v@lY\hbox{\rlap{\kern-\wd\Gb@x\kern-\v@leur%
     \raise-.5\ht\Gb@x\hbox{\raise\delt@\hbox{\raise.5\dp\Gb@x\hbox{\unhcopy\Gb@x}}}}%
     \c@nterpt}}}
\ctr@ld@f\def\FigWp@sigcE{\rlap{\kern\v@lX\raise\v@lY\hbox{\c@nterpt\kern\v@leur%
    \raise-.5\ht\Gb@x\hbox{\raise\delt@\hbox{\raise.5\dp\Gb@x\hbox{\unhcopy\Gb@x}}}}}}
\ctr@ld@f\def\FigWBB@xgcWE{\v@lZ=\ht\Gb@x\advance\v@lZ\dp\Gb@x%
    \advance\v@lX\the\let@xte\v@leur\advance\v@lY\delt@\advance\v@lY.5\v@lZ%
    \b@undb@x{\v@lX}{\v@lY}%
    \advance\v@lX\the\let@xte\wd\Gb@x\advance\v@lY-\v@lZ\b@undb@x{\v@lX}{\v@lY}}
\ctr@ld@f\def\figwritebn#1:#2(#3){{\def\Vc@rrect{\v@lZ=\v@leur}\Figwrit@NS#1:{#2}(#3)}\ignorespaces}
\ctr@ld@f\def\figwritebs#1:#2(#3){{\def\Vc@rrect{\v@lZ=-\v@leur}\Figwrit@NS#1:{#2}(#3)}\ignorespaces}
\ctr@ld@f\def\figwritebw#1:#2(#3){{\let\figWp@si=\FigWp@sibW\let\figWBB@x=\FigWBB@xbWE%
    \let@xte={-}\FigWrit@L#1:{#2}(#3,0pt)}\ignorespaces}
\ctr@ld@f\def\figwritebe#1:#2(#3){{\let\figWp@si=\FigWp@sibE\let\figWBB@x=\FigWBB@xbWE%
    \let@xte={}\FigWrit@L#1:{#2}(#3,0pt)}\ignorespaces}
\ctr@ld@f\def\FigWp@sibW{\rlap{\kern\v@lX\raise\v@lY\hbox{\rlap{\kern-\wd\Gb@x\kern-\v@leur%
          \hbox{\unhcopy\Gb@x}}\c@nterpt}}}
\ctr@ld@f\def\FigWp@sibE{\rlap{\kern\v@lX\raise\v@lY\hbox{\c@nterpt\kern\v@leur%
          \hbox{\unhcopy\Gb@x}}}}
\ctr@ld@f\def\FigWBB@xbWE{\v@lZ=\ht\Gb@x\advance\v@lZ\dp\Gb@x%
    \advance\v@lX\the\let@xte\v@leur\advance\v@lY\ht\Gb@x\b@undb@x{\v@lX}{\v@lY}%
    \advance\v@lX\the\let@xte\wd\Gb@x\advance\v@lY-\v@lZ\b@undb@x{\v@lX}{\v@lY}}
\ctr@ln@w{newread}\frf@g  \ctr@ln@w{newwrite}\fwf@g
\ctr@ln@w{newif}\ifCUR@PS
\ctr@ln@w{newif}\ifGR@cri
\ctr@ln@w{newif}\ifUse@llipse
\ctr@ln@w{newif}\ifGRdebugm@de \GRdebugm@defalse 
\ctr@ln@w{newif}\ifPDFm@ke
\ifx\pdfliteral\undefined\else\ifnum\pdfoutput>\z@\PDFm@ketrue\fi\fi
\ctr@ld@f\def\initPDF@rDVI{%
\ifPDFm@ke
 \let\figscan=\figscan@E
 \let\newGr@FN=\newGr@FNPDF
 \ctr@ld@f\def\c@mcurveto{c}
 \ctr@ld@f\def\c@mfill{f}
 \ctr@ld@f\def\c@mgsave{q}
 \ctr@ld@f\def\c@mgrestore{Q}
 \ctr@ld@f\def\c@mlineto{l}
 \ctr@ld@f\def\c@mmoveto{m}
 \ctr@ld@f\def\c@msetgray{g}     \ctr@ld@f\def\c@msetgrayStroke{G}
 \ctr@ld@f\def\c@msetcmykcolor{k}\ctr@ld@f\def\c@msetcmykcolorStroke{K}
 \ctr@ld@f\def\c@msetrgbcolor{rg}\ctr@ld@f\def\c@msetrgbcolorStroke{RG}
 \ctr@ld@f\def\d@fprimarC@lor{\CUR@color\space\CUR@colorc@md%
               \space\CUR@color\space\CUR@colorc@mdStroke}
 \ctr@ld@f\def\c@msetdash{d}
 \ctr@ld@f\def\c@msetlinejoin{j}
 \ctr@ld@f\def\c@msetlinewidth{w}
 \ctr@ld@f\def\f@gclosestroke{\immediate\write\fwf@g{s}}
 \ctr@ld@f\def\f@gfill{\immediate\write\fwf@g{\fillc@md}}
 \ctr@ld@f\def\f@gnewpath{}
 \ctr@ld@f\def\f@gstroke{\immediate\write\fwf@g{S}}
\else
 \let\figinsertE=\figinsert
 \let\newGr@FN=\newGr@FNDVI
 \ctr@ld@f\def\c@mcurveto{curveto}
 \ctr@ld@f\def\c@mfill{fill}
 \ctr@ld@f\def\c@mgsave{gsave}
 \ctr@ld@f\def\c@mgrestore{grestore}
 \ctr@ld@f\def\c@mlineto{lineto}
 \ctr@ld@f\def\c@mmoveto{moveto}
 \ctr@ld@f\def\c@msetgray{setgray}          \ctr@ld@f\def\c@msetgrayStroke{}
 \ctr@ld@f\def\c@msetcmykcolor{setcmykcolor}\ctr@ld@f\def\c@msetcmykcolorStroke{}
 \ctr@ld@f\def\c@msetrgbcolor{setrgbcolor}  \ctr@ld@f\def\c@msetrgbcolorStroke{}
 \ctr@ld@f\def\d@fprimarC@lor{\CUR@color\space\CUR@colorc@md}
 \ctr@ld@f\def\c@msetdash{setdash}
 \ctr@ld@f\def\c@msetlinejoin{setlinejoin}
 \ctr@ld@f\def\c@msetlinewidth{setlinewidth}
 \ctr@ld@f\def\f@gclosestroke{\immediate\write\fwf@g{closepath\space stroke}}
 \ctr@ld@f\def\f@gfill{\immediate\write\fwf@g{\fillc@md}}
 \ctr@ld@f\def\f@gnewpath{\immediate\write\fwf@g{newpath}}
 \ctr@ld@f\def\f@gstroke{\immediate\write\fwf@g{stroke}}
\fi}
\ctr@ld@f\def\c@pypsfile#1#2{\c@pyfil@{\immediate\write#1}{#2}}
\ctr@ld@f\def\Figinclud@PDF#1#2{\openin\frf@g=#1\pdfliteral{q #2 0 0 #2 0 0 cm}%
    \c@pyfil@{\pdfliteral}{\frf@g}\pdfliteral{Q}\closein\frf@g}
\ctr@ln@w{newif}\ifmored@ta
\ctr@ln@m\bl@nkline
\ctr@ld@f\def\c@pyfil@#1#2{\def\bl@nkline{\par}{\catcode`\%=12
    \loop\ifeof#2\mored@tafalse\else\mored@tatrue\immediate\read#2 to\tr@c
    \ifx\tr@c\bl@nkline\else#1{\tr@c}\fi\fi\ifmored@ta\repeat}}
\ctr@ld@f\def\keln@mun#1#2|{\def\l@debut{#1}\def\l@suite{#2}}
\ctr@ld@f\def\keln@mde#1#2#3|{\def\l@debut{#1#2}\def\l@suite{#3}}
\ctr@ld@f\def\keln@mtr#1#2#3#4|{\def\l@debut{#1#2#3}\def\l@suite{#4}}
\ctr@ld@f\def\keln@mqu#1#2#3#4#5|{\def\l@debut{#1#2#3#4}\def\l@suite{#5}}
\ctr@ld@f\let\@psffilein=\frf@g 
\ctr@ln@w{newif}\if@psffileok    
\ctr@ln@w{newif}\if@psfbbfound   
\ctr@ln@w{newif}\if@psfverbose   
\@psfverbosetrue
\ctr@ln@m\@psfllx \ctr@ln@m\@psflly
\ctr@ln@m\@psfurx \ctr@ln@m\@psfury
\ctr@ln@m\resetcolonc@tcode
\ctr@ld@f\def\@psfgetbb#1{\global\@psfbbfoundfalse%
\global\def\@psfllx{0}\global\def\@psflly{0}%
\global\def\@psfurx{30}\global\def\@psfury{30}%
\openin\@psffilein=#1\relax
\ifeof\@psffilein\errmessage{I couldn't open #1, will ignore it}\else
   \edef\resetcolonc@tcode{\catcode`\noexpand\:\the\catcode`\:\relax}%
   {\@psffileoktrue \chardef\other=12
    \def\do##1{\catcode`##1=\other}\dospecials \catcode`\ =10 \resetcolonc@tcode
    \loop
       \read\@psffilein to \@psffileline
       \ifeof\@psffilein\@psffileokfalse\else
          \expandafter\@psfaux\@psffileline:. \\%
       \fi
   \if@psffileok\repeat
   \if@psfbbfound\else
    \if@psfverbose\message{No bounding box comment in #1; using defaults}\fi\fi
   }\closein\@psffilein\fi}%
\ctr@ln@m\@psfbblit
\ctr@ln@m\@psfpercent
{\catcode`\%=12 \global\let\@psfpercent=
\ctr@ln@m\@psfaux
\long\def\@psfaux#1#2:#3\\{\ifx#1\@psfpercent
   \def\testit{#2}\ifx\testit\@psfbblit
      \@psfgrab #3 . . . \\%
      \@psffileokfalse
      \global\@psfbbfoundtrue
   \fi\else\ifx#1\par\else\@psffileokfalse\fi\fi}%
\ctr@ld@f\def\@psfempty{}%
\ctr@ld@f\def\@psfgrab #1 #2 #3 #4 #5\\{%
\global\def\@psfllx{#1}\ifx\@psfllx\@psfempty
      \@psfgrab #2 #3 #4 #5 .\\\else
   \global\def\@psflly{#2}%
   \global\def\@psfurx{#3}\global\def\@psfury{#4}\fi}%
\ctr@ld@f\def\PSwrit@cmd#1#2#3{{\Figg@tXY{#1}\c@lprojSP\b@undb@x{\v@lX}{\v@lY}%
    \v@lX=\ptT@ptps\v@lX\v@lY=\ptT@ptps\v@lY%
    \immediate\write#3{\repdecn@mb{\v@lX}\space\repdecn@mb{\v@lY}\space#2}}}
\ctr@ld@f\def\PSwrit@cmdS#1#2#3#4#5{{\Figg@tXY{#1}\c@lprojSP\b@undb@x{\v@lX}{\v@lY}%
    \global\result@t=\v@lX\global\result@@t=\v@lY%
    \v@lX=\ptT@ptps\v@lX\v@lY=\ptT@ptps\v@lY%
    \immediate\write#3{\repdecn@mb{\v@lX}\space\repdecn@mb{\v@lY}\space#2}}%
    \edef#4{\the\result@t}\edef#5{\the\result@@t}}
\ctr@ld@f\def\update@ttr#1#2#3{\Figdisc@rdLTS{#3}{\n@mref}%
    \ifx\n@mref\D@FTref#2{#1}\else#2{#3}\fi}
\ctr@ld@f\def\D@FTref{default}
\ctr@ld@f\def\W@rnmesAttr#1#2{%
    \immediate\write16{*** Unknown attribute: \BS@ #1(..., #2=...)}}
\ctr@ld@f\def\W@rnmeskwd#1#2{%
    \immediate\write16{*** Unknown keyword #2 in \BS@ #1}}
\ctr@ld@f\def\W@rnmesIgn#1{\immediate\write16{*** \BS@ #1 is ignored inside a
     \BS@ figdrawbegin-\BS@ figdrawend block.}}
\ctr@ld@f\def\Psset@lti#1=#2|{\keln@mtr#1|%
    \def\n@mref{blc}\ifx\l@debut\n@mref\update@ttr\D@FTref\P@setblcolor{#2}\else
    \def\n@mref{bld}\ifx\l@debut\n@mref\update@ttr\D@FTref\P@setbldash{#2}\else
    \def\n@mref{blw}\ifx\l@debut\n@mref\update@ttr\D@FTref\P@setblwidth{#2}\else
    \def\n@mref{sqc}\ifx\l@debut\n@mref\update@ttr\D@FTref\P@setsqcolor{#2}\else
    \def\n@mref{sqd}\ifx\l@debut\n@mref\update@ttr\D@FTref\P@setsqdash{#2}\else
    \def\n@mref{sqw}\ifx\l@debut\n@mref\update@ttr\D@FTref\P@setsqwidth{#2}\else
    \W@rnmesAttr{figset altitude}{#1}\fi\fi\fi\fi\fi\fi}
\ctr@ln@m\DDV@blcolor
\ctr@ld@f\def\P@setblcolor#1{\edef\DDV@blcolor{#1}}
\ctr@ln@m\DDV@bldash
\ctr@ld@f\def\P@setbldash#1{\edef\DDV@bldash{#1}}
\ctr@ln@m\DDV@blwidth
\ctr@ld@f\def\P@setblwidth#1{\edef\DDV@blwidth{#1}}
\ctr@ln@m\DDV@sqcolor
\ctr@ld@f\def\P@setsqcolor#1{\edef\DDV@sqcolor{#1}}
\ctr@ln@m\DDV@sqdash
\ctr@ld@f\def\P@setsqdash#1{\edef\DDV@sqdash{#1}}
\ctr@ln@m\DDV@sqwidth
\ctr@ld@f\def\P@setsqwidth#1{\edef\DDV@sqwidth{#1}}
\ctr@ld@f\def\figdrawaltitude#1[#2,#3,#4]{{\ifCUR@PS\ifGR@cri%
    \PSc@mment{altitude Square Dim=#1, Triangle=[#2 / #3,#4]}%
    \s@uvc@ntr@l\et@tpsaltitude\resetc@ntr@l{2}\figptorthoprojline-5:=#2/#3,#4/%
    \figvectP -1[#3,#4]\n@rminf{\v@leur}{-1}\vecunit@{-3}{-1}%
    \figvectP -1[-5,#3]\n@rminf{\v@lmin}{-1}\figvectP -2[-5,#4]\n@rminf{\v@lmax}{-2}%
    \ifdim\v@lmin<\v@lmax\s@mme=#3\else\v@lmax=\v@lmin\s@mme=#4\fi%
    \figvectP -4[-5,#2]\vecunit@{-4}{-4}\delt@=#1\unit@%
    \edef\t@ille{\repdecn@mb{\delt@}}\figpttra-1:=-5/\t@ille,-3/%
    \figptstra-3=-5,-1/\t@ille,-4/\figdrawline[#2,-5]%
    \Pss@tspecifSt{color=\DDV@sqcolor,dash=\DDV@sqdash,width=\DDV@sqwidth}%
    \figdrawline[-1,-2,-3]%
    \Psrest@reSt{color=\DDV@sqcolor,dash=\DDV@sqdash,width=\DDV@sqwidth}%
    \ifdim\v@leur<\v@lmax%
    \Pss@tspecifSt{color=\DDV@blcolor,dash=\DDV@bldash,width=\DDV@blwidth}%
    \figdrawline[-5,\the\s@mme]%
    \Psrest@reSt{color=\DDV@blcolor,dash=\DDV@bldash,width=\DDV@blwidth}%
    \fi\PSc@mment{End altitude}\resetc@ntr@l\et@tpsaltitude\fi\fi}}
\ctr@ld@f\def\Ps@rcerc#1;#2(#3,#4){\ellBB@x#1;#2,#2(#3,#4,0)%
    \f@gnewpath{\delt@=#2\unit@\delt@=\ptT@ptps\delt@%
    \BdingB@xfalse%
    \PSwrit@cmd{#1}{\repdecn@mb{\delt@}\space #3\space #4\space arc}{\fwf@g}}}
\ctr@ln@m\figdrawarccirc
\ctr@ld@f\def\Q@arccircDD#1;#2(#3,#4){\ifCUR@PS\ifGR@cri%
    \PSc@mment{arccircDD Center=#1 ; Radius=#2 (Ang1=#3, Ang2=#4)}%
    \iffillm@de\Ps@rcerc#1;#2(#3,#4)%
    \f@gfill%
    \else\Ps@rcerc#1;#2(#3,#4)\f@gstroke\fi%
    \PSc@mment{End arccircDD}\fi\fi}
\ctr@ld@f\def\Q@arccircTD#1,#2,#3;#4(#5,#6){{\ifCUR@PS\ifGR@cri\s@uvc@ntr@l\et@tpsarccircTD%
    \PSc@mment{arccircTD Center=#1,P1=#2,P2=#3 ; Radius=#4 (Ang1=#5, Ang2=#6)}%
    \setc@ntr@l{2}\c@lExtAxes#1,#2,#3(#4)\Q@arcellPATD#1,-4,-5(#5,#6)%
    \PSc@mment{End arccircTD}\resetc@ntr@l\et@tpsarccircTD\fi\fi}}
\ctr@ld@f\def\c@lExtAxes#1,#2,#3(#4){%
    \figvectPTD-5[#1,#2]\vecunit@{-5}{-5}\figvectNTD-4[#1,#2,#3]\vecunit@{-4}{-4}%
    \figvectNVTD-3[-4,-5]\delt@=#4\unit@\edef\r@yon{\repdecn@mb{\delt@}}%
    \figpttra-4:=#1/\r@yon,-5/\figpttra-5:=#1/\r@yon,-3/}
\ctr@ln@m\figdrawarccircP
\ctr@ld@f\def\Q@arccircPDD#1;#2[#3,#4]{{\ifCUR@PS\ifGR@cri\s@uvc@ntr@l\et@tpsarccircPDD%
    \PSc@mment{arccircPDD Center=#1; Radius=#2, [P1=#3, P2=#4]}%
    \Ps@ngleparam#1;#2[#3,#4]\ifdim\v@lmin>\v@lmax\advance\v@lmax\DePI@deg\fi%
    \edef\@ngdeb{\repdecn@mb{\v@lmin}}\edef\@ngfin{\repdecn@mb{\v@lmax}}%
    \figdrawarccirc#1;\r@dius(\@ngdeb,\@ngfin)%
    \PSc@mment{End arccircPDD}\resetc@ntr@l\et@tpsarccircPDD\fi\fi}}
\ctr@ld@f\def\Q@arccircPTD#1;#2[#3,#4,#5]{{\ifCUR@PS\ifGR@cri\s@uvc@ntr@l\et@tpsarccircPTD%
    \PSc@mment{arccircPTD Center=#1; Radius=#2, [P1=#3, P2=#4, P3=#5]}%
    \setc@ntr@l{2}\c@lExtAxes#1,#3,#5(#2)\figdrawarcellPP#1,-4,-5[#3,#4]%
    \PSc@mment{End arccircPTD}\resetc@ntr@l\et@tpsarccircPTD\fi\fi}}
\ctr@ld@f\def\Ps@ngleparam#1;#2[#3,#4]{\setc@ntr@l{2}%
    \figvectPDD-1[#1,#3]\vecunit@{-1}{-1}\Figg@tXY{-1}\arct@n\v@lmin(\v@lX,\v@lY)%
    \figvectPDD-2[#1,#4]\vecunit@{-2}{-2}\Figg@tXY{-2}\arct@n\v@lmax(\v@lX,\v@lY)%
    \v@lmin=\rdT@deg\v@lmin\v@lmax=\rdT@deg\v@lmax%
    \v@leur=#2pt\maxim@m{\mili@u}{-\v@leur}{\v@leur}%
    \edef\r@dius{\repdecn@mb{\mili@u}}}
\ctr@ld@f\def\Ps@rcercBz#1;#2(#3,#4){\Ps@rellBz#1;#2,#2(#3,#4,0)}
\ctr@ld@f\def\Ps@rellBz#1;#2,#3(#4,#5,#6){%
    \ellBB@x#1;#2,#3(#4,#5,#6)\BdingB@xfalse%
    \c@lNbarcs{#4}{#5}\v@leur=#4pt\setc@ntr@l{2}\figptell-13::#1;#2,#3(#4,#6)%
    \f@gnewpath\PSwrit@cmd{-13}{\c@mmoveto}{\fwf@g}%
    \s@mme=\z@\bcl@rellBz#1;#2,#3(#6)\BdingB@xtrue}
\ctr@ld@f\def\bcl@rellBz#1;#2,#3(#4){\relax%
    \ifnum\s@mme<\p@rtent\advance\s@mme\@ne%
    \advance\v@leur\delt@\edef\@ngle{\repdecn@mb\v@leur}\figptell-14::#1;#2,#3(\@ngle,#4)%
    \advance\v@leur\delt@\edef\@ngle{\repdecn@mb\v@leur}\figptell-15::#1;#2,#3(\@ngle,#4)%
    \advance\v@leur\delt@\edef\@ngle{\repdecn@mb\v@leur}\figptell-16::#1;#2,#3(\@ngle,#4)%
    \figptscontrolDD-18[-13,-14,-15,-16]%
    \PSwrit@cmd{-18}{}{\fwf@g}\PSwrit@cmd{-17}{}{\fwf@g}%
    \PSwrit@cmd{-16}{\c@mcurveto}{\fwf@g}%
    \figptcopyDD-13:/-16/\bcl@rellBz#1;#2,#3(#4)\fi}
\ctr@ld@f\def\Ps@rell#1;#2,#3(#4,#5,#6){\ellBB@x#1;#2,#3(#4,#5,#6)%
    \f@gnewpath{\v@lmin=#2\unit@\v@lmin=\ptT@ptps\v@lmin%
    \v@lmax=#3\unit@\v@lmax=\ptT@ptps\v@lmax\BdingB@xfalse%
    \PSwrit@cmd{#1}%
    {#6\space\repdecn@mb{\v@lmin}\space\repdecn@mb{\v@lmax}\space #4\space #5\space ellipse}{\fwf@g}}%
    \global\Use@llipsetrue}
\ctr@ln@m\figdrawarcell
\ctr@ld@f\def\Q@arcellDD#1;#2,#3(#4,#5,#6){{\ifCUR@PS\ifGR@cri%
    \PSc@mment{arcellDD Center=#1 ; XRad=#2, YRad=#3 (Ang1=#4, Ang2=#5, Inclination=#6)}%
    \iffillm@de\Ps@rell#1;#2,#3(#4,#5,#6)%
    \f@gfill%
    \else\Ps@rell#1;#2,#3(#4,#5,#6)\f@gstroke\fi%
    \PSc@mment{End arcellDD}\fi\fi}}
\ctr@ld@f\def\Q@arcellTD#1;#2,#3(#4,#5,#6){{\ifCUR@PS\ifGR@cri\s@uvc@ntr@l\et@tpsarcellTD%
    \PSc@mment{arcellTD Center=#1 ; XRad=#2, YRad=#3 (Ang1=#4, Ang2=#5, Inclination=#6)}%
    \setc@ntr@l{2}\figpttraC -8:=#1/#2,0,0/\figpttraC -7:=#1/0,#3,0/%
    \figvectC -4(0,0,1)\figptsrot -8=-8,-7/#1,#6,-4/\Q@arcellPATD#1,-8,-7(#4,#5)%
    \PSc@mment{End arcellTD}\resetc@ntr@l\et@tpsarcellTD\fi\fi}}
\ctr@ln@m\figdrawarcellPA
\ctr@ld@f\def\Q@arcellPADD#1,#2,#3(#4,#5){{\ifCUR@PS\ifGR@cri\s@uvc@ntr@l\et@tpsarcellPADD%
    \PSc@mment{arcellPADD Center=#1,PtAxis1=#2,PtAxis2=#3 (Ang1=#4, Ang2=#5)}%
    \setc@ntr@l{2}\figvectPDD-1[#1,#2]\vecunit@DD{-1}{-1}\v@lX=\ptT@unit@\result@t%
    \edef\XR@d{\repdecn@mb{\v@lX}}\Figg@tXY{-1}\arct@n\v@lmin(\v@lX,\v@lY)%
    \v@lmin=\rdT@deg\v@lmin\edef\Inclin@{\repdecn@mb{\v@lmin}}%
    \figgetdist\YR@d[#1,#3]\Q@arcellDD#1;\XR@d,\YR@d(#4,#5,\Inclin@)%
    \PSc@mment{End arcellPADD}\resetc@ntr@l\et@tpsarcellPADD\fi\fi}}
\ctr@ld@f\def\Q@arcellPATD#1,#2,#3(#4,#5){{\ifCUR@PS\ifGR@cri\s@uvc@ntr@l\et@tpsarcellPATD%
    \PSc@mment{arcellPATD Center=#1,PtAxis1=#2,PtAxis2=#3 (Ang1=#4, Ang2=#5)}%
    \iffillm@de\Ps@rellPATD#1,#2,#3(#4,#5)%
    \f@gfill%
    \else\Ps@rellPATD#1,#2,#3(#4,#5)\f@gstroke\fi%
    \PSc@mment{End arcellPATD}\resetc@ntr@l\et@tpsarcellPATD\fi\fi}}
\ctr@ld@f\def\Ps@rellPATD#1,#2,#3(#4,#5){\let\c@lprojSP=\relax%
    \setc@ntr@l{2}\figvectPTD-1[#1,#2]\figvectPTD-2[#1,#3]\c@lNbarcs{#4}{#5}%
    \v@leur=#4pt\c@lptellP{#1}{-1}{-2}\Figptpr@j-5:/-3/%
    \f@gnewpath\PSwrit@cmdS{-5}{\c@mmoveto}{\fwf@g}{\X@un}{\Y@un}%
    \edef\C@nt@r{#1}\s@mme=\z@\bcl@rellPATD}
\ctr@ld@f\def\bcl@rellPATD{\relax%
    \ifnum\s@mme<\p@rtent\advance\s@mme\@ne%
    \advance\v@leur\delt@\c@lptellP{\C@nt@r}{-1}{-2}\Figptpr@j-4:/-3/%
    \advance\v@leur\delt@\c@lptellP{\C@nt@r}{-1}{-2}\Figptpr@j-6:/-3/%
    \advance\v@leur\delt@\c@lptellP{\C@nt@r}{-1}{-2}\Figptpr@j-3:/-3/%
    \v@lX=\z@\v@lY=\z@\Figtr@nptDD{-5}{-5}\Figtr@nptDD{2}{-3}%
    \divide\v@lX\@vi\divide\v@lY\@vi%
    \Figtr@nptDD{3}{-4}\Figtr@nptDD{-1.5}{-6}\v@lmin=\v@lX\v@lmax=\v@lY%
    \v@lX=\z@\v@lY=\z@\Figtr@nptDD{2}{-5}\Figtr@nptDD{-5}{-3}%
    \divide\v@lX\@vi\divide\v@lY\@vi\Figtr@nptDD{-1.5}{-4}\Figtr@nptDD{3}{-6}%
    \BdingB@xfalse%
    \Figp@intregDD-4:(\v@lmin,\v@lmax)\PSwrit@cmdS{-4}{}{\fwf@g}{\X@de}{\Y@de}%
    \Figp@intregDD-4:(\v@lX,\v@lY)\PSwrit@cmdS{-4}{}{\fwf@g}{\X@tr}{\Y@tr}%
    \BdingB@xtrue\PSwrit@cmdS{-3}{\c@mcurveto}{\fwf@g}{\X@qu}{\Y@qu}%
    \B@zierBB@x{1}{\Y@un}(\X@un,\X@de,\X@tr,\X@qu)%
    \B@zierBB@x{2}{\X@un}(\Y@un,\Y@de,\Y@tr,\Y@qu)%
    \edef\X@un{\X@qu}\edef\Y@un{\Y@qu}\figptcopyDD-5:/-3/\bcl@rellPATD\fi}
\ctr@ld@f\def\c@lNbarcs#1#2{%
    \delt@=#2pt\advance\delt@-#1pt\maxim@m{\v@lmax}{\delt@}{-\delt@}%
    \v@leur=\v@lmax\divide\v@leur45 \p@rtentiere{\p@rtent}{\v@leur}\advance\p@rtent\@ne%
    \s@mme=\p@rtent\multiply\s@mme\thr@@\divide\delt@\s@mme}
\ctr@ld@f\def\figdrawarcellPP#1,#2,#3[#4,#5]{{\ifCUR@PS\ifGR@cri\s@uvc@ntr@l\et@tpsarcellPP%
    \PSc@mment{arcellPP Center=#1,PtAxis1=#2,PtAxis2=#3 [Point1=#4, Point2=#5]}%
    \setc@ntr@l{2}\figvectP-2[#1,#3]\vecunit@{-2}{-2}\v@lmin=\result@t%
    \invers@{\v@lmax}{\v@lmin}%
    \figvectP-1[#1,#2]\vecunit@{-1}{-1}\v@leur=\result@t%
    \v@leur=\repdecn@mb{\v@lmax}\v@leur\edef\AsB@{\repdecn@mb{\v@leur}}
    \c@lAngle{#1}{#4}{\v@lmin}\edef\@ngdeb{\repdecn@mb{\v@lmin}}%
    \c@lAngle{#1}{#5}{\v@lmax}\ifdim\v@lmin>\v@lmax\advance\v@lmax\DePI@deg\fi%
    \edef\@ngfin{\repdecn@mb{\v@lmax}}\figdrawarcellPA#1,#2,#3(\@ngdeb,\@ngfin)%
    \PSc@mment{End arcellPP}\resetc@ntr@l\et@tpsarcellPP\fi\fi}}
\ctr@ld@f\def\c@lAngle#1#2#3{\figvectP-3[#1,#2]%
    \c@lproscal\delt@[-3,-1]\c@lproscal\v@leur[-3,-2]%
    \v@leur=\AsB@\v@leur\arct@n#3(\delt@,\v@leur)#3=\rdT@deg#3}
\ctr@ln@w{newif}\if@rrowratio\@rrowratiotrue
\ctr@ln@w{newif}\if@rrowhfill
\ctr@ln@w{newif}\if@rrowhout
\ctr@ld@f\def\Psset@rrowhe@d#1=#2|{\keln@mun#1|%
    \def\n@mref{a}\ifx\l@debut\n@mref\update@ttr\D@FTarrowheadangle\Q@s@tarrowheadangle{#2}\else
    \def\n@mref{f}\ifx\l@debut\n@mref\update@ttr\D@FTarrowheadfill\Q@s@tarrowheadfill{#2}\else
    \def\n@mref{l}\ifx\l@debut\n@mref\update@ttr\D@FTarrowheadlength\Q@s@tarrowheadlength{#2}\else
    \def\n@mref{o}\ifx\l@debut\n@mref\update@ttr\D@FTarrowheadout\Q@s@tarrowheadout{#2}\else
    \def\n@mref{r}\ifx\l@debut\n@mref\update@ttr\D@FTarrowheadratio\Q@s@tarrowheadratio{#2}\else
    \W@rnmesAttr{figset arrowhead}{#1}\fi\fi\fi\fi\fi}
\ctr@ln@m\@rrowheadangle
\ctr@ln@m\C@AHANG \ctr@ln@m\S@AHANG \ctr@ln@m\UNSS@N
\ctr@ld@f\def\Q@s@tarrowheadangle#1{\edef\@rrowheadangle{#1}{\c@ssin{\C@}{\S@}{#1}%
    \xdef\C@AHANG{\C@}\xdef\S@AHANG{\S@}\v@lmax=\S@ pt%
    \invers@{\v@leur}{\v@lmax}\maxim@m{\v@leur}{\v@leur}{-\v@leur}%
    \xdef\UNSS@N{\the\v@leur}}}
\ctr@ld@f\def\Q@s@tarrowheadfill#1{\expandafter\set@rrowhfill#1:}
\ctr@ld@f\def\set@rrowhfill#1#2:{\if#1n\@rrowhfillfalse\else\@rrowhfilltrue\fi}
\ctr@ld@f\def\Q@s@tarrowheadout#1{\expandafter\set@rrowhout#1:}
\ctr@ld@f\def\set@rrowhout#1#2:{\if#1n\@rrowhoutfalse\else\@rrowhouttrue\fi}
\ctr@ln@m\@rrowheadlength
\ctr@ld@f\def\Q@s@tarrowheadlength#1{\edef\@rrowheadlength{#1}\@rrowratiofalse}
\ctr@ln@m\@rrowheadratio
\ctr@ld@f\def\Q@s@tarrowheadratio#1{\edef\@rrowheadratio{#1}\@rrowratiotrue}
\ctr@ln@m\D@FTarrowheadlength
\ctr@ld@f\def\figresetarrowhead{%
    \Q@s@tarrowheadangle{\D@FTarrowheadangle}%
    \Q@s@tarrowheadfill{\D@FTarrowheadfill}%
    \Q@s@tarrowheadout{\D@FTarrowheadout}%
    \Q@s@tarrowheadratio{\D@FTarrowheadratio}%
    \d@fm@cdim\D@FTarrowheadlength{\D@FTh@rdahlength}
    \Q@s@tarrowheadlength{\D@FTarrowheadlength}}
\ctr@ld@f\def\D@FTarrowheadratio{0.1}
\ctr@ld@f\def\D@FTarrowheadangle{20}
\ctr@ld@f\def\D@FTarrowheadfill{no}
\ctr@ld@f\def\D@FTarrowheadout{no}
\ctr@ld@f\def\D@FTh@rdahlength{8pt}
\ctr@ln@m\figdrawarrow
\ctr@ld@f\def\Q@arrowDD[#1,#2]{{\ifCUR@PS\ifGR@cri\s@uvc@ntr@l\et@tpsarrow%
    \PSc@mment{arrowDD [Pt1,Pt2]=[#1,#2]}\Q@s@tfillmode{no}%
    \Q@arrowheadDD[#1,#2]\setc@ntr@l{2}\figdrawline[#1,-3]%
    \PSc@mment{End arrowDD}\resetc@ntr@l\et@tpsarrow\fi\fi}}
\ctr@ld@f\def\Q@arrowTD[#1,#2]{{\ifCUR@PS\ifGR@cri\s@uvc@ntr@l\et@tpsarrowTD%
    \PSc@mment{arrowTD [Pt1,Pt2]=[#1,#2]}\resetc@ntr@l{2}%
    \Figptpr@j-5:/#1/\Figptpr@j-6:/#2/\let\c@lprojSP=\relax\Q@arrowDD[-5,-6]%
    \PSc@mment{End arrowTD}\resetc@ntr@l\et@tpsarrowTD\fi\fi}}
\ctr@ln@m\figdrawarrowhead
\ctr@ld@f\def\Q@arrowheadDD[#1,#2]{{\ifCUR@PS\ifGR@cri\s@uvc@ntr@l\et@tpsarrowheadDD%
    \if@rrowhfill\def\@hangle{-\@rrowheadangle}\else\def\@hangle{\@rrowheadangle}\fi%
    \if@rrowratio%
    \if@rrowhout\def\@hratio{-\@rrowheadratio}\else\def\@hratio{\@rrowheadratio}\fi%
    \PSc@mment{arrowheadDD Ratio=\@hratio, Angle=\@hangle, [Pt1,Pt2]=[#1,#2]}%
    \Ps@rrowhead\@hratio,\@hangle[#1,#2]%
    \else%
    \if@rrowhout\def\@hlength{-\@rrowheadlength}\else\def\@hlength{\@rrowheadlength}\fi%
    \PSc@mment{arrowheadDD Length=\@hlength, Angle=\@hangle, [Pt1,Pt2]=[#1,#2]}%
    \Ps@rrowheadfd\@hlength,\@hangle[#1,#2]%
    \fi%
    \PSc@mment{End arrowheadDD}\resetc@ntr@l\et@tpsarrowheadDD\fi\fi}}
\ctr@ld@f\def\Q@arrowheadTD[#1,#2]{{\ifCUR@PS\ifGR@cri\s@uvc@ntr@l\et@tpsarrowheadTD%
    \PSc@mment{arrowheadTD [Pt1,Pt2]=[#1,#2]}\resetc@ntr@l{2}%
    \Figptpr@j-5:/#1/\Figptpr@j-6:/#2/\let\c@lprojSP=\relax\Q@arrowheadDD[-5,-6]%
    \PSc@mment{End arrowheadTD}\resetc@ntr@l\et@tpsarrowheadTD\fi\fi}}
\ctr@ld@f\def\Ps@rrowhead#1,#2[#3,#4]{\v@leur=#1\p@\maxim@m{\v@leur}{\v@leur}{-\v@leur}%
    \ifdim\v@leur>\Cepsil@n{
    \PSc@mment{@rrowhead Ratio=#1, Angle=#2, [Pt1,Pt2]=[#3,#4]}\v@leur=\UNSS@N%
    \v@leur=\CUR@width\v@leur\v@leur=\ptpsT@pt\v@leur\delt@=.5\v@leur
    \setc@ntr@l{2}\figvectPDD-3[#4,#3]%
    \Figg@tXY{-3}\v@lX=#1\v@lX\v@lY=#1\v@lY\Figv@ctCreg-3(\v@lX,\v@lY)%
    \vecunit@{-4}{-3}\mili@u=\result@t%
    \ifdim#2pt>\z@\v@lXa=-\C@AHANG\delt@%
     \edef\c@ef{\repdecn@mb{\v@lXa}}\figpttraDD-3:=-3/\c@ef,-4/\fi%
    \edef\c@ef{\repdecn@mb{\delt@}}%
    \v@lXa=\mili@u\v@lXa=\C@AHANG\v@lXa%
    \v@lYa=\ptpsT@pt\p@\v@lYa=\CUR@width\v@lYa\v@lYa=\sDcc@ngle\v@lYa%
    \advance\v@lXa-\v@lYa\gdef\sDcc@ngle{0}%
    \ifdim\v@lXa>\v@leur\edef\c@efendpt{\repdecn@mb{\v@leur}}%
    \else\edef\c@efendpt{\repdecn@mb{\v@lXa}}\fi%
    \Figg@tXY{-3}\v@lmin=\v@lX\v@lmax=\v@lY%
    \v@lXa=\C@AHANG\v@lmin\v@lYa=\S@AHANG\v@lmax\advance\v@lXa\v@lYa%
    \v@lYa=-\S@AHANG\v@lmin\v@lX=\C@AHANG\v@lmax\advance\v@lYa\v@lX%
    \setc@ntr@l{1}\Figg@tXY{#4}\advance\v@lX\v@lXa\advance\v@lY\v@lYa%
    \setc@ntr@l{2}\Figp@intregDD-2:(\v@lX,\v@lY)%
    \v@lXa=\C@AHANG\v@lmin\v@lYa=-\S@AHANG\v@lmax\advance\v@lXa\v@lYa%
    \v@lYa=\S@AHANG\v@lmin\v@lX=\C@AHANG\v@lmax\advance\v@lYa\v@lX%
    \setc@ntr@l{1}\Figg@tXY{#4}\advance\v@lX\v@lXa\advance\v@lY\v@lYa%
    \setc@ntr@l{2}\Figp@intregDD-1:(\v@lX,\v@lY)%
    \ifdim#2pt<\z@\fillm@detrue\figdrawline[-2,#4,-1]
    \else\figptstraDD-3=#4,-2,-1/\c@ef,-4/\s@uvdash{\typ@dash}\Q@s@tdash{\D@FTdash}%
    \figdrawline[-2,-3,-1]\Q@s@tdash{\typ@dash}\fi
    \ifdim#1pt>\z@\figpttraDD-3:=#4/\c@efendpt,-4/\else\figptcopyDD-3:/#4/\fi%
    \PSc@mment{End @rrowhead}}\fi}
\ctr@ld@f\def\sDcc@ngle{0}
\ctr@ld@f\def\Ps@rrowheadfd#1,#2[#3,#4]{{%
    \PSc@mment{@rrowheadfd Length=#1, Angle=#2, [Pt1,Pt2]=[#3,#4]}%
    \setc@ntr@l{2}\figvectPDD-1[#3,#4]\n@rmeucDD{\v@leur}{-1}\v@leur=\ptT@unit@\v@leur%
    \invers@{\v@leur}{\v@leur}\v@leur=#1\v@leur\edef\R@tio{\repdecn@mb{\v@leur}}%
    \Ps@rrowhead\R@tio,#2[#3,#4]\PSc@mment{End @rrowheadfd}}}
\ctr@ln@m\figdrawarrowBezier
\ctr@ld@f\def\Q@arrowBezierDD[#1,#2,#3,#4]{{\ifCUR@PS\ifGR@cri\s@uvc@ntr@l\et@tpsarrowBezierDD%
    \PSc@mment{arrowBezierDD Control points=#1,#2,#3,#4}\setc@ntr@l{2}%
    \if@rrowratio\c@larclengthDD\v@leur,10[#1,#2,#3,#4]\else\v@leur=\z@\fi%
    \Ps@rrowB@zDD\v@leur[#1,#2,#3,#4]%
    \PSc@mment{End arrowBezierDD}\resetc@ntr@l\et@tpsarrowBezierDD\fi\fi}}
\ctr@ld@f\def\Q@arrowBezierTD[#1,#2,#3,#4]{{\ifCUR@PS\ifGR@cri\s@uvc@ntr@l\et@tpsarrowBezierTD%
    \PSc@mment{arrowBezierTD Control points=#1,#2,#3,#4}\resetc@ntr@l{2}%
    \Figptpr@j-7:/#1/\Figptpr@j-8:/#2/\Figptpr@j-9:/#3/\Figptpr@j-10:/#4/%
    \let\c@lprojSP=\relax\ifnum\CUR@proj<\tw@\Q@arrowBezierDD[-7,-8,-9,-10]%
    \else\f@gnewpath\PSwrit@cmd{-7}{\c@mmoveto}{\fwf@g}%
    \if@rrowratio\c@larclengthDD\mili@u,10[-7,-8,-9,-10]\else\mili@u=\z@\fi%
    \p@rtent=\NBz@rcs\advance\p@rtent\m@ne\subB@zierTD\p@rtent[#1,#2,#3,#4]%
    \f@gstroke%
    \advance\v@lmin\p@rtent\delt@
    \v@leur=\v@lmin\advance\v@leur0.33333 \delt@\edef\unti@rs{\repdecn@mb{\v@leur}}%
    \v@leur=\v@lmin\advance\v@leur0.66666 \delt@\edef\deti@rs{\repdecn@mb{\v@leur}}%
    \figptcopyDD-8:/-10/\c@lsubBzarc\unti@rs,\deti@rs[#1,#2,#3,#4]%
    \figptcopyDD-8:/-4/\figptcopyDD-9:/-3/\Ps@rrowB@zDD\mili@u[-7,-8,-9,-10]\fi%
    \PSc@mment{End arrowBezierTD}\resetc@ntr@l\et@tpsarrowBezierTD\fi\fi}}
\ctr@ld@f\def\c@larclengthDD#1,#2[#3,#4,#5,#6]{{\p@rtent=#2\figptcopyDD-5:/#3/%
    \delt@=\p@\divide\delt@\p@rtent\c@rre=\z@\v@leur=\z@\s@mme=\z@%
    \loop\ifnum\s@mme<\p@rtent\advance\s@mme\@ne\advance\v@leur\delt@%
    \edef\T@{\repdecn@mb{\v@leur}}\figptBezierDD-6::\T@[#3,#4,#5,#6]%
    \figvectPDD-1[-5,-6]\n@rmeucDD{\mili@u}{-1}\advance\c@rre\mili@u%
    \figptcopyDD-5:/-6/\repeat\global\result@t=\ptT@unit@\c@rre}#1=\result@t}
\ctr@ld@f\def\Ps@rrowB@zDD#1[#2,#3,#4,#5]{{\Q@s@tfillmode{no}%
    \if@rrowratio\delt@=\@rrowheadratio#1\else\delt@=\@rrowheadlength pt\fi%
    \v@leur=\C@AHANG\delt@\edef\R@dius{\repdecn@mb{\v@leur}}%
    \FigptintercircB@zDD-5::0,\R@dius[#5,#4,#3,#2]%
    \Q@s@tarrowheadlength{\repdecn@mb{\delt@}}\Q@arrowheadDD[-5,#5]%
    \let\n@rmeuc=\n@rmeucDD\figgetdist\R@dius[#5,-3]%
    \FigptintercircB@zDD-6::0,\R@dius[#5,#4,#3,#2]%
    \figptBezierDD-5::0.33333[#5,#4,#3,#2]\figptBezierDD-3::0.66666[#5,#4,#3,#2]%
    \figptscontrolDD-5[-6,-5,-3,#2]\Q@BezierDD1[-6,-5,-4,#2]}}
\ctr@ln@m\figdrawarrowcirc
\ctr@ld@f\def\Q@arrowcircDD#1;#2(#3,#4){{\ifCUR@PS\ifGR@cri\s@uvc@ntr@l\et@tpsarrowcircDD%
    \PSc@mment{arrowcircDD Center=#1 ; Radius=#2 (Ang1=#3,Ang2=#4)}%
    \Q@s@tfillmode{no}\Pscirc@rrowhead#1;#2(#3,#4)%
    \setc@ntr@l{2}\figvectPDD -4[#1,-3]\vecunit@{-4}{-4}%
    \Figg@tXY{-4}\arct@n\v@lmin(\v@lX,\v@lY)%
    \v@lmin=\rdT@deg\v@lmin\v@leur=#4pt\advance\v@leur-\v@lmin%
    \maxim@m{\v@leur}{\v@leur}{-\v@leur}%
    \ifdim\v@leur>\DemiPI@deg\relax\ifdim\v@lmin<#4pt\advance\v@lmin\DePI@deg%
    \else\advance\v@lmin-\DePI@deg\fi\fi\edef\ar@ngle{\repdecn@mb{\v@lmin}}%
    \ifdim#3pt<#4pt\figdrawarccirc#1;#2(#3,\ar@ngle)\else\figdrawarccirc#1;#2(\ar@ngle,#3)\fi%
    \PSc@mment{End arrowcircDD}\resetc@ntr@l\et@tpsarrowcircDD\fi\fi}}
\ctr@ld@f\def\Q@arrowcircTD#1,#2,#3;#4(#5,#6){{\ifCUR@PS\ifGR@cri\s@uvc@ntr@l\et@tpsarrowcircTD%
    \PSc@mment{arrowcircTD Center=#1,P1=#2,P2=#3 ; Radius=#4 (Ang1=#5, Ang2=#6)}%
    \resetc@ntr@l{2}\c@lExtAxes#1,#2,#3(#4)\let\c@lprojSP=\relax%
    \figvectPTD-11[#1,-4]\figvectPTD-12[#1,-5]\c@lNbarcs{#5}{#6}%
    \if@rrowratio\v@lmax=\degT@rd\v@lmax\edef\D@lpha{\repdecn@mb{\v@lmax}}\fi%
    \advance\p@rtent\m@ne\mili@u=\z@%
    \v@leur=#5pt\c@lptellP{#1}{-11}{-12}\Figptpr@j-9:/-3/%
    \f@gnewpath\PSwrit@cmdS{-9}{\c@mmoveto}{\fwf@g}{\X@un}{\Y@un}%
    \edef\C@nt@r{#1}\s@mme=\z@\bcl@rcircTD\f@gstroke%
    \advance\v@leur\delt@\c@lptellP{#1}{-11}{-12}\Figptpr@j-5:/-3/%
    \advance\v@leur\delt@\c@lptellP{#1}{-11}{-12}\Figptpr@j-6:/-3/%
    \advance\v@leur\delt@\c@lptellP{#1}{-11}{-12}\Figptpr@j-10:/-3/%
    \figptscontrolDD-8[-9,-5,-6,-10]%
    \if@rrowratio\c@lcurvradDD0.5[-9,-8,-7,-10]\advance\mili@u\result@t%
    \maxim@m{\mili@u}{\mili@u}{-\mili@u}\mili@u=\ptT@unit@\mili@u%
    \mili@u=\D@lpha\mili@u\advance\p@rtent\@ne\divide\mili@u\p@rtent\fi%
    \Ps@rrowB@zDD\mili@u[-9,-8,-7,-10]%
    \PSc@mment{End arrowcircTD}\resetc@ntr@l\et@tpsarrowcircTD\fi\fi}}
\ctr@ld@f\def\bcl@rcircTD{\relax%
    \ifnum\s@mme<\p@rtent\advance\s@mme\@ne%
    \advance\v@leur\delt@\c@lptellP{\C@nt@r}{-11}{-12}\Figptpr@j-5:/-3/%
    \advance\v@leur\delt@\c@lptellP{\C@nt@r}{-11}{-12}\Figptpr@j-6:/-3/%
    \advance\v@leur\delt@\c@lptellP{\C@nt@r}{-11}{-12}\Figptpr@j-10:/-3/%
    \figptscontrolDD-8[-9,-5,-6,-10]\BdingB@xfalse%
    \PSwrit@cmdS{-8}{}{\fwf@g}{\X@de}{\Y@de}\PSwrit@cmdS{-7}{}{\fwf@g}{\X@tr}{\Y@tr}%
    \BdingB@xtrue\PSwrit@cmdS{-10}{\c@mcurveto}{\fwf@g}{\X@qu}{\Y@qu}%
    \if@rrowratio\c@lcurvradDD0.5[-9,-8,-7,-10]\advance\mili@u\result@t\fi%
    \B@zierBB@x{1}{\Y@un}(\X@un,\X@de,\X@tr,\X@qu)%
    \B@zierBB@x{2}{\X@un}(\Y@un,\Y@de,\Y@tr,\Y@qu)%
    \edef\X@un{\X@qu}\edef\Y@un{\Y@qu}\figptcopyDD-9:/-10/\bcl@rcircTD\fi}
\ctr@ld@f\def\Pscirc@rrowhead#1;#2(#3,#4){{%
    \PSc@mment{circ@rrowhead Center=#1 ; Radius=#2 (Ang1=#3,Ang2=#4)}%
    \v@leur=#2\unit@\edef\s@glen{\repdecn@mb{\v@leur}}\v@lY=\z@\v@lX=\v@leur%
    \resetc@ntr@l{2}\Figv@ctCreg-3(\v@lX,\v@lY)\figpttraDD-5:=#1/1,-3/%
    \figptrotDD-5:=-5/#1,#4/%
    \figvectPDD-3[#1,-5]\Figg@tXY{-3}\v@leur=\v@lX%
    \ifdim#3pt<#4pt\v@lX=\v@lY\v@lY=-\v@leur\else\v@lX=-\v@lY\v@lY=\v@leur\fi%
    \Figv@ctCreg-3(\v@lX,\v@lY)\vecunit@{-3}{-3}%
    \if@rrowratio\v@leur=#4pt\advance\v@leur-#3pt\maxim@m{\mili@u}{-\v@leur}{\v@leur}%
    \mili@u=\degT@rd\mili@u\v@leur=\s@glen\mili@u\edef\s@glen{\repdecn@mb{\v@leur}}%
    \mili@u=#2\mili@u\mili@u=\@rrowheadratio\mili@u\else\mili@u=\@rrowheadlength pt\fi%
    \figpttraDD-6:=-5/\s@glen,-3/\v@leur=#2pt\v@leur=2\v@leur%
    \invers@{\v@leur}{\v@leur}\c@rre=\repdecn@mb{\v@leur}\mili@u
    \mili@u=\c@rre\mili@u=\repdecn@mb{\c@rre}\mili@u%
    \v@leur=\p@\advance\v@leur-\mili@u
    \invers@{\mili@u}{2\v@leur}\delt@=\c@rre\delt@=\repdecn@mb{\mili@u}\delt@%
    \xdef\sDcc@ngle{\repdecn@mb{\delt@}}
    \sqrt@{\mili@u}{\v@leur}\arct@n\v@leur(\mili@u,\c@rre)%
    \v@leur=\rdT@deg\v@leur
    \ifdim#3pt<#4pt\v@leur=-\v@leur\fi%
    \if@rrowhout\v@leur=-\v@leur\fi\edef\cor@ngle{\repdecn@mb{\v@leur}}%
    \figptrotDD-6:=-6/-5,\cor@ngle/\Q@arrowheadDD[-6,-5]%
    \PSc@mment{End circ@rrowhead}}}
\ctr@ln@m\figdrawarrowcircP
\ctr@ld@f\def\Q@arrowcircPDD#1;#2[#3,#4]{{\ifCUR@PS\ifGR@cri%
    \PSc@mment{arrowcircPDD Center=#1; Radius=#2, [P1=#3,P2=#4]}%
    \s@uvc@ntr@l\et@tpsarrowcircPDD\Ps@ngleparam#1;#2[#3,#4]%
    \ifdim\v@leur>\z@\ifdim\v@lmin>\v@lmax\advance\v@lmax\DePI@deg\fi%
    \else\ifdim\v@lmin<\v@lmax\advance\v@lmin\DePI@deg\fi\fi%
    \edef\@ngdeb{\repdecn@mb{\v@lmin}}\edef\@ngfin{\repdecn@mb{\v@lmax}}%
    \figdrawarrowcirc#1;\r@dius(\@ngdeb,\@ngfin)%
    \PSc@mment{End arrowcircPDD}\resetc@ntr@l\et@tpsarrowcircPDD\fi\fi}}
\ctr@ld@f\def\Q@arrowcircPTD#1;#2[#3,#4,#5]{{\ifCUR@PS\ifGR@cri\s@uvc@ntr@l\et@tpsarrowcircPTD%
    \PSc@mment{arrowcircPTD Center=#1; Radius=#2, [P1=#3,P2=#4,P3=#5]}%
    \figgetangleTD\@ngfin[#1,#3,#4,#5]\v@leur=#2pt%
    \maxim@m{\mili@u}{-\v@leur}{\v@leur}\edef\r@dius{\repdecn@mb{\mili@u}}%
    \ifdim\v@leur<\z@\v@lmax=\@ngfin pt\advance\v@lmax-\DePI@deg%
    \edef\@ngfin{\repdecn@mb{\v@lmax}}\fi\Q@arrowcircTD#1,#3,#5;\r@dius(0,\@ngfin)%
    \PSc@mment{End arrowcircPTD}\resetc@ntr@l\et@tpsarrowcircPTD\fi\fi}}
\ctr@ld@f\def\figdrawaxes#1(#2){{\ifCUR@PS\ifGR@cri\s@uvc@ntr@l\et@tpsaxes%
    \PSc@mment{axes Origin=#1 Range=(#2)}\an@lys@xes#2,:\resetc@ntr@l{2}%
    \ifx\t@xt@\empty\ifTr@isDim\Q@@xes#1(0,#2,0,#2,0,#2)\else\Q@@xes#1(0,#2,0,#2)\fi%
    \else\Q@@xes#1(#2)\fi\PSc@mment{End axes}\resetc@ntr@l\et@tpsaxes\fi\fi}}
\ctr@ld@f\def\an@lys@xes#1,#2:{\def\t@xt@{#2}}
\ctr@ln@m\Q@@xes
\ctr@ld@f\def\Q@@xesDD#1(#2,#3,#4,#5){%
    \figpttraC-5:=#1/#2,0/\figpttraC-6:=#1/#3,0/\Q@arrowDD[-5,-6]%
    \figpttraC-5:=#1/0,#4/\figpttraC-6:=#1/0,#5/\Q@arrowDD[-5,-6]}
\ctr@ld@f\def\Q@@xesTD#1(#2,#3,#4,#5,#6,#7){%
    \figpttraC-7:=#1/#2,0,0/\figpttraC-8:=#1/#3,0,0/\Q@arrowTD[-7,-8]%
    \figpttraC-7:=#1/0,#4,0/\figpttraC-8:=#1/0,#5,0/\Q@arrowTD[-7,-8]%
    \figpttraC-7:=#1/0,0,#6/\figpttraC-8:=#1/0,0,#7/\Q@arrowTD[-7,-8]}
\ctr@ln@m\newGr@FN
\ctr@ld@f\def\newGr@FNPDF#1{\s@mme=\Gr@FNb\advance\s@mme\@ne\xdef\Gr@FNb{\number\s@mme}}
\ctr@ld@f\def\newGr@FNDVI#1{\newGr@FNPDF{}\xdef#1{\jobname GI\Gr@FNb.anx}}
\ctr@ld@f\def\figdrawbegin#1{\newGr@FN\DefGIfilen@me\gdef\@utoFN{0}%
    \def\t@xt@{#1}\relax\ifx\t@xt@\empty\GRupdatem@detrue%
    \gdef\@utoFN{1}\Psb@ginfig\DefGIfilen@me\else\expandafter\Psb@ginfigNu@#1 :\fi}
\ctr@ld@f\def\Psb@ginfigNu@#1 #2:{\def\t@xt@{#1}\relax\ifx\t@xt@\empty\def\t@xt@{#2}%
    \ifx\t@xt@\empty\GRupdatem@detrue\gdef\@utoFN{1}\Psb@ginfig\DefGIfilen@me%
    \else\Psb@ginfigNu@#2:\fi\else\Psb@ginfig{#1}\fi}
\ctr@ln@m\PSfilen@me \ctr@ln@m\auxfilen@me
\ctr@ld@f\def\Psb@ginfig#1{\ifCUR@PS\else%
    \edef\PSfilen@me{#1}\edef\auxfilen@me{\jobname.anx}%
    \ifGRupdatem@de\GR@critrue\else\openin\frf@g=\PSfilen@me\relax%
    \ifeof\frf@g\GR@critrue\else\GR@crifalse\fi\closein\frf@g\fi%
    \CUR@PStrue\c@ldefproj\expandafter\setupd@te\D@FTupdate:%
    \ifGR@cri\initb@undb@x%
    \immediate\openout\fwf@g=\auxfilen@me\initpss@ttings\fi%
    \fi}
\ctr@ld@f\def\Gr@FNb{0}
\ctr@ld@f\def\figforTeXFileno{\Gr@FNb}
\ctr@ld@f\def\figforTeXFigno{0 }
\ctr@ld@f\def\figforTeXnextFigno{1 }
\ctr@ld@f\edef\DefGIfilen@me{\jobname GI.anx}
\ctr@ld@f\def\initpss@ttings{\figreset{altitude,arrowhead,curve,general,flowchart,mesh,trimesh}%
    \Use@llipsefalse}
\ctr@ld@f\def\B@zierBB@x#1#2(#3,#4,#5,#6){{\c@rre=\t@n\epsil@n
    \v@lmax=#4\advance\v@lmax-#5\v@lmax=\thr@@\v@lmax\advance\v@lmax#6\advance\v@lmax-#3%
    \mili@u=#4\mili@u=-\tw@\mili@u\advance\mili@u#3\advance\mili@u#5%
    \v@lmin=#4\advance\v@lmin-#3\maxim@m{\v@leur}{-\v@lmax}{\v@lmax}%
    \maxim@m{\delt@}{-\mili@u}{\mili@u}\maxim@m{\v@leur}{\v@leur}{\delt@}%
    \maxim@m{\delt@}{-\v@lmin}{\v@lmin}\maxim@m{\v@leur}{\v@leur}{\delt@}%
    \ifdim\v@leur>\c@rre\invers@{\v@leur}{\v@leur}\edef\Uns@rM@x{\repdecn@mb{\v@leur}}%
    \v@lmax=\Uns@rM@x\v@lmax\mili@u=\Uns@rM@x\mili@u\v@lmin=\Uns@rM@x\v@lmin%
    \maxim@m{\v@leur}{-\v@lmax}{\v@lmax}\ifdim\v@leur<\c@rre%
    \maxim@m{\v@leur}{-\mili@u}{\mili@u}\ifdim\v@leur<\c@rre\else%
    \invers@{\mili@u}{\mili@u}\v@leur=-0.5\v@lmin%
    \v@leur=\repdecn@mb{\mili@u}\v@leur\m@jBBB@x{\v@leur}{#1}{#2}(#3,#4,#5,#6)\fi%
    \else\delt@=\repdecn@mb{\mili@u}\mili@u\v@leur=\repdecn@mb{\v@lmax}\v@lmin%
    \advance\delt@-\v@leur\ifdim\delt@<\z@\else\invers@{\v@lmax}{\v@lmax}%
    \edef\Uns@rAp{\repdecn@mb{\v@lmax}}\sqrt@{\delt@}{\delt@}%
    \v@leur=-\mili@u\advance\v@leur\delt@\v@leur=\Uns@rAp\v@leur%
    \m@jBBB@x{\v@leur}{#1}{#2}(#3,#4,#5,#6)%
    \v@leur=-\mili@u\advance\v@leur-\delt@\v@leur=\Uns@rAp\v@leur%
    \m@jBBB@x{\v@leur}{#1}{#2}(#3,#4,#5,#6)\fi\fi\fi}}
\ctr@ld@f\def\m@jBBB@x#1#2#3(#4,#5,#6,#7){{\relax\ifdim#1>\z@\ifdim#1<\p@%
    \edef\T@{\repdecn@mb{#1}}\v@lX=\p@\advance\v@lX-#1\edef\UNmT@{\repdecn@mb{\v@lX}}%
    \v@lX=#4\v@lY=#5\v@lZ=#6\v@lXa=#7\v@lX=\UNmT@\v@lX\advance\v@lX\T@\v@lY%
    \v@lY=\UNmT@\v@lY\advance\v@lY\T@\v@lZ\v@lZ=\UNmT@\v@lZ\advance\v@lZ\T@\v@lXa%
    \v@lX=\UNmT@\v@lX\advance\v@lX\T@\v@lY\v@lY=\UNmT@\v@lY\advance\v@lY\T@\v@lZ%
    \v@lX=\UNmT@\v@lX\advance\v@lX\T@\v@lY%
    \ifcase#2\or\v@lY=#3\or\v@lY=\v@lX\v@lX=#3\fi\b@undb@x{\v@lX}{\v@lY}\fi\fi}}
\ctr@ld@f\def\PsB@zier#1[#2]{{\f@gnewpath%
    \s@mme=\z@\def\list@num{#2,0}\extrairelepremi@r\p@int\de\list@num%
    \PSwrit@cmdS{\p@int}{\c@mmoveto}{\fwf@g}{\X@un}{\Y@un}\p@rtent=#1\bclB@zier}}
\ctr@ld@f\def\bclB@zier{\relax%
    \ifnum\s@mme<\p@rtent\advance\s@mme\@ne\BdingB@xfalse%
    \extrairelepremi@r\p@int\de\list@num\PSwrit@cmdS{\p@int}{}{\fwf@g}{\X@de}{\Y@de}%
    \extrairelepremi@r\p@int\de\list@num\PSwrit@cmdS{\p@int}{}{\fwf@g}{\X@tr}{\Y@tr}%
    \BdingB@xtrue%
    \extrairelepremi@r\p@int\de\list@num\PSwrit@cmdS{\p@int}{\c@mcurveto}{\fwf@g}{\X@qu}{\Y@qu}%
    \B@zierBB@x{1}{\Y@un}(\X@un,\X@de,\X@tr,\X@qu)%
    \B@zierBB@x{2}{\X@un}(\Y@un,\Y@de,\Y@tr,\Y@qu)%
    \edef\X@un{\X@qu}\edef\Y@un{\Y@qu}\bclB@zier\fi}
\ctr@ln@m\figdrawBezier
\ctr@ld@f\def\Q@BezierDD#1[#2]{\ifCUR@PS\ifGR@cri%
    \PSc@mment{BezierDD N arcs=#1, Control points=#2}%
    \iffillm@de\PsB@zier#1[#2]%
    \f@gfill%
    \else\PsB@zier#1[#2]\f@gstroke\fi%
    \PSc@mment{End BezierDD}\fi\fi}
\ctr@ln@m\et@tpsBezierTD
\ctr@ld@f\def\Q@BezierTD#1[#2]{\ifCUR@PS\ifGR@cri\s@uvc@ntr@l\et@tpsBezierTD%
    \PSc@mment{BezierTD N arcs=#1, Control points=#2}%
    \iffillm@de\PsB@zierTD#1[#2]%
    \f@gfill%
    \else\PsB@zierTD#1[#2]\f@gstroke\fi%
    \PSc@mment{End BezierTD}\resetc@ntr@l\et@tpsBezierTD\fi\fi}
\ctr@ld@f\def\PsB@zierTD#1[#2]{\ifnum\CUR@proj<\tw@\PsB@zier#1[#2]\else\PsB@zier@TD#1[#2]\fi}
\ctr@ld@f\def\PsB@zier@TD#1[#2]{{\f@gnewpath%
    \s@mme=\z@\def\list@num{#2,0}\extrairelepremi@r\p@int\de\list@num%
    \let\c@lprojSP=\relax\setc@ntr@l{2}\Figptpr@j-7:/\p@int/%
    \PSwrit@cmd{-7}{\c@mmoveto}{\fwf@g}%
    \loop\ifnum\s@mme<#1\advance\s@mme\@ne\extrairelepremi@r\p@intun\de\list@num%
    \extrairelepremi@r\p@intde\de\list@num\extrairelepremi@r\p@inttr\de\list@num%
    \subB@zierTD\NBz@rcs[\p@int,\p@intun,\p@intde,\p@inttr]\edef\p@int{\p@inttr}\repeat}}
\ctr@ld@f\def\subB@zierTD#1[#2,#3,#4,#5]{\delt@=\p@\divide\delt@\NBz@rcs\v@lmin=\z@%
    {\Figg@tXY{-7}\edef\X@un{\the\v@lX}\edef\Y@un{\the\v@lY}%
    \s@mme=\z@\loop\ifnum\s@mme<#1\advance\s@mme\@ne%
    \v@leur=\v@lmin\advance\v@leur0.33333 \delt@\edef\unti@rs{\repdecn@mb{\v@leur}}%
    \v@leur=\v@lmin\advance\v@leur0.66666 \delt@\edef\deti@rs{\repdecn@mb{\v@leur}}%
    \advance\v@lmin\delt@\edef\trti@rs{\repdecn@mb{\v@lmin}}%
    \figptBezierTD-8::\trti@rs[#2,#3,#4,#5]\Figptpr@j-8:/-8/%
    \c@lsubBzarc\unti@rs,\deti@rs[#2,#3,#4,#5]\BdingB@xfalse%
    \PSwrit@cmdS{-4}{}{\fwf@g}{\X@de}{\Y@de}\PSwrit@cmdS{-3}{}{\fwf@g}{\X@tr}{\Y@tr}%
    \BdingB@xtrue\PSwrit@cmdS{-8}{\c@mcurveto}{\fwf@g}{\X@qu}{\Y@qu}%
    \B@zierBB@x{1}{\Y@un}(\X@un,\X@de,\X@tr,\X@qu)%
    \B@zierBB@x{2}{\X@un}(\Y@un,\Y@de,\Y@tr,\Y@qu)%
    \edef\X@un{\X@qu}\edef\Y@un{\Y@qu}\figptcopyDD-7:/-8/\repeat}}
\ctr@ld@f\def\NBz@rcs{2}
\ctr@ld@f\def\c@lsubBzarc#1,#2[#3,#4,#5,#6]{\figptBezierTD-5::#1[#3,#4,#5,#6]%
    \figptBezierTD-6::#2[#3,#4,#5,#6]\Figptpr@j-4:/-5/\Figptpr@j-5:/-6/%
    \figptscontrolDD-4[-7,-4,-5,-8]}
\ctr@ln@m\figdrawcirc
\ctr@ld@f\def\Q@circDD#1(#2){\ifCUR@PS\ifGR@cri\PSc@mment{circDD Center=#1 (Radius=#2)}%
    \Q@arccircDD#1;#2(0,360)\PSc@mment{End circDD}\fi\fi}
\ctr@ld@f\def\Q@circTD#1,#2,#3(#4){\ifCUR@PS\ifGR@cri%
    \PSc@mment{circTD Center=#1,P1=#2,P2=#3 (Radius=#4)}%
    \Q@arccircTD#1,#2,#3;#4(0,360)\PSc@mment{End circTD}\fi\fi}
\ctr@ln@m\p@urcent
{\catcode`\%=12\gdef\p@urcent{
\ctr@ld@f\def\PSc@mment#1{\ifGRdebugm@de\immediate\write\fwf@g{\p@urcent\space#1}\fi}
\ctr@ln@m\acc@louv \ctr@ln@m\acc@lfer
{\catcode`\[=1\catcode`\{=12\gdef\acc@louv[{}}
{\catcode`\]=2\catcode`\}=12\gdef\acc@lfer{}]]
\ctr@ld@f\def\PSdict@{\ifUse@llipse%
    \immediate\write\fwf@g{/ellipsedict 9 dict def ellipsedict /mtrx matrix put}%
    \immediate\write\fwf@g{/ellipse \acc@louv ellipsedict begin}%
    \immediate\write\fwf@g{ /endangle exch def /startangle exch def}%
    \immediate\write\fwf@g{ /yrad exch def /xrad exch def}%
    \immediate\write\fwf@g{ /rotangle exch def /y exch def /x exch def}%
    \immediate\write\fwf@g{ /savematrix mtrx currentmatrix def}%
    \immediate\write\fwf@g{ x y translate rotangle rotate xrad yrad scale}%
    \immediate\write\fwf@g{ 0 0 1 startangle endangle arc}%
    \immediate\write\fwf@g{ savematrix setmatrix end\acc@lfer def}%
    \fi\PShe@der{EndProlog}}
\ctr@ld@f\def\Pssetc@rve#1=#2|{\keln@mun#1|%
    \def\n@mref{r}\ifx\l@debut\n@mref\update@ttr\D@FTroundness\Q@s@troundness{#2}\else
    \W@rnmesAttr{figset curve}{#1}\fi}
\ctr@ln@m\curv@roundness
\ctr@ld@f\def\Q@s@troundness#1{\edef\curv@roundness{#1}}
\ctr@ld@f\def\D@FTroundness{0.2} 
\ctr@ln@m\figdrawcurve
\ctr@ld@f\def\Q@curveDD[#1]{{\ifCUR@PS\ifGR@cri\PSc@mment{curveDD Points=#1}%
    \s@uvc@ntr@l\et@tpscurveDD%
    \iffillm@de\Psc@rveDD\curv@roundness[#1]%
    \f@gfill%
    \else\Psc@rveDD\curv@roundness[#1]\f@gstroke\fi%
    \PSc@mment{End curveDD}\resetc@ntr@l\et@tpscurveDD\fi\fi}}
\ctr@ld@f\def\Q@curveTD[#1]{{\ifCUR@PS\ifGR@cri%
    \PSc@mment{curveTD Points=#1}\s@uvc@ntr@l\et@tpscurveTD\let\c@lprojSP=\relax%
    \iffillm@de\Psc@rveTD\curv@roundness[#1]%
    \f@gfill%
    \else\Psc@rveTD\curv@roundness[#1]\f@gstroke\fi%
    \PSc@mment{End curveTD}\resetc@ntr@l\et@tpscurveTD\fi\fi}}
\ctr@ld@f\def\Psc@rveDD#1[#2]{%
    \def\list@num{#2}\extrairelepremi@r\Ak@\de\list@num%
    \extrairelepremi@r\Ai@\de\list@num\extrairelepremi@r\Aj@\de\list@num%
    \f@gnewpath\PSwrit@cmdS{\Ai@}{\c@mmoveto}{\fwf@g}{\X@un}{\Y@un}%
    \setc@ntr@l{2}\figvectPDD -1[\Ak@,\Aj@]%
    \@ecfor\Ak@:=\list@num\do{\figpttraDD-2:=\Ai@/#1,-1/\BdingB@xfalse%
       \PSwrit@cmdS{-2}{}{\fwf@g}{\X@de}{\Y@de}%
       \figvectPDD -1[\Ai@,\Ak@]\figpttraDD-2:=\Aj@/-#1,-1/%
       \PSwrit@cmdS{-2}{}{\fwf@g}{\X@tr}{\Y@tr}\BdingB@xtrue%
       \PSwrit@cmdS{\Aj@}{\c@mcurveto}{\fwf@g}{\X@qu}{\Y@qu}%
       \B@zierBB@x{1}{\Y@un}(\X@un,\X@de,\X@tr,\X@qu)%
       \B@zierBB@x{2}{\X@un}(\Y@un,\Y@de,\Y@tr,\Y@qu)%
       \edef\X@un{\X@qu}\edef\Y@un{\Y@qu}\edef\Ai@{\Aj@}\edef\Aj@{\Ak@}}}
\ctr@ld@f\def\Psc@rveTD#1[#2]{\ifnum\CUR@proj<\tw@\Psc@rvePPTD#1[#2]\else\Psc@rveCPTD#1[#2]\fi}
\ctr@ld@f\def\Psc@rvePPTD#1[#2]{\setc@ntr@l{2}%
    \def\list@num{#2}\extrairelepremi@r\Ak@\de\list@num\Figptpr@j-5:/\Ak@/%
    \extrairelepremi@r\Ai@\de\list@num\Figptpr@j-3:/\Ai@/%
    \extrairelepremi@r\Aj@\de\list@num\Figptpr@j-4:/\Aj@/%
    \f@gnewpath\PSwrit@cmdS{-3}{\c@mmoveto}{\fwf@g}{\X@un}{\Y@un}%
    \figvectPDD -1[-5,-4]%
    \@ecfor\Ak@:=\list@num\do{\Figptpr@j-5:/\Ak@/\figpttraDD-2:=-3/#1,-1/%
       \BdingB@xfalse\PSwrit@cmdS{-2}{}{\fwf@g}{\X@de}{\Y@de}%
       \figvectPDD -1[-3,-5]\figpttraDD-2:=-4/-#1,-1/%
       \PSwrit@cmdS{-2}{}{\fwf@g}{\X@tr}{\Y@tr}\BdingB@xtrue%
       \PSwrit@cmdS{-4}{\c@mcurveto}{\fwf@g}{\X@qu}{\Y@qu}%
       \B@zierBB@x{1}{\Y@un}(\X@un,\X@de,\X@tr,\X@qu)%
       \B@zierBB@x{2}{\X@un}(\Y@un,\Y@de,\Y@tr,\Y@qu)%
       \edef\X@un{\X@qu}\edef\Y@un{\Y@qu}\figptcopyDD-3:/-4/\figptcopyDD-4:/-5/}}
\ctr@ld@f\def\Psc@rveCPTD#1[#2]{\setc@ntr@l{2}%
    \def\list@num{#2}\extrairelepremi@r\Ak@\de\list@num%
    \extrairelepremi@r\Ai@\de\list@num\extrairelepremi@r\Aj@\de\list@num%
    \Figptpr@j-7:/\Ai@/%
    \f@gnewpath\PSwrit@cmd{-7}{\c@mmoveto}{\fwf@g}%
    \figvectPTD -9[\Ak@,\Aj@]%
    \@ecfor\Ak@:=\list@num\do{\figpttraTD-10:=\Ai@/#1,-9/%
       \figvectPTD -9[\Ai@,\Ak@]\figpttraTD-11:=\Aj@/-#1,-9/%
       \subB@zierTD\NBz@rcs[\Ai@,-10,-11,\Aj@]\edef\Ai@{\Aj@}\edef\Aj@{\Ak@}}}
\ctr@ld@f\def\figdrawend{\ifCUR@PS\ifGR@cri\immediate\closeout\fwf@g%
    \immediate\openout\fwf@g=\PSfilen@me\relax%
    \ifPDFm@ke\PSBdingB@x\else%
    \immediate\write\fwf@g{\p@urcent\string!PS-Adobe-2.0 EPSF-2.0}%
    \PShe@der{Creator\string: TeX (fig4tex.tex)}%
    \PShe@der{Title\string: \PSfilen@me}%
    \PShe@der{CreationDate\string: \the\day/\the\month/\the\year}%
    \PSBdingB@x%
    \PShe@der{EndComments}\PSdict@\fi%
    \immediate\write\fwf@g{\c@mgsave}%
    \openin\frf@g=\auxfilen@me\c@pypsfile\fwf@g\frf@g\closein\frf@g%
    \immediate\write\fwf@g{\c@mgrestore}%
    \PSc@mment{End of file.}\immediate\closeout\fwf@g%
    \immediate\openout\fwf@g=\auxfilen@me\immediate\closeout\fwf@g%
    \immediate\write16{File \PSfilen@me\space created.}\fi\fi\CUR@PSfalse\GR@critrue}
\ctr@ld@f\def\PShe@der#1{\immediate\write\fwf@g{\p@urcent\p@urcent#1}}
\ctr@ld@f\def\PSBdingB@x{{\v@lX=\ptT@ptps\c@@rdXmin\v@lY=\ptT@ptps\c@@rdYmin%
     \v@lXa=\ptT@ptps\c@@rdXmax\v@lYa=\ptT@ptps\c@@rdYmax%
     \PShe@der{BoundingBox\string: \repdecn@mb{\v@lX}\space\repdecn@mb{\v@lY}%
     \space\repdecn@mb{\v@lXa}\space\repdecn@mb{\v@lYa}}}}
\ctr@ld@f\def\figdrawfcconnect[#1]{{\ifCUR@PS\ifGR@cri\PSc@mment{fcconnect Points=#1}%
    \Q@s@tfillmode{no}\s@uvc@ntr@l\et@tpsfcconnect\resetc@ntr@l{2}%
    \fcc@nnect@[#1]\resetc@ntr@l\et@tpsfcconnect\PSc@mment{End fcconnect}\fi\fi}}
\ctr@ld@f\def\fcc@nnect@[#1]{\let\N@rm=\n@rmeucDD\def\list@num{#1}%
    \extrairelepremi@r\Ai@\de\list@num\edef\pr@m{\Ai@}\v@leur=\z@\p@rtent=\@ne\c@llgtot%
    \ifcase\fclin@typ@\edef\list@num{[\pr@m,#1,\Ai@}\expandafter\figdrawcurve\list@num]%
    \else\ifdim\fclin@r@d\p@>\z@\Pslin@conge[#1]\else\figdrawline[#1]\fi\fi%
    \v@leur=\@rrowp@s\v@leur\edef\list@num{#1,\Ai@,0}%
    \extrairelepremi@r\Ai@\de\list@num\mili@u=\epsil@n\c@llgpart%
    \advance\mili@u-\epsil@n\advance\mili@u-\delt@\advance\v@leur-\mili@u%
    \ifcase\fclin@typ@\invers@\mili@u\delt@%
    \ifnum\@rrowr@fpt>\z@\advance\delt@-\v@leur\v@leur=\delt@\fi%
    \v@leur=\repdecn@mb\v@leur\mili@u\edef\v@lt{\repdecn@mb\v@leur}%
    \extrairelepremi@r\Ak@\de\list@num%
    \figvectPDD-1[\pr@m,\Aj@]\figpttraDD-6:=\Ai@/\curv@roundness,-1/%
    \figvectPDD-1[\Ak@,\Ai@]\figpttraDD-7:=\Aj@/\curv@roundness,-1/%
    \delt@=\@rrowheadlength\p@\delt@=\C@AHANG\delt@\edef\R@dius{\repdecn@mb{\delt@}}%
    \ifcase\@rrowr@fpt%
    \FigptintercircB@zDD-8::\v@lt,\R@dius[\Ai@,-6,-7,\Aj@]\Q@arrowheadDD[-5,-8]\else%
    \FigptintercircB@zDD-8::\v@lt,\R@dius[\Aj@,-7,-6,\Ai@]\Q@arrowheadDD[-8,-5]\fi%
    \else\advance\delt@-\v@leur%
    \p@rtentiere{\p@rtent}{\delt@}\edef\C@efun{\the\p@rtent}%
    \p@rtentiere{\p@rtent}{\v@leur}\edef\C@efde{\the\p@rtent}%
    \figptbaryDD-5:[\Ai@,\Aj@;\C@efun,\C@efde]\ifcase\@rrowr@fpt%
    \delt@=\@rrowheadlength\unit@\delt@=\C@AHANG\delt@\edef\t@ille{\repdecn@mb{\delt@}}%
    \figvectPDD-2[\Ai@,\Aj@]\vecunit@{-2}{-2}\figpttraDD-5:=-5/\t@ille,-2/\fi%
    \Q@arrowheadDD[\Ai@,-5]\fi}
\ctr@ld@f\def\c@llgtot{\@ecfor\Aj@:=\list@num\do{\figvectP-1[\Ai@,\Aj@]\N@rm\delt@{-1}%
    \advance\v@leur\delt@\advance\p@rtent\@ne\edef\Ai@{\Aj@}}}
\ctr@ld@f\def\c@llgpart{\extrairelepremi@r\Aj@\de\list@num\figvectP-1[\Ai@,\Aj@]\N@rm\delt@{-1}%
    \advance\mili@u\delt@\ifdim\mili@u<\v@leur\edef\pr@m{\Ai@}\edef\Ai@{\Aj@}\c@llgpart\fi}
\ctr@ld@f\def\Pslin@conge[#1]{\ifnum\p@rtent>\tw@{\def\list@num{#1}%
    \extrairelepremi@r\Ai@\de\list@num\extrairelepremi@r\Aj@\de\list@num%
    \figptcopy-6:/\Ai@/\figvectP-3[\Ai@,\Aj@]\vecunit@{-3}{-3}\v@lmax=\result@t%
    \@ecfor\Ak@:=\list@num\do{\figvectP-4[\Aj@,\Ak@]\vecunit@{-4}{-4}%
    \minim@m\v@lmin\v@lmax\result@t\v@lmax=\result@t%
    \det@rm\delt@[-3,-4]\maxim@m\mili@u{\delt@}{-\delt@}\ifdim\mili@u>\Cepsil@n%
    \ifdim\delt@>\z@\figgetangleDD\Angl@[\Aj@,\Ak@,\Ai@]\else%
    \figgetangleDD\Angl@[\Aj@,\Ai@,\Ak@]\fi%
    \v@leur=\PI@deg\advance\v@leur-\Angl@\p@\divide\v@leur\tw@%
    \edef\Angl@{\repdecn@mb\v@leur}\c@ssin{\C@}{\S@}{\Angl@}\v@leur=\fclin@r@d\unit@%
    \v@leur=\S@\v@leur\mili@u=\C@\p@\invers@\mili@u\mili@u%
    \v@leur=\repdecn@mb{\mili@u}\v@leur%
    \minim@m\v@leur\v@leur\v@lmin\edef\t@ille{\repdecn@mb{\v@leur}}%
    \figpttra-5:=\Aj@/-\t@ille,-3/\figdrawline[-6,-5]\figpttra-6:=\Aj@/\t@ille,-4/%
    \figvectNVDD-3[-3]\figvectNVDD-8[-4]\inters@cDD-7:[-5,-3;-6,-8]%
    \ifdim\delt@>\z@\figdrawarccircP-7;\fclin@r@d[-5,-6]\else\figdrawarccircP-7;\fclin@r@d[-6,-5]\fi%
    \else\figdrawline[-6,\Aj@]\figptcopy-6:/\Aj@/\fi
    \edef\Ai@{\Aj@}\edef\Aj@{\Ak@}\figptcopy-3:/-4/}\figdrawline[-6,\Aj@]}\else\figdrawline[#1]\fi}
\ctr@ld@f\def\figdrawfcnode[#1]#2{{\ifCUR@PS\ifGR@cri\PSc@mment{fcnode Points=#1}%
    \s@uvc@ntr@l\et@tpsfcnode\resetc@ntr@l{2}%
    \def\t@xt@{#2}\ifx\t@xt@\empty\def\g@tt@xt{\setbox\Gb@x=\hbox{\Figg@tT{\p@int}}}%
    \else\def\g@tt@xt{\setbox\Gb@x=\hbox{#2}}\fi%
    \v@lmin=\h@rdfcXp@dd\advance\v@lmin\Xp@dd\unit@\multiply\v@lmin\tw@%
    \v@lmax=\h@rdfcYp@dd\advance\v@lmax\Yp@dd\unit@\multiply\v@lmax\tw@%
    \Figv@ctCreg-8(\unit@,-\unit@)\def\list@num{#1}%
    \delt@=\CUR@width bp\divide\delt@\tw@%
    \fcn@de\PSc@mment{End fcnode}\resetc@ntr@l\et@tpsfcnode\fi\fi}}
\ctr@ld@f\def\d@butn@de{\g@tt@xt\v@lX=\wd\Gb@x%
    \v@lY=\ht\Gb@x\advance\v@lY\dp\Gb@x\advance\v@lX\v@lmin\advance\v@lY\v@lmax}
\ctr@ld@f\def\fcn@deE{%
    \@ecfor\p@int:=\list@num\do{\d@butn@de\v@lX=\unssqrttw@\v@lX\v@lY=\unssqrttw@\v@lY%
    \ifdim\thickn@ss\p@>\z@
    \v@lXa=\v@lX\advance\v@lXa\delt@\v@lXa=\ptT@unit@\v@lXa\edef\XR@d{\repdecn@mb\v@lXa}%
    \v@lYa=\v@lY\advance\v@lYa\delt@\v@lYa=\ptT@unit@\v@lYa\edef\YR@d{\repdecn@mb\v@lYa}%
    \arct@n\v@leur(\v@lXa,\v@lYa)\v@leur=\rdT@deg\v@leur\edef\@nglde{\repdecn@mb\v@leur}%
    {\c@lptellDD-2::\p@int;\XR@d,\YR@d(\@nglde)}
    \advance\v@leur-\PI@deg\edef\@nglun{\repdecn@mb\v@leur}%
    {\c@lptellDD-3::\p@int;\XR@d,\YR@d(\@nglun)}%
    \figptstra-6=-3,-2,\p@int/\thickn@ss,-8/\Q@s@tfillmode{yes}%
    \Pss@tspecifSt{color=\DDV@thickcolor}%
    \figdrawline[-2,-3,-6,-5]\figdrawarcell-4;\XR@d,\YR@d(\@nglun,\@nglde,0)%
    \Psrest@reSt{color=\DDV@thickcolor}\fi
    \v@lX=\ptT@unit@\v@lX\v@lY=\ptT@unit@\v@lY%
    \edef\XR@d{\repdecn@mb\v@lX}\edef\YR@d{\repdecn@mb\v@lY}%
    \Q@s@tfillmode{yes}\Pss@tspecifSt{color=\fcbgc@lor}%
    \figdrawarcell\p@int;\XR@d,\YR@d(0,360,0)%
    \Q@s@tfillmode{no}\Psrest@reSt{color=\fcbgc@lor}\figdrawarcell\p@int;\XR@d,\YR@d(0,360,0)}}
\ctr@ld@f\def\fcn@deL{\delt@=\ptT@unit@\delt@\edef\t@ille{\repdecn@mb\delt@}%
    \@ecfor\p@int:=\list@num\do{\Figg@tXYa{\p@int}\d@butn@de%
    \ifdim\v@lX>\v@lY\itis@Ktrue\else\itis@Kfalse\fi%
    \advance\v@lXa-\v@lX\Figp@intreg-1:(\v@lXa,\v@lYa)%
    \advance\v@lXa\v@lX\advance\v@lYa-\v@lY\Figp@intreg-2:(\v@lXa,\v@lYa)%
    \advance\v@lXa\v@lX\advance\v@lYa\v@lY\Figp@intreg-3:(\v@lXa,\v@lYa)%
    \advance\v@lXa-\v@lX\advance\v@lYa\v@lY\Figp@intreg-4:(\v@lXa,\v@lYa)%
    \ifdim\thickn@ss\p@>\z@
    \Figg@tXYa{\p@int}\Q@s@tfillmode{yes}\Pss@tspecifSt{color=\DDV@thickcolor}%
    \c@lpt@xt{-1}{-4}\c@lpt@xt@\v@lXa\v@lYa\v@lX\v@lY\c@rre\delt@%
    \Figp@intregDD-9:(\v@lZ,\v@lYa)\Figp@intregDD-11:(\v@lZa,\v@lYa)%
    \c@lpt@xt{-4}{-3}\c@lpt@xt@\v@lYa\v@lXa\v@lY\v@lX\delt@\c@rre%
    \Figp@intregDD-12:(\v@lXa,\v@lZ)\Figp@intregDD-10:(\v@lXa,\v@lZa)%
    \ifitis@K\figptstra-7=-9,-10,-11/\thickn@ss,-8/\figdrawline[-9,-11,-5,-6,-7]\else%
    \figptstra-7=-10,-11,-12/\thickn@ss,-8/\figdrawline[-10,-12,-5,-6,-7]\fi%
    \Psrest@reSt{color=\DDV@thickcolor}\fi
    \Q@s@tfillmode{yes}\Pss@tspecifSt{color=\fcbgc@lor}\figdrawline[-1,-2,-3,-4]%
    \Q@s@tfillmode{no}\Psrest@reSt{color=\fcbgc@lor}\figdrawline[-1,-2,-3,-4,-1]}}
\ctr@ld@f\def\c@lpt@xt#1#2{\figvectN-7[#1,#2]\vecunit@{-7}{-7}\figpttra-5:=#1/\t@ille,-7/%
    \figvectP-7[#1,#2]\Figg@tXY{-7}\c@rre=\v@lX\delt@=\v@lY\Figg@tXY{-5}}
\ctr@ld@f\def\c@lpt@xt@#1#2#3#4#5#6{\v@lZ=#6\invers@{\v@lZ}{\v@lZ}\v@leur=\repdecn@mb{#5}\v@lZ%
    \v@lZ=#2\advance\v@lZ-#4\mili@u=\repdecn@mb{\v@leur}\v@lZ%
    \v@lZ=#3\advance\v@lZ\mili@u\v@lZa=-\v@lZ\advance\v@lZa\tw@#1}
\ctr@ld@f\def\fcn@deR{\@ecfor\p@int:=\list@num\do{\Figg@tXYa{\p@int}\d@butn@de%
    \advance\v@lXa-0.5\v@lX\advance\v@lYa-0.5\v@lY\Figp@intreg-1:(\v@lXa,\v@lYa)%
    \advance\v@lXa\v@lX\Figp@intreg-2:(\v@lXa,\v@lYa)%
    \advance\v@lYa\v@lY\Figp@intreg-3:(\v@lXa,\v@lYa)%
    \advance\v@lXa-\v@lX\Figp@intreg-4:(\v@lXa,\v@lYa)%
    \ifdim\thickn@ss\p@>\z@
    \Q@s@tfillmode{yes}\Pss@tspecifSt{color=\DDV@thickcolor}%
    \Figv@ctCreg-5(-\delt@,-\delt@)\figpttra-9:=-1/1,-5/%
    \Figv@ctCreg-5(\delt@,-\delt@)\figpttra-10:=-2/1,-5/%
    \Figv@ctCreg-5(\delt@,\delt@)\figpttra-11:=-3/1,-5/%
    \figptstra-7=-9,-10,-11/\thickn@ss,-8/\figdrawline[-9,-11,-5,-6,-7]%
    \Psrest@reSt{color=\DDV@thickcolor}\fi
    \Q@s@tfillmode{yes}\Pss@tspecifSt{color=\fcbgc@lor}\figdrawline[-1,-2,-3,-4]%
    \Q@s@tfillmode{no}\Psrest@reSt{color=\fcbgc@lor}\figdrawline[-1,-2,-3,-4,-1]}}
\ctr@ld@f\def\Pssetfl@wchart#1=#2|{\keln@mtr#1|%
    \def\n@mref{arr}\ifx\l@debut\n@mref\expandafter\keln@mtr\l@suite|%
     \def\n@mref{owp}\ifx\l@debut\n@mref\update@ttr\D@FTfcarrowposition\P@setfcarrowposition{#2}\else
     \def\n@mref{owr}\ifx\l@debut\n@mref\update@ttr\D@FTfcarrowrefpt\P@setfcarrowrefpt{#2}\else
     \W@rnmesAttr{figset flowchart}{#1}\fi\fi\else%
    \def\n@mref{bgc}\ifx\l@debut\n@mref\update@ttr\D@FTfcbgcolor\P@setfcbgcolor{#2}\else
    \def\n@mref{lin}\ifx\l@debut\n@mref\update@ttr\D@FTfcline\P@setfcline{#2}\else
    \def\n@mref{pad}\ifx\l@debut\n@mref\update@ttr\D@FTfcxpadding\P@setfcxpadding{#2}%
                                       \update@ttr\D@FTfcypadding\P@setfcypadding{#2}\else
    \def\n@mref{rad}\ifx\l@debut\n@mref\update@ttr\D@FTfcradius\P@setfcradius{#2}\else
    \def\n@mref{sha}\ifx\l@debut\n@mref\update@ttr\D@FTfcshape\P@setfcshape{#2}\else
    \def\n@mref{thi}\ifx\l@debut\n@mref\expandafter\keln@mtr\l@suite|%
     \def\n@mref{ckc}\ifx\l@debut\n@mref\update@ttr\D@FTref\P@setfcthickcolor{#2}\else
     \def\n@mref{ckn}\ifx\l@debut\n@mref\update@ttr\D@FTfcthickness\P@setfcthickness{#2}\else
     \W@rnmesAttr{figset flowchart}{#1}\fi\fi\else%
    \def\n@mref{xpa}\ifx\l@debut\n@mref\update@ttr\D@FTfcxpadding\P@setfcxpadding{#2}\else
    \def\n@mref{ypa}\ifx\l@debut\n@mref\update@ttr\D@FTfcypadding\P@setfcypadding{#2}\else
    \W@rnmesAttr{figset flowchart}{#1}\fi\fi\fi\fi\fi\fi\fi\fi\fi}
\ctr@ln@m\@rrowp@s
\ctr@ld@f\def\P@setfcarrowposition#1{\edef\@rrowp@s{#1}}
\ctr@ln@m\@rrowr@fpt
\ctr@ld@f\def\P@setfcarrowrefpt#1{\setfcr@fpt#1|}
\ctr@ld@f\def\setfcr@fpt#1#2|{\if#1e\def\@rrowr@fpt{1}\else\def\@rrowr@fpt{0}\fi}
\ctr@ln@m\fcbgc@lor
\ctr@ld@f\def\P@setfcbgcolor#1{\edef\fcbgc@lor{#1}}
\ctr@ln@m\fclin@typ@
\ctr@ld@f\def\P@setfcline#1{\setfccurv@#1|}
\ctr@ld@f\def\setfccurv@#1#2|{\if#1c\def\fclin@typ@{0}\else\def\fclin@typ@{1}\fi}
\ctr@ln@m\fclin@r@d
\ctr@ld@f\def\P@setfcradius#1{\edef\fclin@r@d{#1}}
\ctr@ln@m\fcn@de \ctr@ln@m\fcsh@pe
\ctr@ln@m\h@rdfcXp@dd \ctr@ln@m\h@rdfcYp@dd
\ctr@ld@f\def\P@setfcshape#1{\setfcshap@#1|}
\ctr@ld@f\def\setfcshap@#1#2|{%
    \if#1e\let\fcn@de=\fcn@deE\def\h@rdfcXp@dd{4pt}\def\h@rdfcYp@dd{4pt}%
     \edef\fcsh@pe{ellipse}\else%
    \if#1l\let\fcn@de=\fcn@deL\def\h@rdfcXp@dd{4pt}\def\h@rdfcYp@dd{4pt}%
     \edef\fcsh@pe{lozenge}\else%
          \let\fcn@de=\fcn@deR\def\h@rdfcXp@dd{6pt}\def\h@rdfcYp@dd{6pt}%
     \edef\fcsh@pe{rectangle}\fi\fi}
\ctr@ln@m\DDV@thickcolor
\ctr@ld@f\def\P@setfcthickcolor#1{\edef\DDV@thickcolor{#1}}
\ctr@ln@m\thickn@ss
\ctr@ld@f\def\P@setfcthickness#1{\edef\thickn@ss{#1}}
\ctr@ln@m\Xp@dd
\ctr@ld@f\def\P@setfcxpadding#1{\edef\Xp@dd{#1}}
\ctr@ln@m\Yp@dd
\ctr@ld@f\def\P@setfcypadding#1{\edef\Yp@dd{#1}}
\ctr@ld@f\def\figdrawline[#1]{{\ifCUR@PS\ifGR@cri\PSc@mment{line Points=#1}%
    \let\figdrawlign@=\Pslign@P\Pslin@{#1}\PSc@mment{End line}\fi\fi}}
\ctr@ld@f\def\figdrawlineF#1{{\ifCUR@PS\ifGR@cri\PSc@mment{lineF Filename=#1}%
    \let\figdrawlign@=\Pslign@F\Pslin@{#1}\PSc@mment{End lineF}\fi\fi}}
\ctr@ld@f\def\figdrawlineC(#1){{\ifCUR@PS\ifGR@cri\PSc@mment{lineC}%
    \let\figdrawlign@=\Pslign@C\Pslin@{#1}\PSc@mment{End lineC}\fi\fi}}
\ctr@ld@f\def\Pslin@#1{\iffillm@de\figdrawlign@{#1}%
    \f@gfill%
    \else\figdrawlign@{#1}\ifx\derp@int\premp@int%
    \f@gclosestroke%
    \else\f@gstroke\fi\fi}
\ctr@ld@f\def\Pslign@P#1{\def\list@num{#1}\extrairelepremi@r\p@int\de\list@num%
    \edef\premp@int{\p@int}\f@gnewpath%
    \PSwrit@cmd{\p@int}{\c@mmoveto}{\fwf@g}%
    \@ecfor\p@int:=\list@num\do{\PSwrit@cmd{\p@int}{\c@mlineto}{\fwf@g}%
    \edef\derp@int{\p@int}}}
\ctr@ld@f\def\Pslign@F#1{\s@uvc@ntr@l\et@tPslign@F\setc@ntr@l{2}\openin\frf@g=#1\relax%
    \ifeof\frf@g\message{*** File #1 not found !}\end\else%
    \read\frf@g to\tr@c\edef\premp@int{\tr@c}\expandafter\extr@ctCF\tr@c:%
    \f@gnewpath\PSwrit@cmd{-1}{\c@mmoveto}{\fwf@g}%
    \loop\read\frf@g to\tr@c\ifeof\frf@g\mored@tafalse\else\mored@tatrue\fi%
    \ifmored@ta\expandafter\extr@ctCF\tr@c:\PSwrit@cmd{-1}{\c@mlineto}{\fwf@g}%
    \edef\derp@int{\tr@c}\repeat\fi\closein\frf@g\resetc@ntr@l\et@tPslign@F}
\ctr@ln@m\extr@ctCF
\ctr@ld@f\def\extr@ctCFDD#1 #2:{\v@lX=#1\unit@\v@lY=#2\unit@\Figp@intregDD-1:(\v@lX,\v@lY)}
\ctr@ld@f\def\extr@ctCFTD#1 #2 #3:{\v@lX=#1\unit@\v@lY=#2\unit@\v@lZ=#3\unit@%
    \Figp@intregTD-1:(\v@lX,\v@lY,\v@lZ)}
\ctr@ld@f\def\Pslign@C#1{\s@uvc@ntr@l\et@tPslign@C\setc@ntr@l{2}%
    \def\list@num{#1}\extrairelepremi@r\p@int\de\list@num%
    \edef\premp@int{\p@int}\f@gnewpath%
    \expandafter\Pslign@C@\p@int:\PSwrit@cmd{-1}{\c@mmoveto}{\fwf@g}%
    \@ecfor\p@int:=\list@num\do{\expandafter\Pslign@C@\p@int:%
    \PSwrit@cmd{-1}{\c@mlineto}{\fwf@g}\edef\derp@int{\p@int}}%
    \resetc@ntr@l\et@tPslign@C}
\ctr@ld@f\def\Pslign@C@#1 #2:{{\def\t@xt@{#1}\ifx\t@xt@\empty\Pslign@C@#2:
    \else\extr@ctCF#1 #2:\fi}}
\ctr@ld@f\def\Pssetm@sh#1=#2|{\keln@mde#1|%
    \def\n@mref{co}\ifx\l@debut\n@mref\update@ttr\D@FTref\P@setmeshcolor{#2}\else
    \def\n@mref{da}\ifx\l@debut\n@mref\update@ttr\D@FTref\P@setmeshdash{#2}\else
    \def\n@mref{di}\ifx\l@debut\n@mref\update@ttr\D@FTmeshdiag\Q@s@tmeshdiag{#2}\else
    \def\n@mref{wi}\ifx\l@debut\n@mref\update@ttr\D@FTref\P@setmeshwidth{#2}\else
    \W@rnmesAttr{figset mesh}{#1}\fi\fi\fi\fi}
\ctr@ln@m\c@ntrolmesh
\ctr@ld@f\def\Q@s@tmeshdiag#1{\edef\c@ntrolmesh{#1}}
\ctr@ld@f\def\D@FTmeshdiag{0}    
\ctr@ln@m\DDV@meshcolor
\ctr@ld@f\def\P@setmeshcolor#1{\edef\DDV@meshcolor{#1}}
\ctr@ln@m\DDV@meshdash
\ctr@ld@f\def\P@setmeshdash#1{\edef\DDV@meshdash{#1}}
\ctr@ln@m\DDV@meshwidth
\ctr@ld@f\def\P@setmeshwidth#1{\edef\DDV@meshwidth{#1}}
\ctr@ld@f\def\figdrawmesh#1,#2[#3,#4,#5,#6]{{\ifCUR@PS\ifGR@cri%
    \PSc@mment{mesh N1=#1, N2=#2, Quadrangle=[#3,#4,#5,#6]}\s@uvc@ntr@l\et@tpsmesh%
    \Pss@tspecifSt{color=\DDV@meshcolor,dash=\DDV@meshdash,width=\DDV@meshwidth}%
    \setc@ntr@l{2}%
    \ifnum#1>\@ne\Psmeshp@rt#1[#3,#4,#5,#6]\fi%
    \ifnum#2>\@ne\Psmeshp@rt#2[#4,#5,#6,#3]\fi%
    \ifnum\c@ntrolmesh>\z@\Psmeshdi@g#1,#2[#3,#4,#5,#6]\fi%
    \ifnum\c@ntrolmesh<\z@\Psmeshdi@g#2,#1[#4,#5,#6,#3]\fi%
    \Psrest@reSt{color=\DDV@meshcolor,dash=\DDV@meshdash,width=\DDV@meshwidth}%
    \figdrawline[#3,#4,#5,#6,#3]\PSc@mment{End mesh}\resetc@ntr@l\et@tpsmesh\fi\fi}}
\ctr@ld@f\def\Psmeshp@rt#1[#2,#3,#4,#5]{{\l@mbd@un=\@ne\l@mbd@de=#1\loop%
    \ifnum\l@mbd@un<#1\advance\l@mbd@de\m@ne\figptbary-1:[#2,#3;\l@mbd@de,\l@mbd@un]%
    \figptbary-2:[#5,#4;\l@mbd@de,\l@mbd@un]\figdrawline[-1,-2]\advance\l@mbd@un\@ne\repeat}}
\ctr@ld@f\def\Psmeshdi@g#1,#2[#3,#4,#5,#6]{\figptcopy-2:/#3/\figptcopy-3:/#6/%
    \l@mbd@un=\z@\l@mbd@de=#1\loop\ifnum\l@mbd@un<#1%
    \advance\l@mbd@un\@ne\advance\l@mbd@de\m@ne\figptcopy-1:/-2/\figptcopy-4:/-3/%
    \figptbary-2:[#3,#4;\l@mbd@de,\l@mbd@un]%
    \figptbary-3:[#6,#5;\l@mbd@de,\l@mbd@un]\Psmeshdi@gp@rt#2[-1,-2,-3,-4]\repeat}
\ctr@ld@f\def\Psmeshdi@gp@rt#1[#2,#3,#4,#5]{{\l@mbd@un=\z@\l@mbd@de=#1\loop%
    \ifnum\l@mbd@un<#1\figptbary-5:[#2,#5;\l@mbd@de,\l@mbd@un]%
    \advance\l@mbd@de\m@ne\advance\l@mbd@un\@ne%
    \figptbary-6:[#3,#4;\l@mbd@de,\l@mbd@un]\figdrawline[-5,-6]\repeat}}
\ctr@ln@m\figdrawnormal
\ctr@ld@f\def\Q@normalDD#1,#2[#3,#4]{{\ifCUR@PS\ifGR@cri%
    \PSc@mment{normal Length=#1, Lambda=#2 [Pt1,Pt2]=[#3,#4]}%
    \s@uvc@ntr@l\et@tpsnormal\resetc@ntr@l{2}\figptendnormal-6::#1,#2[#3,#4]%
    \figptcopyDD-5:/-1/\figdrawarrow[-5,-6]%
    \PSc@mment{End normal}\resetc@ntr@l\et@tpsnormal\fi\fi}}
\ctr@ld@f\def\figreset#1{\trtlis@rg{#1}{\Psreset@}}
\ctr@ld@f\def\Psreset@#1|{\def\t@xt@{#1}\ifx\t@xt@\empty\P@resetg@n
    \else\keln@mde#1|%
    \def\n@mref{al}\ifx\l@debut\n@mref%
        \figset altitude(blcolor=default,bldash=default,blwidth=default,%
        sqcolor=default,sqdash=default,sqwidth=default)\else
    \def\n@mref{ar}\ifx\l@debut\n@mref\figresetarrowhead\else
    \def\n@mref{cu}\ifx\l@debut\n@mref\figset curve(roundness=\D@FTroundness)\else
    \def\n@mref{ge}\ifx\l@debut\n@mref\P@resetg@n\else
    \def\n@mref{fl}\ifx\l@debut\n@mref%
        \figset flowchart(arrowp=\D@FTfcarrowposition,arrowr=\D@FTfcarrowrefpt,%
	bgcolor=\D@FTfcbgcolor,line=\D@FTfcline,radius=\D@FTfcradius,%
	shape=\D@FTfcshape,thickcolor=default,thickness=\D@FTfcthickness,%
	xpadd=\D@FTfcxpadding,ypadd=\D@FTfcypadding)\else
    \def\n@mref{me}\ifx\l@debut\n@mref\figset mesh(diag=\D@FTmeshdiag,%
        color=default,dash=default,width=default)\else
    \def\n@mref{tr}\ifx\l@debut\n@mref%
        \figset trimesh(color=default,dash=default,width=default)\else
    \W@rnmeskwd{figreset}{#1}\fi\fi\fi\fi\fi\fi\fi\fi}
\ctr@ld@f\def\P@resetg@n{\figset (color=\D@FTcolor,dash=\D@FTdash,fill=\D@FTfill,%
    join=\D@FTjoin,width=\D@FTwidth)}
\ctr@ld@f\def\figset#1(#2){\def\t@xt@{#1}\ifx\t@xt@\empty\trtlis@rg{#2}{\Pssetg@n}
    \else\keln@mde#1|%
    \def\n@mref{al}\ifx\l@debut\n@mref\trtlis@rg{#2}{\Psset@lti}\else
    \def\n@mref{ar}\ifx\l@debut\n@mref\trtlis@rg{#2}{\Psset@rrowhe@d}\else
    \def\n@mref{cu}\ifx\l@debut\n@mref\trtlis@rg{#2}{\Pssetc@rve}\else
    \def\n@mref{fl}\ifx\l@debut\n@mref\trtlis@rg{#2}{\Pssetfl@wchart}\else
    \def\n@mref{ge}\ifx\l@debut\n@mref\trtlis@rg{#2}{\Pssetg@n}\else
    \def\n@mref{me}\ifx\l@debut\n@mref\trtlis@rg{#2}{\Pssetm@sh}\else
    \def\n@mref{pr}\ifx\l@debut\n@mref\ifCUR@PS\W@rnmesIgn{figset proj(...)}%
     \else\trtlis@rg{#2}{\Figsetpr@j}\fi\else
    \def\n@mref{tr}\ifx\l@debut\n@mref\trtlis@rg{#2}{\Pssettrim@sh}\else
    \def\n@mref{wr}\ifx\l@debut\n@mref\let\M@cro=\Figsetwr@te\trtlis@rgtok{#2,|}\else
    \W@rnmeskwd{figset}{#1}\fi\fi\fi\fi\fi\fi\fi\fi\fi\fi\ignorespaces}
\ctr@ld@f\def\figsetdefault#1(#2){\ifCUR@PS\W@rnmesIgn{figsetdefault}\else%
    \def\t@xt@{#1}\ifx\t@xt@\empty\trtlis@rg{#2}{\Pssd@g@n}\else\keln@mun#1|
    \def\n@mref{a}\ifx\l@debut\n@mref\trtlis@rg{#2}{\Pssd@@rrowhe@d}\else
    \def\n@mref{c}\ifx\l@debut\n@mref\trtlis@rg{#2}{\Pssd@c@rve}\else
    \def\n@mref{g}\ifx\l@debut\n@mref\trtlis@rg{#2}{\Pssd@g@n}\else
    \def\n@mref{f}\ifx\l@debut\n@mref\trtlis@rg{#2}{\Pssd@fl@wchart}\else
    \def\n@mref{m}\ifx\l@debut\n@mref\trtlis@rg{#2}{\Pssd@m@sh}\else
    \W@rnmeskwd{figsetdefault}{#1}\fi\fi\fi\fi\fi\fi\initpss@ttings\fi}
\ctr@ld@f\def\Pssd@g@n#1=#2|{\keln@mun#1|%
    \def\n@mref{c}\ifx\l@debut\n@mref\edef\D@FTcolor{#2}\else
    \def\n@mref{d}\ifx\l@debut\n@mref\edef\D@FTdash{#2}\else
    \def\n@mref{f}\ifx\l@debut\n@mref\edef\D@FTfill{#2}\else
    \def\n@mref{j}\ifx\l@debut\n@mref\edef\D@FTjoin{#2}\else
    \def\n@mref{u}\ifx\l@debut\n@mref\edef\D@FTupdate{#2}\Q@s@tupdate{#2}\else
    \def\n@mref{w}\ifx\l@debut\n@mref\edef\D@FTwidth{#2}\else
    \W@rnmesAttr{figsetdefault}{#1}\fi\fi\fi\fi\fi\fi}
\ctr@ld@f\def\Pssd@@rrowhe@d#1=#2|{\keln@mun#1|%
    \def\n@mref{a}\ifx\l@debut\n@mref\edef\D@FTarrowheadangle{#2}\else
    \def\n@mref{f}\ifx\l@debut\n@mref\edef\D@FTarrowheadfill{#2}\else
    \def\n@mref{l}\ifx\l@debut\n@mref\y@tiunit{#2}\ifunitpr@sent%
     \edef\D@FTh@rdahlength{#2}\else\edef\D@FTh@rdahlength{#2pt}%
     \message{*** \BS@ figsetdefault (..., #1=#2, ...) : unit is missing, pt is assumed.}%
     \fi\else
    \def\n@mref{o}\ifx\l@debut\n@mref\edef\D@FTarrowheadout{#2}\else
    \def\n@mref{r}\ifx\l@debut\n@mref\edef\D@FTarrowheadratio{#2}\else
    \W@rnmesAttr{figsetdefault arrowhead}{#1}\fi\fi\fi\fi\fi}
\ctr@ld@f\def\Pssd@c@rve#1=#2|{\keln@mun#1|%
    \def\n@mref{r}\ifx\l@debut\n@mref\edef\D@FTroundness{#2}\else%
    \W@rnmesAttr{figsetdefault curve}{#1}\fi}
\ctr@ld@f\def\Pssd@fl@wchart#1=#2|{\keln@mtr#1|%
    \def\n@mref{arr}\ifx\l@debut\n@mref\expandafter\keln@mtr\l@suite|%
     \def\n@mref{owp}\ifx\l@debut\n@mref\edef\D@FTfcarrowposition{#2}\else
     \def\n@mref{owr}\ifx\l@debut\n@mref\edef\D@FTfcarrowrefpt{#2}\else
                     \W@rnmesAttr{figsetdefault flowchart}{#1}\fi\fi\else%
    \def\n@mref{bgc}\ifx\l@debut\n@mref\edef\D@FTfcbgcolor{#2}\else
    \def\n@mref{lin}\ifx\l@debut\n@mref\edef\D@FTfcline{#2}\else
    \def\n@mref{pad}\ifx\l@debut\n@mref\edef\D@FTfcxpadding{#2}%
                    \edef\D@FTfcypadding{#2}\else
    \def\n@mref{rad}\ifx\l@debut\n@mref\edef\D@FTfcradius{#2}\else
    \def\n@mref{sha}\ifx\l@debut\n@mref\edef\D@FTfcshape{#2}\else
    \def\n@mref{thi}\ifx\l@debut\n@mref\expandafter\keln@mtr\l@suite|%
     \def\n@mref{ckn}\ifx\l@debut\n@mref\edef\D@FTfcthickness{#2}\else
                     \W@rnmesAttr{figsetdefault flowchart}{#1}\fi\else%
    \def\n@mref{xpa}\ifx\l@debut\n@mref\edef\D@FTfcxpadding{#2}\else
    \def\n@mref{ypa}\ifx\l@debut\n@mref\edef\D@FTfcypadding{#2}\else
    \W@rnmesAttr{figsetdefault flowchart}{#1}\fi\fi\fi\fi\fi\fi\fi\fi\fi}
\ctr@ld@f\def\D@FTfcarrowposition{0.5}
\ctr@ld@f\def\D@FTfcarrowrefpt{start}
\ctr@ld@f\def\D@FTfcbgcolor{1}
\ctr@ld@f\def\D@FTfcline{polygon}
\ctr@ld@f\def\D@FTfcradius{0}
\ctr@ld@f\def\D@FTfcshape{rectangle}
\ctr@ld@f\def\D@FTfcthickness{0}
\ctr@ld@f\def\D@FTfcxpadding{0}
\ctr@ld@f\def\D@FTfcypadding{0}
\ctr@ld@f\def\Pssd@m@sh#1=#2|{\keln@mun#1|%
    \def\n@mref{d}\ifx\l@debut\n@mref\edef\D@FTmeshdiag{#2}\else%
    \W@rnmesAttr{figsetdefault mesh}{#1}\fi}
\ctr@ln@w{newif}\iffillm@de
\ctr@ld@f\def\Q@s@tfillmode#1{\expandafter\setfillm@de#1:}
\ctr@ld@f\def\setfillm@de#1#2:{\if#1n\fillm@defalse\else\fillm@detrue\fi}
\ctr@ld@f\def\D@FTfill{no}     
\ctr@ln@w{newif}\ifGRupdatem@de
\ctr@ld@f\def\Q@s@tupdate#1{\ifCUR@PS\W@rnmesIgn{figset (update=...)}%
    \else\expandafter\setupd@te#1:\fi}
\ctr@ld@f\def\setupd@te#1#2:{\if#1n\GRupdatem@defalse\else\GRupdatem@detrue\fi}
\ctr@ld@f\def\D@FTupdate{no}     
\ctr@ln@m\CUR@color \ctr@ln@m\CUR@colorc@md
\ctr@ld@f\def\s@uvcolor#1{\edef#1{\CUR@color}}
\ctr@ld@f\def\D@FTcolor{0}       
\ctr@ld@f\def\Pssetc@lor#1{\ifGR@cri\result@tent=\@ne\expandafter\c@lnbV@l#1 :%
    \def\CUR@color{}\def\CUR@colorc@md{}%
    \ifcase\result@tent\or\Q@s@tgray{#1}\or\or\Q@s@trgb{#1}\or\Q@s@tcmyk{#1}\fi\fi}
\ctr@ln@m\CUR@colorc@mdStroke
\ctr@ld@f\def\Q@s@tcmyk#1{\ifGR@cri\def\CUR@color{#1}\def\CUR@colorc@md{\c@msetcmykcolor}%
    \def\CUR@colorc@mdStroke{\c@msetcmykcolorStroke}%
    \ifCUR@PS\PSc@mment{setcmyk Color=#1}\us@primarC@lor\fi\fi}
\ctr@ld@f\def\Q@s@trgb#1{\ifGR@cri\def\CUR@color{#1}\def\CUR@colorc@md{\c@msetrgbcolor}%
    \def\CUR@colorc@mdStroke{\c@msetrgbcolorStroke}%
    \ifCUR@PS\PSc@mment{setrgb Color=#1}\us@primarC@lor\fi\fi}
\ctr@ld@f\def\Q@s@tgray#1{\ifGR@cri\def\CUR@color{#1}\def\CUR@colorc@md{\c@msetgray}%
    \def\CUR@colorc@mdStroke{\c@msetgrayStroke}%
    \ifCUR@PS\PSc@mment{setgray Gray level=#1}\us@primarC@lor\fi\fi}
\ctr@ln@m\fillc@md
\ctr@ld@f\def\us@primarC@lor{\immediate\write\fwf@g{\d@fprimarC@lor}%
    \let\fillc@md=\prfillc@md}
\ctr@ld@f\def\prfillc@md{\d@fprimarC@lor\space\c@mfill}
\ctr@ld@f\def\c@lnbV@l#1 #2:{\def\t@xt@{#1}\relax\ifx\t@xt@\empty\c@lnbV@l#2:
    \else\c@lnbV@l@#1 #2:\fi}
\ctr@ld@f\def\c@lnbV@l@#1 #2:{\def\t@xt@{#2}\ifx\t@xt@\empty%
    \def\t@xt@{#1}\ifx\t@xt@\empty\advance\result@tent\m@ne\fi
    \else\advance\result@tent\@ne\c@lnbV@l@#2:\fi}
\ctr@ld@f\def\Blackcmyk{0 0 0 1}
\ctr@ld@f\def\Whitecmyk{0 0 0 0}
\ctr@ld@f\def\Cyancmyk{1 0 0 0}
\ctr@ld@f\def\Magentacmyk{0 1 0 0}
\ctr@ld@f\def\Yellowcmyk{0 0 1 0}
\ctr@ld@f\def\Redcmyk{0 1 1 0}
\ctr@ld@f\def\Greencmyk{1 0 1 0}
\ctr@ld@f\def\Bluecmyk{1 1 0 0}
\ctr@ld@f\def\Graycmyk{0 0 0 0.50}
\ctr@ld@f\def\BrickRedcmyk{0 0.89 0.94 0.28} 
\ctr@ld@f\def\Browncmyk{0 0.81 1 0.60} 
\ctr@ld@f\def\ForestGreencmyk{0.91 0 0.88 0.12} 
\ctr@ld@f\def\Goldenrodcmyk{ 0 0.10 0.84 0} 
\ctr@ld@f\def\Marooncmyk{0 0.87 0.68 0.32} 
\ctr@ld@f\def\Orangecmyk{0 0.61 0.87 0} 
\ctr@ld@f\def\Purplecmyk{0.45 0.86 0 0} 
\ctr@ld@f\def\RoyalBluecmyk{1. 0.50 0 0} 
\ctr@ld@f\def\Violetcmyk{0.79 0.88 0 0} 
\ctr@ld@f\def\Blackrgb{0 0 0}
\ctr@ld@f\def\Whitergb{1 1 1}
\ctr@ld@f\def\Redrgb{1 0 0}
\ctr@ld@f\def\Greenrgb{0 1 0}
\ctr@ld@f\def\Bluergb{0 0 1}
\ctr@ld@f\def\Cyanrgb{0 1 1}
\ctr@ld@f\def\Magentargb{1 0 1}
\ctr@ld@f\def\Yellowrgb{1 1 0}
\ctr@ld@f\def\Grayrgb{0.5 0.5 0.5}
\ctr@ld@f\def\Chocolatergb{0.824 0.412 0.118}
\ctr@ld@f\def\DarkGoldenrodrgb{0.722 0.525 0.043}
\ctr@ld@f\def\DarkOrangergb{1 0.549 0}
\ctr@ld@f\def\Firebrickrgb{0.698 0.133 0.133}
\ctr@ld@f\def\ForestGreenrgb{0.133 0.545 0.133}
\ctr@ld@f\def\Goldrgb{1 0.843 0}
\ctr@ld@f\def\HotPinkrgb{1 0.412 0.706}
\ctr@ld@f\def\Maroonrgb{0.690 0.188 0.376}
\ctr@ld@f\def\Pinkrgb{1 0.753 0.796}
\ctr@ld@f\def\RoyalBluergb{0.255 0.412 0.882}
\ctr@ld@f\def\Pssetg@n#1=#2|{\keln@mun#1|%
    \def\n@mref{c}\ifx\l@debut\n@mref\update@ttr\D@FTcolor\Pssetc@lor{#2}\else
    \def\n@mref{d}\ifx\l@debut\n@mref\update@ttr\D@FTdash\Q@s@tdash{#2}\else
    \def\n@mref{f}\ifx\l@debut\n@mref\update@ttr\D@FTfill\Q@s@tfillmode{#2}\else
    \def\n@mref{j}\ifx\l@debut\n@mref\update@ttr\D@FTjoin\Q@s@tjoin{#2}\else
    \def\n@mref{u}\ifx\l@debut\n@mref\update@ttr\D@FTupdate\Q@s@tupdate{#2}\else
    \def\n@mref{w}\ifx\l@debut\n@mref\update@ttr\D@FTwidth\Q@s@twidth{#2}\else
    \W@rnmesAttr{figset}{#1}\fi\fi\fi\fi\fi\fi}
\ctr@ln@m\CUR@dash
\ctr@ld@f\def\s@uvdash#1{\edef#1{\CUR@dash}}
\ctr@ld@f\def\D@FTdash{1}        
\ctr@ld@f\def\Q@s@tdash#1{\ifGR@cri\edef\CUR@dash{#1}\ifCUR@PS\expandafter\Pssetd@sh#1 :\fi\fi}
\ctr@ld@f\def\Pssetd@shI#1{\PSc@mment{setdash Index=#1}\ifcase#1%
    \or\immediate\write\fwf@g{[] 0 \c@msetdash}
    \or\immediate\write\fwf@g{[6 2] 0 \c@msetdash}
    \or\immediate\write\fwf@g{[4 2] 0 \c@msetdash}
    \or\immediate\write\fwf@g{[2 2] 0 \c@msetdash}
    \or\immediate\write\fwf@g{[1 2] 0 \c@msetdash}
    \or\immediate\write\fwf@g{[2 4] 0 \c@msetdash}
    \or\immediate\write\fwf@g{[3 5] 0 \c@msetdash}
    \or\immediate\write\fwf@g{[3 3] 0 \c@msetdash}
    \or\immediate\write\fwf@g{[3 5 1 5] 0 \c@msetdash}
    \or\immediate\write\fwf@g{[6 4 2 4] 0 \c@msetdash}
    \fi}
\ctr@ld@f\def\Pssetd@sh#1 #2:{{\def\t@xt@{#1}\ifx\t@xt@\empty\Pssetd@sh#2:
    \else\def\t@xt@{#2}\ifx\t@xt@\empty\Pssetd@shI{#1}\else\s@mme=\@ne\def\debutp@t{#1}%
    \an@lysd@sh#2:\ifodd\s@mme\edef\debutp@t{\debutp@t\space\finp@t}\def\finp@t{0}\fi%
    \PSc@mment{setdash Pattern=#1 #2}%
    \immediate\write\fwf@g{[\debutp@t] \finp@t\space\c@msetdash}\fi\fi}}
\ctr@ld@f\def\an@lysd@sh#1 #2:{\def\t@xt@{#2}\ifx\t@xt@\empty\def\finp@t{#1}\else%
    \edef\debutp@t{\debutp@t\space#1}\advance\s@mme\@ne\an@lysd@sh#2:\fi}
\ctr@ln@m\CUR@width
\ctr@ld@f\def\s@uvwidth#1{\edef#1{\CUR@width}}
\ctr@ld@f\def\D@FTwidth{0.4}     
\ctr@ld@f\def\Q@s@twidth#1{\ifGR@cri\edef\CUR@width{#1}\ifCUR@PS%
    \PSc@mment{setwidth Width=#1}\immediate\write\fwf@g{#1 \c@msetlinewidth}\fi\fi}
\ctr@ln@m\CUR@join
\ctr@ld@f\def\s@uvjoin#1{\edef#1{\CUR@join}}
\ctr@ld@f\def\D@FTjoin{miter}   
\ctr@ld@f\def\Q@s@tjoin#1{\ifGR@cri\edef\CUR@join{#1}\ifCUR@PS\expandafter\Pssetj@in#1:\fi\fi}
\ctr@ld@f\def\Pssetj@in#1#2:{\PSc@mment{setjoin join=#1}%
    \if#1r\def\t@xt@{1}\else\if#1b\def\t@xt@{2}\else\def\t@xt@{0}\fi\fi%
    \immediate\write\fwf@g{\t@xt@\space\c@msetlinejoin}}
\ctr@ld@f\def\Pss@tspecifSt#1{\trtlis@rg{#1}{\Pss@tspecifSt@}}
\ctr@ld@f\def\Pss@tspecifSt@#1=#2|{\keln@mun#1|%
    \def\n@mref{c}\ifx\l@debut\n@mref\def\n@mref{#2}\ifx\n@mref\D@FTref\else%
     \s@uvcolor{\typ@color}\Pssetc@lor{#2}\fi\else
    \def\n@mref{d}\ifx\l@debut\n@mref\def\n@mref{#2}\ifx\n@mref\D@FTref\else%
     \s@uvdash{\typ@dash}\Q@s@tdash{#2}\fi\else
    \def\n@mref{j}\ifx\l@debut\n@mref\def\n@mref{#2}\ifx\n@mref\D@FTref\else%
     \s@uvjoin{\typ@join}\Q@s@tjoin{#2}\fi\else
    \def\n@mref{w}\ifx\l@debut\n@mref\def\n@mref{#2}\ifx\n@mref\D@FTref\else%
     \s@uvwidth{\typ@width}\Q@s@twidth{#2}\fi\else
    \W@rnmeskwd{Pss@tspecifSt}{#1}\fi\fi\fi\fi}
\ctr@ld@f\def\Psrest@reSt#1{\trtlis@rg{#1}{\Psrest@reSt@}}
\ctr@ld@f\def\Psrest@reSt@#1=#2|{\keln@mun#1|%
    \def\n@mref{c}\ifx\l@debut\n@mref\def\n@mref{#2}\ifx\n@mref\D@FTref\else%
     \Pssetc@lor{\typ@color}\fi\else
    \def\n@mref{d}\ifx\l@debut\n@mref\def\n@mref{#2}\ifx\n@mref\D@FTref\else%
     \Q@s@tdash{\typ@dash}\fi\else
    \def\n@mref{j}\ifx\l@debut\n@mref\def\n@mref{#2}\ifx\n@mref\D@FTref\else%
     \Q@s@tjoin{\typ@join}\fi\else
    \def\n@mref{w}\ifx\l@debut\n@mref\def\n@mref{#2}\ifx\n@mref\D@FTref\else%
     \Q@s@twidth{\typ@width}\fi\else
    \W@rnmeskwd{Psrest@reSt}{#1}\fi\fi\fi\fi}
\ctr@ld@f\def\Pssettrim@sh#1=#2|{\keln@mde#1|%
    \def\n@mref{co}\ifx\l@debut\n@mref\update@ttr\D@FTref\P@settmeshcolor{#2}\else
    \def\n@mref{da}\ifx\l@debut\n@mref\update@ttr\D@FTref\P@settmeshdash{#2}\else
    \def\n@mref{wi}\ifx\l@debut\n@mref\update@ttr\D@FTref\P@settmeshwidth{#2}\else
    \W@rnmesAttr{figset trimesh}{#1}\fi\fi\fi}
\ctr@ln@m\DDV@tmeshcolor
\ctr@ld@f\def\P@settmeshcolor#1{\edef\DDV@tmeshcolor{#1}}
\ctr@ln@m\DDV@tmeshdash
\ctr@ld@f\def\P@settmeshdash#1{\edef\DDV@tmeshdash{#1}}
\ctr@ln@m\DDV@tmeshwidth
\ctr@ld@f\def\P@settmeshwidth#1{\edef\DDV@tmeshwidth{#1}}
\ctr@ld@f\def\figdrawtrimesh#1[#2,#3,#4]{{\ifCUR@PS\ifGR@cri%
    \PSc@mment{trimesh Type=#1, Triangle=[#2,#3,#4]}%
    \s@uvc@ntr@l\et@tpstrimesh\ifnum#1>\@ne%
    \Pss@tspecifSt{color=\DDV@tmeshcolor,dash=\DDV@tmeshdash,width=\DDV@tmeshwidth}%
    \setc@ntr@l{2}%
    \Pstrimeshp@rt#1[#2,#3,#4]\Pstrimeshp@rt#1[#3,#4,#2]\Pstrimeshp@rt#1[#4,#2,#3]%
    \Psrest@reSt{color=\DDV@tmeshcolor,dash=\DDV@tmeshdash,width=\DDV@tmeshwidth}%
    \fi\figdrawline[#2,#3,#4,#2]%
    \PSc@mment{End trimesh}\resetc@ntr@l\et@tpstrimesh\fi\fi}}
\ctr@ld@f\def\Pstrimeshp@rt#1[#2,#3,#4]{{\l@mbd@un=\@ne\l@mbd@de=#1\loop\ifnum\l@mbd@de>\@ne%
    \advance\l@mbd@de\m@ne\figptbary-1:[#2,#3;\l@mbd@de,\l@mbd@un]%
    \figptbary-2:[#2,#4;\l@mbd@de,\l@mbd@un]\figdrawline[-1,-2]%
    \advance\l@mbd@un\@ne\repeat}}
\initpr@lim\initpss@ttings\initPDF@rDVI
\ctr@ln@w{newbox}\figBoxA
\ctr@ln@w{newbox}\figBoxB
\ctr@ln@w{newbox}\figBoxC
\catcode`\@=12

\allowdisplaybreaks
\DisableLigatures{encoding = *, family = * }
\numberwithin{equation}{section}
\newtheorem{theorem}{Theorem}[section]
\newtheorem{proposition}{Proposition}[section]

\newtheorem{definition}{Definition}[section]
\newtheorem{remark}{Remark}[section]
\numberwithin{equation}{section}
\numberwithin{theorem}{section}
\numberwithin{remark}{section}
\numberwithin{lemma}{section}
\numberwithin{proposition}{section}
\numberwithin{definition}{section}
\newcommand{\norm}[2]{{\left\|#1\right\|}_{#2}}
\newcommand{\fl}[2]{(-d_x^2)^{#1}#2}
\newcommand{\rfl}[2]{A^{#1}_{\Omega}#2}
\newcommand{\hp}[1]{\hphantom{#1}}
\newcommand{\cns}{c_{N,s}}
\newcommand{\ccs}{c_{1,s}}
\newcommand{\ffl}[2]{(-d_x^{\,2})^{#1}#2}
\newcommand{\flh}[2]{\frac{1}{\Gamma(-s)}\int_0^{+\infty}\Big(e^{t\Delta}#2 - #2\Big)\frac{dt}{t^{1+#1}}}
\newcommand{\kernel}[1]{|x-y|^{#1}}
\newcommand{\dkj}{\delta_{kj}}
\newcommand{\intr}[1]{\underset{#1}{\int}}
\newcommand{\Do}[1]{D_{#1}}
\newcommand{\Hs}{H^s_0(\Omega)}
\newcommand{\ue}[1]{#1^{\,\varepsilon}}
\newcommand{\xHdot}[1]{\dot{H}^{#1}}
\newcommand{\ha}[2]{\mathbf{H}_{#1}^{#2}}
\newcommand{\lhi}{\mathcal{L}_i^h}
\newcommand{\NN}{\mathbb{N}}
\newcommand{\ZZ}{\mathbb{Z}}
\newcommand{\RR}{\mathbb{R}}
\newcommand{\CC}{\mathbb{C}}
\newcommand{\TT}{\mathbf{T}}
\newcommand\inter[1]{\llbracket #1\rrbracket}
\newcommand\mesh{\mathfrak{M}}
\usepackage{tikz}
\usepackage{pgfplots}
\usepackage{pgfplotstable}
\usetikzlibrary{matrix,external,fit}
 \pgfplotsset{surface/.style={ %
               xmax=0.34,%
               axis z line=center,%
               axis x line=center,%
               axis y line=center,%
               zmin=-1,%
               clip=false,%
               extra x ticks={1},%
               extra x tick label={$T=1$},%
               xtick={0},%
               ytick=\empty}}
\newcommand{\controldomain}[2]{ \addplot3 [surf,fill=violet,mesh/rows=2] coordinates {(0,#1,0) (1,#1,0) (0,#2,0) (1,#2,0)}; }
\newcommand{\controldomainT}[3]{ \addplot3 [surf,fill=violet,mesh/rows=2] coordinates {(0,#1,0) (#3,#1,0) (0,#2,0) (#3,#2,0)}; }

\newcommand{\couplingdomain}[2]{ \addplot3 [surf,fill=green,mesh/rows=2] coordinates {(0,#1,0) (1,#1,0) (0,#2,0) (1,#2,0)}; }
\newcommand{\couplingdomainT}[3]{ \addplot3 [surf,fill=green,mesh/rows=2] coordinates {(0,#1,0) (#3,#1,0) (0,#2,0) (#3,#2,0)}; }
 \pgfplotsset{erreurs/.style={scale=1,
     legend cell align=left,
     legend pos=outer north east,
     legend plot pos=right,
     legend style={cells={anchor=east},draw=none},
     xlabel=$h$,
    xmin=0.001,xmax=0.05}}
\tikzset{pente/.style={opacity=0.6}}
\pgfplotsset{cout/.style={black,mark=diamond*,mark size=2.5,mark options={fill=gray}}}
\pgfplotsset{cible/.style={black,mark=square*,mark size=2.5,mark options={fill=gray}}}
\pgfplotsset{cibleyT/.style={black,mark=*,mark size=2.5,mark options={fill=gray}}}
\pgfplotsset{CG/.style={black,mark=otimes*,mark size=2.5,mark options={fill=gray}}}
\pgfplotsset{solex/.style={black,mark=*,mark size=2.5,mark options={fill=gray}}}
\pgfplotsset{energie/.style={black,mark=triangle*,mark size=2.5,mark options={fill=gray}}}

\newcommand{\MyQuote}[1]{\vspace{0.5cm}
     \parbox{12cm}{\em #1}\hspace*{1cm}($\mathcal{P}$)\\[0.5cm]}
\title[FE approximation of the 1-d fractional Poisson equation]{A Finite Element approximation of the one-dimensional fractional Poisson equation with applications to numerical control}
\author{U.~Biccari}
\address{Umberto Biccari, DeustoTech, University of Deusto, 48007 Bilbao, Basque Country, Spain.}
\address{Umberto Biccari, Facultad Ingenier\'{\i}a, Universidad de Deusto, Avda Universidades 24, 48007 Bilbao, Basque Country, Spain.}
\email{umberto.biccari@deusto.es,u.biccari@gmail.com}
\author{V.~Hern\'andez-Santamar\'ia}
\address{Victor Hern\'andez-Santamar\'ia, DeustoTech, University of Deusto, 48007 Bilbao, Basque Country, Spain.}
\address{Victor Hern\'andez-Santamar\'ia, Facultad Ingenier\'{\i}a, Universidad de Deusto, Avda Universidades 24, 48007 Bilbao, Basque Country, Spain.}
\email{victor.santamaria@deusto.es}
\thanks{The work of Umberto Biccari was partially supported by the Advanced Grant DYCON (Dynamic Control) of the European Research Council Executive Agency, by the MTM2014-52347 Grant of the MINECO (Spain) and by the Air Force Office of Scientific Research under the Award No: FA9550-15-1-0027. The work of V\'ictor Hern\'andez-Santamar\'ia was partially supported by the Advanced Grant DYCON (Dynamic Control) of the European Research Council Executive Agency.}

\begin{document}

\bibliographystyle{acm}

\maketitle

\begin{abstract}
We present a finite element (FE) scheme for the numerical approximation of the solution to a non-local Poisson equation involving the one-dimensional fractional Laplacian $\fl{s}{}$ on the interval $(-L,L)$. In particular, we include the complete computations for obtaining the stiffness matrix, starting from the variational formulation of the problem. The problem being one-dimensional, the values of the matrix can be explicitly calculated, without need of any numerical integration, thus obtaining an algorithm which is very efficient in terms of the computational cost.
As an application, we analyze the corresponding parabolic equation from the point of view of controllability properties, employing the penalized Hilbert Uniqueness Method (HUM) for computing the numerical approximation of the null-control, acting from an open subset $\omega\subset(-L,L)$. In accordance to the theory, our numerical simulations show: (1) that the method solves the elliptic equation with an acceptable approximation in the natural functional setting, (2) that the parabolic problem is null-controllable for $s>1/2$ and (3) that for $s\leq 1/2$ we only have approximate controllability. 
\end{abstract}

\section{Introduction}\label{intro_sec}

In this work, we present a finite element (FE) scheme for the numerical approximation of the solution to the following non-local Poisson equation
\begin{align}\label{PE}
	\begin{cases}
		\fl{s}{u} = f, \quad & x\in(-L,L)
		\\
		u\equiv 0, & x\in\RR\setminus(-L,L).
	\end{cases}
\end{align}

In \eqref{PE}, $f$ is a given function and, for all $s\in(0,1)$, $\fl{s}{}$ denotes the one-dimensional fractional Laplace operator, which is defined as the following singular integral
\begin{align*}
	\fl{s}{u}(x) = \ccs\,P.V.\,\int_{\RR}\frac{u(x)-u(y)}{|x-y|^{1+2s}}\,dy. 
\end{align*}
Here, $\ccs$ is a normalization constant given by
\begin{align*}
	\ccs = \frac{s2^{2s}\Gamma\left(\frac{1+2s}{2}\right)}{\sqrt{\pi}\Gamma(1-s)},
\end{align*}
where $\Gamma$ is the usual Gamma function. 

We have to mention that, for having a completely rigorous definition of the fractional Laplace operator, it is necessary to introduce also the class of functions $u$ for which computing $\fl{s}{u}$ makes sense. We postpone this discussion to the next section.

The analysis of non-local operators and non-local PDEs is a topic in continuous development.
A motivation for this growing interest relies in the large number of possible applications in the modeling of several complex phenomena for which a local approach turns up to be inappropriate or limiting.
Indeed, there is an ample spectrum of situations in which a non-local equation gives a
significantly better description than a PDE of the problem one wants to analyze.
Among others, we mention applications in turbulence (\cite{bakunin2008turbulence}), anomalous transport and diffusion (\cite{bologna2000anomalous,meerschaert2012fractional}), elasticity (\cite{dipierro2015dislocation}), image processing (\cite{gilboa2008nonlocal}), porous media flow (\cite{vazquez2012nonlinear}), wave propagation in heterogeneous high contrast media (\cite{zhu2014modeling}). Also, it is well known that the fractional Laplacian is the generator of s-stable processes, and it is often used in stochastic models with applications, for instance, in mathematical finance (\cite{levendorskii2004pricing,pham1997optimal}).

One of the main differences between these non-local models and classical Partial Differential Equations is that the fulfilment of a non-local equation at a point involves the values of the function far away from that point.

In the recent past, the fractional Laplacian has been widely analyzed also from the point of view of numerical analysis. We refer, for instance, to the work \cite{acosta2017fractional} of Acosta and Borthagaray (see also \cite{acosta2017short}). There, the authors present a FE scheme for implementing the solution of \eqref{PE} in a bounded domain $\Omega\subset\RR^2$. In particular, they provide appropriate quadrature rules in order to solve numerically the variational formulation associated to the problem. Moreover, in \cite{acosta2017fractional} it is also developped an accurate analysis of the efficiency of the FE method, employing several existing results. The techniques of \cite{acosta2017short,acosta2017fractional} have then been applied in \cite{acosta2017finite}, combined with a convolution quadrature approach, for solving evolution equations involving the fractional Laplacian. For the sake of completeness, we also mention \cite{bonito2015numerical}, where it is presented a discretization of the so-called \textit{spectral fractional Laplacian} (see Eq. \eqref{fl_spec}) and its application to the evolutionary case \cite{bonito2017approximation}, and \cite{nochetto2015pde}, where the solution of \eqref{PE} is implemented applying the well known extension of Caffarelli and Silvestre (\cite{caffarelli2007extension}).   

Our method deals with a FE approximation in one space dimension for the fractional Poisson equation. The main novelty of our work, with respect to \cite{acosta2017short,acosta2017fractional}, relies on the fact that, since we are dealing with the one-dimensional case, we will not need any quadrature rule and each entry of the stiffness matrix can be computed explicitly. This has the great advantage of significantly reducing the computational cost of the algorithm and, therefore, our discretization method is suitable for being included in more general applications.  

A natural example is given by the numerical resolution of the following control problem: given any $T>0$, find a control function $g\in L^2((-L,L)\times(0,T))$ such that the corresponding solution to the parabolic problem 
\begin{align}\label{heat_frac}
	\begin{cases}
		z_t + \fl{s}{z} = g\mathbf{1}_{\omega},\quad & (x,t)\in (-1,1)\times(0,T)
		\\
		z=0, & (x,t)\in[\,\RR\setminus (-1,1)\,]\times(0,T)
		\\
		z(x,0)=z_0(x), & x\in (-1,1)
	\end{cases}
\end{align} 
satisfies $z(x,T)=0$. 

The approach that we will employ for solving this control problem is based on the penalized Hilbert Uniqueness Method (\cite{boyer2013penalised}), which relies on some classical works of Glowinski and Lions (\cite{glowinski1995exact,glowinski2008exact}).   

This paper is organized as follows. In Section \ref{theor_sec}, we present some existing theoretical results for the problems that we are going to analyze. In particular, we give a more accurate definition of the fractional Laplace operator and we introduce the variational formulation associated to \eqref{PE} (needed for the development of the FE scheme). Concerning the parabolic problem \eqref{heat_frac}, we present a couple of controllability results, which will help us in the verification of the accuracy of the numerical method. In Section \ref{fe_sec}, we describe our FE method for the elliptic equation \eqref{PE} and we present the algorithm for the penalized HUM, employed for the numerical control of \eqref{heat_frac}. In Section \ref{res_numerical} we present and comment the results of our numerical simulations. Finally, in Appendix \ref{appendix} we include the complete details for computing the stiffness matrix associated to our FE scheme.

\section{Preliminary results}\label{theor_sec}

In this Section, we introduce some preliminary result that will be useful in the remaining of the paper.

\subsection{Elliptic problem}

We start by giving a more rigorous definition of the fractional Laplace operator, as we have anticipated in Section \ref{intro_sec}. Let
\begin{align*}
	\mathcal L^1_s(\RR) :=\left\{ u:\RR\longrightarrow\RR\,:\; u\textrm{ measurable },\;\int_{\RR}\frac{|u(x)|}{(1+|x|)^{1+2s}}\,dx<\infty\right\}.
\end{align*}
For any $u\in\mathcal L_s^1$ and $\varepsilon>0$ we set 
\begin{align*}
	(-d_x^2)^s_{\varepsilon} u(x) = \ccs\,\int_{|x,y|>\varepsilon}\frac{u(x)-u(y)}{|x-y|^{1+2s}}\,dy,\;\;\; x\in\RR.
\end{align*}
The fractional Laplacian is then defined by the following singular integral
\begin{align}\label{fl}
	\fl{s}{u}(x) = \ccs\,P.V.\,\int_{\RR}\frac{u(x)-u(y)}{|x-y|^{1+2s}}\,dy = \lim_{\varepsilon\to 0^+} (-d_x^2)^s_{\varepsilon} u(x), \;\;\; x\in\RR,
\end{align}
provided that the limit exists. 

We notice that if $0<s<1/2$ and $u$ is smooth, for example bounded and Lipschitz continuous on $\RR$, then the integral in \eqref{fl} is in fact not really singular near $x$ (see e.g. \cite[Remark 3.1]{dihitchhiker}). Moreover, $\mathcal L_s^1(\RR)$ is the right space for which $v:= (-d_x^2)^s_{\varepsilon} u$ exists for every $\varepsilon > 0$, $v$ being also continuous at the continuity points of $u$.

Let us now introduce the variational formulation associated to equation \eqref{PE}, which will be the starting point for the development of the FE approximation of the problem we are considering. That is, find $u\in H^s_0(-L,L)$ such that
\begin{align*}
	a(u,v) = \int_{-L}^L fv\,dx,	
\end{align*}
for all $v\in H_0^s(-L,L)$, where the bilinear form $a(\cdot,\cdot):H^s_0(-L,L)\times H^s_0(-L,L)\to \RR$ is given by
\begin{align*}
	a(u,v)=\frac{\ccs}{2} \int_{\RR}\int_{\RR}\frac{(u(x)-u(y))(v(x)-v(y))}{|x-y|^{1+2s}}\,dxdy.	
\end{align*}
Here, $H^s_0(-L,L)$ denotes the space 
\begin{align*}
	H^s_0(-L,L) :=\Big\{\, u\in H^s(\RR)\,:\,u=0 \textrm{ in } \RR\setminus(-L,L)\,\Big\}, 
\end{align*}
while with $H^s(\RR)$ we indicate the classical fractional Sobolev space of order $s$. We refer to \cite{dihitchhiker} for a complete description of these spaces.  

Since the bilinear form $a$ is continuous and coercive, Lax-Milgram Theorem immediately implies existence and uniqueness of solutions to the Dirichlet problem \eqref{PE}. In more detail, if $f\in H^{-s}(-L,L)$, then \eqref{PE} admits a unique weak solution $u\in H_0^s(-L,L)$ (see, e.g., \cite[Proposition 2.1]{biccari2017local}). Here $H^{-s}(-L,L)$ stands for the dual space of $H^s_0(-L,L)$. Furthermore, in the literature it is possible to find improved regularity results for the solution to \eqref{PE}, both in H\"older and Sobolev spaces. The interested reader may refer, for instance, to \cite{acosta2017fractional,biccari2017local,leonori2015basic,ros2014dirichlet,ros2014extremal}.

\subsection{Parabolic problem}

As we mentioned in Section \ref{intro_sec}, the main goal of the present paper is to obtain a FE discretization of the fractional Laplacian. An application for this approximation will then be the numerical resolution of the fractional heat equation \eqref{heat_frac} and the associated control problem. Before doing that, let us recall the following definitions of controllability.

\begin{definition}
System \eqref{heat_frac} is said to be \textit{null-controllable} at time $T$ if, for any $z_0\in L^2(-1,1)$, there exists $g\in L^2(\omega\times(0,T))$ such that the corresponding solution $z$ satisfies 
\begin{align*}
	z(x,T)=0.
\end{align*}

\end{definition}

\begin{definition}
System \eqref{heat_frac} is said to be \textit{approximately controllable} at time $T$ if, for any $z_0,z_T\in L^2(-1,1)$ and any $\delta>0$, there exists $g\in L^2(\omega\times(0,T))$ such that the corresponding solution $z$ satisfies \begin{align*}
	\norm{z(x,T)-z_T}{L^2(-1,1)}<\delta.
\end{align*}
\end{definition}

Therefore, given any initial datum $z_0\in L^2(-1,1)$ we are interested in computing numerically the control function $g$ that drives the solution $z$ to zero in time $T$. Before describing the methodology that we shall adopt, we recall the existing theoretical results on the controllability of the fractional heat equation \eqref{heat_frac}. This will give us a hint about what we should expect from our simulations, and will provide a validation of the accuracy of our numerical method.

First of all, it is worth to mention that the existence, uniqueness and regularity of the solutions to \eqref{heat_frac} has been studied by several authors. Among others, we mention the works \cite{biccari2017parabolic,fernandez2016boundary,leonori2015basic}. 

Concerning now the control problem, we have to mention that, to the best of our knowledge, there are few results in the literature on the null-controllability of the fractional heat equation, and none of them is for a problem involving the fractional Laplacian in its integral form \eqref{fl}. The existing results, instead, deal with the \textit{spectral} definition of the fractional Laplace operator, which is given as follows.

Let $\{\psi_k,\lambda_k\}_{k\in\NN}\subset H_0^1(-1,1)\times\RR^+$ be the set of normalized eigenfunctions and eigenvalues of the Laplace operator in $(-1,1)$ with homogeneous Dirichlet boundary conditions, so that $\{\psi_k\}_{k\in\NN}$ is an orthonormal basis of $L^2(-1,1)$ and         
\begin{align*}
	\begin{cases}
		-d_x^2\psi_k =\lambda_k\psi_k, & x\in (-1,1), 
		\\
		\psi_k(-1)=\psi_k(1)=0.
	\end{cases}
\end{align*}

Then, the \textit{spectral fractional Laplacian} $(-d_x^2)^s_S$ is defined by
\begin{align}\label{fl_spec}
	(-d_x^2)^s_S u(x) = \sum_{k\geq 1}\langle u,\psi_k\rangle \lambda_k^s\psi_k(x),
\end{align}
firstly for $u\in C_0^{\infty}(-1,1)$ and then for $u\in H_0^s(-1,1)$ employing a density argument.

It is important to notice that the spectral fractional Laplacian and the fractional Laplacian defined as in \eqref{fl} are two different operators. Indeed, definition \eqref{fl_spec} depends on the choice of the domain (in this case, $(-1,1)$), while the integral definition does not. For a complete discussion on the differences of these two operators, we refer to \cite{servadei2014spectrum}.

The fractional heat equation involving the operator $(-d_x^2)^s_S$ has been analyzed in \cite{micu2006controllability}, where the authors proved its null controllability, provided that $s>1/2$. For $s\leq 1/2$, instead, null controllability does not hold, not even for $T$ large. This negative result is based on the equivalence (consequence of M\"untz Theorem, see, e.g., \cite[Page 24]{schwartz1958etude}) between the controllability property (more specifically, the possibility of proving an observability inequality), and the following condition for the eigenvalues of the operator considered
\begin{align}\label{eigen_cond}
	\sum_{k\geq 1} \frac{1}{\lambda_k}<\infty,
\end{align} 
which is clearly not satisfied in the case of the spectral fractional Laplacian when $s\leq 1/2$, since in that case the eigenvalues are $\lambda_k = (k\pi)^{2s}$. Finally, in \cite{miller2006controllability}, the same result as in \cite{micu2006controllability} is obtained in a multi-dimensional setting, by means of a  \textit{spectral observability condition} for a negative self-adjoint operator, which allows to prove the null-controllability of the semi-group that it generates.

Even if we are not aware of any controllability result, neither positive nor negative, for the parabolic equation involving the integral fractional Laplacian, at least in the one space dimension, these properties are easily achievable. In more detail, we have the following.

\begin{proposition}
For all $z_0\in L^2(-1,1)$ the parabolic problem \eqref{heat_frac} is null-controllable with a control function $g\in L^2(\omega\times(0,T))$ if and only if $s>1/2$.  
\end{proposition}

\begin{proof}
First of all, multiplying \eqref{heat_frac} by $\varphi$ and integrating over $(-1,1)\times (0,T)$, it is straightforward to check that $z(x,T)=0$ if and only if 
\begin{align}\label{control_id}
	\int_0^T\int_{-1}^1 \varphi(x,t)g(x,t)\mathbf{1}_{\omega}(x)\,dxdt = -\int_{-1}^1 u_0(x)\varphi(x,0)\,dx,
\end{align}
for all $\varphi^T\in L^2(-1,1)$, where $\varphi(x,t)$ is the unique solution to the adjoint system
\begin{align}\label{heat_frac_adj}
	\begin{cases}
		-\varphi_t + \fl{s}{\varphi} = 0, & (x,t)\in (-1,1)\times(0,T)
		\\
		\varphi = 0, & (x,t)\in [\,\RR\setminus(-1,1)\,]\times(0,T)
		\\
		\varphi(x,T) = \varphi^T(x), & x\in (-1,1).
	\end{cases}
\end{align}

In turn, it is classical that \eqref{control_id} is equivalent to the existence of a constant $C>0$ such that the following observability inequality holds
\begin{align}\label{obs}
	\norm{\varphi(x,0)}{L^2(-1,1)}^2\leq C\int_0^T\left|\,\int_{-1}^1 \varphi(x,t)g(x,t)\mathbf{1}_{\omega}(x)\,dx\,\right|^2\,dt,
\end{align}

Notice that $\varphi$ can be written in terms of the basis of eigenfunctions $\{\varrho_k\}_{k\geq 1}$. Namely,
\begin{align}\label{adj_sol}
	\varphi(x,t) = \sum_{k\geq 1} \varphi_ke^{-\lambda_k(T-t)}\varrho_k(x), 
\end{align}
where $\varphi_k = \langle \varphi^T,\varrho_k\rangle$ and, for $k\geq 1$, $\varrho(x)$ are the solutions to the following eigenvalue problem 
\begin{align*}
	\begin{cases}
		\fl{s}{\varrho_k} = \lambda_k\varrho_k, & x\in (-1,1), \;\; k\in\NN
		\\
		\varrho_k = 0, & x\in \,\RR\setminus(-1,1).
	\end{cases}
\end{align*}

Now, plugging \eqref{adj_sol} into \eqref{obs}, using the orthonormality of the eigenfunctions $\varrho_k$ and employing the change of variables $T-t\mapsto t$, the observability inequality becomes 
\begin{align}\label{obs_spectr}
	\sum_{k\geq 1} |\varphi_k|^2e^{-2\lambda_k T} \leq C\int_0^T\left|\,\sum_{k\geq 1} \varphi_kg_k(t)e^{-\lambda_k t}\right|^2\,dt, 
\end{align}
where $g_k = \langle g\mathbf{1}_{\omega},\varrho_k\rangle$. 

By means of M\"untz Theorem, inequalities of the form \eqref{obs_spectr} are well known to be true if and only if \eqref{eigen_cond} holds. On the other hand, according to \cite{kulczycki2010spectral,kwasnicki2012eigenvalues} we have 
\begin{align*}
	\lambda_k = \left(\frac{k\pi}{2}-\frac{(1-s)\pi}{4}\right)^{2s}+O\left(\frac{1}{k}\right).
\end{align*}
 
Therefore, we easily see that the condition \eqref{eigen_cond} is satisfied if and only if $s>1/2$. If $s\leq 1/2$, instead, the series diverges, since it behaves as an harmonic series. In conclusion, the observability inequality \eqref{obs} is proved when $s>1/2$, but it is false when $s\leq 1/2$. This concludes the proof. 
\end{proof}

Even if for $s\leq 1/2$ null controllability for \eqref{heat_frac} fails, we still have the following result of approximate controllability. 

\begin{proposition}
Let $s\in(0,1)$. For all $z_0\in L^2(-1,1)$, there exists a control function $g\in L^2(\omega\times(0,T))$ such that the unique solution $z$ to the parabolic problem \eqref{heat_frac} is approximately controllable.
\end{proposition}

\begin{proof}

The result will be a consequence of the following unique continuation property for the solution to the adjoint equation \eqref{heat_frac_adj}

\MyQuote{
Given $s\in(0,1)$ and $\varphi^T_0\in L^2(-1,1)$, let $\varphi$ be the unique solution to the system \eqref{heat_frac_adj}. Let $\omega\subset (-1,1)$ be an arbitrary open set. If $\varphi = 0$ on $\omega\times(0,T)$, then $\varphi = 0$ on $(-1,1)\times(0,T)$.}

Therefore, we are reduced to the proof of the property ($\mathcal P$). To this end, let us recall that $\varphi$ can be expressed in the form \eqref{adj_sol} and let us assume that 
\begin{align}\label{uc}
	\varphi=0 \textrm{ in } \omega\times(0,T). 
\end{align}

Let $\{\psi_{k_j}\}_{1\leq k\leq m_k}$ be an orthonormal basis of $\ker(\lambda_k-\fl{s}{})$. Then, \eqref{adj_sol} can be rewritten as
\begin{align*}
	\varphi(x,t) = \sum_{k\geq 1} \left(\sum_{j=1}^{m_k} \varphi_{k_j}\psi_{k_j}(x)\right)e^{-\lambda_k(T-t)}, \;\;\; (x,t)\in (-1,1)\times(-\infty, T). 
\end{align*}

Let $z\in\CC$ with $\eta:=\Re(z)>0$ and let $N\in\NN$. Since the functions $\psi_{k_j}$, $1\leq j\leq m_k$, $1\leq k\leq N$ are orthonormal, we have that
\begin{align*}
	\norm{\sum_{k=1}^N \left(\sum_{j=1}^{m_k} \varphi_{k_j}\psi_{k_j}(x)\right)e^{z(t-T)}e^{-\lambda_k(T-t)}}{L^2(-1,1)}^2 & \leq \sum_{k=1}^N \left(\sum_{j=1}^{m_k} |\varphi_{k_j}|^2\right)e^{2\eta(t-T)}e^{-2\lambda_k(T-t)}
	\\
	& \leq \sum_{k\geq 1} \left(\sum_{j=1}^{m_k} |\varphi_{k_j}|^2\right)e^{2\eta(t-T)}e^{-2\lambda_k(T-t)} \leq Ce^{2\eta(t-T)}\norm{\varphi^T}{L^2(-1,1)}^2.
\end{align*}

Hence, letting 
\begin{align*}
w_N(x,t):= \sum_{k=1}^N \left(\sum_{j=1}^{m_k} \varphi_{k_j}\psi_{k_j}(x)\right)e^{z(t-T)}e^{-\lambda_k(T-t)},
\end{align*}
we have shown that $\norm{w_T(x,t)}{L^2(-1,1)}\leq Ce^{\eta(t-T)}\norm{\varphi^T}{L^2(-1,1)}$. Moreover, we have
\begin{align*}
	\int_{-\infty}^T e^{\eta(t-T)}\norm{\varphi^T}{L^2(-1,1)}\,dt = \frac{1}{\eta}\norm{\varphi^T}{L^2(-1,1)}\int_0^{+\infty} e^{-\tau}\,d\tau = \frac{1}{\eta}\norm{\varphi^T}{L^2(-1,1)}.
\end{align*}

Therefore, we can apply the Dominated Convergence Theorem, obtaining
\begin{align}\label{dct_identity}
	\lim_{N\to+\infty}\int_{-\infty}^T w_N(x,t)\,dt &= \int_{-\infty}^T\lim_{N\to+\infty} w_N(x,t)\,dt = \int_{-\infty}^T e^{z(t-T)} \sum_{k=1}^{+\infty} \left(\sum_{j=1}^{m_k} \varphi_{k_j}\psi_{k_j}(x)\right)e^{-\lambda_k(T-t)}\,dt \notag
	\\
	&=\sum_{k=1}^{+\infty}\sum_{j=1}^{m_k} \varphi_{k_j}\psi_{k_j}(x) \int_{-\infty}^T e^{z(t-T)}e^{-\lambda_k(T-t)}\,dt =\sum_{k=1}^{+\infty}\sum_{j=1}^{m_k} \varphi_{k_j}\psi_{k_j}(x) \int_0^{+\infty} e^{-(z+\lambda_k)\tau}\,d\tau \notag
	\\
	&=\sum_{k=1}^{+\infty}\sum_{j=1}^{m_k} \frac{\varphi_{k_j}}{z+\lambda_k}\psi_{k_j}(x), \;\;\; x\in (-1,1)\;\Re(z)>0.
\end{align} 

It follows from \eqref{uc} and \eqref{dct_identity} that 
\begin{align*}
	\sum_{k=1}^{+\infty}\sum_{j=1}^{m_k} \frac{\varphi_{k_j}}{z+\lambda_k}\psi_{k_j}(x)=0, \;\;\; x\in\omega,\;\Re(z)>0.
\end{align*}

This holds for every $z\in\CC\setminus\{-\lambda_k\}_{k\in\NN}$, using the analytic continuation in $z$. Hence, taking a suitable small circle around $-\lambda_{\ell}$ not including $\{-\lambda_k\}_{k\neq\ell}$ and integrating on that circle we get that
\begin{align*}
	w_\ell:=\sum_{j=1}^{m_{\ell}} \varphi_{\ell_j}\psi_{\ell_j}(x)=0, \;\;\; x\in\omega.
\end{align*}

According to \cite[Theorem 1.4]{fall2014unique}, $\fl{s}{}$ has the unique continuation property in the sense that if $\lambda_k$ is an eigenvalue of $\fl{s}{}$ on (-1,1) with Dirichlet boundary conditions, and $(\fl{s}{}-\lambda_k)\varrho_k = 0$ in $(-1,1)$ with $\varrho_k = 0$ in $\omega$, then $\varrho_k = 0$ in $(-1,1)$. 
This can applied to $w_{\ell}$, in order to conclude $w_{\ell} = 0$ in $(-1,1)$ for every $\ell$. Since $\{\psi_{\ell_j}\}_{1\leq j\leq m_{\ell}}$ are linearly independent in $L^2(-1,1)$, we get $\varphi_{\ell_j} = 0$, $1\leq j\leq m_k$, $\ell\in\NN$. It follows that $\varphi^T=0$ and hence, $\varphi=0$ in $(-1,1)\times(0,T)$, meaning that $\varphi$ enjoys the property $(\mathcal{P})$. As an immediate consequence, we have that our original equation \eqref{heat_frac} is approximately controllable. This last fact being classical (see, e.g., \cite[Theorem 2.5]{keyantuo2016interior}), we will leave the details to the reader. Our proof is then concluded. 
\end{proof}

\section{Development of the numerical scheme}\label{fe_sec}

We devote this Section to the description of the numerical scheme that we are going to employ. Let us start with the elliptic case. 
\subsection{Finite element approximation of the elliptic problem}\label{fe_ell_sec}
In order to solve numerically \eqref{PE}, we will develop a finite element scheme on a uniform mesh. To this purpose, let us firstly recall the variational formulation associated to our problem: find $u\in H_0^s(-L,L)$ such that
\begin{align}\label{WF}
	\frac{\ccs}{2} \int_{\RR}\int_{\RR}\frac{(u(x)-u(y))(v(x)-v(y))}{|x-y|^{1+2s}}\,dxdy = \int_{-L}^L fv\,dx,	
\end{align}
for all $v\in H_0^s(-L,L)$. 

Let us introduce a partition of the interval $(-L,L)$ as follows:
\begin{align*}
	-L = x_0<x_1<\ldots <x_i<x_{i+1}<\ldots<x_{N+1}=L\,,
\end{align*}
with $x_{i+1}=x_i+h$, $i=0,\ldots N$. We call $\mathfrak{M}$ the mesh composed by the points $\{x_i\,:\, i=1,\ldots,N\}$, while the set of the boundary points is denoted $\partial\mathfrak{M}:=\{x_0,x_{N+1}\}$. Now, define $K_i:=[x_i,x_{i+1}]$ and consider the discrete space 
\begin{align*}
	V_h :=\Big\{v\in H_0^s(-L,L)\,\big|\, \left. v\,\right|_{K_i}\in \mathcal{P}^1\Big\},
\end{align*} 
where $\mathcal{P}^1$ is the space of the continuous and piece-wise linear functions. Hence, we approximate \eqref{WF} with the following discrete problem: find $u_h\in V_h$ such that
\begin{align*}
	\frac{\ccs}{2} \int_{\RR}\int_{\RR}\frac{(u_h(x)-u_h(y))(v_h(x)-v_h(y))}{|x-y|^{1+2s}}\,dxdy = \int_{-L}^L fv_h\,dx,	
\end{align*}
for all $v_h\in V_h$. If now we indicate with $\big\{\phi_i\big\}_{i=1}^N$ a basis of $V_h$, it will be sufficient that \eqref{WFD} is satisfied for all the functions of the basis, since any element of $V_h$ is a linear combination of them. Therefore the problem takes the following form
\begin{align}\label{WFD}
	\frac{\ccs}{2} \int_{\RR}\int_{\RR}\frac{(u_h(x)-u_h(y))(\phi_i(x)-\phi_i(y))}{|x-y|^{1+2s}}\,dxdy = \int_{-L}^L fv_h\,dx,\;\;\; i=1,\ldots,N.	
\end{align}
Clearly, since $u_h\in V_h$, we have 
\begin{align*}
	u_h(x) = \sum_{j=1}^N u_j\phi_j(x),
\end{align*} 
where the coefficients $u_j$ are, a priori, unknown. In this way, \eqref{WFD} is reduced to solve the linear system $\mathcal A_h u=F$, where the stiffness matrix $\mathcal A_h\in \RR^{N\times N}$ has components
\begin{align}\label{stiffness_nc}
	a_{i,j}=\frac{\ccs}{2} \int_{\RR}\int_{\RR}\frac{(\phi_i(x)-\phi_i(y))(\phi_j(x)-\phi_j(y))}{|x-y|^{1+2s}}\,dxdy,	
\end{align}
while the vector $F\in\RR^N$ is given by $F=(F_1,\ldots,F_N)$ with
\begin{align*}
	F_i = \langle f,\phi_i\rangle = \int_{-L}^L f\phi_i\,dx,\;\;\; i=1,\ldots,N.
\end{align*}

Moreover, the basis $\big\{\phi_i\big\}_{i=1}^N$ that we will employ is the classical one in which each $\phi_i$ is the tent function with $supp(\phi_i)=(x_{i-1},x_{i+1})$ and verifying $\phi_i(x_j)=\delta_{i,j}$. In particular, for $x\in\{x_{i-1},x_i,x_{i+1}\}$ the $i^{th}$ function of the basis is explicitly defined as (see Fig. \ref{basis}) 
\begin{align}\label{basis_fun}
	\phi_i(x)= 1-\frac{|x-x_i|}{h}.
\end{align} 

\begin{figure}[h]
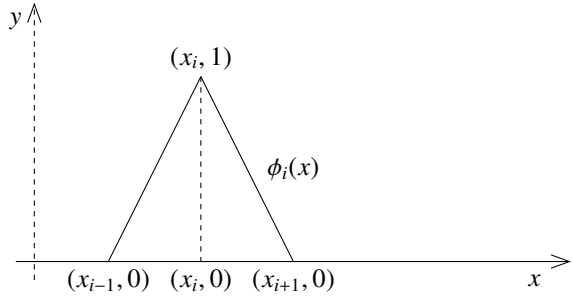

\figinit{0.7pt}
\figpt 1:(-100,0) \figpt 2:(200,0)
\figpt 11:(-90,-10) \figpt 12:(-90,140)
\figpt 3:(-50,0) \figpt 4:(0,100) 
\figpt 5:(50,0) \figpt 6:(0,0)
%
\figpt 7:(-50,-10) \figpt 8:(0,110) 
\figpt 9:(50,-10) \figpt 10:(0,-10)

\figpt 13:(-100,130) \figpt 14:(180,-10)
\figpt 15:(50,50)

\figdrawbegin{}
\figdrawarrow[1,2]
\figdrawline[3,4]
\figdrawline[4,5]
\figset(dash=4)
\figdrawline[4,6]
\figdrawarrow[11,12]

\figdrawend

\figvisu{\figBoxA}{}{
\figwritec [7]{$(x_{i-1},0)$}
\figwritec [8]{$(x_i,1)$}
\figwritec [9]{$(x_{i+1},0)$}
\figwritec [10]{$(x_i,0)$}
\figwritec [13]{$y$}
\figwritec [14]{$x$}
\figwritec [15]{$\phi_i(x)$}
}
\centerline{\box\figBoxA}
\caption{Basis function $\phi_i(x)$ on its support $(x_{i-1},x_{i+1})$.}\label{basis}
\end{figure}


\noindent Let us now describe our algorithm. Before that, we shall make the following preliminary comments.
\begin{remark}\label{rem_prel}
The following fact are worth noticing.
\begin{enumerate}
	\item It is evident from the definition \eqref{stiffness_nc} that $\mathcal A_h$ is symmetric. Therefore, in our algorithm we will only need to compute the values $a_{i,j}$ with $j\geq i$.
	
	\item Due to the non-local nature of the problem, the matrix $\mathcal A_h$ is full. However, while computing its components, we will encounter many simplifications, due to the fact that $supp(\phi_i)\cap supp(\phi_j) =\emptyset$ for $j\geq i+2$.
	
	\item While computing the values $a_{i,j}$, we will only work on the mesh $\mathfrak{M}$, not considering the points of the set $\partial\mathfrak{M}$. In this way, we will ensure that the basis functions $\phi_i$ satisfy the zero Dirichlet boundary conditions. In other words, in our FE approximation we are considering only the functions from $\phi_1$ to $\phi_N$. Instead, if we considered the points $x_0$ and $x_{N+1}$, then we would need to introduce in our discretization also the basis functions $\phi_0$ and $\phi_{N+1}$, which take value one at the boundary, and this would not be consistent with the continuous problem. Fig. \ref{basis_int} provides a graphical explanation of this last discussion.      	
\end{enumerate}
\end{remark}

\begin{figure}[h]
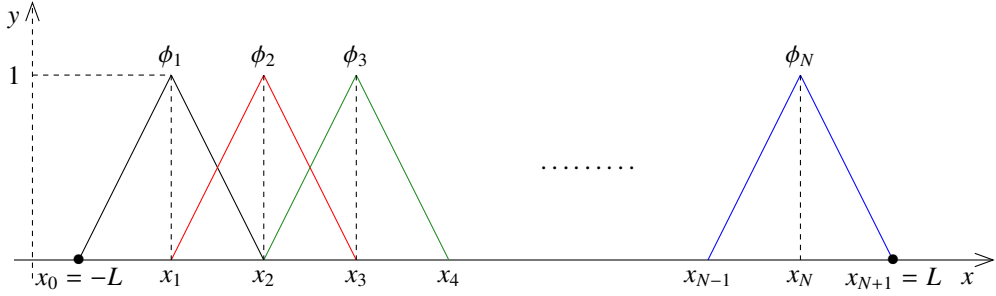

\figinit{0.7pt}
\figpt 1:(-110,0) \figpt 2:(420,0)
\figpt 11:(-100,-10) \figpt 12:(-100,140)
\figpt 3:(-75,0) \figpt 4:(-25,100) 
\figpt 5:(25,0) \figpt 6:(-25,0)

\figpt 31:(-25,0) \figpt 41:(25,100) 
\figpt 51:(75,0) \figpt 61:(25,0)

\figpt 32:(25,0) \figpt 42:(75,100) 
\figpt 52:(125,0) \figpt 62:(75,0)

\figpt 33:(265,0) \figpt 43:(315,100) 
\figpt 53:(365,0) \figpt 63:(315,0)

\figpt 200:(-100,100)
%
\figpt 7:(-75,-10) \figpt 8:(-25,110) 
\figpt 9:(25,-10) \figpt 10:(-25,-10)

\figpt 71:(-25,-10) \figpt 81:(25,110) 
\figpt 91:(75,-10) \figpt 101:(25,-10)

\figpt 72:(25,-10) \figpt 82:(75,110) 
\figpt 92:(125,-10) \figpt 102:(75,-10)

\figpt 73:(265,-10) \figpt 83:(315,110) 
\figpt 93:(365,-10) \figpt 103:(315,-10)

\figpt 13:(-110,130) \figpt 14:(405,-10)
\figpt 15:(50,50) \figpt 16:(250,50)

\figpt 201:(-110,100) \figpt 202:(200,50)

\figdrawbegin{}
\figdrawarrow[1,2]
\figdrawline[3,4]
\figdrawline[4,5]
\figset (color=\Redrgb)
\figdrawline[31,41]
\figdrawline[41,51]
\figset (color=default)
\figset (color=\ForestGreenrgb)
\figdrawline[32,42]
\figdrawline[42,52]
\figset (color=default)
\figset (color=\Bluergb)
\figdrawline[33,43]
\figdrawline[43,53]
\figset (color=default)
\figset(dash=4)
\figdrawline[4,6]
\figdrawline[200,4]
\figdrawline[41,61]
\figdrawline[42,62]
\figdrawline[43,63]
\figdrawarrow[11,12]

\figdrawend

\figvisu{\figBoxA}{}{
\figwritec [7]{$x_0=-L$}
\figwritec [8]{$\phi_1$}
\figwritec [9]{$x_2$}
\figwritec [10]{$x_1$}
\figwritec [81]{$\phi_2$}
\figwritec [91]{$x_3$}
\figwritec [82]{$\phi_3$}
\figwritec [92]{$x_4$}
\figwritec [73]{$x_{N-1}$}
\figwritec [83]{$\phi_N$}
\figwritec [93]{$x_{N+1}=L$}
\figwritec [103]{$x_N$}
\figwritec [13]{$y$}
\figwritec [14]{$x$}
\figwritec [201]{$1$}
\figwritec [202]{$\ldots\ldots\ldots$}
\figwritec[3,53]{$\bullet$}
}
\centerline{\box\figBoxA}
\caption{Basis functions $\phi_i(x)$ on the whole interval $(-L,L)$.}\label{basis_int}
\end{figure}


We now start building the stiffness matrix $\mathcal A_h$. This will be done it in three steps, since the values of the matrix can be computed differentiating among three well defined regions: the upper triangle, corresponding to $j\geq i+2$, the upper diagonal corresponding to $j=i+1$ and the diagonal, corresponding to $j=i$ (see Fig. \ref{matrix_fig}). In fact, as it will be clear during our computations, in each of these regions the intersections among the support of the basis functions are different, thus generating different values of the bilinear form. In what follows, we will briefly present which will be the contributions to the matrix in each of these three steps, including the complete computations as an appendix at the end of the paper. 

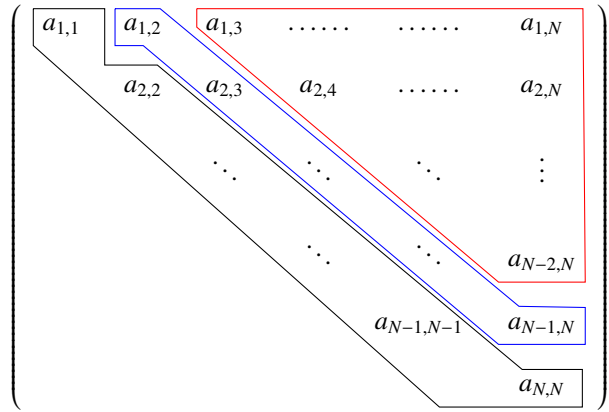
\begin{figure}[!h]

\centering
\begin{tikzpicture}
 \matrix[matrix of math nodes,left delimiter = (,right delimiter = ),row sep=10pt,column sep = 10pt] (m)
 {
 a_{1,1}  & a_{1,2} & a_{1,3} & \ldots\ldots & \ldots\ldots & a_{1,N}\\
   &a_{2,2} & a_{2,3} & a_{2,4} & \ldots\ldots &a_{2,N}\\
   & & \ddots & \ddots & \ddots &\vdots\\
   & & & \ddots & \ddots &a_{N-2,N}\\	   
   & & & & \hspace{-0.25cm}a_{N-1,N-1} &a_{N-1,N}\\
   & & & & & a_{N,N}\\
 };
\draw[color=red] (m-1-3.north west) -- (m-1-6.north east) -- +(0.14,0) -- (m-4-6.south east) -- (m-4-6.south west) -- (m-1-3.west) -- (m-1-3.north west);
\draw[color=blue] (m-1-2.north west) -- +(0.6,0) -- +(5.37,-3.96) -- (m-5-6.north east) -- (m-5-6.south east) -- (m-5-6.south west) -- (m-1-2.south) -- (m-1-2.south west) -- (m-1-2.north west); 
\draw[color=black] (m-1-1.north west) -- (m-1-1.north east) -- +(0.2,0) -- +(0.2,-0.75) -- +(0.9,-0.75) -- +(5.75,-4.8) -- +(6.55,-4.8) -- +(6.55,-5.3) -- +(4.65,-5.3) -- (m-1-1.south west) -- (m-1-1.north west); 
\end{tikzpicture}
\caption{Structure of the stiffness matrix $\mathcal{A}_h$.}\label{matrix_fig}
\end{figure}

%

\subsubsection*{Step 1: $j\geq i+2$}
As we mentioned in Remark \ref{rem_prel}, in this case we have $supp(\phi_i)\cap supp(\phi_j) =\emptyset$ (see also Fig. \ref{basis_upp_tri}). Hence, \eqref{stiffness_nc} is reduced to computing only the integral
\begin{align}\label{elem_noint}
	a_{i,j}=-2 \int_{x_{j-1}}^{x_{j+1}}\int_{x_{i-1}}^{x_{i+1}}\frac{\phi_i(x)\phi_j(y)}{|x-y|^{1+2s}}\,dxdy.
\end{align}

\begin{figure}[h]
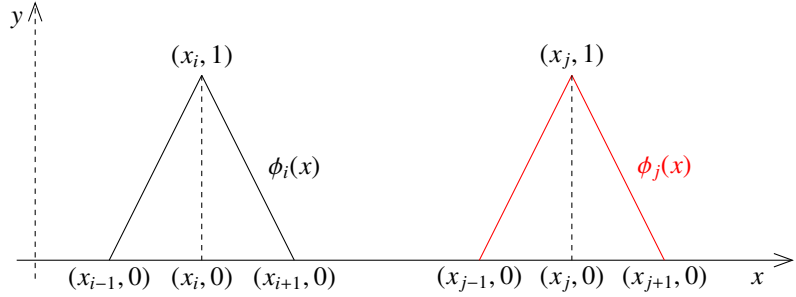

\figinit{0.7pt}
\figpt 1:(-100,0) \figpt 2:(320,0)
\figpt 11:(-90,-10) \figpt 12:(-90,140)
\figpt 3:(-50,0) \figpt 4:(0,100) 
\figpt 5:(50,0) \figpt 6:(0,0)

\figpt 31:(150,0) \figpt 41:(200,100) 
\figpt 51:(250,0) \figpt 61:(200,0)
%
\figpt 7:(-50,-10) \figpt 8:(0,110) 
\figpt 9:(50,-10) \figpt 10:(0,-10)

\figpt 71:(150,-10) \figpt 81:(200,110) 
\figpt 91:(250,-10) \figpt 101:(200,-10)

\figpt 13:(-100,130) \figpt 14:(300,-10)
\figpt 15:(50,50) \figpt 16:(250,50)

\figdrawbegin{}
\figdrawarrow[1,2]
\figdrawline[3,4]
\figdrawline[4,5]
\figset (color=\Redrgb)
\figdrawline[31,41]
\figdrawline[41,51]
\figset (color=default)
\figset(dash=4)
\figdrawline[4,6]
\figdrawline[41,61]
\figdrawarrow[11,12]

\figdrawend

\figvisu{\figBoxA}{}{
\figwritec [7]{$(x_{i-1},0)$}
\figwritec [8]{$(x_i,1)$}
\figwritec [9]{$(x_{i+1},0)$}
\figwritec [10]{$(x_i,0)$}
\figwritec [71]{$(x_{j-1},0)$}
\figwritec [81]{$(x_j,1)$}
\figwritec [91]{$(x_{j+1},0)$}
\figwritec [101]{$(x_j,0)$}
\figwritec [13]{$y$}
\figwritec [14]{$x$}
\figwritec [15]{$\phi_i(x)$}
\figwritec [16]{$\color{red}\phi_j(x)$}
}
\centerline{\box\figBoxA}
\caption{Basis functions $\phi_i(x)$ and $\phi_j(x)$ for $j\geq i+1$. In this case, the supports are disjoint.}\label{basis_upp_tri}
\end{figure}


Taking into account the definition of the basis function \eqref{basis_fun}, from \eqref{elem_noint} we obtain
\begin{align*}
	a_{i,j}=-2 \int_{x_{j-1}}^{x_{j+1}}\int_{x_{i-1}}^{x_{i+1}}\frac{\left(1-\frac{|x-x_i|}{h}\right)\left(1-\frac{|y-x_j|}{h}\right)}{|x-y|^{1+2s}}\,dxdy.
\end{align*}

Finally, this last integral can be computed explicitly employing the following  change of variables:
\begin{align}\label{cv}
	\frac{x-x_i}{h}=\hat{x},\;\;\; \frac{y-x_i}{h}=\hat{y}.
\end{align}

In this way, for the elements $a_{i,j}$, $j\geq i+2$, we get the following values: 
\begin{align*}
	a_{i,j} = \begin{cases}
			\displaystyle -h^{1-2s}\,\frac{4(k+1)^{3-2s} + 4(k-1)^{3-2s}-6k^{3-2s}-(k+2)^{3-2s}-(k-2)^{3-2s}}{2s(1-2s)(1-s)(3-2s)}, & k=j-i,\;\; \displaystyle s\neq\frac{1}{2} 
			\\
			\\
			-4(j-i+1)^2\log(j-i+1)-4(j-i-1)^2\log(j-i-1) &  \displaystyle s=\frac{1}{2},\;\; j> i+2
			\\
			\;\;\;+6(j-i)^2\log(j-i)+(j-i+2)^2\log(j-i+2)+(j-i-2)^2\log(j-i-2), 
			\\
			\\
			56\ln(2)-36\ln(3), & \displaystyle s=\frac{1}{2},\;\; j= i+2.
		\end{cases}	
\end{align*}

\subsubsection*{Step 2: $j= i+1$}
This is the most cumbersome case, since it is the one with the most interactions between the basis functions (see Fig. \ref{basis_upp_dia}). 

\begin{figure}[h]
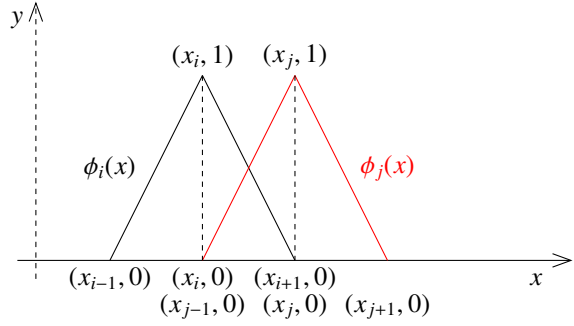

\figinit{0.7pt}
\figpt 1:(-100,0) \figpt 2:(200,0)
\figpt 11:(-90,-10) \figpt 12:(-90,140)
\figpt 3:(-50,0) \figpt 4:(0,100) 
\figpt 5:(50,0) \figpt 6:(0,0)

\figpt 31:(0,0) \figpt 41:(50,100) 
\figpt 51:(100,0) \figpt 61:(50,0)
%
\figpt 7:(-50,-10) \figpt 8:(0,110) 
\figpt 9:(50,-10) \figpt 10:(0,-10)

\figpt 71:(0,-25) \figpt 81:(50,110) 
\figpt 91:(100,-25) \figpt 101:(50,-25)

\figpt 13:(-100,130) \figpt 14:(180,-10)
\figpt 15:(-50,50) \figpt 16:(100,50)

\figdrawbegin{}
\figdrawarrow[1,2]
\figdrawline[3,4]
\figdrawline[4,5]
\figset (color=\Redrgb)
\figdrawline[31,41]
\figdrawline[41,51]
\figset (color=default)
\figset(dash=4)
\figdrawline[4,6]
\figdrawline[41,61]
\figdrawarrow[11,12]

\figdrawend

\figvisu{\figBoxA}{}{
\figwritec [7]{$(x_{i-1},0)$}
\figwritec [8]{$(x_i,1)$}
\figwritec [9]{$(x_{i+1},0)$}
\figwritec [10]{$(x_i,0)$}
\figwritec [71]{$(x_{j-1},0)$}
\figwritec [81]{$(x_j,1)$}
\figwritec [91]{$(x_{j+1},0)$}
\figwritec [101]{$(x_j,0)$}
\figwritec [13]{$y$}
\figwritec [14]{$x$}
\figwritec [15]{$\phi_i(x)$}
\figwritec [16]{$\color{red}\phi_j(x)$}
}
\centerline{\box\figBoxA}
\caption{Basis functions $\phi_i(x)$ and $\phi_{i+1}(x)$. In this case, the intersection of the supports is the interval $[x_i,x_{i+1}]$.}\label{basis_upp_dia}
\end{figure}


According to \eqref{stiffness_nc}, and using the symmetry of the integral with respect to the bisector $y=x$, we have 
	\begin{align*}
	a_{i,i+1}= & \int_{\RR}\int_{\RR}\frac{(\phi_i(x)-\phi_i(y))(\phi_{i+1}(x)-\phi_{i+1}(y))}{|x-y|^{1+2s}}\,dxdy
	\\
	= & \int_{x_{i+1}}^{+\infty}\int_{x_{i+1}}^{+\infty} \ldots\,dxdy + 2\int_{x_{i+1}}^{+\infty}\int_{x_i}^{x_{i+1}} \ldots\,dxdy + 2\int_{x_{i+1}}^{+\infty}\int_{-\infty}^{x_i} \ldots\,dxdy 
	\\
	& + \int_{x_i}^{x_{i+1}}\int_{x_i}^{x_{i+1}} \ldots\,dxdy + 2\int_{x_i}^{x_{i+1}}\int_{-\infty}^{x_i} \ldots\,dxdy + \int_{-\infty}^{x_i}\int_{-\infty}^{x_i} \ldots\,dxdy 
	\\
	:= & Q_1 + Q_2 + Q_3 + Q_4 + Q_5 + Q_6.
\end{align*}

These contributions will be calculated separately, employing changes of variables analogous to \eqref{cv}. After several computations, we obtain
\begin{align*}
	a_{i,i+1} = \begin{cases}
					\displaystyle h^{1-2s}\frac{3^{3-2s}-2^{5-2s}+7}{2s(1-2s)(1-s)(3-2s)}, & \displaystyle s\neq \frac{1}{2}
					\\
					9\ln 3-16\ln 2, & \displaystyle s=\frac{1}{2}.
				\end{cases}	
\end{align*}

\subsubsection*{Step 3: $j= i$}
As a last step, we fill the diagonal of the matrix $\mathcal A_h$, which collects the values corresponding to the case $\phi_i(x)=\phi_j(x)$ (see Fig. \ref{basis_dia}). We have
	\begin{align*}
	a_{i,i}= & \int_{\RR}\int_{\RR}\frac{(\phi_i(x)-\phi_i(y))^2}{|x-y|^{1+2s}}\,dxdy
	\\
	= & \int_{x_{i+1}}^{+\infty}\int_{x_{i+1}}^{+\infty} \ldots\,dxdy + 2\int_{x_{i+1}}^{+\infty}\int_{x_{i-1}}^{x_{i+1}} \ldots\,dxdy + \int_{x_{i+1}}^{+\infty}\int_{-\infty}^{x_{i-1}} \ldots\,dxdy 
	\\
	& + \int_{x_{i-1}}^{x_{i+1}}\int_{x_{i-1}}^{x_{i+1}} \ldots\,dxdy + 2\int_{-\infty}^{x_{i-1}}\int_{x_{i-1}}^{x_{i+1}} \ldots\,dxdy + + \int_{-\infty}^{x_{i-1}}\int_{x{i+1}}^{+\infty} \ldots\,dxdy 
	\\
	& +  \int_{-\infty}^{x_{i-1}}\int_{-\infty}^{x_{i-1}} \ldots\,dxdy := R_1 + R_2 + R_3 + R_4 + R_5 + R_6 + R_7.
\end{align*}

\begin{figure}[h]
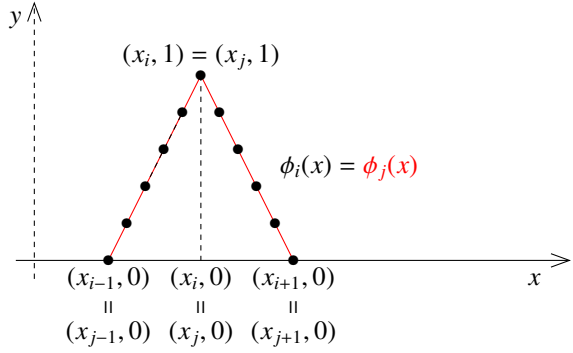

\figinit{0.7pt}
\figpt 1:(-100,0) \figpt 2:(200,0)
\figpt 11:(-90,-10) \figpt 12:(-90,140)
\figpt 3:(-50,0) \figpt 4:(0,100) 
\figpt 5:(50,0) \figpt 6:(0,0)

\figpt 31:(-40,20) \figpt 41:(-30,40) 
\figpt 51:(-20,60) \figpt 61:(-10,80)

\figpt 32:(40,20) \figpt 42:(30,40) 
\figpt 52:(20,60) \figpt 62:(10,80)
%
\figpt 7:(-50,-10) \figpt 8:(0,110) 
\figpt 9:(50,-10) \figpt 10:(0,-10)

\figpt 71:(-50,-40) \figpt 81:(50,110) 
\figpt 91:(50,-40) \figpt 101:(0,-40)

\figpt 72:(-50,-25) \figpt 82:(50,110) 
\figpt 92:(50,-25) \figpt 102:(0,-25)

\figpt 13:(-100,130) \figpt 14:(180,-10)
\figpt 15:(-50,50) \figpt 16:(80,50)

\figdrawbegin{}
\figdrawarrow[1,2]
\figset (color=\Redrgb)
\figdrawline[3,4]
\figdrawline[4,5]
\figset (width=default,color=default)
\figset(dash=4)
\figdrawline[4,6]
\figdrawline[41,61]
\figdrawarrow[11,12]

\figdrawend

\figvisu{\figBoxA}{}{
\figwritec [7]{$(x_{i-1},0)$}
\figwritec [8]{$(x_i,1)=(x_j,1)$}
\figwritec [9]{$(x_{i+1},0)$}
\figwritec [10]{$(x_i,0)$}
\figwritec [71]{$(x_{j-1},0)$}
\figwritec [91]{$(x_{j+1},0)$}
\figwritec [101]{$(x_j,0)$}
\figwritec [72,92,102]{$\shortparallel$}
\figwritec [13]{$y$}
\figwritec [14]{$x$}
\figwritec [16]{$\phi_i(x)=\color{red}\phi_j(x)$}
\figset write(mark=$\figBullet$)
\figwritep[3,31,41,51,61,4,32,42,52,62,5]
}
\centerline{\box\figBoxA}
\caption{Basis functions $\phi_i(x)$ and $\phi_j(x)$. In this case, the two functions coincide.}\label{basis_dia}
\end{figure}


Once again, the terms $R_i$, $i=1,\ldots,7$ will be computed separately, obtaining
\begin{align*}
	a_{i,i} = \begin{cases}
			\displaystyle h^{1-2s}\,\frac{2^{3-2s}-4}{s(1-2s)(1-s)(3-2s)}, & \displaystyle s\neq\frac{1}{2}
			\\
			\\
			8\ln 2, & \displaystyle s=\frac{1}{2}.			
			\end{cases}	
\end{align*}

\subsubsection*{Conclusion}
Summarizing, we have the following values for the elements of the stiffness matrix $\mathcal{A}_h$: for $s\neq 1/2$
\begin{align*}
	a_{i,j} =  -h^{1-2s} \begin{cases}
			\displaystyle \,\frac{4(k+1)^{3-2s} + 4(k-1)^{3-2s}-6k^{3-2s}-(k+2)^{3-2s}-(k-2)^{3-2s}}{2s(1-2s)(1-s)(3-2s)}, &  \displaystyle k=j-i,\,k\geq 2
			\\
			\\
			\displaystyle\frac{3^{3-2s}-2^{5-2s}+7}{2s(1-2s)(1-s)(3-2s)}, & \displaystyle j=i+1
			\\
			\\
			\displaystyle\frac{2^{3-2s}-4}{s(1-2s)(1-s)(3-2s)}, & \displaystyle j=i.
		\end{cases}	
\end{align*}

For $s=1/2$, instead, we have

\begin{align*}
	a_{i,j} = \begin{cases}
			-4(j-i+1)^2\log(j-i+1)-4(j-i-1)^2\log(j-i-1) 
			\\
			\;\;\;+6(j-i)^2\log(j-i)+(j-i+2)^2\log(j-i+2)+(j-i-2)^2\log(j-i-2), &  \displaystyle j> i+2
			\\
			\\
			56\ln(2)-36\ln(3), & \displaystyle j= i+2.
			\\
			\\
			\displaystyle 9\ln 3-16\ln 2, & \displaystyle j=i+1
			\\
			\\
			\displaystyle 8\ln 2, & \displaystyle j=i.
		\end{cases}	
\end{align*}

\begin{remark}
We point out the following facts:
\begin{enumerate}
	\item The matrix $\mathcal A_h$ has the structure of a $N$-diagonal matrix, meaning that value of its elements remain constant along its diagonals. This is in analogy with the tridiagonal matrix approximating the classical Laplace operator. Notice, however, that in our case we obtain a full matrix. This is consistent with the non-local nature of the operator that we are discretizing.
	
	\item The value of each element $a_{i,j}$ is given explicitly, and it only depends on $i$, $j$, $s$ and $h$. In other words, when approximating the left hand side of \eqref{WFD}, no numerical integration is needed. This significantly improve the efficiency of our method.
	
	\item For $s=1/2$, the elements $a_{i,j}$ do not depend on the value of $h$ which, in turn, is a function of $N$. This implies that, in this particular case, no matter how many points we consider in our mesh, the matrix $\mathcal A_h$ will always have the same entries. 
\end{enumerate}
\end{remark}

In Section \ref{res_numerical} below, we will present the numerical simulations associated to  the elliptic problem \eqref{PE}, discussing in detail the convergence properties of the algorithm. 

\subsection{Control problem for the fractional heat equation}

Let us now give a brief description of the so called penalised Hilbert Uniqueness Method (HUM in what follows) that we shall employ for computing the controls for equation \eqref{heat_frac}. Here, we will mostly refer to the work of Boyer \cite{boyer2013penalised}.

We start recalling the classical HUM, as it has been introduced in the pioneering works \cite{glowinski1995exact,glowinski2008exact}. Let $(E,\langle\cdot,\cdot\rangle)$ be a Hilbert space whose norm is denoted by $\norm{\cdot}{}$. Let $(A,\mathcal D(A))$ be an unbounded operator in $E$ such that $-A$ generates an analytic semi-group in $E$ that we indicate by $t\mapsto e^{-tA}$. Also, we denote $(A^*,\mathcal D(A^*))$ the adjoint of this operator and by $t\mapsto e^{-tA^*}$ the corresponding semi-group. 

Let $(U,[\cdot,\cdot])$ be another Hilbert space whose norm is denoted by $\inter{\cdot}$. Let  $B$ be an unbounded operator from $U$ to $\mathcal D(A^*)'$ and let $B^*:\mathcal D(A^*)\to U$ be its adjoint. Let $T>0$ be given and, for any $y_0\in E$ and $v\in L^2(0,T;U)$, let us consider the non-homogeneous evolution problem  
\begin{align}\label{abstract_pb}
	\begin{cases}
		y_t+Ay=Bv, & t\in[0,T]
		\\
		y(0)=y_0.
	\end{cases}
\end{align}

The well posedness of \eqref{abstract_pb} is guaranteed by \cite[Theorem 2.37]{coron2007control}. From now on, we will refer to the solution as $t\mapsto y_{v,y_0}(t)$.

Notice that we have $y_{0,y_0}(t)= e^{-tA}y_0$. Moreover, for simplicity, the solution at time $T$, which is of particular interest in what follows, will be denoted by 
\begin{align*}
	y_{v,y_0}(T)=\mathcal{L}_T(v|y_0).
\end{align*} 
The linear operator $\mathcal{L}_T(\cdot|\cdot)$ is then continuous from $L^2(0,T;U)\times E$ into $E$.

In the framework of both controllability notions that we introduced in Section \ref{theor_sec}, if one control exists it is certainly not unique. For instance, the classical HUM approach consists in finding the control with the minimal $L^2(0,T;U)$-norm. Nevertheless, even though in this way we can identify a precise control, its computation can be a difficult task due to the nature of the constraints involved (see, e.g., \cite{fernandez2000cost,micu2004introduction}). 

Because of what we just described, it is convenient to deal with a penalized version of the above mentioned optimization problems.

In the penalized version of the HUM, we look for a control that is solution to a different optimization problem. In particular, we for any $\varepsilon>0$, we shall find

\begin{align}\label{min_ve}
	v_\varepsilon=\min_{v\in L^2(0,T;U)} F_\varepsilon (v)
\end{align}
where
\begin{align*}
	F_\varepsilon(v):=\frac{1}{2}\int_0^T\inter{v(t)}^2\,dt+\frac{1}{2\varepsilon}\norm{\mathcal L(v|y_0)}{}^2, \quad \forall v\in L^2(0,T;U).
\end{align*}

Notice that, for any $\varepsilon > 0$, the functional $F_\varepsilon$ has a unique minimiser on $L^2(0,T;U)$  that we denote by $v_\varepsilon$. This is due to the fact that $F_\varepsilon$ is strictly convex, continuous and coercive. 

However, the space $L^2(0,T;U)$ in which one has to minimise $F_\varepsilon$ is a quite big one and it depends on the time $T$. This makes the minimization problem computationally expensive. On the other hand, this issue can be bypassed by considering a different optimization problem, defined on the smaller space $E$. Namely we have to find  

\begin{align}\label{min_je}
	q^T_\varepsilon=\min_{q\in E} J_\varepsilon (q^T)
\end{align}
where
\begin{align}\label{penalized_fun}
	J_\varepsilon(q^T):=\frac{1}{2}\int_0^T\inter{B^*e^{-(T-t)A^*}q^T}^2\,dt+\frac{\varepsilon}{2}\norm{q^T}{}^2 + \left\langle y_0,e^{-TA^*}q^T\right\rangle, \quad \forall q^T\in E.
\end{align}

Notice that \eqref{min_ve} and \eqref{min_je} are equivalent since, according to \cite[Proposition 1.5]{boyer2013penalised}, for any $\varepsilon > 0$, the minimisers $v_\varepsilon$ and $q_\varepsilon^T$ of the functionals $F_\varepsilon$ and $J_\varepsilon$, respectively, are related through the formula
\begin{align*}
	v_\varepsilon = B^*e^{-(T-t)A^*}q_\varepsilon^T, \textrm{ for a.e. } t\in(0,T).
\end{align*} 

Notice also that we can express the approximate and null controllability properties of the system, for a given initial datum $y_0$, in terms of the behaviour of the penalised HUM approach described above. In particular we have 

\begin{theorem}[Theorem 1.7 of \cite{boyer2013penalised}]\label{theorem_hum}
Problem \eqref{abstract_pb} is approximately controllable from the initial datum $y_0$ if and only if we have
\begin{align*}
	\mathcal{L}_T(v_\varepsilon|y_0) = y_{v_\varepsilon,y_0}(T)\rightarrow 0,\;\;\;\textrm{ as }\;\varepsilon\to 0.
\end{align*}
Problem \eqref{abstract_pb} is null-controllable from the initial datum $y_0$ if and only if we have
\begin{align*}
	M_{y_0}^2:=2\sup_{\varepsilon>0}\left( \inf_{L^2(0,T;U)}F_\varepsilon\right)<+\infty.
\end{align*}
In this case, we have 
\begin{align*}
	\inter{v_\varepsilon}_{L^2(0,T;U)}\leq M_{y_0},\nonumber
	\\
	\norm{\mathcal L_T(v_\varepsilon|y_0)}{}\leq M_{y_0}\sqrt{\varepsilon}.
\end{align*} 
\end{theorem} 

Since the fractional Laplacian $\fl{s}{}$ has the properties required for the operator $A$, the penalized HUM that we just described can be applied to the control problem \eqref{heat_frac}. To this purpose, let us now present its numerical implementation. 

Having obtained a FE approximation $\mathcal A_h$ of the operator $\fl{s}{}$, we can compute the fully-discrete version of \eqref{heat_frac}. For any given mesh $\mathfrak M$ and any integer $M>0$, we set $\delta t=T/M$ and we consider an implicit Euler method, with respect to the time variable. More precisely, we consider
\begin{align}\label{frac_heat_num}
	\begin{cases}
		\displaystyle\mathcal M_h \frac{z^{n+1}-z^n}{\delta t}+\mathcal A_h z^{n+1}=\mathbf{1}_\omega v_h^{n+1}, \quad \forall n\in \left\{1,\ldots,M-1\right\}
		\\
		z^0=z_0, 
	\end{cases}
\end{align}
where $z_0\in \mathbb R^\mesh$ and $\mathcal M_h$ is the classical mass matrix with entries $m_{i,j}=\langle \phi_i,\phi_j\rangle$.

Here, $v_{h,\delta t}=(v_h^n)_{1\leq n\leq M}$ is a fully-discrete control function whose cost, that is the discrete $L_{\delta t}^2(0,T;\mathbb R^\mesh)$-norm, is defined by
\begin{align*}
\|v_{\delta t}\|_{L_{\delta t}^2(0,T;\mathbb R^\mesh)}:=\left(\sum_{i=1}^M\delta t |v^n|^2_{L^2(\RR^\mesh)}\right)^{1/2},
\end{align*}
and where $| \cdot |_{L^2(\RR^\mesh)}$ stands for the norm associated to the $L^2$-inner product on $\mathbb{R}^\mathfrak M$
\begin{align*}
(u,v)_{L^2(\RR^\mesh)}=h \sum_{i=1}^N u_i v_i.
\end{align*}

With the above notation and according to the penalized HUM strategy, we introduce, for some penalization parameter $\varepsilon>0$, the following primal fully-discrete functional 
\begin{align*}
F_{\varepsilon,h,\delta t}(v_{\delta t})=\sum_{n=1}^{M}\delta t |v^n|^2_{L^2(\RR^\mesh)}+\frac{1}{2\varepsilon}|z^M|^2_{L^2(\RR^\mesh)}, \quad \forall \, v_{\delta t}\in L_{\delta t}^2(0,T;\mathbb R^\mesh),
\end{align*}
that we wish to minimize onto the whole fully-discrete control space $L_{\delta t}^2(0,T;\mathbb R^\mesh)$ and where $z^M$ is the final value of the controlled problem \eqref{frac_heat_num}. 

We can apply Fenchel-Rockafellar theory results to obtain the corresponding dual functional, which reads as follows
\begin{align}\label{dual_fully}
	J_{\varepsilon,h,\delta t}(\varphi^T)=\frac{1}{2}\sum_{n=1}^{M}\delta t|\mathbf 1_\omega\varphi|^2_{L^2(\RR^\mesh)}+\frac{\varepsilon}{2}|\varphi^T|^2_{L^2(\RR^\mesh)}+(\varphi^1,y_0)_{L^2(\RR^\mesh)}, \quad \forall \varphi^T\in L_{h,\delta t}^2(0,1)
\end{align}
where $\varphi=(\varphi^n)_{1\leq n\leq M+1}$ is solution to the adjoint system
\begin{align}\label{frac_adj_num}
	\begin{cases}
		\displaystyle\mathcal M_h \frac{\varphi^n-\varphi^{n+1}}{\delta t}+\mathcal A_h \varphi^n=0, & \forall n\in\inter{1,M}
		\\
		\varphi^{M+1}=\varphi^T. 
	\end{cases}
\end{align}

Notice that \eqref{dual_fully} is the fully-discrete approximation of \eqref{penalized_fun}. Moreover, it can be readily verified that this functional has a unique minimizer without any additional assumption on the problem. Therefore, by minimizing  \eqref{dual_fully}, and from duality theory, we obtain a control function 
\begin{align*}
	v_{\varepsilon,h,\delta t}=\left(\mathbf{1}_\omega\varphi_{\varepsilon,h,\delta t}^n\right)_{1\leq n\leq M},
\end{align*}
where $\varphi_\epsilon$ is the solution to \eqref{frac_adj_num} evaluated in the optimal datum $\varphi_\varepsilon^T$. 

Thus, the optimal penalized control always exists and is unique. Deducing controllability properties amounts to study the behavior of this control with respect to the penalization parameter $\varepsilon$, in connection with the discretization parameters.  

It is well known that, in general, we cannot expect for a given bounded family of initial data that the fully-discrete controls are uniformly bounded when the discretization parameters $h$, $\delta t$ and the penalization term $\varepsilon$ tend to zero independently. 

Instead, we expect to obtain uniform bounds by taking the penalization parameter $\varepsilon=\phi(h)$ that tends to zero in connection with the mesh size not too fast (see \cite{boyer2013penalised}) and a time step $\delta t$ verifying some weak condition of the kind $\delta t\leq \zeta(h)$ where $\zeta$ tends to zero logaritmically when $h\to 0$ (see \cite{boyer2011uniform}).

These facts will be confirmed by the numerical simulations that we are going to present in Section \ref{control_exp} below, by observing the behavior of the norm of the control, the optimal energy $\inf F_\varepsilon$, and the norm of the solution at time $T$. In this way, as predicted by Theorem \ref{theorem_hum}, we obtain a numerical evidence of the properties of null and approximate controllability for equation \eqref{heat_frac}, in accordance with the theoretical results in Section \ref{theor_sec}. 

\section{Numerical results} \label{res_numerical}
In this Section, we present the numerical simulations corresponding to the algorithm previously described, and we provide a complete discussion of the results obtained. 
First of all, in order to test numerically the accuracy of our method, we use the following problem 
\begin{align}\label{PE_real}
	\left\{\begin{array}{ll}
		\fl{s}{u} = 1, & x\in(-L,L)
		\\
		u\equiv 0, & x\in\RR\setminus(-L,L).
	\end{array}\right.
\end{align}
In this particular case, the solution can be computed exactly and it is given in \cite{getoor1961first}. It reads as follows, 
\begin{align}\label{real_sol}
	u(x)=\frac{2^{-2s}\sqrt{\pi}}{\Gamma\left(\frac{1+2s}{2}\right)\Gamma(1+s)}\Big(L^2-x^2\Big)^s.
\end{align}

In Fig.  \ref{smas12}, we show a comparison for different values of $s$ between the exact solution \eqref{real_sol} and the computed numerical approximation. Here we consider $L=1$ and $N=50$. One can notice that when $s=0.1$ (and also for other small values of s), the computed solution is to a certain extent different from the exact solution. However, one should be careful with such result and a more precise analysis of the error should be carried. 
\begin{figure}[!h]
	\pgfplotstableread{fl_01_50_sym.txt}{\datpu}
	\pgfplotstableread{fl_04_50_sym.txt}{\ddatpu}
	\pgfplotstableread{fl_05_50_sym.txt}{\Ddatpu}
	\pgfplotstableread{fl_08_50_sym.txt}{\Dddatpu}

		\subfloat[$s=0.1$]{
		\begin{tikzpicture}[scale=0.8]
		\begin{axis}[xmin=-1.1, xmax=1.1,legend style={at={(0.77,0.25)}}]
		
			\addplot [color=blue, mark=none,  thick] table[x=0,y=1]{\datpu};	
			\addplot [color=red, mark=x, only marks, thick] table[x=2,y=3]{\datpu};		
			\addlegendentry{Numerical solution}
			\addlegendentry{Real solution}
		\end{axis}
	\end{tikzpicture}
	\label{s01a}
	}
		\hspace{1cm}\subfloat[$s=0.4$]{
		\begin{tikzpicture}[scale=0.8]
		\begin{axis}[xmin=-1.1, xmax=1.1,legend style={at={(0.77,0.25)}}]
			\addplot [color=blue, mark=none, thick] table[x=0,y=1]{\ddatpu};	
			\addplot [color=red, mark=x, only marks, thick] table[x=2,y=3]{\ddatpu};		
		\end{axis}
	\end{tikzpicture}
	}
\\
		\subfloat[$s=0.5$]{
		\begin{tikzpicture}[scale=0.8]
		\begin{axis}[xmin=-1.1, xmax=1.1,legend style={at={(0.77,0.25)}}]
			\addplot [color=blue, mark=none, thick] table[x=0,y=1]{\Ddatpu};	
			\addplot [color=red, mark=x, only marks, thick] table[x=2,y=3]{\Ddatpu};		
		\end{axis}
	\end{tikzpicture}
	}
		\hspace{1cm}\subfloat[$s=0.8$]{
		\begin{tikzpicture}[scale=0.8]
		\begin{axis}[xmin=-1.1, xmax=1.1,legend style={at={(0.77,0.25)}}]
			\addplot [color=blue, mark=none, thick] table[x=0,y=1]{\Dddatpu};	
			\addplot [color=red, mark=x, only marks, thick] table[x=2,y=3]{\Dddatpu};		
		\end{axis}
	\end{tikzpicture}
	}
	\caption{Plot for different values of $s$.}
	\label{smas12}
\end{figure}
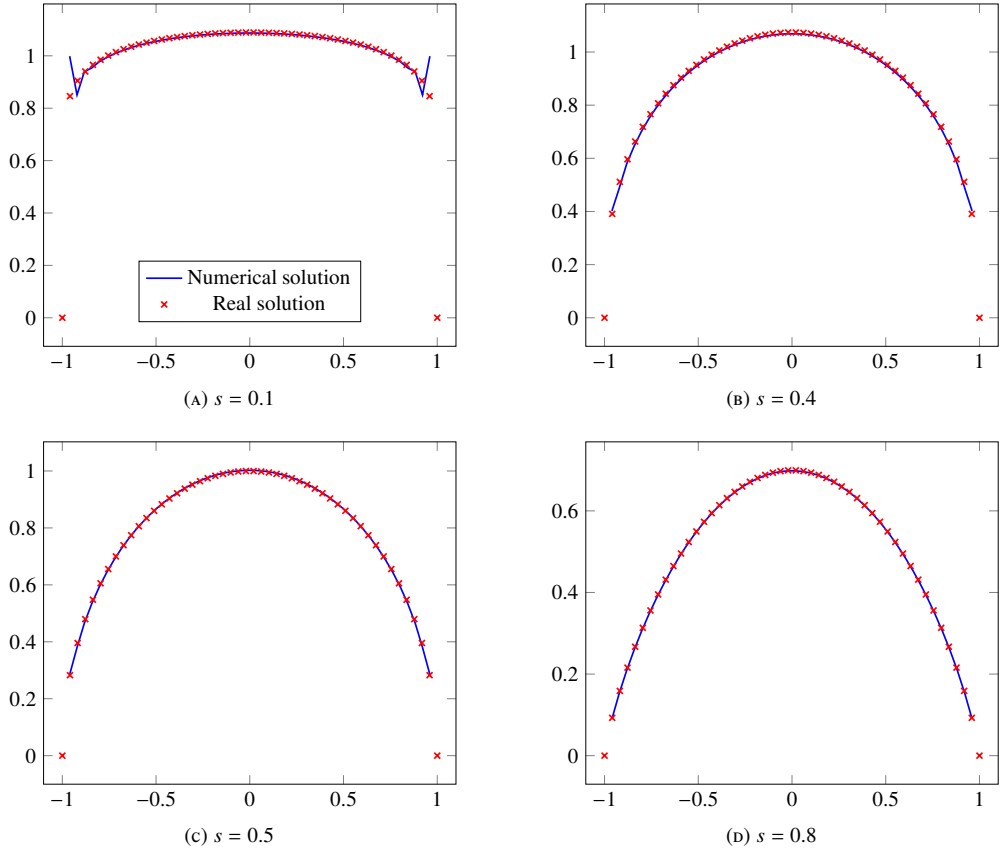


In the same spirit as in \cite{acosta2017short}, the computation of the error in the space $H_0^s(-L,L)$ can be readily done by using the definition of the bilinear form, namely
\begin{align*}
\|u-u_h\|^2_{H_0^s(-L,L)}&=a(u-u_h,u-u_h) \\
&=a(u,u-u_h) \\
&=\int_{-L}^{L}f(x)\left(u(x)-u_h(x)\right)dx,
\end{align*}
where have used the orthogonality condition $a(v_h,u-u_h)=0$ $\forall v_h \in V_h$.

For this particular test, since $f\equiv 1$ in $(-L,L)$, the problem is therefore reduced to
\begin{align*}
\|u-u_h\|_{H^s_0(-L,L)}=\left(\int_{-L}^{L}\left( u(x)-u_h(x) \right)\,dx\right)^{1/2}
\end{align*}
where the right-hand side can be easily computed, since we have the closed formula 
\begin{align*}
\int_{-L}^{L}u\,dx= \frac{\pi L^{2s+1}}{2^{2s}\Gamma(s+\frac{1}{2})\Gamma(s+\frac{3}{2})}
\end{align*}
and the term corresponding to $\int_{-L}^{L}u_h$ can be carried out numerically. 

In Fig.  \ref{error}, we present the computational errors evaluated for different values of $s$ and $h$. 

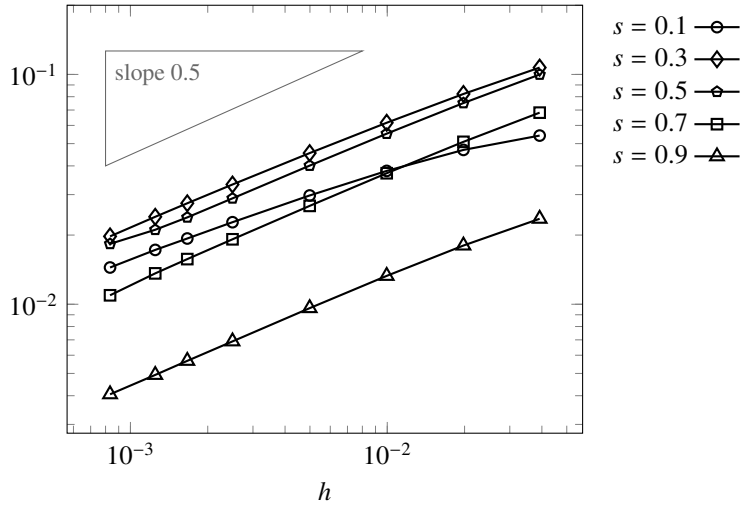
\begin{figure}[!h]
 \centering
  \pgfplotstableread{result_convergence_01.org}{\datpu}
  \pgfplotstableread{result_convergence_03.org}{\datpt}
   \pgfplotstableread{result_convergence_05.org}{\datpcnum}
  \pgfplotstableread{result_convergence_07.org}{\datps}
  \pgfplotstableread{result_convergence_09.org}{\datpn}
   \pgfplotsset{
     legend cell align=left,
     legend pos=outer north east,
     legend plot pos=right,
     legend style={cells={anchor=east},draw=none},
     }
%
  \begin{tikzpicture}
  \label{conv}
  \begin{loglogaxis}[xlabel=$h$, ymax=2e-1]
	\addplot [color=black, mark=o, thick] table [x=0, y=1] {\datpu};
	\addlegendentry{$s=0.1$};
	\addplot [color=black, mark=diamond, mark size=3 pt, thick] table [x=0, y=1] {\datpt};
	\addlegendentry{${s}=0.3$};
	\addplot [color=black, mark=pentagon, thick] table [x=0, y=1] {\datpcnum};
	\addlegendentry{${s}=0.5$};
	\addplot [color=black, mark=square, thick] table [x=0, y=1] {\datps};
	\addlegendentry{$ s=0.7$};
	\addplot [color=black, mark=triangle, mark size=3 pt, thick] table [x=0, y=1] {\datpn};
	\addlegendentry{$ s=0.9$};
 \draw [pente]  (axis cs: 0.0008,4e-2) -- ++ (axis cs: 1, {10^(0.5)}) -- ++ (axis cs: 10, 1) -- cycle;
 \node at (axis cs:0.0008,10e-2) [right,pente] {\small slope $0.5$};
  \end{loglogaxis}
\end{tikzpicture}%
\caption{Convergence of the error.}
\label{error}
\end{figure}


The rates of convergence shown are of order (in $h$) of $1/2$. This is in accordance with the following result: 
\begin{theorem}[Theorem 4.6 of \cite{acosta2017short}]
For the solution $u$ of \eqref{WF} and its FE approximation $u_h$ given by \eqref{WFD}, if $h$ is sufficiently small, the following estimates hold
\begin{align*}
&\|u-u_h\|_{H^s_0(-L,L)}\leq C h^{1/2}|\!\ln h|\,\|f\|_{C^{1/2-s}(-L,L)}, \quad \textnormal{if}\quad s<1/2, \\
&\|u-u_h\|_{H^s_0(-L,L)}\leq C h^{1/2} |\!\ln h|\, \|f\|_{L^\infty(-L,L)}, \quad \textnormal{if}\quad  s=1/2 \\
&\|u-u_h\|_{H^s_0(-L,L)}\leq \tfrac{C}{2s-1} h^{1/2} \sqrt{|\!\ln h|}\, \|f\|_{C^\beta(-L,L)}, \quad \textnormal{if} \quad s>1/2,
\end{align*}
where $C$ is a positive constant not depending on $h$. 
\end{theorem}

Moreover, Fig.  \ref{error} shows that the convergence rate is maintained also for small values of $s$. This confirms that the behavior shown in Fig. \ref{smas12} (A) is not in contrast with the known theoretical results. Indeed, since it is well-known that the notion of trace is not defined for the spaces $H^s(-L,L)$ with $s\leq 1/2$ (see \cite{lions1968problemes,tartar2007introduction}), it is somehow natural that we cannot expect a point-wise convergence in this case.  

As a further validation of this fact, in Fig.  \ref{linfty_error} we plot the behavior of the $L^{\infty}$-norm of the difference between the real and the numerical solution to \eqref{PE_real}. It is shown that, increasing the number of point of discretization, this norm is decreasing with a rate (in $h$) of $0.1$. This confirms that, refining the mesh, also for small values of $s$ the numerical method gives an acceptable approximation of the real solution to the model problem considered.   

\begin{figure}[!h]
\centering
  \pgfplotstableread{res.org}{\datpu}
     \pgfplotsset{
     legend cell align=left,
     legend pos=outer north east,
     legend plot pos=right,
     legend style={cells={anchor=east},draw=none},
     }
  \begin{tikzpicture}
  \label{conv}
  \begin{loglogaxis}
	\addplot [color=black, mark=o, thick] table [x=0, y=1] {\datpu};
 \draw [pente]  (axis cs: 0.0004,0.8) -- ++ (axis cs: 1, {10^(0.1)}) -- ++ (axis cs: 10, 1) -- cycle;
 \node at (axis cs:0.0004,0.98) [right,pente] {\small slope $0.1$};
  \end{loglogaxis}
\end{tikzpicture}%

\caption{Convergence of the error in the norm $L^\infty$.}\label{linfty_error}
\end{figure}
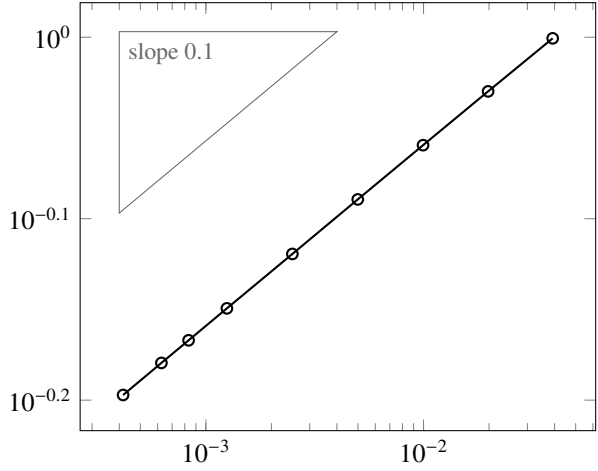


\subsection{Control experiments}\label{control_exp}

To address the actual computation of fully-discrete controls for a given problem, we use the methodology described, for instance, in \cite{glowinski2008exact}. We apply an optimization algorithm to the dual functional \eqref{dual_fully}. Since these functionals are quadratic and coercive, the conjugate gradient is a natural and quite simple choice.

In the same spirit as \cite{boyer2013penalised}, the computation of the gradient at each iteration amounts to solve first the homogeneous equation 

\begin{align*}
	\begin{cases}
		\displaystyle\mathcal M_h \frac{\varphi^n-\varphi^{n+1}}{\delta t}+\mathcal A_h \varphi^n=0, & \forall n\in\inter{1,M}
		\\
		\varphi^{M+1}=\varphi^T. 
	\end{cases}
\end{align*}

Then, set $v^n=\mathbf{1}_\omega\varphi^n$ and finally solve 

\begin{align*}
	\begin{cases}
		\displaystyle\mathcal M_h \frac{z^{n+1}-z^n}{\delta t}+\mathcal A_h z^{n+1}=\mathbf{1}_\omega \varphi_h^{n+1}, \quad \forall n\in \left\{1,\ldots,M-1\right\}
		\\
		z^0=0. 
	\end{cases}
\end{align*}

In this way, the procedure to compute the control for a given problem basically requires to solve two parabolic equations: a homogenous backward equation associated with the final data $\varphi^T$, and a non-homogeneous forward problem with zero initial data. 

We present now some results obtained with the described methodology. In accordance with the discussion in Section \ref{fe_ell_sec}, we use the finite-element approximation of $\fl{s}{}$ for the space discretization and the implicit Euler scheme in the time variable. We denote by $N$ the number of points in the mesh and by $M$ the number of time intervals. As discussed in \cite{boyer2013penalised}, the results in this kind of problems does not depend too much in the time step, as soon as it is chosen to ensure at least the same accuracy as the space discretization. The same remains true here, and therefore we always take $M=2000$ in order to concentrate the discussion on the dependency of the results with respect to the mesh size $h$ and the parameter $s$.

As we mentioned, we choose the penalization term $\varepsilon$ as a function of $h$. A reasonable practical rule (\cite{boyer2013penalised}) is to systematically choose $\phi(h)\sim h^{2p}$ where $p$ is the order of accuracy in space of the numerical method employed for the discretization of the spatial operator involved (in this case the fractional Laplacian \eqref{fl}). We recall that for the elliptic problem that we are considering, this order of convergence is $1/2$. Thus, hereinafter we always assume $\varepsilon=\phi(h)=h$.

We begin by plotting on Fig.  \ref{sol_surf} the time evolution of the uncontrolled solution as well as the controlled solution. Here, we set $s=0.8$, $\omega=(-0.3,0.8)$ and $T=0.3$, and as an initial condition we take $z_0(x) = \sin(\pi x)$. The control domain is represented as highlighted zone on the plane $(t,x)$. As expected, we observe that the uncontrolled solution is damped with time, but does not reach zero at time $T$, while the controlled solution does. 

\begin{figure}[ht]
 \centering
\subfloat[The adjoint][Uncontrolled solution]{
\begin{tikzpicture}
\begin{axis}[surface,zmax=1.0]
  
\addplot3 [mesh/ordering=y varies, surf, shader=flat,draw=black,opacity=0.6] file {soly_s08_sc.dat};
\node at (axis cs:0.34,0,0) [below] {$T=0.3$};

\end{axis}
\end{tikzpicture}}
\hspace{1 cm}
  \subfloat[The state][Controlled solution ({\tikz \fill [green] (0,0) rectangle (0.2,0.2);}=control domain)]{
\begin{tikzpicture}
\begin{axis}[surface, zmax=1.0]

  \couplingdomainT{-0.3}{0.8}{0.3};
  
\addplot3 [mesh/ordering=y varies, surf, shader=flat,draw=black,opacity=0.6] file {soly_s08_cc.dat};
\node at (axis cs:0.34,0,0) [below] {$T=0.3$};

\end{axis}
\end{tikzpicture}}%
\caption{Time evolution of system \eqref{frac_heat_num}.}\label{fig_heat_frac}
\label{sol_surf}
\end{figure}
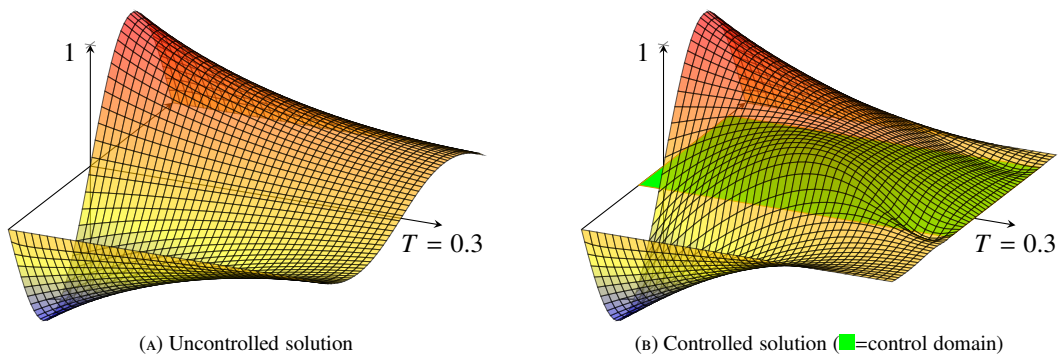


In figure \ref{figure_case1}, we present the computed values of various quantities of interest when the mesh size goes to zero. More precisely, we observe that the control cost $\|v_{\delta t}\|_{L^2_{\delta t}(0,T;\mathbb R^\mesh)}$ and the optimal energy $\inf F_{\phi(h),h,\delta t}$ remain bounded as $h\to 0$. On the other hand, we see that 
\begin{align}\label{control_norm_behavior}
	|y^M|_{L^2(\RR^\mesh)}\,\sim\,C\sqrt{\phi(h)}=Ch^{1/2}. 
\end{align}

We know that, for $s=0.8$, system \eqref{heat_frac} is null controllable. This is now confirmed by \eqref{control_norm_behavior}, according to Theorem \ref{theorem_hum}.  In fact, the same experiment can be repeated for different values of $s>1/2$, obtaining the same conclusions. 
\begin{figure}
  \centering
\begin{tikzpicture}
  \begin{loglogaxis}[erreurs, ymin=5e-4,ymax=0.2e1,title={$s=0.8$}]
 
    \pgfplotstableread[ignore chars={|},skip first n=2]{heat_frac_s=08.org}\resultats

    \addplot[cout] table[x=dx,y=Nv] \resultats;
    \addlegendentry{Cost of the control};
    \addplot[cible] table[x=dx,y=NyT] \resultats;
    \addlegendentry{Size of $y^M$};
    \addplot[energie] table[x=dx,y=Inf_eps(F_eps)] \resultats;
    \addlegendentry{Optimal energy};
    \draw [pente]  (axis cs: 0.003,84e-4) -- ++ (axis cs: 1, 10^0.5) -- ++ (axis cs: 10, 1) -- cycle;
    \node at (axis cs:0.003,20e-3) [right,pente] {\small slope $0.5$};
    
  \end{loglogaxis}
  
\end{tikzpicture}
\caption{Convergence properties of the method for controllability of the fractional heat equation. }\label{figure_case1}
\end{figure}
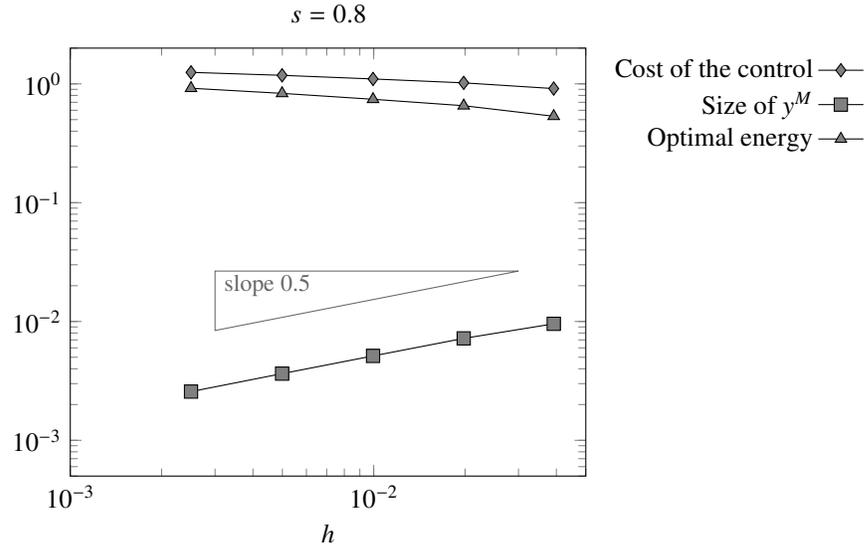


According to the discussion in Section \ref{theor_sec}, one can prove that null controllability does not hold for system \eqref{heat_frac} in the case $s\leq 1/2$. However approximate controllability can be proved by means of the unique continuation property of the operator $\fl{s}{}$. We would like to illustrate this property in Fig. \ref{figure_case3}.


We observe that the results are different from what we obtained in Fig.  \ref{figure_case1}. In fact, the cost of the control and the optimal energy increase in both cases, while the target $y^M$ tends to zero with a slower rate than $h^{1/2}$. This seems to confirm that a uniform observability estimate for \eqref{heat_frac} does not hold and that we can only expect to have approximate controllability (see Theorem \ref{theorem_hum}).
 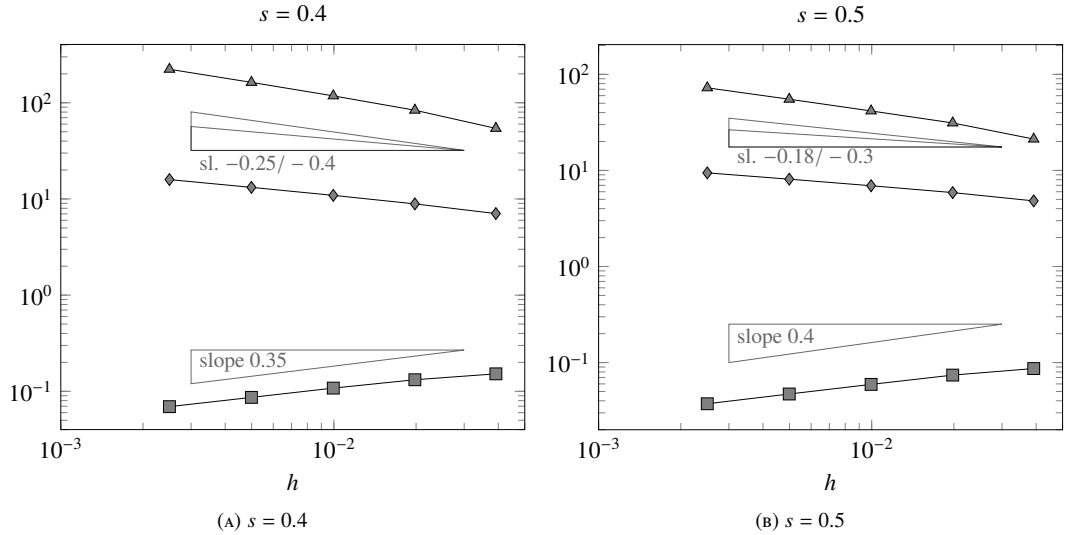
\begin{figure}
 \centering
 \subfloat[$s=0.4$]{
 \begin{tikzpicture}[scale=0.9]
  \begin{loglogaxis}[erreurs, ymin=4e-2,ymax=40.4e1,title={$s=0.4$}]
 
    \pgfplotstableread[ignore chars={|},skip first n=2]{heat_frac_s=04.org}\resultats

    \addplot[cout] table[x=dx,y=Nv] \resultats;
    \addplot[cible] table[x=dx,y=NyT] \resultats;
    \addplot[energie] table[x=dx,y=Inf_eps(F_eps)] \resultats;
    \draw [pente]  (axis cs: 0.003,1.2e-1) -- ++ (axis cs: 1, 10^0.35) -- ++ (axis cs: 10, 1) -- cycle;
    \node at (axis cs:0.003,20e-2) [right,pente] {\small slope $0.35$};
    
    \draw [pente]  (axis cs: 0.003,80.4e0) -- ++ (axis cs: 1, {10^(-0.4)}) -- ++ (axis cs: 10, 1) -- cycle;
    \draw [pente]  (axis cs: 0.003,56.7e0) -- ++ (axis cs: 1, {10^(-0.25)}) -- ++ (axis cs: 10, 1) -- cycle;
    \node at (axis cs:0.003,23e0) [right,pente] {\small sl. $-0.25/-0.4$};
    
  \end{loglogaxis}
  \end{tikzpicture}
  }
 \subfloat[$s=0.5$]{
  \begin{tikzpicture}[scale=0.9]
  \begin{loglogaxis}[erreurs, ymin=2e-2,ymax=20.4e1,title={$s=0.5$}]
    \pgfplotstableread[ignore chars={|},skip first n=2]{heat_frac_s=05.org}\resultats

    \addplot[cout] table[x=dx,y=Nv] \resultats;
    \addplot[cible] table[x=dx,y=NyT] \resultats;
    \addplot[energie] table[x=dx,y=Inf_eps(F_eps)] \resultats;
    \draw [pente]  (axis cs: 0.003,10e-2) -- ++ (axis cs: 1, 10^0.4) -- ++ (axis cs: 10, 1) -- cycle;
    \node at (axis cs:0.003,18e-2) [right,pente] {\small slope $0.4$};
    
    \draw [pente]  (axis cs: 0.003,2.65e1) -- ++ (axis cs: 1, {10^(-0.18)}) -- ++ (axis cs: 10, 1) -- cycle;
    \draw [pente]  (axis cs: 0.003,3.5e1) -- ++ (axis cs: 1, {10^(-0.3)}) -- ++ (axis cs: 10, 1) -- cycle;
    \node at (axis cs:0.003,1.35e1) [right,pente] {\small sl. $-0.18/-0.3$};
    
  \end{loglogaxis}
\end{tikzpicture}
}
\caption{Convergence properties of the method for $s<1/2$. Same legend as in Fig.  \ref{figure_case1}}\label{figure_case3}
\end{figure}

{\appendix

\section{Explicit computations of the elements of the matrix $\mathcal A_h$}\label{appendix}

We present here the explicit computations for each element $a_{i,j}$ of the stiffness matrix, completing the discussion that we started in Section \ref{fe_sec}.

\subsubsection*{Step 1: $j\geq i+2$}
We recall that, in this case, the value of $a_{i,j}$ is given by the integral
\begin{align}\label{elem_noint_app}
	a_{i,j}=-2 \int_{x_{j-1}}^{x_{j+1}}\int_{x_{i-1}}^{x_{i+1}}\frac{\phi_i(x)\phi_j(y)}{|x-y|^{1+2s}}\,dxdy.
\end{align}

In Fig. \ref{upp_tri}, we give a scheme of the region of interaction (marked in grey) between the basis functions in this case. 
\begin{figure}[h]
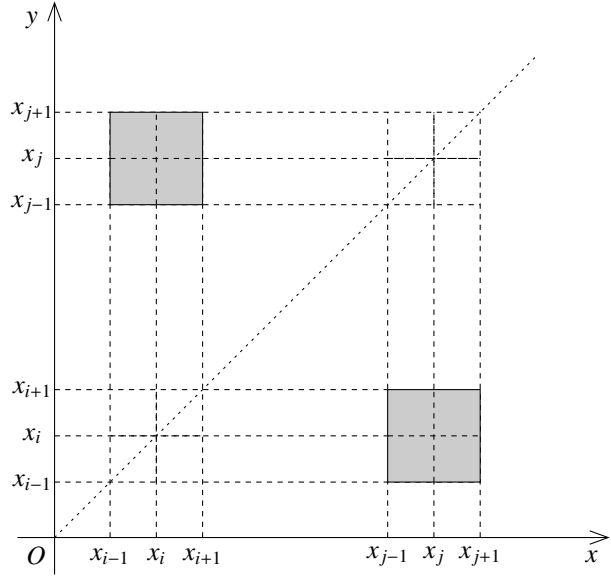

\figinit{0.7pt}
\figpt 0:(-90,-10)
\figpt 1:(-100,0) \figpt 2:(220,0)
\figpt 3:(-80,-20) \figpt 4:(-80,290)

\figpt 5:(-50,0) \figpt 6:(-50,230)
\figpt 7:(-25,0) \figpt 8:(-25,230)
\figpt 9:(0,0) \figpt 10:(0,230)
\figpt 11:(-80,30) \figpt 12:(150,30)
\figpt 13:(-80,55) \figpt 14:(150,55)
\figpt 15:(-80,80) \figpt 16:(150,80)
\figpt 17:(100,0) \figpt 18:(100,230)
\figpt 19:(125,0) \figpt 20:(125,230)
\figpt 21:(150,0) \figpt 22:(150,230)
\figpt 23:(-80,180) \figpt 24:(150,180)
\figpt 25:(-80,205) \figpt 26:(150,205)
\figpt 27:(-80,230) \figpt 28:(150,230)

\figpt 29:(-50,30) \figpt 30:(0,30)
\figpt 31:(-50,80) \figpt 32:(0,80)

\figpt 33:(-50,180) \figpt 34:(0,180)
\figpt 35:(-50,230) \figpt 36:(0,230)

\figpt 37:(100,180) \figpt 38:(150,180)
\figpt 39:(100,230) \figpt 40:(150,230)

\figpt 41:(100,30) \figpt 42:(150,30)
\figpt 43:(100,80) \figpt 44:(150,80)

\figpt 45:(-50,55) \figpt 46:(0,55)
\figpt 47:(-25,30) \figpt 48:(-25,80)

\figpt 49:(-50,205) \figpt 50:(0,205)
\figpt 51:(-25,180) \figpt 52:(-25,230)

\figpt 53:(100,205) \figpt 54:(150,205)
\figpt 55:(125,180) \figpt 56:(125,230)

\figpt 57:(100,55) \figpt 58:(150,55)
\figpt 59:(125,30) \figpt 60:(125,80)

\figpt 61:(-50,-10) \figpt 62:(-25,-10) \figpt 63:(0,-10) 
\figpt 64:(-92,30) \figpt 65:(-92,55) \figpt 66:(-92,80) 
\figpt 67:(100,-10) \figpt 68:(125,-10) \figpt 69:(150,-10) 
\figpt 70:(-92,180) \figpt 71:(-92,205) \figpt 72:(-92,230) 

\figpt 73:(-92,280) \figpt 74:(210,-10)


\figpt 75:(-80,0) \figpt 76:(180,260)

\figdrawbegin{}
\figdrawarrow[1,2]
\figdrawarrow[3,4]
\figset(dash=4)
\figdrawline[5,6]
\figdrawline[7,8]
\figdrawline[9,10]
\figdrawline[11,12]
\figdrawline[13,14]
\figdrawline[15,16]
\figdrawline[17,18]
\figdrawline[19,20]
\figdrawline[21,22]
\figdrawline[23,24]
\figdrawline[25,26]
\figdrawline[27,28]
\figset(dash=default)

\figset(fillmode=yes, color=0.8)
\figdrawline[33,34,36,35,33]
\figdrawline[41,42,44,43,41]
\figset (fillmode=no, color=default)
\figdrawline[33,34,36,35,33]
\figdrawline[41,42,44,43,41]

\figset(dash=4)
\figdrawline[45,46]
\figdrawline[47,48]
\figdrawline[49,50]
\figdrawline[51,52]
\figdrawline[53,54]
\figdrawline[55,56]
\figdrawline[57,58]
\figdrawline[59,60]

\figset(dash=5)
\figdrawline[75,76]

\figdrawend

\figvisu{\figBoxA}{}{
\figwritec [0]{$O$}
\figwritec [61,64]{$x_{i-1}$}
\figwritec [62,65]{$x_i$}
\figwritec [63,66]{$x_{i+1}$}
\figwritec [67,70]{$x_{j-1}$}
\figwritec [68,71]{$x_j$}
\figwritec [69,72]{$x_{j+1}$}
\figwritec [73]{$y$}
\figwritec [74]{$x$}
}
\centerline{\box\figBoxA}
\caption{Interactions between the basis function $\phi_i$ and $\phi_j$ when $j\geq i+2$.}\label{upp_tri}
\end{figure}


Now, taking into account the definition of the basis function \eqref{basis_fun}, the integral \eqref{elem_noint_app} becomes
\begin{align*}
	a_{i,j}=-2 \int_{x_{j-1}}^{x_{j+1}}\int_{x_{i-1}}^{x_{i+1}}\frac{\left(1-\frac{|x-x_i|}{h}\right)\left(1-\frac{|y-x_j|}{h}\right)}{|x-y|^{1+2s}}\,dxdy.
\end{align*}

Let us introduce the following change of variables:
\begin{align*}
	\frac{x-x_i}{h}=\hat{x},\;\;\; \frac{y-x_i}{h}=\hat{y}.
\end{align*}

Then, rewriting (with some abuse of notations since there is no possibility of confusion) $\hat{x}=x$ and $\hat{y}=y$, we get 
\begin{align}\label{elem_noint_cv}
	a_{i,j}=-2h^{1-2s} \int_{-1}^1\int_{-1}^1\frac{(1-|x\,|\,)(1-|y\,|\,)}{|x-y+i-j\,|^{1+2s}}\,dxdy.
\end{align}

The integral \eqref{elem_noint_cv} can be computed explicitly in the following way. First of all, for simplifying the notation, let us define $k=j-i$. We have 
\begin{align*}
	a_{i,j} = &\, -2h^{1-2s} \int_{-1}^1\int_{-1}^1\frac{(1-|x\,|\,)(1-|y\,|\,)}{|x-y+i-j\,|^{1+2s}}\,dxdy =-2h^{1-2s} \int_{-1}^1\int_{-1}^1\frac{(1-|x\,|\,)(1-|y\,|\,)}{|x-y-k\,|^{1+2s}}\,dxdy
	\\
	= &\, -2h^{1-2s} \int_0^1\int_0^1\frac{(1-x)(1-y)}{(y-x+k)^{1+2s}}\,dxdy - 2h^{1-2s} \int_0^1\int_{-1}^0\frac{(1+x)(1-y)}{(y-x+k)^{1+2s}}\,dxdy 
	\\
	&- 2h^{1-2s} \int_{-1}^0\int_0^1\frac{(1-x)(1+y)}{(y-x+k)^{1+2s}}\,dxdy - 2h^{1-2s} \int_{-1}^0\int_{-1}^0\frac{(1+x)(1+y)}{(y-x+k)^{1+2s}}\,dxdy
	\\
	= &\, - 2h^{1-2s}(B_1 + B_2 + B_3 + B_4).
\end{align*}

These terms $B_i$, $i=1,2,3,4$, can be computed integrating by parts several times. In more detail, we have
\begin{align*}
	& B_1 = \frac{1}{4s(1-2s)}\left[2k^{1-2s}-\frac{(k+1)^{2-2s}-(k-1)^{2-2s}}{1-s}-\frac{2k^{3-2s}-(k+1)^{3-2s}-(k-1)^{3-2s}}{(1-s)(3-2s)}\right]
	\\
	& B_2 = \frac{1}{4s(1-2s)}\left[-2k^{1-2s}+\frac{2(k+1)^{2-2s}-2k^{2-2s}}{1-s}+\frac{2(k+1)^{3-2s}-k^{3-2s}-(k+2)^{3-2s}}{(1-s)(3-2s)}\right]
	\\
	& B_3 = \frac{1}{4s(1-2s)}\left[-2k^{1-2s}+\frac{2k^{2-2s}-2(k-1)^{2-2s}}{1-s}+\frac{2(k-1)^{3-2s}-k^{3-2s}-(k-2)^{3-2s}}{(1-s)(3-2s)}\right]
	\\
	& B_4 = \frac{1}{4s(1-2s)}\left[2k^{1-2s}-\frac{(k+1)^{2-2s}-(k-1)^{2-2s}}{1-s}-\frac{2k^{3-2s}-(k+1)^{3-2s}-(k-1)^{3-2s}}{(1-s)(3-2s)}\right].
\end{align*} 

Therefore, we obtain
\begin{align*}
	a_{i,j} = - h^{1-2s}\,\frac{4(k+1)^{3-2s} + 4(k-1)^{3-2s}-6k^{3-2s}-(k+2)^{3-2s}-(k-2)^{3-2s}}{2s(1-2s)(1-s)(3-2s)}.
\end{align*} 

We notice that, when $s=1/2$, both the numerator and the denominator of the expression above are zero. Hence, in this particular case, it would not be possible to introduce the value that we just encountered in our code. Nevertheless, this difficulty can be overcome noting that we can easily compute
\begin{align*}
	\lim_{s\to\frac{1}{2}} &- h^{1-2s}\,\frac{4(k+1)^{3-2s} + 4(k-1)^{3-2s}-6k^{3-2s}-(k+2)^{3-2s}-(k-2)^{3-2s}}{2s(1-2s)(1-s)(3-2s)}
	\\
	& = -4(k+1)^2\log(k+1)-4(k-1)^2\log(k-1)+6k^2\log(k)+(k+2)^2\log(k+2)+(k-2)^2\log(k-2),
\end{align*} 
if $k\neq 2$. When $k=2$, instead, since 
\begin{align*}
	\lim_{k\to 2} (k-2)^2\log(k-2) =0,
\end{align*}
the corresponding value $a_{i,j}=a_{i,i+2}$ if given by 
\begin{align*}
	a_{i,i+2} = 56\ln(2)-36\ln(3).
\end{align*}

\subsubsection*{Step 2: $j= i+1$}
This is the most cumbersome case, since it is the one with the most interactions between the basis functions (see Fig. \ref{basis_upp_dia}). According to \eqref{stiffness_nc}, and using the symmetry of the integral with respect to the bisector $y=x$, we have 
	\begin{align*}
	a_{i,i+1}= & \int_{\RR}\int_{\RR}\frac{(\phi_i(x)-\phi_i(y))(\phi_{i+1}(x)-\phi_{i+1}(y))}{|x-y|^{1+2s}}\,dxdy
	\\
	= & \int_{x_{i+1}}^{+\infty}\int_{x_{i+1}}^{+\infty} \ldots\,dxdy + 2\int_{x_{i+1}}^{+\infty}\int_{x_i}^{x_{i+1}} \ldots\,dxdy + 2\int_{x_{i+1}}^{+\infty}\int_{-\infty}^{x_i} \ldots\,dxdy 
	\\
	& + \int_{x_i}^{x_{i+1}}\int_{x_i}^{x_{i+1}} \ldots\,dxdy + 2\int_{x_i}^{x_{i+1}}\int_{-\infty}^{x_i} \ldots\,dxdy + \int_{-\infty}^{x_i}\int_{-\infty}^{x_i} \ldots\,dxdy 
	\\
	:= & Q_1 + Q_2 + Q_3 + Q_4 + Q_5 + Q_6.
\end{align*}

In Fig. \ref{upp_dia}, we give a scheme of the regions of interactions between the basis functions $\phi_i$ and $\phi_{i+1}$ enlightening the domain of integration of the $Q_i$. The regions in grey are the ones that produce a contribution to $a_{i,i+1}$, while on the regions in white the integrals will be zero.
\begin{figure}[h]
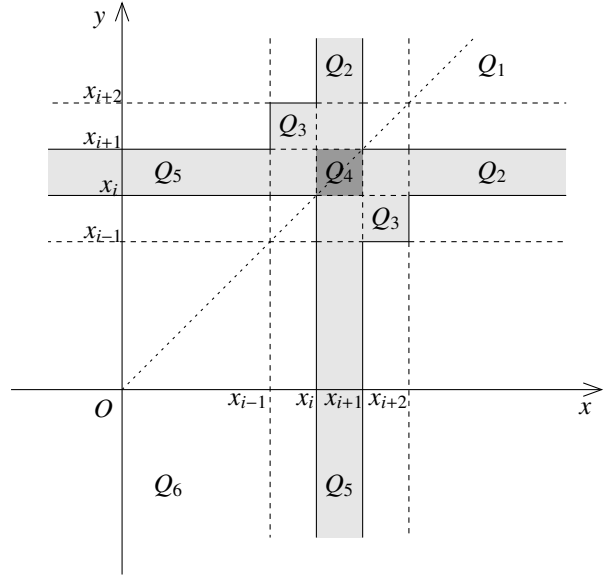

\figinit{0.7pt}
\figpt 0:(-50,70)
\figpt 1:(-100,80) \figpt 2:(220,80)
\figpt 3:(-40,-20) \figpt 4:(-40,290)

\figpt 5:(40,0) \figpt 6:(40,270)
\figpt 7:(65,0) \figpt 8:(65,270)
\figpt 9:(90,0) \figpt 10:(90,270)
\figpt 11:(-80,160) \figpt 12:(200,160)
\figpt 13:(-80,185) \figpt 14:(200,185)
\figpt 15:(-80,210) \figpt 16:(200,210)
%

\figpt 29:(65,210) \figpt 30:(90,210)
\figpt 31:(65,235) \figpt 32:(40,235) 
\figpt 33:(40,210) \figpt 34:(90,185)
\figpt 35:(115,185) \figpt 36:(115,160)
\figpt 37:(90,160) \figpt 38:(65,185)
\figpt 39:(40,185) \figpt 40:(65,160) 
\figpt 41:(115,0) \figpt 42:(115,270) 
\figpt 43:(115,210) \figpt 44:(-80,235) 
\figpt 45:(90,235) \figpt 46:(200,235) 

\figpt 47:(-40,80) \figpt 48:(150,270)

\figpt 49:(-50,164) \figpt 50:(-47,189) 
\figpt 51:(-50,215) \figpt 52:(-50,240) 

\figpt 53:(28,74) \figpt 54:(59,74) 
\figpt 55:(80,74) \figpt 56:(104,74) 

\figpt 73:(-52,280) \figpt 74:(210,70)

\figpt 75:(160,255) \figpt 76:(77.5,255)
\figpt 77:(160,197.5) \figpt 78:(52.5,222.5)
\figpt 79:(102.5,172.5) \figpt 80:(77.5,197.5)
\figpt 81:(77.5,28) \figpt 82:(-15,197.5)
\figpt 83:(-15,28)

\figdrawbegin{}

\figset(fillmode=yes, color=0.9)
\figdrawline[29,30,10,8,29]
\figdrawline[29,31,32,33,29]
\figdrawline[30,34,14,16,30]
\figdrawline[35,36,37,34,35]
\figdrawline[34,38,7,9,34]
\figdrawline[38,29,15,13,38]
\figset(fillmode=yes, color=0.6)
\figdrawline[29,30,34,38,29]
\figset(fillmode=no, color=0)
\figdrawarrow[1,2]
\figdrawarrow[3,4]
\figdrawline[8,31]
\figdrawline[10,30]
\figdrawline[30,16]
\figdrawline[35,14]
\figdrawline[35,36]
\figdrawline[36,37]
\figdrawline[37,9]
\figdrawline[38,7]
\figdrawline[38,13]
\figdrawline[33,15]
\figdrawline[32,33]
\figdrawline[32,31]
\figset(dash=4)
\figdrawline[6,32]
\figdrawline[39,5]
\figdrawline[40,11]
\figdrawline[36,12]
\figdrawline[36,41]
\figdrawline[42,43]
\figdrawline[32,44]
\figdrawline[45,46]

\figdrawline[33,30]
\figdrawline[31,45]
\figdrawline[38,35]
\figdrawline[40,37]
\figdrawline[33,39]
\figdrawline[31,38]
\figdrawline[30,37]
\figdrawline[43,35]
\figset(dash=5)
\figdrawline[47,48]
\figdrawend

\figvisu{\figBoxA}{}{
\figwritec [0]{$O$}
\figwritec [49,53]{$x_{i-1}$}
\figwritec [50,54]{$x_i$}
\figwritec [51,55]{$x_{i+1}$}
\figwritec [52,56]{$x_{i+2}$}
\figwritec [73]{$y$}
\figwritec [74]{$x$}
\figwritec [75]{$Q_1$}
\figwritec [76,77]{$Q_2$}
\figwritec [78,79]{$Q_3$}
\figwritec [80]{$Q_4$}
\figwritec [81,82]{$Q_5$}
\figwritec [83]{$Q_6$}
}
\centerline{\box\figBoxA}
\caption{Interactions between the basis function $\phi_i$ and $\phi_{i+1}$.}\label{upp_dia}
\end{figure}


Le us now compute the terms $Q_i$, $i=1,\ldots,6$, separately. 
\subsubsection*{Computation of $Q_1$}
Since $\phi_i = 0$ on the domain of integration we have
\begin{align*}
	Q_1 &= \int_{x_{i+1}}^{+\infty}\int_{x_{i+1}}^{+\infty} \frac{\phi_{i+1}(x)-\phi_{i+1}(y)}{|x-y|^{1+2s}}\,dxdy 
	\\
	&= \int_{x_{i+1}}^{+\infty}\int_{x_{i+1}}^{+\infty} \frac{\phi_{i+1}(x)}{|x-y|^{1+2s}}\,dxdy - \int_{x_{i+1}}^{+\infty}\int_{x_{i+1}}^{+\infty} \frac{\phi_{i+1}(y)}{|x-y|^{1+2s}}\,dxdy = 0.
\end{align*}

\subsubsection*{Computation of $Q_2$}
We have
\begin{align*}
	Q_2 &= 2\int_{x_{i+1}}^{+\infty}\int_{x_i}^{x_{i+1}} \frac{\phi_i(x)(\phi_{i+1}(x)-\phi_{i+1}(y))}{|x-y|^{1+2s}}\,dxdy. 
\end{align*}

Now, using Fubini's theorem we can exchange the order of the integrals, obtaining 
\begin{align*}
	Q_2 &= 2\int_{x_i}^{x_{i+1}}\phi_i(x)\phi_{i+1}(x)\left(\int_{x_{i+1}}^{+\infty} \frac{dy}{|x-y|^{1+2s}}\right)\,dx - 2\int_{x_{i+1}}^{x_{i+2}}\int_{x_i}^{x_{i+1}} \frac{\phi_i(x)\phi_{i+1}(y)}{|x-y|^{1+2s}}\,dxdy 
	\\
	&= \frac{1}{s}\int_{x_i}^{x_{i+1}}\frac{\phi_i(x)\phi_{i+1}(x)}{(x_{i+1}-x)^{2s}}\,dx - 2\int_{x_{i+1}}^{x_{i+2}}\int_{x_i}^{x_{i+1}} \frac{\phi_i(x)\phi_{i+1}(y)}{|x-y|^{1+2s}}\,dxdy
	\\
	&= \frac{1}{s}\int_{x_i}^{x_{i+1}}\frac{\left(1-\frac{|x-x_i|}{h}\right)\left(1-\frac{|x-x_{i+1}|}{h}\right)}{(x_{i+1}-x)^{2s}}\,dx - 2\int_{x_{i+1}}^{x_{i+2}}\int_{x_i}^{x_{i+1}} \frac{\left(1-\frac{|x-x_i|}{h}\right)\left(1-\frac{|y-x_{i+1}|}{h}\right)}{|x-y|^{1+2s}}\,dxdy:= Q_2^1 + Q_2^2.
\end{align*}

The two integrals above can be computed explicitly. Indeed, employing the change of variables
\begin{align*}
	\frac{x_{i+1}-x}{h}=\hat{x},
\end{align*}
and then renaming $\hat{x}=x$, $R_2^1$ becomes
\begin{align*}
	Q_2^1=\frac{h^{1-2s}}{s}\int_0^1 x^{1-2s}(1-x)\,dx = \frac{h^{1-2s}}{s(2-2s)(3-2s)}.
\end{align*}

For computing $Q_2^2$, instead, we introduce the change of variables
\begin{align}\label{cv3}
	\frac{x_i-x}{h}=\hat{x},\;\;\;\frac{y-x_{i+1}}{h}=\hat{y},
\end{align}
and we obtain
\begin{align*}
	Q_2^2 = -2h^{1-2s}\int_0^1\int_0^1\frac{(1-x)(1-y)}{(y-x+1)^{1+2s}}\,dxdy = h^{1-2s}\frac{2^{1-2s}+s-2}{s(1-s)(3-2s)}.
\end{align*}

Adding the two contributions, we get the following expression for the term $R_2$
\begin{align*}
	Q_2 = h^{1-2s}\frac{2^{2-2s}+2s-3}{s(2-2s)(3-2s)}.
\end{align*}

\subsubsection*{Computation of $Q_3$}
In this case, we simply take into account the intervals in which the basis functions are supported, so that we obtain
\begin{align*}
	Q_3 &= -2\int_{x_{i+1}}^{x_{i+2}}\int_{x_{i-1}}^{x_i} \frac{\phi_i(x)\phi_{i+1}(y)}{|x-y|^{1+2s}}\,dxdy = - 2\int_{x_{i+1}}^{x_{i+2}}\int_{x_{i-1}}^{x_i} \frac{\left(1-\frac{|x-x_i|}{h}\right)\left(1-\frac{|y-x_{i+1}|}{h}\right)}{|x-y|^{1+2s}}\,dxdy.
\end{align*}

This integral can be computed applying again \eqref{cv3}, and we get
\begin{align}\label{Q3}
	Q_3 = -2h^{1-2s}\int_0^1\int_{-1}^0 \frac{(1+x)(1-y)}{(y-x+1)^{1+2s}}\,dxdy = h^{1-2s}\frac{13-5\cdot 2^{3-2s}+3^{3-2s}+s(2^{4-2s}-14)+4s^2}{2s(1-2s)(1-s)(3-2s)},
\end{align}
if $s\neq 1/2$. If $s=1/2$, instead, we have 
\begin{align*}
	Q_3 &= -2\int_0^1\int_{-1}^0 \frac{(1+x)(1-y)}{(y-x+1)^2}\,dxdy = 1+9\ln 3-16\ln(2).
\end{align*}

Notice that this last value could have been computed directly from \eqref{Q3}, by taking the limit as $s\to 1/2$ in that expression, being this limit exactly $1+9\ln 3-16\ln(2)$.

\subsubsection*{Computation of $Q_4$}
In this case, we are in the intersection of the supports of $\phi_i$ and $\phi_{i+1}$. Therefore, we have
\begin{align*}
	Q_4 &= \int_{x_i}^{x_{i+1}}\int_{x_i}^{x_{i+1}} \frac{(\phi_i(x)-\phi_i(y))(\phi_{i+1}(x)-\phi_{i+1}(y))}{|x-y|^{1+2s}}\,dxdy. 
\end{align*}

Moreover, we notice that, this time, it is possible that $x=y$, meaning that $Q_4$ could be a singular integral. To deal with this difficulty, we will exploit the explicit definition of the basis function. We have (see also Fig. \ref{basis2})
\begin{align*}
	\phi_i(x) = 1-\frac{x-x_i}{h}, \;\;\; x\in (x_i,x_{i+1}),
	\\
	\phi_{i+1}(x) = \frac{x_{i+1}-x}{h}, \;\;\; x\in (x_i,x_{i+1}).
\end{align*}
\begin{figure}[h]
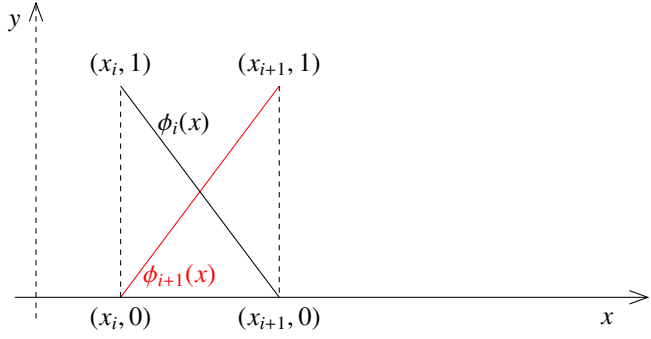

\figinit{0.8pt}
\figpt 1:(-100,0) \figpt 2:(200,0)
\figpt 11:(-90,-10) \figpt 12:(-90,140)
\figpt 3:(-50,0) \figpt 4:(25,100) 
\figpt 5:(-50,100) \figpt 6:(25,0)
%
\figpt 7:(-50,-10) \figpt 8:(25,110) 
\figpt 9:(-50,110) \figpt 10:(25,-10)

\figpt 13:(-100,130) \figpt 14:(180,-10)
\figpt 15:(-20,82) \figpt 16:(-22,10)

\figdrawbegin{}
\figdrawarrow[1,2]
\figset (color=\Redrgb)
\figdrawline[3,4]
\figset (color=default)
\figdrawline[5,6]
\figset(dash=4)
\figdrawline[4,6]
\figdrawarrow[11,12]
\figdrawline[3,5]
\figdrawend

\figvisu{\figBoxA}{}{
\figwritec [7]{$(x_i,0)$}
\figwritec [8]{$(x_{i+1},1)$}
\figwritec [9]{$(x_i,1)$}
\figwritec [10]{$(x_{i+1},0)$}
\figwritec [13]{$y$}
\figwritec [14]{$x$}
\figwritec [15]{$\phi_i(x)$}
\figwritec [16]{$\color{red}\phi_{i+1}(x)$}
}
\centerline{\box\figBoxA}
\caption{Functions $\phi_i(x)$ and $\phi_{i+1}(x)$ on the interval $(x_i,x_{i+1})$.}\label{basis2}
\end{figure}

$\newline$
Therefore, 
\begin{align*}
	(\phi_i(x)-\phi_i(y))(\phi_{i+1}(x)-\phi_{i+1}(y)) = \left(\frac{y-x}{h}\right)\left(\frac{x-y}{h}\right) = -\frac{|x-y|^2}{h^2},
\end{align*}
and the integral becomes
\begin{align*}
	Q_4 &= -\int_{x_i}^{x_{i+1}}\int_{x_i}^{x_{i+1}} |x-y|^{1-2s}\,dxdy = -\frac{h^{1-2s}}{(1-s)(3-2s)}. 
\end{align*}

\subsubsection*{Computation of $Q_5$}
Here the procedure is analogous to the one for $Q_2$ before. Using again Fubini's theorem we have
\begin{align*}
	Q_5 &= 2\int_{x_i}^{x_{i+1}}\phi_i(y)\phi_{i+1}(y)\left(\int_{-\infty}^{x_i} \frac{dx}{|x-y|^{1+2s}}\right)dy - 2\int_{x_i}^{x_{i+1}}\int_{x_{i-1}}^{x_i} \frac{\phi_i(x)\phi_{i+1}(y)}{|x-y|^{1+2s}}\,dxdy 
	\\
	&= \frac{1}{s}\int_{x_i}^{x_{i+1}}\frac{\phi_i(y)\phi_{i+1}(y)}{(y-x_i)^{2s}}\,dy - 2\int_{x_i}^{x_{i+1}}\int_{x_{i-1}}^{x_i} \frac{\phi_i(x)\phi_{i+1}(y)}{|x-y|^{1+2s}}\,dxdy. 
\end{align*}
Applying again \eqref{cv3}, it is now easy to check that $Q_5=Q_2$.

\subsubsection*{Computation of $Q_6$}
In analogy with what we did for $Q_1$, we can show that also $Q_6=0$.

\subsubsection*{Conclusion}
The elements $a_{i,i+1}$ are now given by the sum $2Q_2+Q_3+Q_4$, according to the corresponding values that we computed. In particular, we have
\begin{align*}
	a_{i,i+1} = \begin{cases}
					\displaystyle h^{1-2s}\frac{3^{3-2s}-2^{5-2s}+7}{2s(1-2s)(1-s)(3-2s)}, & \displaystyle s\neq \frac{1}{2}
					\\
					9\ln 3-16\ln 2, & \displaystyle s=\frac{1}{2}.
				\end{cases}	
\end{align*}

\subsubsection*{Step 3: $j= i$}
As a last step, we fill the diagonal of the matrix $\mathcal A_h$. In this case we have
	\begin{align*}
	a_{i,i}= & \int_{\RR}\int_{\RR}\frac{(\phi_i(x)-\phi_i(y))^2}{|x-y|^{1+2s}}\,dxdy
	\\
	= & \int_{x_{i+1}}^{+\infty}\int_{x_{i+1}}^{+\infty} \ldots\,dxdy + 2\int_{x_{i+1}}^{+\infty}\int_{x_{i-1}}^{x_{i+1}} \ldots\,dxdy + \int_{x_{i+1}}^{+\infty}\int_{-\infty}^{x_{i-1}} \ldots\,dxdy 
	\\
	& + \int_{x_{i-1}}^{x_{i+1}}\int_{x_{i-1}}^{x_{i+1}} \ldots\,dxdy + 2\int_{-\infty}^{x_{i-1}}\int_{x_{i-1}}^{x_{i+1}} \ldots\,dxdy + + \int_{-\infty}^{x_{i-1}}\int_{x{i+1}}^{+\infty} \ldots\,dxdy 
	\\
	& +  \int_{-\infty}^{x_{i-1}}\int_{-\infty}^{x_{i-1}} \ldots\,dxdy := R_1 + R_2 + R_3 + R_4 + R_5 + R_6 + R_7.
\end{align*}

In Fig. \ref{upp_dia}, we give a scheme of the regions of interactions between the basis functions $\phi_i(x)$ and $\phi_i(y)$ enlightening the domain of integration of the $R_i$. The regions in grey are the ones that produce a contribution to $a_{i,i}$, while on the regions in white the integrals will be zero.
\begin{figure}
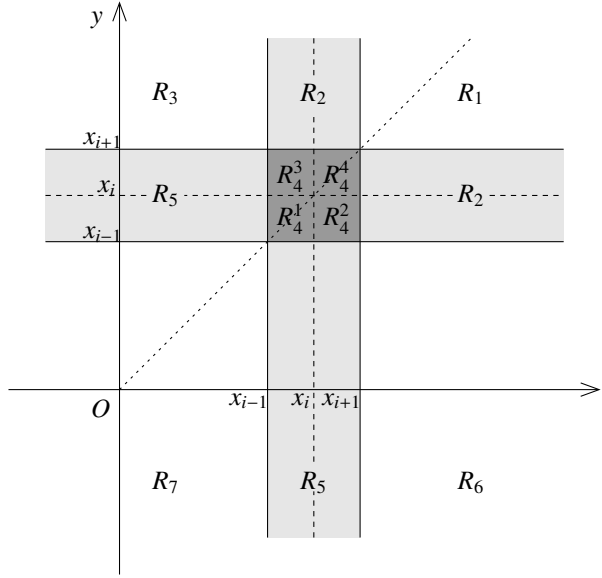

\figinit{0.7pt}
\figpt 0:(-50,70)
\figpt 1:(-100,80) \figpt 2:(220,80)
\figpt 3:(-40,-20) \figpt 4:(-40,290)

\figpt 5:(40,0) \figpt 6:(40,270)
\figpt 7:(65,0) \figpt 8:(65,270)
\figpt 9:(90,0) \figpt 10:(90,270)
\figpt 11:(-80,160) \figpt 12:(200,160)
\figpt 13:(-80,185) \figpt 14:(200,185)
\figpt 15:(-80,210) \figpt 16:(200,210)

\figpt 53:(65,250) \figpt 54:(65,250)
\figpt 55:(65,230) \figpt 56:(65,25)
\figpt 57:(65,41) \figpt 58:(200,210)

\figpt 59:(-25,185) \figpt 60:(140,185)
\figpt 61:(-5,185) \figpt 62:(160,185)
\figpt 63:(65,230) \figpt 64:(65,25)

%

\figpt 29:(40,210) \figpt 30:(90,210)
\figpt 31:(40,160) \figpt 32:(90,160)

\figpt 33:(-40,80) \figpt 34:(150,270)

\figpt 35:(-49,164) \figpt 36:(-47,189) \figpt 37:(-49,215) 
\figpt 38:(30,74) \figpt 39:(58,74) \figpt 40:(80,74) 

\figpt 41:(150,240) \figpt 42:(65,240) \figpt 43:(150,185) 
\figpt 44:(-15,240) \figpt 45:(52.5,172.5) \figpt 46:(77.5,172.5) 
\figpt 47:(52.5,195.5) \figpt 48:(77.5,195.5) \figpt 49:(-15,185) 
\figpt 50:(65,30) \figpt 51:(150,30) \figpt 52:(-15,30) 

\figpt 73:(-52,280) \figpt 74:(210,70)

\figdrawbegin{}

\figset(fillmode=yes, color=0.9)
\figdrawline[29,30,10,6,29]
\figdrawline[30,32,12,16,30]
\figdrawline[32,31,5,9,32]
\figdrawline[31,11,15,29,31]
\figset(fillmode=yes, color=0.6)
\figdrawline[29,30,32,31,29]
\figset(fillmode=no, color=0)
\figdrawarrow[1,2]
\figdrawarrow[3,4]
\figdrawline[5,6]
\figdrawline[9,10]
\figdrawline[11,12]
\figdrawline[15,16]
\figset(dash=4)
\figdrawline[7,56]
\figdrawline[57,55]
\figdrawline[53,8]
\figdrawline[13,59]
\figdrawline[60,61]
\figdrawline[62,14]
\figset(dash=5)
\figdrawline[33,34]
\figdrawend

\figvisu{\figBoxA}{}{
\figwritec [0]{$O$}
\figwritec [35,38]{$x_{i-1}$}
\figwritec [36,39]{$x_i$}
\figwritec [37,40]{$x_{i+1}$}
\figwritec [73]{$y$}
\figwritec [41]{$R_1$}
\figwritec [42,43]{$R_2$}
\figwritec [44]{$R_3$}
\figwritec [45]{$R_4^1$}
\figwritec [46]{$R_4^2$}
\figwritec [47]{$R_4^3$}
\figwritec [48]{$R_4^4$}
\figwritec [49,50]{$R_5$}
\figwritec [51]{$R_6$}
\figwritec [52]{$R_7$}
}
\centerline{\box\figBoxA}
\caption{Interactions between the basis function $\phi_i(x)$ and $\phi_i(y)$.}\label{dia}
\end{figure}

Le us now compute the terms $R_i$, $i=1,\ldots,7$, separately. First of all, according to Fig. \ref{dia} we have that $R_1=R_3=R_6=R_7=0.$ This is due to the fact that the corresponding regions are all away from the support of the basis functions. 
\subsubsection*{Computation of $R_2$}
Since $\phi_i(y) = 0$ on the domain of integrations we have
\begin{align*}
	R_2 &= 2\int_{x_{i+1}}^{+\infty}\int_{x_{i-1}}^{x_{i+1}} \frac{\phi_i^2(x)}{|x-y|^{1+2s}}\,dxdy = 2\int_{x_{i-1}}^{x_{i+1}}\phi_i^2(x)\left(\int_{x_{i+1}}^{+\infty} \frac{dy}{|x-y|^{1+2s}}\right)\,dx = \frac{1}{s}\int_{x_{i-1}}^{x_{i+1}}\frac{\phi_i^2(x)}{(x_{i+1}-x)^{2s}}\,dxdy.
\end{align*}

This integral is computed employing \eqref{cv}, and we obtain
\begin{align*}
	R_2 = \frac{h^{1-2s}}{s}\int_{-1}^1 \frac{(1-|x|\,)^2}{(1-x)^{2s}}\,dx = h^{1-2s}\frac{4s-6+2^{3-2s}}{s(1-2s)(1-s)(3-2s)}, 
\end{align*}
if $s\neq 1/2$. If $s=1/2$, instead, we have
\begin{align*}
	R_2 = 2\int_{-1}^1 \frac{(1-|x|\,)^2}{1-x}\,dx = 2\ln 16-4.
\end{align*}

\subsubsection*{Computation of $R_4$}
In this case, we are in the intersection of the supports of $\phi_i(x)$ and $\phi_i(y)$. Therefore, we have
\begin{align*}
	R_4 &= \int_{x_{i-1}}^{x_{i+1}}\int_{x_{i-1}}^{x_{i+1}} \frac{(\phi_i(x)-\phi_i(y))^2}{|x-y|^{1+2s}}\,dxdy. 
\end{align*}

In order to compute this integral, we divide it in four components as follows:
\begin{align*}
	R_4 &= \int_{x_{i-1}}^{x_i}\int_{x_{i-1}}^{x_i} \ldots\,dxdy + \int_{x_{i-1}}^{x_i}\int_{x_i}^{x_{i+1}} \ldots\,dxdy +
	\int_{x_i}^{x_{i+1}}\int_{x_{i-1}}^{x_i} \ldots\,dxdy +  \int_{x_i}^{x_{i+1}}\int_{x_i}^{x_{i+1}} \ldots\,dxdy
	\\
	&= R_4^1 + R_4^2 + R_4^3 + R_4^4.
\end{align*}

Moreover, we notice that, due to symmetry reason, we have $R_4^2 = R_4^3$. Therefore, we can compute only one of this two terms and add its value twice when building the matrix $\mathcal A_h$. Also, notice that in these two region it cannot happen that $x=y$. On the other hand, $R_4^1$ and $R_4^4$ may be singular integrals, and we shall deal with them as we did before. 

\subsubsection*{Computation of $R_4^1$}
Using again the explicit expression of the basis functions we find 
\begin{align*}
	(\phi_i(x)-\phi_i(y))^2 = \frac{|x-y|^2}{h^2},
\end{align*}
and the integral becomes
\begin{align*}
	R_4^1 &= \int_{x_{i-1}}^{x_i}\int_{x_{i-1}}^{x_i} |x-y|^{1-2s}\,dxdy = \frac{h^{1-2s}}{(1-s)(3-2s)}. 
\end{align*}

\subsubsection*{Computation of $R_4^2$}
In this case, we simply have
\begin{align*}
	R_4^2 &= \int_{x_{i-1}}^{x_i}\int_{x_i}^{x_{i+1}} \frac{(\phi_i(x)-\phi_i(y))^2}{|x-y|^{1+2s}}\,dxdy.
\end{align*}
Employing \eqref{cv} we obtain
\begin{align*}
	R_4^2 = h^{1-2s}\int_{-1}^0\int_0^1 \frac{(x+y)^2}{(x-y)^{1+2s}}\,dxdy = h^{1-2s}\frac{2s^2-5s+4-2^{2-2s}}{s(1-2s)(1-s)(3-2s)}, 
\end{align*}
if $s\neq 1/2$. If $s=1/2$, instead, we get
\begin{align*}
	R_4^2 = \int_{-1}^0\int_0^1 \frac{(x+y)^2}{(x-y)^2}\,dxdy = 3-4\ln 2.
\end{align*}
\subsubsection*{Computation of $R_4^4$}
Also in this case we can use the explicit expression of the basis functions and the integral becomes
\begin{align*}
	R_4^4 &= \int_{x_i}^{x_{i+1}}\int_{x_i}^{x_{i+1}} |x-y|^{1-2s}\,dxdy = \frac{h^{1-2s}}{(1-s)(3-2s)}=R_4^1. 
\end{align*}
Adding the values that we just computed, we therefore obtain
\begin{align*}
	R_4 = 2(R_4^1+R_4^2) = \begin{cases}
					\displaystyle h^{1-2s}\frac{8-8s-2^{3-2s}}{2s(1-2s)(1-s)(3-2s)}, & \displaystyle s\neq \frac{1}{2}
					\\
					8\ln 3-8\ln 2, & \displaystyle s=\frac{1}{2}.
				\end{cases}	
\end{align*}

\subsubsection*{Computation of $R_5$}
Since, once again, $\phi_i(y) = 0$ on the domain of integration we have
\begin{align*}
	R_5 &= 2\int_{-\infty}^{x_{i-1}}\int_{x_{i-1}}^{x_{i+1}} \frac{\phi_i^2(x)}{|x-y|^{1+2s}}\,dxdy = 2\int_{x_{i-1}}^{x_{i+1}}\phi_i^2(x)\left(\int_{-\infty}^{x_{i-1}} \frac{dy}{|x-y|^{1+2s}}\right)\,dx = \frac{1}{s}\int_{x_{i-1}}^{x_{i+1}}\frac{\phi_i^2(x)}{(x-x_{i-1})^{2s}}\,dxdy.
\end{align*}
Employing one last time \eqref{cv}, we get
\begin{align*}
	R_5 = \frac{h^{1-2s}}{s}\int_{-1}^1 \frac{(1-|x|\,)^2}{(1+x)^{2s}}\,dx = h^{1-2s}\frac{4s-6+2^{3-2s}}{s(1-2s)(3-2s)(1-s)}=R_2,
\end{align*}
if $s\neq 1/2$. If $s=1/2$, instead, we have
\begin{align*}
	R_5 = 2\int_{-1}^1 \frac{(1-|x|\,)^2}{1+x}\,dx = 8\ln 2-4.
\end{align*}

\subsubsection*{Conclusion}
The elements $a_{i,i}$ are now given by the sum $2R_2+R_4$, according to the corresponding values that we computed. In particular, we have 
\begin{align*}
	a_{i,i} = \begin{cases}
			\displaystyle h^{1-2s}\,\frac{2^{3-2s}-4}{s(1-2s)(1-s)(3-2s)}, & \displaystyle s\neq\frac{1}{2}
			\\
			\\
			8\ln 2, & \displaystyle s=\frac{1}{2}.			
			\end{cases}	
\end{align*}

}

\section*{Acknowledgements}
The authors wish to acknowledge F. Boyer (Institut de Math\'ematiques de Toulouse) for his contribution in the development of the numerical implementation of the control problem. Moreover, a special thanks goes to J. Loh\'eac (Laboratoire de Sciences Num\'eriques de Nantes), for helping with some computations in Section \ref{fe_sec}.  

\bibliography{biblio}

\end{document}